\documentclass[leqno,11pt]{amsart}
\usepackage[applemac]{inputenc}
\usepackage{amssymb, amsmath}
\usepackage{amssymb,amsfonts}
\usepackage{amsmath,latexsym,stmaryrd}
\usepackage{pdfsync}
\usepackage{bbm}

\allowdisplaybreaks
\usepackage{
 ulem,
 amsfonts}
 \usepackage{amsmath}
\usepackage{amsmath,esint,bigints, amssymb, mathrsfs, amsthm,mathbbol,upgreek,exscale}

\DeclareMathAlphabet{\mathpzc}{OT1}{pzc}{m}{it}
 \usepackage{mathtools}
\usepackage{
 ulem,
 amsfonts}
 \usepackage{amsmath}
\usepackage{amsmath,esint, amssymb, amsthm,mathbbol,upgreek,exscale}

 \usepackage{mathtools}

\usepackage[colorlinks=true, pdfstartview=FitV, linkcolor=blue, citecolor=blue,
 urlcolor=blue]{hyperref}
\usepackage{color}

\date{}

\newcommand{\Rea}{\operatorname{Re}}

\newcommand\N{\mathbb{N}}

\newcommand\R{\mathbb{R}} 
\newcommand\C{\mathbb{C}} 
\newcommand\D{\mathbb{D}} 

\newcommand\E{\varepsilon}

\newcommand\T{\mathbb{T}}

\newcommand\EE{\varepsilon}

\newcommand\ttt{{\texttt{t}}}

\theoremstyle{plain}
\newtheorem{definition}{Definition}

\numberwithin{equation}{section}
\newtheorem{theorem}{Theorem}[section]
\newtheorem{proposition}[theorem]{Proposition}
\newtheorem{lemma}[theorem]{Lemma}
\newtheorem{remark}[theorem]{Remark}

\newtheorem{coro}[theorem]{Corollary}

\DeclareMathOperator{\Real}{Re}

\usepackage[textmath,displaymath,floats,graphics,delayed]{preview}
 \title[Non uniform rotating vortices]{
 Non uniform rotating vortices and periodic orbits for the two-dimensional Euler Equations}

\author[C. Garc\'ia]{Claudia Garc\'ia}
\address{Departamento de Matem\'atica Aplicada and Excellence Research Unit ``Modeling Nature'' (MNat), Facultad de Ciencias. Universidad de Granada\\ 18071-Granada, Spain. \& IRMAR, Universit\'e de Rennes 1\\ Campus de
Beaulieu\\ 35~042 Rennes cedex\\ France}
\email{claudiagarcia@ugr.es}
\author[T. Hmidi]{Taoufik Hmidi}
\address{IRMAR, Universit\'e de Rennes 1\\ Campus de
Beaulieu\\ 35~042 Rennes cedex\\ France}
\email{thmidi@univ-rennes1.fr}
\author[J. Soler]{Juan Soler}
\address{Departamento de Matem\'atica Aplicada and Excellence Research Unit ``Modeling Nature'' (MNat), Facultad de Ciencias. Universidad de Granada\\ 18071-Granada, Spain}
\email{jsoler@ugr.es}
{\thanks{This work has been partially supported by the MINECO-Feder (Spain) research grant number MTM2014-53406-R, the Junta de Andaluc\'ia (Spain) Project FQM 954 (C.G. \& J.S.), the MECD (Spain) research grant FPU15/04094 (C.G.) and the ERC project FAnFArE no.63751.
}}


\begin{document}

\begin{abstract}
This paper concerns the study  of  some special ordered   structures   in turbulent  flows. 
In particular,  a systematic and relevant methodology    is proposed to construct non trivial and non radial rotating vortices  with non necessarily uniform densities and with different $m$--fold symmetries, $m\ge 1$.   In particular, a complete  study is provided for the truncated quadratic density $(A|x|^2+B){\bf{1}}_{\D}(x)$, with $\D$ the unit disc. We exhibit different behaviors with respect to the coefficients $A$ and $B$ describing the rarefaction of bifurcating  curves. 
\end{abstract}

\maketitle{}
\tableofcontents
\section{Introduction}

The main goal of this paper is to investigate  the emergence of  some special ordered   structures   in turbulent  flows. 
The search for Euler and Navier-Sokes solutions  is a classical problem of permanent relevance that seeks to understand the complexity and dynamics of certain singular structures in Fluid Mechanics. Only a few solutions are known without much information about their dynamics. 

We will focus on the two-dimensional Euler equations that  can be  written in  the velocity--vorticity formulation as follows
\begin{eqnarray}   \label{Eulereq}	           
       \left\{\begin{array}{ll}
          	\omega_t+(v\cdot \nabla) \omega=0, &\text{ in $[0,+\infty)\times\mathbb{R}^2$}, \\
         	 v=K*\omega,&\text{ in $[0,+\infty)\times\mathbb{R}^2$}, \\
         	 \omega(t=0,x)=\omega_0(x),& \text{ with $x\in\mathbb{R}^2$}.
       \end{array}\right.
\end{eqnarray}
 The second equation  links the velocity to the vorticity through the so called Biot--Savart  law, where $K(x)=\frac{1}{2\pi}\frac{x^\perp}{|x|^2}$.
This is a Hamiltonian system that develops   various  interesting behaviors at different levels, which are in the center of intensive  research activities.  Lots of studies have been devoted to the existence and stability of relative equilibria (in general, translating and rotating steady-state solutions called V-states).
We point out  that  despite the complexity of the motion  and the deformation process that the vorticity undergoes, some special vortices  subsist  without any deformation and keep their shape during the motion. These fascinating   and intriguing  structures   illustrate somehow the emergence of the order from disordered and chaotic  motion.
 The first known example in the literature goes back to Kirchhoff who discovered that a vorticity uniformly distributed inside an elliptic shape performs a uniform rotation about its center  with constant angular velocity. Notice  that the solutions of vortex patch type (solutions with piecewise constant vorticity)  have motivated important mathematical achievements in recent years. For example, the  existence of global solutions in this setting  is rigorously obtained by Yudovich  \cite{Yudovich}. The $L^1$ assumption can be replaced by a m--fold condition of symmetry thanks to the work of Elgindi and Jeong \cite{Elg}.
The main feature of  the vortex patch problem is the persistance of this  structure due to the  transportation  of the vorticity by the flow.  However, the  regularity persistence of the boundary with $\mathscr{C}^{1,\alpha}$-regularity is delicate and  was first
shown by Chemin \cite{Chemin}, and then via different techniques by Bertozzi and Constantin \cite{B-C} and Serfati \cite{Serf}.  Coming back to the emergence of relative equilibria, uniformly rotating m-fold patches with lower symmetries generalizing Kirchhoff ellipses  were discovered numerically  by Deem and Zabusky \cite{DeemZabusky}. 
Having this kind of  V-states solutions in mind, Burbea \cite{Burbea} designed a  rigorous approach  to generate them close to a Rankine vortex through  complex analytical tools and bifurcation theory. Later this idea was improved and extended to different directions: regularity of the boundary, various topologies, effects of the rigid boundary, and different nonlinear transport equations. For the first subject, the  regularity of the contour was analyzed in \cite{Cas0-Cor0-Gom,Cas-Cor-Gom,HmidiMateuVerdera}. There, it was proved that close to the unit disc the boundary of the rotating patches are not only $\mathscr{C}^\infty$ but analytic. As to the second point,  similar results with   richer structures  have been obtained  for  doubly connected patches \cite{DelaHozHmidiMateuVerdera,HmidiHozMateuVerdera}. The existence of small loops in the bifurcation diagram has been  achieved very recently in \cite{Renault-Hmidi}.  For disconnected patches, the existence of co-rotating and counter--rotating vortex pairs was discussed in 
\cite{H-M}. We mention that the bifurcation theory is so robust that partial results have been extended to different models such as the generalized surface quasi-geostrophic equations \cite{Cas-Cor-Gom,Hassa-Hmi} or Shallow-water quasi-geostrophic equations \cite{D-H-R}, but the computations turn  out to be much more involved in those cases.

It should be noted that the particularity of the  rotating patches is that the dynamics is reduced to the motion of a finite number of  curves in the complex plane, and therefore the implementation of the bifurcation is straightforward. However, the construction of smooth rotating vortices is much more intricate due to the size of the kernel of the linearized operator, which is in general infinite dimensional because it contains at least every radial function. Some strategies  have been elaborated in order to capture some non trivial rotating smooth solutions. The first result amounts  to Castro, C\'ordoba  and G\'omez-Serrano
 \cite{C-C-GS-2} who established for the SQG and Euler equations the existence of 3-fold smooth rotating vortices using a reformulation of the equation through the level sets of the vorticity. However the spectral study turns to be highly complex and they use  computer--assisted proofs  to check  the suitable  spectral properties. In a recent paper 
\cite{CastroCordobaGomezSerrano} the same   authors removed the computer assistance part and proved the existence of $\mathscr{C}^2$ rotating vorticity with m-fold symmetry, for any $m\geq2.$ The  proof relies on the desingularization and bifurcation from the vortex patch problem. We point out that the profile of the  vorticity is constant outside a very thin region where the transition occurs, and the  thickness of this region serves as a bifurcation parameter. Remark  that different   variational arguments  were developed in \cite{Burton,Gra-Smets}.

The  main objective of this paper is to construct a systematic scheme  which turns to be relevant  to detect non trivial  rotating vortices    with non uniform densities, far from  the patches but close to some known  radial profiles. Actually, we are looking for compactly supported rotating vortices  in the form
\begin{eqnarray}\label{rotatingsol}
\omega(t,x)=\omega_0(e^{-it\Omega} x), \quad \omega_0=(f\circ\Phi^{-1}) {\bf{1}}_D, \quad  \forall x\in\R^2,
\end{eqnarray}
where $\Omega$ is the angular velocity, ${\bf{1}}_D$ is  the characteristic function of  a smooth simply connected domain $D$, the real function $f:\overline{\D}\to \R$ denotes  the density profile  and $\Phi:\D\to D$ is a conformal mapping from the unit disc $\D$ into  $D$. It is a known fact that an initial vorticity $\omega_0$ with velocity $v_0$ generates a rotating solution, with constant angular velocity $\Omega$, if and only if
\begin{equation}\label{FirstEqB}
(v_0(x)-\Omega x^\perp)\cdot\nabla\omega_0(x)=0,\quad \forall x\in\R^2.
\end{equation}
Thus the ansatz \eqref{rotatingsol} is a solution of Euler equations \eqref{Eulereq}   if and only if the following  equations 
\begin{eqnarray}
(v(x)-\Omega x^\perp)\cdot\nabla (f\circ\Phi^{-1})(x)&=&0,\quad\hbox{in}\quad D,\label{rotatingeq2d}\\
(v(x)-\Omega x^\perp)\cdot(f\circ\Phi^{-1})(x)\vec{n}(x)&=&0,\quad \hbox{on}\quad \partial D,
\label{rotatingeq2b}
\end{eqnarray}
are  simultaneously satisfied, where $\vec{n}$ is the upward unit normal vector  to the boundary $\partial D$. Regarding its relationship with the issue of finding vortex patches, the problem presented here exhibits a greater complexity. While a rotating vortex patch solution can be described by the boundary equation \eqref{rotatingeq2b},   here we also need to work with the corresponding coupled density equation \eqref{rotatingeq2d}. One major problem that one should face in order to make the bifurcation argument useful is related to the size of the kernel of the linearized operator which is in general infinite-dimensional. In  the vortex patch framework  we overcome this difficulty using  the contour dynamics equation  and by imposing a suitable symmetry on the V-states: they should be invariant by the dihedral group $D_m$.  In this manner we guarantee that the linearized operator becomes a Fredholm operator with zero index. 
In the current context,  we note that all smooth radial functions belong to the kernel. One possible strategy that one could implement is to  filter those non desirable functions from the structure of the function spaces by removing the mode  zero. However, this attempt fails because the space will not be stable by the nonlinearity especially for the density equation  \eqref{rotatingeq2d}: the frequency zero can be obtained from a resonant regime, for example the square of a non vanishing  function  on the disc generates always the zero mode. Even though, if we assume that we were able to  solve this technical problem  by some special  fine  tricks, a  second but more delicate one arises   with the formulation \eqref{FirstEqB}.  The linearized operator around any radial solution is not of Fredholm type: it is smoothing in the radial component. In fact, if $\omega_0$ is radial, then the linearized operator associated with the nonlinear map
$$
F(\omega)(x)=(v(x)-\Omega x^\perp)\cdot\nabla\omega(x),
$$
is given in polar coordinates  by
$$
\mathcal{L} h=\left(\frac{v_\theta^0}{r}-\Omega\right)\partial_\theta h+K(h)\cdot\nabla\omega_0, \quad {K}(h)(x)=\frac{1}{2\pi}\int_{\R^2}\frac{(x-y)^\perp}{|x-y|^2}h(y)dy.
$$ 
The loss of information in the radial direction can not be compensated by the operator ${K}$ which is compact. This means that when using standard function spaces,  the  range of the linearized operator will be of infinite codimension. This discussion illustrates the limitation of working directly with the model  \eqref{FirstEqB}. 
Thus, we should  first proceed  with   reformulating  differently the equation  \eqref{FirstEqB} in order to avoid the preceding technical problems  and capture non radial solutions by a bifurcation argument. We point out that the main obstacle comes  from the density equation \eqref{rotatingeq2d} and the elementary key observation is that a solution to this  equation means that the density is constant along the level sets of the relative stream function. This can be guaranteed if one looks for solutions to the restricted problem,
\begin{equation}\label{DEnsT}
G(\Omega,f,\Phi)(z)\triangleq\mathcal{M}\big(\Omega,f(z)\big)+\frac{1}{2\pi}\int_{D}\log|\Phi(z)-\Phi(y)|f(y)||\Phi^\prime(y)|^2 dy-\frac{1}{2}\Omega|\Phi(z)|^2
=0,
\end{equation} 
for every $ z\in \D$, and for some suitable real function $\mathcal{M}$. The free function $\mathcal{M}$ can be fixed  so that the radial profile is a solution. For instance, as it will be shown in Section \ref{Secdensityeq}, for the radial profile 
\begin{equation}\label{QuadrP}
f_0(r)=Ar^2+B
\end{equation}
we get the explicit form
$$
\mathcal{M}(\Omega,s)=\frac{4\Omega-B}{8A} s-\frac{1}{16 A} s^2+\frac{3B^2+A^2+4AB-8\Omega B}{16 A}.
$$
Moreover, with this reformulation, we can ensure that no other radial solution can be captured around the radial profile except for a singular value, see Proposition \ref{radialfunctions}.

Before stating our result we need to introduce the following set, which is nothing but the singular set introduced later, see \eqref{sing-dos}, in the case of the quadratic profile,
\begin{equation*}
\mathcal{S}_{\textnormal{sing}}=\left\{\frac{A}{4}+\frac{B}{2}-\frac{A(n+1)}{2n(n+2)}-\frac{B}{2n},\quad   n\in\N^\star\cup\{+\infty\}  \right\}.
\end{equation*}
The  main result of this paper concerning the quadratic profile is the following. 
\begin{theorem} \label{TP1}
Let $A>0$, $B\in\R$ and $m$ a positive integer. Then the following results hold true.
\begin{enumerate}
\item If $A+B<0$, then there is    $m_0\in \N$ (depending only on $A$ and $B$)  such that for any $m\geq m_0,$ there exists a branch of non radial rotating solutions 
with $m-$fold symmetry  for the Euler equation, bifurcating from the radial solution  \eqref{QuadrP}   at some given $\Omega_m>\frac{A+2B}{4}$.
\item If $B>A$, then for any integer $m\in \left[1, \frac{B}{A}+\frac18\right]$ or $m\in\left[1,\frac{2B}{A}-\frac{9}{2}\right]$ there exists a branch of non radial rotating solutions  with $m-$fold symmetry  for the Euler equation, bifurcating from the radial solution \eqref{QuadrP} at some given $0\leq\Omega_m<\frac{B}{2}$. However, there is no bifurcation for any symmetry \mbox{ $m\geq \frac{2B}{A}+2$}.
\item If $B>0$ or $B\leq-\frac{A}{1+\epsilon}$ for some $0,0581<\epsilon<1$, then there exists a branch of non radial  $1-$fold symmetry rotating solutions    for the Euler equation, bifurcating from the radial solution \eqref{QuadrP}   at $\Omega_1=0$. 
\item If 
$- \frac{A}{2}< B< 0$ and $\Omega\notin \mathcal{S}_{\textnormal{sing}}$,   then there is no solutions to \eqref{DEnsT} close to the quadratic profile.
\item In the frame of the rotating vortices constructed in $\bf{(1)},\bf{(2)}$ and $\bf{(3)}$, the particle trajectories inside their supports are concentric periodic orbits around the origin.  
 \end{enumerate}
\end{theorem}
This  theorem will be fully detailed in Theorem \ref{singthm}, Theorem \ref{TP2} and Theorem \ref{mainth2}. Before giving some details about the main ideas of the proofs, we wish to draw some useful comments:
\begin{itemize}
\item The upcoming Theorem \ref{mainth2} states that the orbits associated with \eqref{rotatingeq2d} are periodic with smooth period, and at any time the flow is invariant by a rotation of angle $\frac{2\pi}{m}$. Moreover, it generates a group of diffeomorphisms of the closed unit disc.
\item The V-states  constructed in the above theorem have the form $f\circ\Phi^{-1}{\bf{1}}_D$. Also, it is proved  that the density $f$ is $\mathscr{C}^{1,\alpha}(\D)$ and the boundary $\partial D$ is $\mathscr{C}^{2,\alpha}$ with  $\alpha \in (0,1).$ We believe that by implementing the techniques used in \cite{Cas-Cor-Gom} it could be shown  that the density and the domain are analytic. An indication  supporting this intuition   is provided by the generator of the kernel associated with the density equation, see \eqref{hkernel}, which is analytic up to the boundary.
\item The dynamics of the  1--fold symmetric V-states is rich and  very interesting. The branch can survive even in the region where no other symmetry is allowed. It should be noted  that in the context of vortex patches the bifurcation from the disc or the annulus  occurs only  with symmetry $m\geq2$ and never with the symmetry $1$.  The only examples that we know in the literature about the emergence of the symmetry one is the bifurcation from Kirchhoff ellipses \cite{ Cas-Cor-Gom,HmidiMateu} or the presence of the boundary effects \cite{DHHM}.   
\item From the homogeneity of Euler equations the transformation $(A,B,\Omega)\mapsto (-A,-B,-\Omega)$ leads to the same class of solutions  in Theorem $\ref{TP1}.$ This observation allows including in the main theorem the case $A<0.$
\item The assumptions on $A$ and $B$ seen in  Theorem $\ref{TP1}-(1)-(2)$ about the bifurcation cases imply that the radial profile $f_0$ is not changing the sign in the unit disc.  However in the point  $(3)$ the profile can change the sign.
\item
The bifurcation with $m-$fold symmetry, $m\geq1$,  when $B\in(-A,-\frac{A}{2})$  is not well understood. We only know that  we can obtain a branch of 1-fold symmetric solutions bifurcating from $\Omega_1=0$ for $B\in(-A,-\frac{A}{1+\epsilon})$ for some $\epsilon\in(0,1)$, nothing is known for other symmetries. We expect that similarly to the result of Theorem \ref{TP1}-(2), they do exist but only for lower frequencies, and the bifurcation curves are rarefied when $B$ approaches $-\frac{A}{2}.$
\item Let us remark the existence of solutions with lower $m$-fold symmetry  coming from the second point of the above Theorem. Fixing $A$, the number of allowed symmetries increases when $|B|$ increases. We guess that there is a smooth curve when passing from one symmetry to another one, see Fig.1. 
\end{itemize}

\begin{figure}[h]
\begin{center}
 \includegraphics[width=0.6\textwidth]{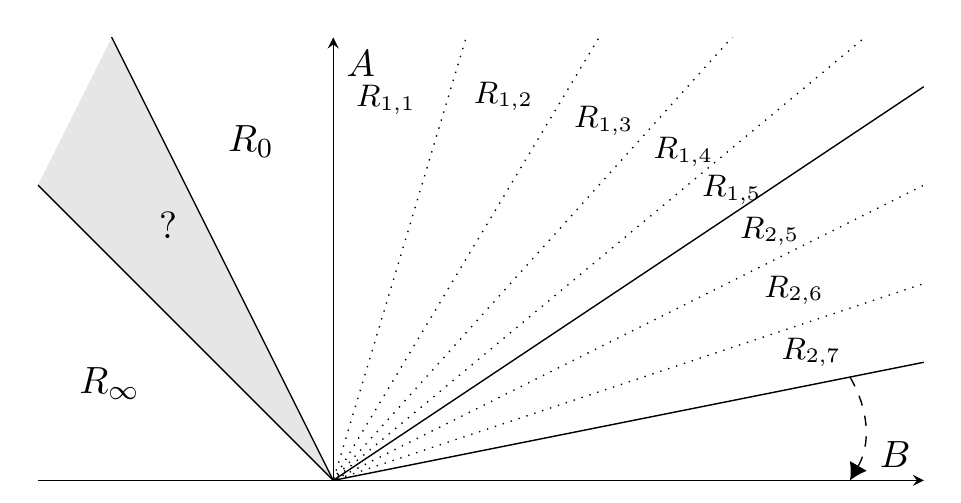}
\caption{This diagram shows the different bifurcation regimes given in Theorem \ref{TP1}, with $A>0$. In the case $B>0$, we can find only a finite number of eigenvalues $\Omega_m$ for which it is possible to obtain a branch of non radial $m$--fold symmetric solutions of the Euler equation. Here, $m\geq 1$  increases as $B$ does. The region $R_{i,j}$ admits solutions with $m$--fold symmetry for $1\leq m \leq i$. In addition, solutions with $m$--fold symmetry for $m>j$ are not found. The transition between $m=i$ and $m=j$ is not known. Notice that in the region $R_\infty$ the bifurcation occurs with  an infinite  countable family of eigevalues. 
However, the bifurcation is not possible in the region $R_0$ but  the transition between $R_0$ and $R_\infty$ is not well-understood due to some spectral problem concerning the linearized operator. We only know the existence of $1$--fold symmetric solutions in a small region.}
\label{fig}
\end{center}
\end{figure}

Let us briefly outline the general strategy we follow to prove the main result and that could be implemented for more general profiles. Using the conformal mapping $\Phi$ we can translate the equations \eqref{rotatingeq2d}--\eqref{rotatingeq2b} into the disc $\D$ and its boundary $\T$. Equations \eqref{rotatingeq2d}--\eqref{rotatingeq2b} depend functionally on the parameters $(\Omega, f, \Phi)$, so that we can write them  as
\begin{eqnarray} \label{FG}
\left\{\begin{array}{lll}
G(\Omega,f,\Phi)(z)&=&0, \quad \forall z\in\D,\\
F(\Omega,f,\Phi)(w) &=& 0,\quad \forall w\in\T ,
\end{array}\right.
\end{eqnarray}
with
\begin{equation*}
 F(\Omega,f,\Phi)(w)\triangleq\textnormal{Im}\left[ \left(\Omega\,\overline{\Phi(w)}-\frac{1}{2\pi}\int_{\D}\frac{f(y)}{{\Phi(w)}-{\Phi(y)}}|\Phi^\prime(y)|^2 dA(y)\right){\Phi^\prime(w)}{w}\right]
=0,\quad \forall w\in\T ,
\end{equation*}
where the functional $G$ is described in \eqref{DEnsT}.
The aim is to parametrize the solutions in $(\Omega,f,\Phi)$ close to some initial radial solution $(\Omega, f_0,\hbox{Id})$, with $f_0$ being a radial profile and $\hbox{Id}$  the identity map.   Then, we will deal with the unknowns $g$ and $\phi$ defined by
\begin{equation}\label{descomp}
f=f_0+g, \, \, \Phi=\hbox{Id}+\phi.
\end{equation}
Thus,  the equations in \eqref{FG} are parametrized in the form $G(\Omega,g,\phi)(z)=0$ and $F(\Omega,g,\phi)(w)=0$, where   $G(\Omega,0,0)(z)=0$ and $F(\Omega,0,0)(w)=0$. The idea is to start by solving the boundary equation, which would reduce a variable  through a mapping $(\Omega,f)\mapsto \Phi=\mathcal{N}(\Omega,f)$, i.e. to prove that under some restrictions  $F(\Omega,g,\phi)=0$ is equivalent to $ \phi=\mathcal{N}(\Omega,g)$. However, the argument stumbles when we realize that this can only be done outside  a set of singular values
\begin{equation} \label{sing-dos}
\mathcal{S}_{\textnormal{sing}}\triangleq\Big\{\Omega \,:\,   \partial_\phi F(\Omega,0,0) \ \mbox{ is not an isomorphism} \Big\},
\end{equation}
for which the Implicit Function Theorem can be applied. Then, we prove that there exists  an open interval $I$ for $\Omega$ such that  $\overline{I}\subset \R\backslash \mathcal{S}_{\textnormal{sing}}$ and $\mathcal{N}$ is well--defined in appropriated spaces, which will be subspaces of H\"older--continuous functions. Under the hypothesis that $\Omega \in I$, the problem of finding solutions of \eqref{FG} is reduced to solve
\begin{eqnarray} \label{Ggorro}
\widehat{G}(\Omega,g)(z)\triangleq G(\Omega,g,\mathcal{N}(\Omega,g))(z)= 0, \quad \forall z\in\D .
\end{eqnarray}
In order to find time dependent non radial rotating solutions to \eqref{Eulereq} we use the procedure developed in \cite{Burbea} that suggests the bifurcation theory as a tool to generate solutions from a stationary one via the Crandall--Rabinowitz Theorem. The values $\Omega$ that could lead to the bifurcation to non trivial solutions are located in the dispersion set
\begin{equation} \label{disper-dos}
\mathcal{S}_{\textnormal{disp}}\triangleq\left\{\Omega \,:\,  \textnormal{Ker} \, D_g \widehat{G}(\Omega, 0)\neq\{0\}\right\}.
\end{equation}
The problem then consists in verifying  that the singular \eqref{sing-dos} and dispersion \eqref{disper-dos} sets are well-separated, for a correct definition of the interval $I$. Achieving this objective together with the analysis of the dimension properties of the kernel and the codimension of the range of $D_g\widehat{G}(\Omega,0)$, as well as  verifying the transversality property requires a complex and precise spectral and asymptotic analysis.
Although our discussion is quite general, we  focus  our attention  on the  special case of  quadratic profiles \eqref{QuadrP}.  In this case   we obtain  a compact representation of the dispersion set. Indeed, as we shall see in Section \ref{Sec-line}, the resolution of the kernel equation leads to a Volterra type integro-differential equation   that one may solve through transforming it into an ordinary differential equation of second order with polynomial coefficients. Surprisingly, the new  equation can be solved explicitly through  variation of the constant and is connected to  Gauss hypergeometric functions.
The structure of the dispersion set is very subtle and appears to be very sensitive to the parameters $A$ and $B$. Our analysis allows us to highlight   some special regimes on $A$ and $B$, see Proposition \ref{Prop-eig} and Proposition \ref{Propexistcase2}.  

Let us emphasize that the  techniques developed in the quadratic profile   are robust and could be extended to  other profiles. In this direction, we first  provide in Section \ref{Secdensityeq} the explicit expression of the function $\mathcal{M}$  when the density  admits a polynomial or  Gaussian distribution. In general, the explicit resolution of the kernel equations may turn out  to be a very challenging problem. Second, we will notice  in Remark \ref{Rempoly}  that when $f_0=Ar^{2m}+B$ with $m\in\N^\star$, explicit formulas are expected through some elementary transformations and the  kernel elements  are linked also to hypergeometric equations.

In Section \ref{Secorbits} we shall be concerned with the proof of the point ${{(5)}}$ of Theorem \ref{TP1}  concerning the planar trajectories of  the particles located inside the support of the rotating vortices. 
We analyze the properties of periodicity and  symmetries of the solutions  via the study of the associated dynamical Hamiltonian structure in Eulerian coordinates, which was highlighted by Arnold \cite{Arn}. This Hamiltonian nature   of the Euler equations has been the idea behind the study of conservation laws in the hydrodynamics of an ideal fluid \cite{Arn,Meza,Mo,Ol}, as well as  in a certain sense to justify Boltzmann’s principle from classical mechanics \cite{Yau}.

We shall give   in Theorem \ref{mainth2} a precise statement and  prove that close to the quadratic  profile  all the trajectories are periodic orbits located inside  the support of the V-states, enclosing a simply connected domain containing the origin, and are symmetric with respect to the real axis. In addition,  every orbit is invariant by a rotation of angle $\frac{2\pi}{m}$, as it has been proved for the branch of bifurcated solutions, where the parameter $m$ is determined by the spectral properties. The periodicity of the orbits follows from the Hamiltonian structure of the autonomous dynamical system,
\begin{eqnarray} \label{ecsi}
{\partial_t \psi(t,z)}={W(\Omega,f,\Phi)(\psi(t,z))},\quad 
{\psi}(0,z)=z\in\overline{\D},
\end{eqnarray}
where
\begin{eqnarray*}
W(\Omega,f,\Phi)(z)=\left(\frac{i}{2\pi}\int_{\D}\frac{f(y)}{\overline{\Phi(z)}-\overline{\Phi(y)}}|\Phi^\prime(y)|^2 dy-i\Omega\Phi(z)\right)\overline{\Phi^\prime(z)},\quad  z\in \D.
\end{eqnarray*}
Notice  that $W$ is nothing but the pull-back of  the vector filed $v(x)-\Omega x^\perp$ by the conformal mapping $\Phi$.  This vector field  remains Hamiltonian and is tangential to the boundary $\T$. Moreover,  we check  that close to the radial profile, it has only one critical point located at the origin which must be  a center. As a consequence, the trajectories near the origin are organized through periodic  orbits. Since  the trajectories are located in the level sets of the energy functional given by the relative stream function, then using simple arguments we show the limit cycles  are excluded and thus all the trajectories are periodic enclosing the origin which is the only fixed point, which, together with the trajectories defined above, is a way of solving the hyperbolic system \eqref{rotatingeq2d}.  This allows to define  the period map  $z\in \D\mapsto T_z$, whose regularity will be at the same level as the profiles. As a by-product we find the following  equivalent reformulation of the density equation\begin{equation} \label{peri}
f(z)-\frac{1}{T_z}\int_0^{T_z} f(\psi(\tau,z))d\tau=0, \quad \forall z\in \overline{\D}.
\end{equation}

Other approaches to study stationary solutions, and thus to explore the possibilities of bifurcating them, have been proposed through the study of characteristic trajectories associated with stationary velocities
\[
\frac{\partial X(t)}{\partial t} = v(X(t)), \quad X(0) =x \in \R^2,
\]
in connection with the elliptic equation $-\Delta \psi = \omega = \tilde F(\psi)$ ($\psi$ being the current function $\nabla^\perp \psi =v$), see \cite{MajdaBertozzi}.
In this context, the idea of looking for smooth stationary solutions was first developed by Nadirashvili in \cite{Na}, where the geometry (curvature) of streamlines was studied. Luo  and  Shvydkoy \cite{LSh} provide a classification of some kind of homogeneous stationary smooth solutions  with locally finite energy, where new solutions having hyperbolic, parabolic and elliptic structure of streamlines appear. Choffrut and \v{S}ver\'ak \cite{CSv} showed analogies  of  finite-dimensional  results  in  the  infinite-dimensional  setting  of Euler’s equations, under some non-degeneracy assumptions, proving a local one-to-one correspondence between steady-states  and co-adjoint orbits. Then,  Choffrut and Sz\'ekelyhidi \cite{CSz} showed, using an  h-principle \cite{LSz}, that there is  an abundant set of weak, bounded stationary solutions in the neighborhood of any smooth stationary solution. Kiselev and \v{S}ver\'ak \cite{K-S} construct an example of initial data in the disc such that the corresponding solutions for the 2D Euler equation exhibit double exponential growth in the gradient of vorticity, which is related to the lack of Lipschitz regularity, and to an example of the singular stationary solution provided by Bahouri and Chemin \cite{BCh}, which produces a flow map whose H\"older regularity decreases in time.

We finally comment on three recent approaches to the analysis of rotating solutions. The first one concerns  rotating vortex patches. In \cite{HMW}, Hassania, Masmoudi and Wheeler construct continuous curves of rotating vortex patch solutions, where the minimum along the interface of the angular fluid velocity in the rotating frame becomes
arbitrarily small, which agrees with the conjecture about singular limiting patches with $90^{\circ}$ corners \cite{CS,O}.
In the second contribution \cite{CastroCordobaGomezSerrano}, it was studied the existence of  smooth rotating vortices desingularized from a vortex patch, as it was mentioned before. The techniques are  based on the analysis  of the level sets of the vorticity of a global rotating solution. Since the level sets $z(\alpha,\rho,t)$ rotate with constant angular velocity, they satisfy
$\omega(z(\alpha,\rho,t)),t) =f(\rho)$. Thus, in \cite{CastroCordobaGomezSerrano} is studied the problem of bifurcating it for some specific choice of $f$.
In a  broad sense, this result connects with that developed in this paper about the study of orbits and their periodicity. Finally, Bedrossian, Coti Zelati and Vicol in \cite{Bedro} analyze the incompressible 2D Euler equations linearized around a radially symmetric, strictly monotone decreasing vorticity distribution. For sufficiently regular data,  inviscid damping of the $\theta$-dependent radial and angular velocity fields is proved. In this case, the vorticity weakly converges back to radial symmetry as $t
\to \infty$, a phenomenon known as vortex axisymetrization. Also they show that the
$\theta$-dependent angular Fourier modes in the vorticity are ejected from the origin as $t
\to \infty$, resulting in faster inviscid damping rates than those possible with passive scalar evolution (vorticity depletion).

The results and techniques presented in this paper are powerful enough to be extended to other situations and equations such as SQG equations, co--rotating time--dependent solutions, solutions depending only on one variable,...

\section{Preliminaries and statement of the problem}
The aim of this section is to formulate  the equations governing general rotating solutions of the  Euler equations. 
We will also set down some of the tools that we use throughout the paper such as the functional setting or some properties about the extension of Cauchy integrals.

\subsection{Equation for  rotating vortices}
Let us begin with the equations for compactly supported  rotating solutions \eqref{rotatingeq2d}--\eqref{rotatingeq2b} and assume  that  $f$ is not vanishing on the boundary. In the opposite case,  the equation   \eqref{rotatingeq2b} degenerates and becomes trivial, which implies that we just have one equation to analyze. Thus, \eqref{rotatingeq2b} becomes 
\begin{eqnarray}
(v(x)-\Omega x^\perp)\cdot\vec{n}(x)=0,\quad \hbox{on}\quad \partial D.\label{rotatingeq3b}
\end{eqnarray}
 We will rewrite these equations in the unit disc through the  use of the conformal map $\Phi:\D\rightarrow D$. Note that from now on and  for the sake of simplicity we will identify the Euclidean and the complex planes. Then, we write the velocity field as
 $$
 v(x)=\frac{i}{2\pi}\int_D \frac{(f\circ\Phi^{-1})(y)}{\overline{x}-\overline{y}}dA(y), \quad \forall x\in \C,
 $$
where $dA$ refers to the planar Lebesgue measure, getting
 $$
 v(\Phi(z))=\frac{i}{2\pi}\int_\D \frac{f(y)}{\overline{\Phi(z)}-\overline{\Phi(y)}}|\Phi'(y)|^2dA(y), \quad \forall z\in\D.
 $$
Using the conformal parametrization  $\theta\in[0,2\pi]\mapsto \Phi(e^{i\theta})$ of $\partial D$, we find that a normal vector to the boundary is given by $\vec{n}(\Phi(w))=w\Phi^\prime(w)$, with $w\in\T$. In order to deal with \eqref{rotatingeq2d}, we need to  transform carefully the term $\nabla (f\circ \Phi^{-1})(\Phi(z))$ coming from the density equation. Recall that for  any complex function $\varphi:\C\rightarrow \C$ of class $\mathscr{C}^1$ seen as a function of $\R^2$, we can define 
$$
\partial_{\overline{z}}\varphi(z)\triangleq\frac12\left(\partial_{1}\varphi(z)+i\partial_{2} \varphi(z)\right)\quad\hbox{and}\quad \partial_{{z}}\varphi(z)\triangleq\frac12\left(\partial_{1}\varphi(z)-i\partial_{2} \varphi(z)\right),
$$
which are known in the literature  as Wirtinger derivatives. Let us state  some of their basic properties:
$$
\overline{\partial_z\varphi}=\partial_{\overline{z}}\overline{\varphi},\quad \overline{\partial_{\overline{z}}\varphi}=\partial_z\overline{\varphi}.
$$
Given two complex functions $\varphi_1, \varphi_2:\C\rightarrow\C$ of class $\mathscr{C}^1$ in the Euclidean coordinates, the chain rule comes as follows
\begin{eqnarray*}
\partial_z (\varphi_1\circ\varphi_2)&=&\left(\partial_z \varphi_1\circ \varphi_2\right)\partial_z\varphi_2+\left(\partial_{\overline{z}}\varphi_1\circ \varphi_2\right)\partial_z\overline{\varphi_2},\\
\partial_{\overline{z}} (\varphi_1\circ\varphi_2)&=&\left(\partial_z \varphi_1\circ \varphi_2\right)\partial_{\overline{z}}\varphi_2+\left(\partial_{\overline{z}}\varphi_1\circ \varphi_2\right)\partial_{\overline{z}}\overline{\varphi_2}.
\end{eqnarray*}
Moreover, since $\Phi$ is a conformal map, one has that $\partial_{\overline{z}}\Phi=0$. Identifying  the gradient with the operator $2\partial_{\overline{z}}$\, leads to
$$
\nabla (f\circ \Phi^{-1})=2\partial_{\overline{z}}(f\circ \Phi^{-1}).
$$
From  straightforward computations using   the holomorphic structure of $\Phi^{-1}$, combined with  the previous properties of the Wirtinger derivatives, we get
\begin{eqnarray*}
\partial_{\overline{x}} (f\circ \Phi^{-1})(x)&=&(\partial_x f)(\Phi^{-1}(x))\partial_{\overline{x}}\Phi^{-1}(x)+(\partial_{\overline{x}} f)(\Phi^{-1}(x))\partial_{\overline{x}}\overline{\Phi^{-1}(x)}\\
&=&(\partial_{\overline{x}} f)(\Phi^{-1}(x))\overline{(\Phi^{-1})^\prime(x)},
\end{eqnarray*}
where the prime notation $^\prime$ for $\Phi^{-1}$ denotes the complex  derivative  in the holomorphic case.
Using that $(\Phi^{-1}\circ\Phi)(z)=z$ and differentiating it we obtain
$$
(\Phi^{-1})^\prime(\Phi(z))=\frac{\overline{\Phi'(z)}}{|\Phi^\prime(z)|^2},
$$
which implies
\begin{eqnarray}\label{deriv}
\partial_{\overline{z}} (f\circ \Phi^{-1})(\Phi(z))=\frac{\partial_{\overline{z}} f(z) \Phi^\prime(z)}{|\Phi^\prime(z)|^2}.
\end{eqnarray}
Putting everything together in \eqref{rotatingeq2d}-\eqref{rotatingeq3b} we find the following equivalent expression
 \begin{eqnarray}
  W(\Omega,f,\Phi)\cdot\nabla f&=&0,\quad\hbox{in}\quad \D,\label{densityeq}\\ 
W(\Omega,f,\Phi)\cdot\vec{n}&=&0,\quad \hbox{on}\quad \T,\label{boundaryeq}
\end{eqnarray}
where $W$ is given by
\begin{eqnarray}\label{W}
W(\Omega,f,\Phi)(z)=\left(\frac{i}{2\pi}\int_{\D}\frac{f(y)}{\overline{\Phi(z)}-\overline{\Phi(y)}}|\Phi^\prime(y)|^2 dy-i\Omega\Phi(z)\right)\overline{\Phi^\prime(z)}.
\end{eqnarray}
Then, the vector field $W(\Omega,f,\Phi)$ is incompressible. This fact is a consequence  of the  lemma below. Given a vector field $X:\C\rightarrow \C$ of class $\mathscr{C}^1$ in the Euclidean variables, let us associate the divergence operator with the Wirtinger derivatives as follows
\begin{equation} \label{aso}
\textnormal{div} X(z)=2\textnormal{Re}\left[\partial_z X(z)\right].
\end{equation}

\begin{lemma}\label{div}
Given $X:D_1\rightarrow \C$ an incompressible vector field, $\Phi:D_2\rightarrow D_1$ a conformal map, where $D_1,D_2\subset \C$, then
$
(X\circ\Phi)\overline{\Phi'}:D_2\rightarrow \C
$
is incompressible.
\end{lemma}
\begin{proof}
Using \eqref{aso} we have that
$
\textnormal{Re}\left[\partial_x X(x)\right]=0,$ for any $x\in D_1$.

The properties of the Wirtinger derivatives lead to
\begin{eqnarray*}
\partial_z \left[X(\Phi(z))\overline{\Phi'(z)}\right]&=&\partial_z\left[X(\Phi(z))\right]\overline{\Phi^\prime(z)}+X(\Phi(z))\partial_z \overline{\Phi^\prime(z)}\\
&=&(\partial_z X) (\Phi(z))\Phi^\prime(z)\overline{\Phi^\prime(z)}+(\partial_{\overline{z}} X)(\Phi(z))\partial_z \overline{\Phi(z)}\,\overline{\Phi^\prime(z)}\\
&&+X(\Phi(z))\partial_z \overline{\Phi^\prime(z)}\\
&=&(\partial_z X)(\Phi(z))|\Phi^\prime(z)|^2,\quad \forall z\in D_2.
\end{eqnarray*}
Hence, we have that
$$
\textnormal{Re}\left[\partial_z \left(X(\Phi(z))\overline{\Phi'(z)}\right)\right]=|\Phi'(z)|^2 \textnormal{Re}\left[(\partial_z X)(\Phi(z))\right]=0,\quad \forall z\in D_2,
$$
 and $(X\circ\Phi)\overline{\Phi'}$ is incompressible.
\end{proof}

Let us remark that the equation associated to the  V-states in \cite{HmidiMateuVerdera} is nothing but the boundary equation \eqref{boundaryeq}. In \cite{HmidiMateuVerdera},  V-states close to a trivial solution are obtained by means of a perturbation of the domain via a conformal mapping. Since we are perturbing also the initial density, we must analyze one more equation: the density equation \eqref{densityeq}. In order to apply the Crandall--Rabinowitz Theorem we will not deal with \eqref{densityeq} because  it seems not to be suitable when studying the linearized operator. Hence, we will reformulate this equation in Section \ref{Secdensityeq}. Moreover, we will provide an alternative way of writing \eqref{densityeq} in Section \ref{Secorbits} to understand the behavior of the orbits of the dynamical system associated to it.

\subsection{Function spaces}
The right choice of the function spaces will be crucial  in order to construct non radial rotating solutions different to the vortex patches.

Before going into further details we  introduce the classical H\"older spaces in the unit disc $\D$. Let us denote $\mathscr{C}^{0,\alpha}(\D)$ as the set of continuous functions such that
\begin{eqnarray*}
\|f\|_{\mathscr{C}^{0,\alpha}(\D)}\triangleq \|f\|_{L^\infty(\D)}+\sup_{z_1\neq z_2\in\D}\frac{|f(z_1)-f(z_2)|}{|z_1-z_2|^\alpha}<+\infty,
\end{eqnarray*}
for any $\alpha\in(0,1)$. By $\mathscr{C}^{k,\alpha}(\D)$, with $k\in\N$, we denote the $\mathscr{C}^k$ functions whose $k$-order derivative lies \mbox{in $\mathscr{C}^{0,\alpha}(\D)$.}  Recall the Lipschitz space with the semi-norm defined as
\begin{eqnarray*}
\|f\|_{\textnormal{Lip}(\D)}\triangleq \sup_{z_1\neq z_2\in\D}\frac{|f(z_1)-f(z_2)|}{|z_1-z_2|}.
\end{eqnarray*}
Similarly, we define the H\"older spaces $\mathscr{C}^{k,\alpha}(\T)$ in the unit circle $\T$. Let us  supplement these spaces with additional symmetry structures:
$$
\mathscr{C}^{k,\alpha}_s(\D)\triangleq\bigg\{ g: {\D}\to \R\in \mathscr{C}^{k,\alpha}({\D}),\quad g(re^{i\theta})=\sum_{n\geq 0} g_n(r)\cos(n\theta),\, g_n\in\R,\,  \forall z=re^{i\theta}\in {\D} \bigg\},
$$
\begin{equation}\label{g}
 \mathscr{C}^{k,\alpha}_a(\T)\triangleq\bigg\{ \rho: \T\to \R\in \mathscr{C}^{k,\alpha}(\T),\quad \rho(e^{i\theta})=\sum_{n\geq 1}\rho_n\sin(n\theta), \, \, \,  \forall w=e^{i\theta}\in {\T}\bigg\}.
\end{equation}
These spaces  are equipped with the usual  norm $\|\cdot\|_{\mathscr{C}^{k,\alpha}}$. One can easily check that if the functions  $g\in \mathscr{C}^{k,\alpha}_s(\D)$ and $\rho\in \mathscr{C}^{k,\alpha}_a(\T)$, then they satisfy the following properties
\begin{equation}\label{Persis}
g(\overline{z})=g(z), \quad \rho(\overline{w})=-\rho(w), \quad \forall z\in {\D},\,\forall\,  w\in \T.
\end{equation}
The space $\mathscr{C}_s^{k,\alpha}(\D)$ will contain the perturbations of the initial radial density. The condition on $g$ means that this perturbation  is invariant by reflexion on the real axis. Let us remark that we introduce also a radial perturbation coming from the frequency $n=0$, this fact will be a key point in the bifurcation argument.

The second kind of function spaces is $\mathscr{H}\mathscr{C}^{k,\alpha}(\D)$, which is the set  of holomorphic functions $\phi$ in $\D$  belonging to $ \mathscr{C}^{k,\alpha}({\D})$ and satisfying  
\begin{equation*}
  \phi(0)=0,\quad \phi'(0)=0\quad \hbox{and}\quad \overline{\phi(z)}=\phi(\overline{z}), \quad \forall z\in \overline{\D}.
\end{equation*}
With these properties, the function $\phi$ admits the following expansion
\begin{eqnarray}\label{phi}
\phi(z)= z\sum_{n\geq 1} a_n z^{n}, \quad a_n\in\R.
\end{eqnarray}
Thus, we have
$$
\mathscr{H}\mathscr{C}^{k,\alpha}(\T)\triangleq\Big\{\phi\in \mathscr{C}^{k,\alpha}(\T),\quad\phi(w)=w\sum_{n\geq 1}a_nw^n,\, a_n\in\R,\, \forall w\in\T\Big\}.
$$
Notice that if $\Phi\triangleq \hbox{Id}+\phi$ is conformal then  $\Phi(\D)$ is a simply connected domain,  symmetric with respect to the real axis and whose boundary is $\mathscr{C}^{k,\alpha}$. The space $\mathscr{H}\mathscr{C}^{k,\alpha}({\D})$ is a closed subspace of $\mathscr{C}^{k,\alpha}({\D})$ equipped with  the same norm, so it is complete. In the bifurcation argument, we will perturb also the initial domain $\D$ via a conformal map that will lie in this space.

The last condition on $\phi$, given by \eqref{phi}, together with the symmetry condition for the density \eqref{Persis}, means that we are looking for rotating  initial data which admit at least one axis of symmetry. For the rotating patch problem this is the minimal requirement that we should impose and up to now we do not know whether such structures without any prescribed symmetry could exist.

Now, we introduce  the following trace problem concerning  the extension of Cauchy integrals, which is a classical result in complex analysis  and potential  theory. It is directly linked to \cite[Proposition 3.4]{Pommerenke} and \cite[Theorem 2.2]{RubelShieldsTaylor}) and for the convenience of the reader we give a proof.
\begin{lemma}\label{equivnorms}
Let $k\in\N$ and 
$\alpha\in (0,1)$. Denote by   $\mathcal{C}:\mathscr{H}\mathscr{C}^{k,\alpha}(\T)\rightarrow\mathscr{H}\mathscr{C}^{k,\alpha}(\D)$  the linear map defined  by
$$
\phi(w)=w\sum_{n\in\N}a_nw^n,\quad \forall w\in\T\Longrightarrow \mathcal{C}\left(\phi\right)=\sum_{n\in\N}a_nz^n, \quad  \forall z\in\overline{\D}.
$$
Then,  $\mathcal{C}\left(\phi\right)$ is well-defined and continuous.
\end{lemma}
\begin{proof}
First, it is a simple matter to check that the map $\mathcal{C}$ is well-defined. Thus, it remains to check the continuity. We recall from   \cite[Theorem 2.2]{RubelShieldsTaylor}) the following estimates  on the modulus of continuity
\begin{equation}\label{Mod1}
\sup_{  z_1,z_2\in\D\atop |z_1-z_2|\leq \delta } |\mathcal{C}(\phi)(z_1)-\mathcal{C}(\phi)(z_2)|\leq 3\sup_{w_1,w_2\in\T\atop |w_1-w_2|\leq \delta } |\phi(w_1)-\phi(w_2)|,
\end{equation}
for any $\delta<\frac{\pi}{2}$ and for any  continuous function $\mathcal{C}(\phi)$ in $\overline{\D}$, analytic in $\D$  and having trace function $\phi$ on the unit circle.
Therefore,  given $z_1,z_2\in \D$ with $|z_1-z_2|\leq1$, we obtain
\begin{eqnarray*}
 \frac{|\mathcal{C}(\phi)(z_1)-\mathcal{C}(\phi)(z_2)|}{|z_1-z_2|^\alpha}&\leq& 3\sup_{w_1,w_2\in\T\atop |w_1-w_2|\leq|z_1-z_2|  } \frac{|\phi(w_1)-\phi(w_2)|}{|z_1-z_2|^\alpha}\\
 &\leq& 3\sup_{w_1,w_2\in\T\atop |w_1-w_2|\leq1  } \frac{|\phi(w_1)-\phi(w_2)|}{|w_1-w_2|^\alpha} \sup_{w_1,w_2\in\T\atop |w_1-w_2|\leq|z_1-z_2|  } \frac{|w_1-w_2|^\alpha}{|z_1-z_2|^\alpha}\\
 &\le&3\sup_{w_1,w_2\in\T\atop |w_1-w_2|\leq1  } \frac{|\phi(w_1)-\phi(w_2)|}{|w_1-w_2|^\alpha}.
\end{eqnarray*}
Now, let $z\in\D$ and $w\in \T$ such that $|z-w|\leq1$, then we also get from \eqref{Mod1}
\begin{eqnarray*}
|\mathcal{C}(\phi)(z)|\leq |{\phi}(w)|+3\sup_{w_1,w_2\in\T\atop |w_1-w_2|\leq 1 } |\phi(w_1)-\phi(w_2)|
\le  |{\phi}(w)|+3\sup_{w_1,w_2\in\T\atop |w_1-w_2|\leq1  } \frac{|\phi(w_1)-\phi(w_2)|}{|w_1-w_2|^\alpha},
\end{eqnarray*}
which implies that
$$
\|\mathcal{C}(\phi)\|_{L^\infty(\overline{\D})}\leq 3\|\phi\|_{\mathscr{C}^{0,\alpha}(\T)}.
$$
Combining the preceding estimates,  we deduce that 
$$
\|\mathcal{C}(\phi)\|_{\mathscr{C}^{0,\alpha}(\D)}\leq 6\|\phi\|_{\mathscr{C}^{0,\alpha}(\T)}.
$$
Note that this estimate can be extended to higher derivatives and thus we obtain 
$$
\|\mathcal{C}(\phi)\|_{\mathscr{C}^{k,\alpha}(\D)}\leq 6\|\phi\|_{\mathscr{C}^{k,\alpha}(\T)},
$$
which completes the proof.
\end{proof}

\section{Boundary equation}\label{Secboundaryequation}
This section focuses  on studying the second equation \eqref{boundaryeq} concerning the boundary equation and prove that we can parametrize the solutions in $(\Omega,f,\Phi)$ close to the initial radial solution $(f_0,\hbox{Id})$, with $f_0$ being a radial profile,  through a mapping $(\Omega,f)\mapsto \Phi=\mathcal{N}(\Omega,f)$.  We will deal with the unknowns $g$ and $\phi$ defined by
$$
f=f_0+g, \, \, \Phi=\hbox{Id}+\phi.
$$
Equation \eqref{boundaryeq} can be written in the following way
\begin{equation}\label{F}
 F(\Omega,g,\phi)(w)\triangleq\textnormal{Im}\Bigg[ \Bigg(\Omega\,\overline{\Phi(w)}-\frac{1}{2\pi}\int_{\D}\frac{f(y)}{{\Phi(w)}-{\Phi(y)}}|\Phi^\prime(y)|^2 dA(y)\Bigg){\Phi^\prime(w)}{w}\Bigg]
 =0,
\end{equation}
for any $ w\in\T $. Notice that  from this formulation we can retrieve the fact that 
$$
 F(\Omega,0,0)=0, \quad \forall\,\Omega\in\R,
$$
which is compatible with the fact that any radial initial data leads to a stationary solution of the Euler equations. Indeed, this identity follows from  Proposition \ref{integralsprop} which implies that
$$
 \frac{1}{2\pi}\int_{\D}\frac{f_0(y)}{{w}-{y}} dA(y)=  \overline{w}\int_0^{1} s f_0(s)ds, \quad \forall \, w\in \T.
 $$
The idea to solve the nonlinear equation \eqref{F}  is to apply  the Implicit Function Theorem.
Define the open balls
\begin{eqnarray}\label{BAL}
&& \left\{ \begin{array}{lll}B_{ \mathscr{C}_{s}^{k,\alpha}}(g_0,\varepsilon)&=&\Big\{g\in   \mathscr{C}_{s}^{k,\alpha}({\D})\quad \textnormal{s.t.}\quad \|g-g_0\|_{k,\alpha}< \E  \Big\},\\
 B_{\mathscr{H} \mathscr{C}^{k,\alpha}}(\phi_0,\E)&=&\Big\{\phi\in  \mathscr{H} \mathscr{C}^{k,\alpha}({\D})\quad \textnormal{s.t.}\quad \|\phi-\phi_0\|_{k,\alpha}< \E \Big\},
\end{array} \right.
\end{eqnarray}
for $\E>0$, $g_0\in  \mathscr{C}_s^{k,\alpha}({\D})$ and  $\phi_0\in  \mathscr{C}^{k,\alpha}({\D})$.
The first result  concerns the well--definition and regularity of the functional  $F$ introduced in $\eqref{F}$. 
\begin{proposition}\label{PropReqX}
Let $\E\in (0,1)$, then
$F:\R\times B_{ \mathscr{C}_{s}^{1,\alpha}}(0,\E)\times B_{\mathscr{H} \mathscr{C}^{2,\alpha}}(0,\E)\longrightarrow \mathscr{C}^{1,\alpha}_a({\T})
$ 
 is well-defined and  of class $\mathscr{C}^1$. 
\end{proposition}
\begin{remark}
If $\phi\in B_{\mathscr{H} \mathscr{C}^{k,\alpha}}(0,\E)$ with $\E <1$, then $\Phi=\textnormal{Id}+\phi$ is conformal and bi-Lipschitz. 
\end{remark}
\begin{proof}
Let us show that $F\in \mathscr{C}^{1,\alpha}(\T)$. Since $\overline{\Phi},\Phi'\in \mathscr{C}^{1,\alpha}(\T)$ it remains to study the integral term. This is a consequence of Lemma \ref{lemkernel2} in  Appendix \ref{Appotentialtheory}, which yields $F\in \mathscr{C}^{1,\alpha}(\T)$.

Let us turn to the persistence of the symmetry. According to \eqref{Persis} one has to check that
\begin{equation}\label{Sym5}
F(\Omega, g, \phi)(\overline{w})=-F(\Omega, g, \phi)({w}), \quad \forall \, w\in\T.
\end{equation} 
Using the symmetry properties of the density and the conformal mapping we write
\begin{eqnarray*}
F(\Omega, g, \phi)(\overline{w})&=&\textnormal{Im}\Bigg[ \Bigg(\Omega\overline{\Phi(\overline{w})}-\frac{1}{2\pi}\int_{\D}\frac{f(y)}{{\Phi(\overline{w})}-{\Phi(y)}}|\Phi^\prime(y)|^2 dA(y)\Bigg){\Phi^\prime(\overline{w})}\overline{w}\Bigg]\\
&=&\textnormal{Im}\Bigg[ \Bigg(\Omega{\Phi({w})}-\frac{1}{2\pi}\int_{\D}\frac{f(y)}{\overline{\Phi({w})}-{\Phi(y)}}|\Phi^\prime(y)|^2 dA(y)\Bigg)\overline{\Phi^\prime({w})}\overline{w}\Bigg]\\
&=&\textnormal{Im}\Bigg[ \Bigg(\Omega{\Phi({w})}-\frac{1}{2\pi}\overline{\int_{\D}\frac{f(y)}{{\Phi({w})}-{\Phi(y)}}|\Phi^\prime(y)|^2 dA(y)}\Bigg)\overline{\Phi^\prime({w})}\overline{w}\Bigg].
\end{eqnarray*}
Therefore, we get \eqref{Sym5}.
This concludes that $F(\Omega,g,\phi)$ is in $\mathscr{C}^{1,\alpha}_a(\T)$.  Notice that the dependance with respect to $\Omega$ is smooth and we will focus on the G\^ateaux derivatives of  $F$ with respect to $g$ and $\phi$.  Straightforward computations lead to 
\begin{eqnarray}\label{DiffX1}
\nonumber D_g F(\Omega,g,\phi)h(w)&=&-\textnormal{Im}\Bigg[ \frac{w\Phi'(w)}{2\pi}\int_\D\frac{h(y)}{\Phi(w)-\Phi(y)}|\Phi'(y)|^2dA(y) \Bigg],\\
\nonumber D_\phi F(\Omega,g,\phi)k(w)&=&\textnormal{Im}\Bigg[\Omega\overline{k(w)}\Phi'(w)w+\Omega\overline{\Phi(w)}k'(w)w\\
\nonumber&&-\frac{wk'(w)}{2\pi}\int_\D\frac{f(y)}{\Phi(w)-\Phi(y)}|\Phi'(y)|^2dA(y)\\
\nonumber &&+\frac{w\Phi'(w)}{2\pi}\int_\D\frac{k(w)-k(y)}{(\Phi(w)-\Phi(y))^2}f(y)|\Phi'(y)|^2dA(y)\\
&&-\frac{w\Phi'(w)}{\pi}\int_\D\frac{f(y)}{\Phi(w)-\Phi(y)}\textnormal{Re}\bigg[\overline{\Phi'(y)}k'(y)\bigg]dA(y)\Bigg].
\end{eqnarray}
Let us use the operator $\mathscr{F}[\Phi]$ defined in \eqref{operatorF}. Although in Lemma \ref{lemkernel2} $\mathscr{F}[\Phi]$ is defined in $\D$ we can extend it up to the boundary $\overline{\D}$ getting the same result. Hence, all the above expressions can be written through this operator as
\begin{eqnarray*}
\nonumber D_g F(\Omega,g,\phi)h(w)&=&-\textnormal{Im}\Bigg[ \frac{w\Phi'(w)}{2\pi}\mathscr{F}[\Phi](h)(w) \Bigg],\\
\nonumber D_\phi F(\Omega,g,\phi)k(w)&=&\textnormal{Im}\Bigg[\Omega\overline{k(w)}\Phi'(w)w+\Omega\overline{\Phi(w)}k'(w)w-\frac{wk'(w)}{2\pi}\mathscr{F}[\Phi](f)(w)\\
\nonumber &&\left.+\frac{w\Phi'(w)}{2\pi}\int_\D\frac{k(w)-k(y)}{(\Phi(w)-\Phi(y))^2}f(y)|\Phi'(y)|^2dA(y)\right.\\
&&-\left.\frac{w\Phi'(w)}{\pi}\mathscr{F}[\Phi]\left(\frac{\textnormal{Re}\left[\overline{\Phi'(\cdot)}k'(\cdot)\right]}{|\Phi'(\cdot)|^2}\right)(w)\right].
\end{eqnarray*}
Since $\frac{\textnormal{Re}\left[\overline{\Phi'(\cdot)}k'(\cdot)\right]}{|\Phi'(\cdot)|^2}, \Phi', \overline{\Phi},k'\in \mathscr{C}^{1,\alpha}(\D)$ and are continuous with respect to $\Phi$, Lemma \ref{lemkernel2} entails that all the terms except the integral one
lie in $\mathscr{C}^{1,\alpha}(\D)$ and they are continuous with respect to $\Phi$. The continuity with respect to $f$ comes also from the same result. Note that although our unknowns are $(g,\phi)$, studying the continuity with respect to $(g,\phi)$ is equivalent to doing it with respect to $(f,\Phi)$.
We  shall now focus our attention on the integral term by splitting it as follows
\begin{eqnarray*}
\int_\D\frac{k(w)-k(y)}{(\Phi(w)-\Phi(y))^2}f(y)|\Phi'(y)|^2dA(y)&=&\int_\D \frac{(k(w)-k(y))(f(y)-f(w))}{(\Phi(w)-\Phi(y))^2}|\Phi'(y)|^2dA(y)\\
&&+f(w)\int_\D\frac{k(w)-k(y)}{(\Phi(w)-\Phi(y))^2}|\Phi'(y)|^2dA(y)\\
&\triangleq&\mathscr{J}_1[\Phi]f(z)+f(w)\mathscr{J}_2[\Phi].
\end{eqnarray*}
First, we deal with $\mathscr{J}_1[\Phi]$. Clearly
$$
|\mathscr{J}_1[\Phi](w)|\leq C\|f\|_{\mathscr{C}^{1,\alpha}(\D)},
$$
and we define
\begin{eqnarray*}
K(w,y)&\triangleq&\nabla_w \frac{(k(w)-k(y))(f(y)-f(w))}{(\Phi(w)-\Phi(y))^2}\\
&=&-\frac{(k(w)-k(y))}{(\Phi(w)-\Phi(y))^2}\nabla_w f(w) +(f(y)-f(w))\nabla_w \frac{(k(w)-k(y))}{(\Phi(w)-\Phi(y))^2}\\
&\triangleq&-\nabla_w f(w)K_1(w,y)+K_2(w,y).
\end{eqnarray*}
Using the same argument as in  \eqref{K2}, we can check that $K_1$ and $K_2$ verify both the hypotheses of Lemma \ref{lemkernel1}. This implies that $\mathscr{J}_1[\Phi]$ lies in $\mathscr{C}^{1,\alpha}(\D)$. Taking two conformal maps $\Phi_1$ and $\Phi_2$ and estimating $\mathscr{J}_1[\Phi_1]-\mathscr{J}_1[\Phi_2]$, we find integrals similar to those treated in Lemma \ref{lemkernel1}.

Concerning the second integral $\mathscr{J}_2[\Phi]$, which seems to be more singular, we use the Cauchy--Pompeiu's formula \eqref{Cauchy-Pompeiu} to find
$$
\int_\D\frac{k(w)-k(y)}{\Phi(w)-\Phi(y)}|\Phi'(y)|^2dA(y)=\frac{1}{2i}\int_\T \frac{k(w)-k(\xi)}{\Phi(w)-\Phi(\xi)}\overline{\Phi(\xi)}\Phi'(\xi)d\xi.
$$
Differentiating it, we deduce 
\begin{eqnarray*}
\int_\D\frac{k(w)-k(y)}{(\Phi(w)-\Phi(y))^2}|\Phi'(y)|^2dA(y)&=&\frac{k'(w)\overline{\Phi(w)}\pi}{2\Phi'(w)}+\frac{1}{2i}\int_\T \frac{k(\xi)-k(w)}{(\Phi(\xi)-\Phi(w))^2}\overline{\Phi(\xi)}\Phi'(\xi)d\xi\\
&\triangleq&\frac{k'(w)\overline{\Phi(w)}\pi}{2\Phi'(w)}+\frac{1}{2i}\mathscr{T}[\Phi](w).
\end{eqnarray*}
The first term is in $\mathscr{C}^{1,\alpha}(\T)$ and is clearly continuous with respect to $\Phi$. Integration by parts in the second term $\mathscr{T}[\Phi]$ leads to
$$
\mathscr{T}[\Phi](w)=-\int_\T \frac{k'(\xi)\overline{\Phi(\xi)}}{\Phi(\xi)-\Phi(w)}d\xi+\int_\T \frac{k(\xi)-k(w)}{\Phi(\xi)-\Phi(w)}\overline{\xi}^2\overline{\Phi'(\xi)}d\xi.
$$
Differentiating and integrating it  by parts again one obtain the following expression
\begin{eqnarray*}
\mathscr{T}[\Phi]'(w)&=&\Phi'(w)\bigintss_\T \frac{{\partial_\xi\left(\frac{k'(\xi)\overline{\Phi(\xi)}}{\Phi'(\xi)}\right)}}{\Phi(\xi)-\Phi(w)}d\xi-\Phi'(w)\bigintss_\T \frac{\partial_\xi\left(\frac{(k(\xi)-k(w))\overline{\xi}^2\overline{\Phi'(\xi)}}{\Phi'(\xi)}\right)}{\Phi(\xi)-\Phi(w)}d\xi\\
&=&\Phi'(w)\bigintss_\T \frac{{\partial_\xi\left(\frac{k'(\xi)\overline{\Phi(\xi)}}{\Phi'(\xi)}\right)}}{\Phi(\xi)-\Phi(w)}d\xi-\Phi'(w)\bigintss_\T \frac{\frac{k'(\xi)\overline{\xi}^2\overline{\Phi'(\xi)}}{\Phi'(\xi)^2}}{\Phi(\xi)-\Phi(w)}\Phi'(\xi)d\xi\\
&&-\Phi'(w)\bigintss_\T \frac{(k(\xi)-k(w))\partial_\xi\left(\frac{\overline{\xi}^2\overline{\Phi'(\xi)}}{\Phi'(\xi)}\right)}{\Phi(\xi)-\Phi(w)}d\xi\\
&=&\Phi'(w)\mathscr{I}[\Phi]\left(\frac{\partial\left(\frac{k'(\cdot)\overline{\Phi(\cdot)}}{\Phi'(\cdot)}\right)}{\Phi'(\cdot)}\right)(w)-\Phi'(w)\mathscr{I}[\Phi]\left(\frac{k'(\cdot)\overline{(\cdot)}^2\overline{\Phi'(\cdot)}}{\Phi'(\cdot)^2}\right)(w)\\
&&-\Phi'(w)\mathscr{I}[\Phi]\left(k(\cdot)\partial\left(\frac{\overline{\cdot}^2\overline{\Phi'(\cdot)}}{\Phi'(\cdot)}\right)\right)(w)+\Phi'(w)k(w)\mathscr{I}[\Phi]\left(\partial\frac{\overline{(\cdot)}^2\overline{\Phi'(\cdot)}}{\Phi'(\cdot)}\right)(w),
\end{eqnarray*}
where $\mathscr{I}[\Phi]$ is the operator defined in \eqref{operatorI}. Since the functions
$$
\frac{\partial\left(\frac{k'(\cdot)\overline{\Phi(\cdot)}}{\Phi'(\cdot)}\right)}{\Phi'(\cdot)},\quad\frac{k'(\cdot)\overline{(\cdot)}^2\overline{\Phi'(\cdot)}}{\Phi'(\cdot)^2},\quad k,\quad\partial\left(\frac{\overline{(\cdot)}^2\overline{\Phi'(\cdot)}}{\Phi'(\cdot)}\right)\in\mathscr{C}^{0,\alpha}(\D),
$$
and are continuous with respect to $\Phi$, we can use Lemma \ref{lemkernel3}. Hence, these terms lie in $\mathscr{C}^{0,\alpha}(\D)$. Moreover, the same argument gives us the continuity with respect to $\Phi$. To conclude, we use the fact that the G\^ateaux derivatives are continuous with respect to $(g,\phi)$ and so they are in fact Fr\'echet derivatives. \end{proof}

The next task is to  implement the Implicit Function Theorem in order to solve the boundary equation \eqref{F} through a two-parameters  curve solutions in infinite-dimensional spaces.  
Given a radial function $f_0\in \mathscr{C}^{1,\alpha}({\D})$, we associate to it  the singular set
\begin{equation}\label{Interv1}
\mathcal{S}_{\textnormal{sing}}\triangleq\bigg\{\widehat{\Omega}_n\triangleq\int_0^1sf_0(s)ds-\frac{n+1}{n}\int_0^1s^{2n+1}f_0(s)ds,\quad    \forall n\in\N^\star\cup\{+\infty\} \bigg\}.
\end{equation}
This terminology will be later justified in the proof of the next proposition. Actually, this set corresponds to the location of  the points $\Omega$ where the partial  linearized operator $\partial_\phi F(\Omega,0,0)$ is not invertible.
Let us establish the  following result.
\begin{proposition}\label{propImpl}
Let  $f_0:\overline{\D}\to \R$ be a radial function in $\mathscr{C}^{1,\alpha}(\D)$. Let $I$ be an open interval such that  $\overline{I}\subset \R\backslash \mathcal{S}_{\textnormal{sing}}$.
Then, there exists $\varepsilon>0$ and  a $\mathscr{C}^1$ function  
$$\mathcal{N}: I\times B_{\mathscr{C}_{ s}^{1,\alpha}}(0,\varepsilon)\longrightarrow B_{\mathscr{H}\mathscr{C}^{2,\alpha}}(0,\varepsilon),
$$
 with the following property:
$$
 F(\Omega,g,\phi)=0\Longleftrightarrow \phi=\mathcal{N}(\Omega,g),
$$
 for any $ (\Omega,g,\phi)\in I\times B_{\mathscr{C}_{s}^{1,\alpha}}(0,\varepsilon)\times B_{\mathscr{H}\mathscr{C}^{2,\alpha}}(0,\varepsilon)$.
In addition,  we obtain the identity 
$$
D_g\mathcal{N}(\Omega, 0)h(z)=z\sum_{n\geq 1} A_n z^{n},
$$
for any $h\in \mathscr{C}_{s}^{1,\alpha}({\D})$, with  $h(re^{i\theta})=\displaystyle{\sum_{n\geq 0}h_n(r)\cos(n\theta)},$ and 
\begin{equation}\label{An}
A_n={\frac{\displaystyle{\int_0^1s^{n+1}h_n(s)ds}}{2n\big(\widehat{\Omega}_n-\Omega\big)}},
\end{equation}
for any $n\geq 1$ where $\widehat{\Omega}_n$ is defined in \eqref{Interv1}.  Moreover, we have
\begin{equation}\label{ContinQ1}
\|\mathcal{N}(\Omega, 0)h\|_{\mathscr{C}^{2,\alpha}(\D)}\leq  C\|h\|_{\mathscr{C}^{1,\alpha}(\D)}.
\end{equation}
\end{proposition}
\begin{remark}
From  the definition of the function space $\mathscr{C}^{1,\alpha}_s(\D)$ we are adding also a radial perturbation of the initial radial part given by the first mode $n=0$. However, from the expression of $D_g\mathcal{N}(\Omega, 0)h(z)$ the first frequency  disappears and the sum starts at $n=1$. This is an expected fact because   $(\Omega,g,0)$ is a  solution of $F(\Omega,g,\phi)$  for any radial smooth function $g$. This means that $\mathcal{N}(\Omega, g)=0$, and hence  $D_g\mathcal{N}(\Omega,0)h$ is vanishing when  $h$ is radial.  
\end{remark}
\begin{proof}
Applying the Implicit Function Theorem consists in checking that 
$$
D_\phi F(\Omega, 0, 0):\mathscr{H}\mathscr{C}^{2,\alpha}(\D)\longrightarrow \mathscr{C}^{1,\alpha}_a(\T),
$$ 
is an isomorphism. A combination of \eqref{DiffX1} with Proposition \ref{integralsprop} allow us to compute explicitly the differential of $F(\Omega, g,\phi)$ on the initial solution as follows
\begin{eqnarray*}
D_\phi F(\Omega, 0, 0)k(w)&=&\textnormal{Im}\Bigg[\Omega\overline{k(w)}w+\Omega\overline{w}k'(w)w-\frac{wk'(w)}{2\pi}\int_\D\frac{f_0(y)}{w-y}dA(y)\\
&&\left.+\frac{w}{2\pi}\int_\D\frac{k(w)-k(y)}{(w-y)^2}f_0(y)dA(y)\right.-\frac{w}{\pi}\int_\D\frac{f_0(y)}{w-y}\textnormal{Re}[k'(y)]dA(y)\Bigg]\\
&=&\sum_{n\geq 1}a_n\, \textnormal{Im}\Bigg[\Omega \overline{w}^{n}+\Omega(n+1)w^n-(n+1)w^n\int_0^1sf_0(s)ds\\
&&+w^{n}\int_0^1sf_0(s)ds-(n+1)\overline{w}^{n}\int_0^1s^{2n+1}f_0(s)ds\Bigg]\\
&=&\sum_{n\geq 1}a_nn\left\{\Omega-\int_0^1sf_0(s)ds+\frac{n+1}{n}\int_0^1s^{2n+1}f_0(s)ds\right\}\sin(n\theta).
\end{eqnarray*}
Similarly, we get 
\begin{eqnarray}\label{DEDA}
\nonumber D_g F(\Omega,0,0)h(w)&=&-\textnormal{Im} \Bigg[\frac{w}{2\pi}\int_\D \frac{h(y)}{w-y}dA(y)\Bigg]
=-\frac{\pi}{2\pi}\sum_{n\geq 1} \textnormal{Im} \Bigg[w \overline{w}^{n+1}\int_0^1s^{n+1}h_n(s)ds\Bigg]\\
&=&\frac{1}{2}\sum_{n\geq 1} {\int_0^1s^{n+1}h_n(s)ds}\sin(n\theta),
\end{eqnarray}
where 
$$z\mapsto k(z)=\sum_{n\geq1} a_n z^{n+1}\in \mathscr{H}\mathscr{C}^{2,\alpha}(\D)\quad\hbox{and}\quad z\mapsto h(z)=\sum_{n\geq0} h_n(r) \cos(n\theta)\in \mathscr{C}^{1,\alpha}_s(\D),
$$ 
are given as in \eqref{phi} and \eqref{g}, respectively. Then,  we have that $D_\phi F(\Omega,0,0)$ is one--to--one linear mapping and is continuous  according to  Proposition \ref{PropReqX}, for any  $\Omega\in \overline{I}$. Using the Banach Theorem it suffices to check that this mapping is onto. Notice that at the formal level the inverse operator can be easily computed  from the expression of $D_\phi F(\Omega,0,0)$ and it is given by
 \begin{equation}\label{Tbs10}
D_\phi F(\Omega,0,0)^{-1}\rho(z)=z\sum_{n\geq 1}\frac{\rho_n}{n\big(\Omega-\widehat{\Omega}_n\big)}z^{n},
\end{equation}
for any $\displaystyle{\rho(e^{i\theta})=\sum_{n\geq 1}\rho_n\sin(n\theta)}$. Thus the problem reduces to check that
$
D_\phi F(\Omega,0,0)^{-1}\rho\in \mathscr{H}\mathscr{C}^{2,\alpha}({\D}).
$
 First, we will prove that this function is holomorphic inside the unit disc $\D$. For this purpose we use that
 \begin{equation}\label{Tbs1}
 \rho_n=\frac{1}{\pi}\int_0^{2\pi}\rho(e^{i\theta})\sin(n\theta) d\theta.
\end{equation}
 Since $\rho\in L^\infty(\T)$,  we obtain that the coefficients sequences  $(\rho_n)\in \ell^\infty.$ Using the facts that $\lim_{n\to\infty}\widehat\Omega_n=\widehat\Omega_\infty$ and that $\Omega$ is far away from the singular set, then we deduce that the Fourier coefficients  of $D_\phi F(\Omega,0,0)^{-1}\rho$ are bounded. Consequentely, this function is holomorphic inside the unit disc. It remains to check that this function belongs to  $\mathscr{C}^{2,\alpha}(\D)$. By virtue of    Lemma \ref{equivnorms}  it is enough to check that the restriction on the boundary belongs to $\mathscr{C}^{1,\alpha}(\T)$.  First, we must notice that if $\rho \in \mathscr{C}^{1,\alpha}(\T)$, then 
$$
w\mapsto \rho_+(w)\triangleq \sum_{n\geq 1}\rho_nw^n\in \mathscr{C}^{1,\alpha}(\T).
$$
For this purpose, let us write ${\rho}$ in the form
$$
\rho(w)=\frac{1}{2i}\sum_{n\in\mathbb{Z}}\rho_nw^n,\quad\hbox{with}\quad \rho_{-n}=-\rho_n.
$$
Hence, $\rho_+$ is nothing but  the  Szeg\"{o} projection of $\rho$, which is continuous on $\mathscr{C}^{1,\alpha}(\T).$  Note that this latter property is based upon the fact that $\rho_+$ can be expressed from  $\rho$ through the Cauchy integral operator
$$
\rho_+(w)=\frac{1}{\pi}\int_\T \frac{\rho(\xi)}{\xi-w}d\xi,
$$ 
and one may  use  $T(1)-$Theorem of Wittmann for H\"{o}lder spaces, see for instance \cite[Theorem 2.1]{Wittmann} and \cite[page 10]{HmidiMateuVerdera} or Lemma \ref{lemkernel3}. Secondly, we will prove that $D_\phi F(\Omega,0,0)^{-1}\rho\in \mathscr{C}^{2,\alpha}(\T)$. We define
$$
D_\phi F(\Omega,0,0)^{-1}\rho(w)\triangleq wq(w).
$$
Let us show that   $q$ is bounded. Using \eqref{Tbs1} and  integration by parts, we have
$$
|\rho_n|\leq 2\frac{\|\rho^\prime\|_{L^\infty(\T)}}{n},
$$
which implies that $q$ is bounded. To prove higher regularity, we write $q$ as a convolution
$$
q=\rho_+\star K_1,
$$
where 
$$K_1(w)=\sum_{n\geq 1} \frac{w^n}{n\big(\Omega-\widehat{\Omega}_n\big)}.$$
Since $\rho_+\in \mathscr{C}^{1,\alpha}$, we just need to check  that $K_1\in L^1$. To do that, we use the Parseval's identity, which provides that $K_1\in L^2(\T)$:
\begin{eqnarray*}
\|K_1\|_2^2=\sum_{n\geq 1} \frac{1}{n^2\big(\Omega-\widehat{\Omega}_n\big)^2}
\leq C\sum_{n\geq 1} \frac{1}{n^2}
<+\infty,
\end{eqnarray*}
where $C$ is a constant connected to the distance between $\Omega$ and  the singular set $\mathcal{S}_{sing}$ defined in \eqref{Interv1}.  To study its derivative, let us write it as
\begin{eqnarray*}
q'(w)=\overline{w}\sum_{n\geq 1}\frac{\rho_n}{\Omega-\widehat{\Omega}_n}w^n
=\overline{w}\Bigg[\sum_{n\geq 1}\frac{\rho_n}{\beta}w^n+\sum_{n\geq 1}\rho_n\left(\frac{1}{\beta+u_n}-\frac{1}{\beta}\right)w^n\Bigg]
\triangleq\overline{w}\Bigg[\frac{1}{\beta}\rho_+(w)+S\Bigg],
\end{eqnarray*}
where
$$
\beta=\Omega-\widehat{\Omega}_\infty, \quad u_n=\frac{n+1}{n}\int_0^1s^{2n+1}f_0(s)ds.
$$
From the foregoing discussion, we have seen that $\rho_+\in \mathscr{C}^{1,\alpha}(\T)$. As to the term  $S$, it  can be written in convolution form
$$
S=\rho_+\star K_2,
$$
with
$$
K_2(w)=-\sum_{n\geq 1} \frac{u_n}{(\beta+u_n)\beta}w^n.
$$
Since $\rho_+\in \mathscr{C}^{1,\alpha}(\T)$,  then we just need to check that $K_2\in L^1(\T)$ to conclude. Using Parseval's identity we have that $K_2\in L^2$ because
\begin{eqnarray*}
\|K_2\|_2^2=\frac{1}{\beta^2}\sum_{n\geq 1} \frac{u_n^2}{(\beta+u_n)^2}&\leq& C\sum_{n\geq 1} \frac{(n+1)^2}{n^2} \left(\int_0^1s^{2n+1}f_0(s)ds\right)^2\\
&\leq& C\sum_{n\geq 1} \frac{(n+1)^2}{n^2(2n+1)^2} 
<  +\infty.
\end{eqnarray*}
This achieves that $ D_\phi F(\Omega,0,0)^{-1}\rho\in\mathscr{H}\mathscr{C}^{2,\alpha}({\D})$ and consequently the linearized operator $ D_\phi F(\Omega,0,0)$  is an isomorphism.
Hence, the Implicit Function Theorem can be used and it ensures  the existence of a  $\mathscr{C}^1$--function $\mathcal{N}$ such that 
$$
F(\Omega,g,\phi)=0\Longleftrightarrow \phi=\mathcal{N}(\Omega,g),
$$
for any  $(\Omega,g,\phi)\in I\times B_{\mathscr{C}_{s}^{1,\alpha}}(0,\varepsilon)\times B_{\mathscr{H}\mathscr{C}^{2,\alpha}}(0,\varepsilon)$. Differentiating with respect to $g$, we obtain 
\begin{eqnarray*}
D_g F(\Omega, g, \mathcal{N}(\Omega,g))=\partial_g F(\Omega, g, \phi)+\partial_\phi F(\Omega,g,\phi)\circ \partial_g\mathcal{N}(\Omega,g)=0,
\end{eqnarray*}
which yields
$$
\partial_g \mathcal{N}(\Omega,0)h(z)=-\partial_\phi F(\Omega,0,0)^{-1}\circ \partial_g F(\Omega,0,0)h(z).
$$
Then, using \eqref{Tbs10} and \eqref{DEDA}, straightforward computations  show that
$$
D_g\mathcal{N}(\Omega, 0)h(z)=-z\sum_{n\geq 1} \frac{\int_0^1s^{n+1}h_n(s)ds}{2n\big(\Omega-\widehat{\Omega}_n\big)} z^{n}.
$$
This concludes the proof of the announced result. \end{proof}

\section{Density equation}\label{Secdensityeq}
This section aims at studying the density equation \eqref{densityeq} in order to get non radial rotating solutions via the Crandall--Rabinowitz Theorem. We will reformulate it in a more convenient way since we are not able to use the original expression \eqref{densityeq} due to the  structural defect on  its linearized operator as it has been pointed out previously. We must have in mind that under suitable assumptions, the conformal map is recovered from the angular velocity $\Omega$ and the density function via Proposition \ref{propImpl}.

\subsection{Reformulation of the density equation}
Taking  an initial data in the form \eqref{rotatingsol} and noting that if  the density $f$ is fixed close to $f_0$ and $\Omega$ does not lie in  the singular set $\mathcal{S}_{\textnormal{sing}}$, then the conformal mapping is uniquely determined as a consequence of Proposition \ref{propImpl}. Now, we turn to the analysis of  the first equation of \eqref{rotatingeq2d}
that we intend to solve  for a restricted class of  initial densities. The strategy to implement it is to look for  solutions satisfying  the specific  equation 
\begin{eqnarray}\label{specialsolution}
\nonumber\nabla (f\circ\Phi^{-1})(x)&=&\mu\left(\Omega, (f\circ\Phi^{-1})(x)\right) (v(x)-\Omega x^\perp)^\perp\\
&=& \mu\left(\Omega, (f\circ\Phi^{-1})(x)\right) (v^\perp(x)+\Omega x),\quad \forall x\in D,
\end{eqnarray}
for some scalar function $\mu$. One can easily check that any solution of \eqref{specialsolution} is a solution of the initial density equation \eqref{rotatingeq2d} but the reversed is not in general true. Remark that from this latter equation we are looking for  particular solutions due to the precise dependence of the scalar function $\mu$ with respect to $f$.
The scalar function $\mu$ must be fixed in such a way that the radial profile $f_0$, around which we look for non trivial solutions, is also a solution of  \eqref{specialsolution}. Therefore, for any initial radial profile candidate to be bifurcated, we will obtain a different density equation. Notice also that it is not necessary  in general to impose to $\mu$ to be well--defined  on $\R$ but just on some open interval containing the image of  $\overline \D$ by $f_0$.

 Now, let us show how to construct  concretely the function $\mu$. By virtue of Proposition \ref{propImpl}, the associated conformal map to any radial profile is the identity map. Therefore,  it is obvious that a smooth  radial profile $f_0$ is a solution of \eqref{specialsolution} if and only if
\begin{eqnarray*}
f_0'(r)\frac{z}{r}&=&{\mu}(\Omega,f_0(z))  \left[-\frac{1}{2\pi}\int_\D \frac{z-y}{|z-y|^2}f_0(y)dA(y)+\Omega z\right]\\
&=&{\mu}(\Omega,f_0(z)) \left[-\frac{1}{2\pi}\overline{\int_\D\frac{f_0(y)}{z-y}dA(y)}+\Omega z\right]\\
&=&{\mu}(\Omega,f_0(z))\left[-\frac{1}{r^2}\int_0^rsf_0(s)ds+\Omega\right]z,\quad \forall z\in\D, \,r=|z|,
\end{eqnarray*}
 where we have used the explicit computations given in Proposition \ref{integralsprop}. Thus, we infer  that the function ${\mu}$ must satisfy  the compatibility condition
\begin{eqnarray}\label{mutrivial}
{\mu}(\Omega,f_0(r))=\frac{1}{r}\frac{f_0'(r)}{\Omega-\frac{1}{r^2}\int_0^rsf_0(s)ds},\quad \forall\,  r\in (0,1].
\end{eqnarray}
We emphasize  that not all the radial profiles verify the last equation. In fact, we can violate this equation by working with non monotonic profiles. Taking $f_0$ verifying \eqref{mutrivial}, let us go through the above procedure and see how to reformulate the density equation. 
Consider the function
\begin{eqnarray}\label{M}
\mathcal{M}_{f_0}(\Omega,\ttt)=\int_{t_0}^\ttt\frac{1}{\mu(\Omega,s)}ds,
\end{eqnarray}
for some $ t_0\in\R$.  We use the subscript $f_0$ in order to stress that the above function depends on the choice of the initial profile $f_0$. This rigidity is very relevant in our study  and enables us to include the structure of the solution into the formulation. By this way, we expect  to remove the pathological behavior  of the old formulation and  to prepare the problem for the bifurcation arguments.   From the expression of the velocity field, it is obvious that
$$
v^\perp(x)+\Omega x=-\nabla \left(\frac{1}{2\pi}\int_D \log|x-y|(f\circ\Phi^{-1})(y)dA(y)-\frac12\Omega|x|^2\right).
$$
Since  $D$ is a simply connected domain, then integrating   \eqref{specialsolution}  yields to the equivalent form
$$
\mathcal{M}_{f_0}\big(\Omega,(f\circ\Phi^{-1})(x)\big)+\frac{1}{2\pi}\int_D \log|x-y| (f\circ\Phi^{-1})(y)dA(y)-\frac12\Omega|x|^2=\lambda, \quad \forall x\in D,
$$
for some constant $\lambda$. Using  a change of variable through the conformal map $\Phi:\D\rightarrow D$, we obtain the equivalent formulation in the unit disc
\begin{align}\label{specialsolution2}
\mathcal{M}_{f_0}(\Omega,f(z))+\frac{1}{2\pi}\int_\D \log|\Phi(z)-\Phi(y)|f(y)|\Phi'(y)|^2dA(y)-\frac12\Omega|\Phi(z)|^2=\lambda,\quad \forall z\in \D.
\end{align}
It remains to   fix the constant $\lambda$ by using that the initial radial profile should be  a solution of \eqref{specialsolution2}. Thus the last integral identity in Proposition \ref{integralsprop} entails that
\begin{eqnarray}\label{generalalpha}
\lambda=\mathcal{M}_{f_0}(f_0(r))-\int_r^1\frac{1}{\tau}\int_0^\tau sf_0(s)dsd\tau-\frac12\Omega r^2.
\end{eqnarray}
Notice that $\lambda$ does not depend on $r$ since  $f_0$ verifies  \eqref{mutrivial}.
Then, we finally arrive at the following reformulation for the density equation
$$
G_{f_0}(\Omega,g,\phi)(z)\triangleq\mathcal{M}_{f_0}(\Omega,f(z))+\frac{1}{2\pi}\int_\D \log|\Phi(z)-\Phi(y)|f(y)|\Phi'(y)|^2dA(y)-\frac{\Omega}{2}|\Phi(z)|^2-\lambda
=0,
$$
for any $z\in\D. $ The above expression yields 
$$G_{f_0}(\Omega,0,0)=0,\quad\forall \Omega\in\R.
$$  Thanks to Proposition \ref{propImpl}, the conformal mapping is parametrized  outside the singular set by $\Omega$ and $g$ and thus the  equation for  the density becomes  
 \begin{eqnarray}{\label{densityEq}}
 \widehat{G}_{f_0}(\Omega,g)\triangleq G_{f_0}(\Omega,g, \mathcal{N}(\Omega,g))=0.
\end{eqnarray}
Next, let us  analyze the constraint   \eqref{mutrivial}, for  some particular examples. Since we are looking for smooth solutions, it is  convenient to deal with smooth radial profiles. Then, one stands
$$
f_0(r)=\widehat{f_0}(r^2),
$$ 
and thus  \eqref{mutrivial} becomes 
\begin{eqnarray}\label{mutrivial01}
{\mu}(\Omega,\widehat{f_0}(r))=\frac{4 r\widehat{f_0}^\prime(r)}{2\Omega r-\int_0^r\widehat{f_0}(s)ds},\quad \forall\,  r\in (0,1].
\end{eqnarray}
At this stage, there are two ways to proceed. The first one is to start with $\widehat{f_0}$ and reconstruct $\mu$, and the second one is to impose $\mu$ and solve the nonlinear differential equation on $\widehat{f_0}$. This last approach is implicit and more delicate to implement. Therefore, let us proceed with the first approach and apply it to   some special examples.   
 
\subsubsection{Quadratic profiles}
The first example  is the quadratic   profile of the type
\begin{eqnarray*}
f_0(r)=Ar^2+B,
\end{eqnarray*}
where $A,B\in\R$. In this case, $\widehat{f_0}(r)=Ar+B$ and thus  \eqref{mutrivial01} agrees with 
\begin{eqnarray*}
{\mu}(\Omega,\widehat{f_0}(r))
=\frac{4Ar}{2\Omega r-\frac{Ar}{2}-Br}
=\frac{8A}{4\Omega-B-\widehat{f_0}(r)},\quad \forall r\in(0,1].
\end{eqnarray*}
Then, we find
\begin{eqnarray}\label{muQP}
\mu_{}(\Omega,\ttt)=\frac{8A}{4\Omega-B-\ttt},
\end{eqnarray}
which implies from \eqref{M} that  
$$
\mathcal{M}_{f_0}(\Omega,\ttt)=\frac{4\Omega-B}{8A} \ttt-\frac{1}{16 A} \ttt^2.
$$
Thus, using \eqref{generalalpha}, we deduce that
\begin{eqnarray}\label{valuealpha}
\nonumber\lambda&=&\frac{4\Omega-B}{8A} f_0(r)-\frac{1}{16 A} f_0(r)^2-\int_r^1\frac{1}{\tau}\int_0^\tau sf_0(s)dsd\tau-\frac{\Omega r^2}{2}\\
&=&\frac{8\Omega B-3B^2-A^2-4AB}{16 A}.
\end{eqnarray}
As we have  mentioned before, the conformal mapping is determined by $\Omega$ and $g$ and so the last equation  takes the form \eqref{densityEq}.
The subscript $f_0$ will be omitted when we refer to this equation  with the quadratic profile if there is no confusion.

Let us remark some comparison to the vortex patch problem. The case $A=0$ agrees with a vortex patch of the type $f_0(r)=B$. It was mentioned before that the boundary equation studied in Section \ref{Secboundaryequation} is the one studied in \cite{HmidiMateuVerdera} when analyzing the vortex patch problem.  Here we have one more equation in $\D$ given by the density equation. This amounts to look for solutions of the type
$$
\omega_0(x)=(B+g)\left(\Phi^{-1}(x)\right) \, {\bf{1}}_{\Phi(\D)}(x),
$$
which implies that the initial vorticity $B$ of the vortex patch is perturbed by a function that could not to be constant. However, using \eqref{specialsolution} and evaluating in  $f_0(r)=B$, for any $r\in[0,1]$, one gets that for this case $\mu\equiv 0$. Then if you perturb with $g$ the equation to be studied is
$$
\nabla ((f_0+g)\circ \Phi^{-1})(x)=0, \quad \forall x\in\Phi(\D).
$$
Using the conformal map $\Phi$ and changing the variables we arrive at 
$$
\nabla ((f_0+g)\circ \Phi^{-1})(\Phi(z))=0,\quad \forall z\in\D.
$$
By virtue of \eqref{deriv}, the above equation leads to 
$$
\nabla (f_0+g)(z)=0,
$$
which gives us that $g$ must be a constant. Hence, using our approach we get that starting with a vortex patch we just can obtain another vortex patch solution.

\subsubsection{Polynomial profiles}\label{polprofile}
The second example is to consider a general polynomial profile of the type
$$
\widehat{f_0}(r)= Ar^m+B, \quad m\in \N^\star, \quad A\geq0, \quad B\in\R.
$$
From  \eqref{mutrivial01} we obtain that
$$
\mu(\Omega,\ttt)=4m(m+1) A^{\frac1m}\frac{(\ttt-B)^{\frac{m-1}{m}}}{2\Omega(m+1)-mB-\ttt}, \quad \forall \ttt\geq B.
$$
Consequently, we find that
\[
\mathcal{M}_{f_0}(\Omega,\ttt)=\frac{1}{4m(m+1) A^{\frac1m}}\left[\left(2\Omega(m+1)-mB\right)\int_{B}^\ttt(s-B)^{\frac{1-m}{m}}ds-\int_{B}^\ttt s(s-B)^{\frac{1-m}{m}}ds\right].
\]
Remark  that for $m=1$ we recover the previous quadratic profiles. The discussion developed later about  the quadratic profile can be also extended to this polynomial profile as we shall  comment in detail in  Remark \ref{Rempoly}.

\subsubsection{Gaussian profiles}
Another example which is relevant is given by  the Gaussian distribution,
$$
\widehat{f_0}(r)=e^{A r}, \quad A\in \R^\star.
$$
Inserting  $\widehat{f_0}$ into \eqref{mutrivial01} one obtains that
$$
\mu(\Omega,\widehat{f_0}(r))=\frac{4A^2 r}{2\Omega A r-e^{Ar}+1},
$$
and thus
$$
\mu(\Omega,\ttt)=\frac{4A \ln \ttt}{1-\ttt+2\Omega \ln  \ttt}.
$$
Then the formula \eqref{M} allows us to get
\begin{equation*}
\mathcal{M}_{f_0}(\Omega,\ttt)=\frac{1}{4A}\left[2\Omega \ttt+\int_1^\ttt\frac{1-s}{\ln s} ds\right].
\end{equation*}

\subsection{Functional regularity} 
In this section, we will be interested in  the regularity of the functional  $\widehat{G}$ obtained in \eqref{densityEq} for a the quadratic profile. Notice that in the case of the quadratic profile \eqref{QuadrP}, the singular set \eqref{Interv1} becomes
\begin{equation}\label{Interv2}
\mathcal{S}_{\textnormal{sing}}=\left\{\widehat{\Omega}_n\triangleq\frac{A}{4}+\frac{B}{2}-\frac{A(n+1)}{2n(n+2)}-\frac{B}{2n},\quad  \forall n\in\N^\star\cup\{+\infty\}  \right\}.
\end{equation}
\begin{proposition}\label{Gwelldefined}
Let  $f_0$ be the quadratic profile given by \eqref{QuadrP} and $I$ be an open interval with  $\overline{I}\subset\R\backslash \mathcal{S}_{\textnormal{sing}}$. Then, there exists $\varepsilon>0$ such that
$$\widehat{G}:I\times B_{\mathscr{C}_{s}^{1,\alpha}(\D)}(0,\E)\to \mathscr{C}_{s}^{1,\alpha}(\D),
$$ 
is well--defined and of class $\mathscr{C}^1$, where $\widehat{G}$ is defined in \eqref{densityEq} and   $ B_{\mathscr{C}_{s}^{1,\alpha}(\D)}(0,\E)$ in \eqref{BAL}.
\end{proposition}
\begin{proof}
Let us show that $\widehat{G}(\Omega,g)\in\mathscr{C}^{1,\alpha}(\D)$. Clearly, $\mathcal{M}(\Omega,f)$ is polynomial in $f$ and by the algebra  structure of H\"{o}lder spaces we deduce that $\mathcal{M}(\Omega,f)\in\mathscr{C}^{1,\alpha}(\D)$. Since $\Phi\in \mathscr{C}^{2,\alpha}(\D)$, then the  only term which deserves attention is the integral one. It is clear that
$$
\int_\D \log|\Phi(\cdot)-\Phi(y)|f(y)|\Phi'(y)|^2dA(y)\in \mathscr{C}^0(\D).
$$
To estimate its derivative, we note that
\begin{eqnarray*}
\nabla_z \log|\Phi(z)-\Phi(y)|=\frac{\left(\Phi(z)-\Phi(y)\right)\overline{\Phi'(z)}}{|\Phi(z)-\Phi(y)|^2}
=\frac{\overline{\Phi'(z)}}{\overline{\Phi(z)}-\overline{\Phi(y)}},
\end{eqnarray*}
which implies that  
\begin{eqnarray*}
\nabla_z \int_\D \log|\Phi(z)-\Phi(y)|f(y)|\Phi'(y)|^2dA(y)&=&\overline{\Phi'(z)}\int_\D \frac{f(y)}{\overline{\Phi(z)}-\overline{\Phi(y)}}|\Phi'(y)|dA(y)\\
&=&\overline{\Phi'(z)}\, \overline{\mathscr{F}[\Phi](f)(z)},
\end{eqnarray*}
where the operator $\mathscr{F}[\Phi]$ is defined in \eqref{operatorF}. Thus, we can use Lemma \ref{lemkernel2} obtaining that  $\mathscr{F}[\Phi]$ belongs to  $\mathscr{C}^{1,\alpha}(\D)$. Since $\Phi'\in \mathscr{C}^{1,\alpha}(\D)$ we deduce  that the integral term of $\widehat{G}(\Omega,g)$  lies in the space  $\mathscr{C}^{2,\alpha}(\D)$ and is continuous with respect to $(f,\Phi)$. 

Let us check the symmetry property. Take $g$  and $\phi$ satisfying  $g(\overline{z})=g(z)$  and  $\phi(\overline{z})=\overline{\phi(z)}. $
It is a simple matter to verify that
 $$
\mathcal{M}(\Omega,f(\overline{z}))=\mathcal{M}(\Omega,f({z}))\quad \hbox{and}\quad |\Phi(\overline{z})|^2=|\Phi({z})|^2,\quad  \forall \, z\in\overline{\D}.
 $$ 
 For the Newtonian potential, the change of variables $y\mapsto \overline{y}$ leads to
\begin{eqnarray*}
\frac{1}{2\pi}\int_{\D}\log|\Phi(\overline{z})-\Phi(y)| f(y)|\Phi^\prime(y)|^2 dA(y)
&=&\frac{1}{2\pi}\int_{\D}\log|\overline{\Phi(z)}-\Phi(\overline{y})| f(\overline{y})|\Phi^\prime(\overline{y})|^2 dA(y)\\
&=&\frac{1}{2\pi}\int_{\D}\log|\Phi({z})-\Phi(y)| f(y)|\Phi^\prime(y)|^2 dA(y).
\end{eqnarray*}
Let us turn to the computations of the G\^ateaux derivatives, that can be computed as
\begin{eqnarray*}
D_g\widehat{G}(\Omega,g)h(z)=\partial_g G(\Omega,g,\phi)h(z)+\partial_\phi G(\Omega,g,\phi) \circ \partial_g \mathcal{N}(\Omega,g)h(z).
\end{eqnarray*}
By virtue of Proposition \ref{propImpl}, it is known that $\mathcal{N}$ is $\mathscr{C}^1$, which implies that $\partial_g\mathcal{N}(\Omega,g)$ is continuous.  G\^ateaux derivatives are given by
\begin{eqnarray*}
D_g G(\Omega,g,\phi)h(z)\hspace{-0.2cm}&=\hspace{-0.2cm}&D_g \mathcal{M}(\Omega,f(z))h(z)+\frac{1}{2\pi}\int_\D \log|\Phi(z)-\Phi(y)|h(y)|\Phi'(y)|^2dA(y)\\
\hspace{-0.2cm}&=\hspace{-0.2cm}&\frac{4\Omega-B}{8A}h(z)-\frac{1}{8A}f(z)h(z)
+\frac{1}{2\pi}\int_\D \log|\Phi(z)-\Phi(y)|h(y)|\Phi'(y)|^2dA(y),
\end{eqnarray*}
and
\begin{eqnarray*}
D_\phi G(\Omega,g,\phi)k(z)&=&\frac{\textnormal{Re}}{2\pi}\int_\D\frac{k(z)-k(y)}{\Phi(z)-\Phi(y)}f(y)|\Phi'(y)|^2dA(y)-\Omega\textnormal{Re}\left[\overline{\Phi}(z)k(z)\right]\\
&&+\frac{1}{\pi}\int_\D \log|\Phi(z)-\Phi(y)|f(y)\textnormal{Re}\left[\overline{\Phi^\prime(y)}k'(y)\right]dA(y).
\end{eqnarray*}
We focus our attention on the integral terms. The operator $\mathscr{F}[\Phi]$ in \eqref{operatorF} allows us to write 
\begin{eqnarray*}
\int_\D\frac{k(\cdot)-k(y)}{\Phi(\cdot)-\Phi(y)}f(y)|\Phi'(y)|^2dA(y)=k(z)\mathscr{F}[\Phi](f)-\mathscr{F}[\Phi](kf)(z).
\end{eqnarray*}
Thus,  Lemma \ref{lemkernel2} concludes that this term lies in $\mathscr{C}^{1,\alpha}(\D)$ and is continuous with respect \mbox{to $\Phi$.} For the other terms involving the logarithm, we can compute its gradient as before; for instance
\begin{eqnarray*}
&&\nabla_z \int_\D \log|\Phi(z)-\Phi(y)|f(y)\textnormal{Re}\left[\overline{\Phi^\prime(y)}k'(y)\right]dA(y)\\&&
\qquad= \overline{\Phi'(z)}\int_\D \frac{f(y)}{\overline{\Phi(z)}-\overline{\Phi(y)}}\textnormal{Re}\left[\overline{\Phi'(y)}k'(y)\right]dA(y)
 =\overline{\Phi'(z)}\, \overline{\mathscr{F}[\Phi]}\left(\frac{\textnormal{Re}\left[\overline{\Phi'(\cdot)}k'(\cdot)\right]}{|\Phi'(\cdot)|^2}\right)(z).
\end{eqnarray*}
Since $\frac{\textnormal{Re}\left[\overline{\Phi}'(\cdot)k'(\cdot)\right]}{|\Phi'(\cdot)|^2}\in \mathscr{C}^{1,\alpha}(\D)$ and is continuous with respect to $\Phi$, Lemma \ref{lemkernel2} concludes that this term lies in $\mathscr{C}^{1,\alpha}(\D)$ and is continuous with respect to $\Phi$. For the other terms involving a logarithmic part, the same procedure can be done. Trivially, both $D_gG(\Omega,g,\phi)$ and $D_\phi G(\Omega,g,\phi)$ are continuous with respect to $g$. We have obtained that the Gateaux derivatives are continuous with respect to $(g,\phi)$ and hence they are Fr\'echet derivatives.
 \end{proof}

\begin{remark}
Although we have done the previous discussion for the quadratic profile, the same argument may be applied for any radial profile $f_0$. Note that the only difference with the quadratic profile is that the function $\mathcal{M}_{f_0}$ and the constant $\lambda_{f_0}$ will depend on $f_0$. Hence, we just have to study the regularity of function $\mathcal{M}_{f_0}$ in order to give a similar result.
\end{remark}

 \subsection{Radial solutions}
The main goal of this sections is the resolution of \mbox{Equation \eqref{densityEq}} in the class of radial functions but in a small neighborhood  of  the quadratic \mbox{profiles  \eqref{QuadrP}.} We establish that except  for one singular value for $\Omega$,  no radial solutions different from $f_0$  may be found around it. This discussion is essential in order to ensure that with the new reformulation  we avoid the main defect of  the old one \eqref{FirstEqB}: the kernel is infinite-dimensional and  contains radial solutions.
As it was observed before, Proposition \ref{propImpl} gives us that  the  associated conformal mapping of any radial function is the identity map, and therefore  \eqref{densityEq} becomes
$$
\widehat{G}(\Omega,f-f_0)(z)=\frac{4\Omega-B}{8A}f(|z|)-\frac{1}{16A}f^2(|z|)+\frac{1}{2\pi}\int_\D \log|z-y|f(|y|)dA(y)-\frac{\Omega|z|^2}{2}-\lambda=0,
$$
for any $z\in\D$ where $\lambda$ is given by \eqref{valuealpha}. Thus the last integral identity of Proposition \ref{integralsprop} gives
\begin{eqnarray*}
\frac{4\Omega-B}{8A}f(r)-\frac{1}{16A}f^2(r)-\int_{r}^1\frac{1}{\tau}\int_0^\tau sf(s)dsd\tau-\frac12\Omega r^2-\lambda
=0,\quad\forall r\in[0,1].
\end{eqnarray*}
Introduce the function $ G_{\textnormal{rad}}: \R\times \mathscr{C}([0,1];\R) \rightarrow \mathscr{C}([0,1];\R),$ defined by 
$$
G_{\textnormal{rad}}(\Omega,f)(r)=\frac{4\Omega-B}{8A}f(r)-\frac{1}{16A}f^2(r)-\int_{r}^1\frac{1}{\tau}\int_0^\tau sf(s)dsd\tau-\frac{\Omega r^2}{2}-\lambda, \quad \forall r\in [0,1].
$$
 It is obvious that $ G_{\textnormal{rad}}$
is well--defined and furthermore it satisfies 
\begin{equation}\label{Obv23}
 G_{\textnormal{rad}}(\Omega,f_0)=0, \quad \forall\,\, \Omega\,\in\R.
\end{equation}
Through this work,  it will be more convenient to work  with the variable $x$ instead of $\Omega$ defined as
\begin{equation}\label{FormXX}
\frac{1}{x}=\frac{4}{A}\left(\Omega-\frac{B}{2}\right).
\end{equation}
Before stating our result, some properties of the hypergeometric function
$
x\in(-1,1)\mapsto F(1-\sqrt{2},1+\sqrt{2}; 1;x),
$
are needed.  A brief account on some useful properties of Gauss hypergeometric functions will be  discussed later in the Appendix \ref{SecSpecialfunctions}. In view of \eqref{f5}  we obtain   the identity
\begin{eqnarray}\label{relationFH}
 F(1-\sqrt{2},1+\sqrt{2}; 1;x)=\frac{1}{1-x} F(-\sqrt{2},\sqrt{2}; 1;x).
\end{eqnarray}
According to Appendix \ref{SecSpecialfunctions},  we have $F(1-\sqrt{2},1+\sqrt{2};1;0)=1$, and it diverges to $-\infty$ \mbox{at $1$.} This implies that there is at least one root in $(0,1)$. Combined with the fact that its derivative is negative according to \eqref{Diff41},  we may show that this root is unique. Denote this zero by $x_0\in (0,1)$  and   set
\begin{equation}\label{Firstzero}
\Omega_0\triangleq\frac{B}{2}+\frac{A}{4 x_0} .
\end{equation}
Setting  the ball
$$
B(f_0,\E)=\left\{f\in \mathscr{C}([0,1];\R),\quad\|f-f_0\|_{L^\infty}\leq\E\right\},
$$
for any $\E>0$, the first result can be stated as follows.
\begin{proposition}\label{radialfunctions}
Let $f_0$ be the quadratic profile  \eqref{QuadrP}, with $A\in\R^\star, B\in\R$.  Then,  there exists $\E >0$ such that
$$
 G_{\textnormal{rad}}(\Omega,f)=0 \Longleftrightarrow f=f_0,
$$
for any $(\Omega,f)\in I\times B(f_0,\E)$ and  any  bounded \mbox{interval $I$,} with $I\cap \left(\left[\frac{B}{2},\frac{B}{2}+\frac{A}{4}\right]\cup\{\Omega_0\}\right)=\emptyset$.
\end{proposition}
\begin{proof}
We remark that $G_{\textnormal{rad}}$ is a $\mathscr{C}^1$ function on $(\Omega,f)$. The idea is to apply the Implicit Function Theorem to deduce the result.  By differentiation with respect to $f$, one gets  that 
\begin{eqnarray}\label{Difrad}
D_fG_{\textnormal{rad}}(\Omega,f_0)h(r)
\nonumber &=&\frac{4\Omega-B}{8A}h(r)-\frac{1}{8A}f_0(r)h(r)-\int_r^1\frac{1}{\tau}\int_0^\tau sh(s)dsd\tau\\
&=&\frac{4\Omega-Ar^2-2B}{8A}h(r)-\int_r^1\frac{1}{\tau}\int_0^\tau sh(s)dsd\tau,
\end{eqnarray}
for any $h\in \mathscr{C}([0,1];\R)$. Now we shall look for the kernel of this operator, which consists of elements $h$ solving a Volterra  integro-differential  equation of the type
$$
\frac{4\Omega-Ar^2-2B}{8A}h(r)-\int_r^1\frac{1}{\tau}\int_0^\tau sh(s)dsd\tau=0,\quad \forall r\in[0,1].
$$
The assumption $\Omega\notin \left[\frac{B}{2},\frac{B}{2}+\frac{A}{4}\right]$  implies  that $r\in[0,1]\mapsto \frac{4\Omega-Ar^2-2B}{8A}$ is not vanishing and smooth.  Thus from the regularization of the integral, one can check that  any element of the kernel is actually  $\mathscr{C}^\infty$. Our purpose is to derive a differential equation by differentiating successively this integral equation. With the notation \eqref{FormXX} the kernel equation can be written in the form  
\begin{equation*}
\mathcal{L}h(r)\triangleq  \left(\frac{1}{x}-r^2\right) h(r)-8\int_r^1\frac{1}{\tau}\int_0^\tau sh(s)dsd\tau= 0,\quad  \forall r\in[0,1].
\end{equation*}
Remark that the assumptions on $\Omega$ can be translated into $x$, as $x\in(-\infty,1)$ and $x\neq 0$. 
Differentiating the function $\mathcal{L}h$  yields
$$
(\mathcal{L}h)^\prime(r)=\left(\frac{1}{x}-r^2\right)h^\prime(r)-2r h(r)+\frac{8}{r}\int_0^rsh(s)ds.
$$
Multiplying by $r$ and differentiating again we deduce that 
\begin{equation}\label{DiffXY1}
\frac{\left[r(\mathcal{L}h)^\prime(r)\right]^\prime}{r}=\left(\frac{1}{x}-r^2\right)h^{\prime\prime}(r)+\left(\frac{1}{x_0}-5r^2\right)\frac{ h^\prime(r)}{r}+4h(r)=0,\quad \forall r \in(0,1).
\end{equation}
In order to solve the above equation we look for solutions in the form
$$
h(r)=\rho( x\,r^2).
$$
This ansatz can be justified {\it a posteriori} by evoking the uniqueness principle   for ODEs.
Doing the change of variables $y=x\,r^2$, we  transform the preceding equation to 
$$
y(1-y)\rho^{\prime\prime}(y)+(1-3y)\rho^\prime(y)+\rho(y)=0.
$$
Appendix \ref{SecSpecialfunctions} leads to assure that the  only bounded solutions close to zero  to this hypergeometric equation are given by
\begin{eqnarray*}
\rho(y)&=&\gamma\,\,F(1+\sqrt{2},1-\sqrt{2};1;y),\quad\forall \gamma\in\R,
\end{eqnarray*}
and thus
\begin{equation}\label{TSW}
h(r)=\gamma\,\,F(1+\sqrt{2},1-\sqrt{2};1;x\,r^2),\quad\forall \gamma\in\R.
\end{equation}
It is important to note that from  the integral representation \eqref{integ} of hypergeometric functions, we can extend the above solution to   $x\in(-\infty,1)$. Coming back to the equation  \eqref{DiffXY1} and integrating two times, we obtain two real numbers $\alpha, \beta\in \R$ such that
$$
\mathcal{L}h(r)=\alpha \ln r+\beta,\quad\forall r\in(0,1].
$$
Since $\mathcal{L}h\in \mathscr{C}([0,1],\R)$, we obtain that $\alpha=0$ and thus 
$
\mathcal{L}h(r)=\beta.
$
By definition one has  $\mathcal{L}h(1)=\left(\frac{1}{x}-1\right) h(1)$. The fact that $x\neq 1$ implies that $\mathcal{L}h=0$ if and only if
$
h(1)=0.
$ 
According to \eqref{TSW}, this condition is equivalent to 
$
\gamma F(1-\sqrt{2},1+\sqrt{2};1;x)=0.
$
It follows that the kernel is trivial ($\gamma=0$) if and only if $x\neq x_0$, with $x_0$ being the only zero of $F(1-\sqrt{2},1+\sqrt{2};1;\cdot)$. However, for $x=x_0$ the kernel is one-dimensional and is generated by this hypergeometric function. Those claims will be made more rigorous in what follows.

\medskip
\noindent
$\bullet$ {\bf Case $x\neq x_0.$} As we have mentioned before, the kernel is trivial and it remains to check that $\mathcal{L}$ is an isomorphism.  With this aim, it suffices to prove that $\mathcal{L}$ is a Fredholm operator of zero index.  First, we  can split $\mathcal{L}$ as follows
$$
\mathcal{L}\triangleq\mathcal{L}_0+\mathcal{K},\quad \mathcal{L}_0\triangleq\left(\frac{1}{x}-r^2\right)\hbox{Id}\quad\hbox{and}\quad  \mathcal{K}\,h(r)\triangleq-8\int_r^1\frac{1}{\tau}\int_0^\tau sh(s)dsd\tau.
$$
Second, it is obvious that $\mathcal{L}_0: \mathscr{C}([0,1];\R)\to \mathscr{C}([0,1];\R)$ is an isomorphism, it is a Fredholm operator of zero index. Now, since the Fredholm operators with given index are stable by compact perturbation (for more details, see Appendix \ref{Apbif}), then to  check that $\mathcal{L}$ has zero index it is enough to establish  that 
$$\mathcal{K}:\mathscr{C}([0,1];\R)\to \mathscr{C}([0,1];\R),
$$ is compact. One can easily obtain that for $h\in \mathscr{C}([0,1];\R)$ the function $\mathcal{K}h$ belongs to  $\mathscr{C}^1([0,1];\R)$. Furthermore, by change of variables
$$
(\mathcal{K}h)^\prime(r)=r\int_0^1sh(rs) ds\quad \hbox{and}\quad (\mathcal{K}h)^\prime(0)=0, \quad \forall r\in(0,1],
$$
 which implies that 
$$
\|\mathcal{K}h\|_{\mathscr{C}^1}\leq C\|h\|_{L^\infty}.
$$
Since the embedding $\mathscr{C}^{1}([0,1];\R)\hookrightarrow \mathscr{C}([0,1];\R)$ is compact, we find that $\mathcal{K}$ is a compact operator. Finally, we get that $\mathcal{L}$ is an isomorphism. This ensures  that $D_f{G}_{\textnormal{rad}}(\Omega,f_0)$ is an isomorphism, and therefore the Implicit  Function Theorem together with \eqref{Obv23} allow us to deduce that the only solutions of $G(\Omega,f)=0$  in  $I\times B(f_0,\E)$ are given by the trivial \mbox{ones $\{(\Omega,f_0), \Omega\in I\}$.}

\medskip
\noindent
$\bullet$ {\bf Case $x= x_0.$} In this special case the kernel of $\mathcal{L}$ is one--dimensional and is generated by
\begin{equation}\label{kernelZZ1}
\textnormal{Ker}\mathcal{L}=\langle F(1-\sqrt{2},1+\sqrt{2};1;x\, (\cdot)^2)\rangle. 
\end{equation}
This case will be deeply discussed  below in Proposition \ref{propdegen}.
\end{proof}

Let us focus on the case $\Omega=\Omega_0$. From \eqref{kernelZZ1},  the kernel of the linearized operator  is one--dimensional, and we will be able to implement the  Crandall--Rabinowitz Theorem. Our result reads as follows.
\begin{proposition}\label{propdegen}
Let $f_0$ be the quadratic profile  \eqref{QuadrP} with $A\in \R^\star$, $B\in\R$ and fix $\Omega_0$ as \mbox{in \eqref{Firstzero}.}  Then, there exists an open  neighborhood $U$ of $(\Omega_0,f_0)$ in $\R\times \mathscr{C}([0,1];\R)$ and a continuous curve $\xi\in (-a,a)\mapsto (\Omega_\xi, f_\xi)\in U$ with $a>0$ such that
$$
G_{\textnormal{rad}}(\Omega_\xi,f_\xi)=0, \quad \forall\xi\in (-a,a).\,\,\,   
$$
\end{proposition}
\begin{proof}
We must check that the hypotheses of the Crandall--Rabinowitz Theorem are achived. It is clear that
$$
G_{\textnormal{rad}}(\Omega,f_0)=0, \quad \forall\,\Omega\in \R.
$$
It is not difficult to show that the mapping   $(\Omega,f)\mapsto G_{\textnormal{rad}}(\Omega,f)$ is  $\mathscr{C}^1$. In addition,  we have seen  in the foregoing discussion that  $D_f G_{\textnormal{rad}}(\Omega_0,f_0)$ is a Fredholm operator with zero index and its  kernel  is one--dimensional.
Therefore to apply the bifurcation arguments it remains  just to check the transversality condition in the Crandall-Rabinowitz Theorem. Having this in mind, we should first find a practical  characterization for the range  of the linearized operator.   We note that an element  $d\in \mathscr{C}([0,1];\R)$ belongs to the range of $D_f G_{\textnormal{rad}}(\Omega_0,f_0)$ if   the equation
\begin{equation}\label{ImageX1}
\frac{\frac{1}{x_0}-r^2}{8}h(r)-\int_r^1\frac{1}{\tau}\int_0^\tau sh(s)dsd\tau=d(r),\quad \forall r\in[0,1],
\end{equation}
admits a solution $h$ in $\mathscr{C}([0,1];\R)$,  
where $x_0$ is given  by \eqref{FormXX}. Consider the auxiliary \mbox{function $H$} 
$$
H(r)\triangleq\int_r^1\frac{1}{\tau}\int_0^\tau sh(s)dsd\tau,\quad \forall\, r\in[0,1].
$$
Then   $H$ belongs to  $\mathscr{C}^1([0,1];\R)$ and it satisfies the boundary condition  
\begin{equation}\label{Boundar1}
H(1)=0\quad \hbox{and}\quad H^\prime(0)=0.
\end{equation}
In order to write down  an ordinary differential equation for $H$, let us define the linear operator
$$
 \mathcal{L}h(r)\triangleq \int_r^1\frac{1}{\tau}\int_0^\tau sh(s)dsd\tau,\quad \forall r\in[0,1],
$$
for any $h\in \mathscr{C}([0,1];\R)$. Then, we derive successively,
$$
(\mathcal{L}h)'(r)=-\frac{1}{r}\int_0^rsh(s)ds,
$$
and 
\begin{eqnarray*}
r(\mathcal{L}h)''(r)+(\mathcal{L}h)'(r)=-rh(r).
\end{eqnarray*}
Applying this identity to $H$ one arrives at
$$
rH''(r)+H'(r)=-\frac{8rx_0}{1-x_0r^2}\left[H(r)+d(r)\right].
$$
Thus,  $H$ solves the second order differential equation
\begin{equation}\label{OD2E}
r(1-x_0r^2)H''(r)+(1-x_0r^2)H'(r)+8rx_0H(r)=-8rx_0d(r),
\end{equation}
supplemented with  boundary conditions \eqref{Boundar1}. The argument is to come back to the original equation \eqref{ImageX1} and show that the candidate 
$$
h(r)\triangleq \frac{8x_0}{1-x_0r^2}\left[H(r)+d(r)\right],
$$
is actually a solution to this equation. Then, we need to  check that 
$
\mathcal{L}h=H.
$
By setting $\mathpzc{H}\triangleq \mathcal{L}h-H$, we deduce that
$$
r(\mathpzc{H})''(r)+(\mathpzc{H})'(r)=0,
$$
with  the boundary conditions $\mathpzc{H}(1)=0$ and $\mathpzc{H}^\prime(0)=0$, which come from \eqref{Boundar1}. The solution of this differential equation is
$$
\mathpzc{H}(r)=\lambda_0+\lambda_1\ln r,\quad  \forall \lambda_0,\lambda_1\in\R
$$
Since $\mathcal{L}h$ and $H$ belong to $\mathscr{C}([0,1];\R)$, then necessarily  $\lambda_1=0$,  and from the boundary  condition we find $\lambda_0=0$. This implies that $\mathcal{L}h=H$, and it shows finally that solving  \eqref{ImageX1}  is equivalent to solving  \eqref{OD2E}. Now, let us focus on the resolution of \eqref{OD2E}. For this purpose, we proceed by finding a particular solution for the homogeneous equation and use later the method of  variation of  constants. 
Looking for a solution to the homogeneous equation in the form
$
\mathpzc{H}_0(r)=\rho(x_0r^2),
$
and using the variable $y=x_0r^2$, one arrives at
$$
(1-y)y\rho^{\prime\prime}(y)+(1-y)\rho^\prime(y)+2\rho(y)=0.
$$
This is a hypergeometric equation, and one  solution  is given by
$
y\mapsto F(-\sqrt{2},\sqrt{2};1;y).
$
Thus, a particular solution to the homogeneous equation is 
$
\mathpzc{H}_0:\, r\in [0,1]\mapsto F(-\sqrt{2},\sqrt{2};1;x_0r^2).
$
Then, the general solutions for \eqref{OD2E} are given through the formula
$$
\displaystyle{H(r)=\mathpzc{H}_0(r)\left[K_2+\int_\delta^r\frac{1}{\tau\mathpzc{H}_0^2(\tau)}\left[K_1-8x_0\int_0^\tau \frac{s\mathpzc{H}_0(s)}{1-x_0s^2}d(s)ds\right]d\tau\right],\quad \forall r\in[0,1],}
$$
where $K_1,K_2$ are real constants and $\delta\in(0,1)$ is any given number. Since $\mathpzc{H}_0$ is smooth on the interval  $[0,1]$, with $\mathpzc{H}_0(0)=1$, one can check that $H$ admits a singular term close to zero taking the form $K_1\ln r$. This forces $K_1$ to vanish because   $H$  is continuous up to the origin. Therefore, we infer that
$$
\displaystyle{H(r)=\mathpzc{H}_0(r)\left[K_2-8x_0\int_\delta^r\frac{1}{\tau\mathpzc{H}_0^2(\tau)}\int_0^\tau \frac{s\mathpzc{H}_0(s)}{1-x_0s^2}d(s)dsd\tau\right], \quad \forall r\in[0,1].}
$$
The last integral term is convergent at the origin and one may take $\delta=0$. This implies that
 $$
\displaystyle{H(r)=\mathpzc{H}_0(r)\left[K_2-8x_0\int_0^r\frac{1}{\tau\mathpzc{H}_0^2(\tau)}\int_0^\tau \frac{s\mathpzc{H}_0(s)}{1-x_0s^2}d(s)dsd\tau\right], \quad \forall r\in[0,1].}
$$
From this  expression, we deduce the second condition of \eqref{Boundar1}. For the first condition, $H(1)=0$, we first note from \eqref{relationFH} that $F(-\sqrt{2},\sqrt{2};1;x_0)=0$. Then, we can compute the limit at $r=1$ via l'H\^opital's rule leading to 
\begin{eqnarray*}
H(1)=-8x_0\lim_{r\rightarrow 1^-}\mathpzc{H}_0(r)\int_0^r\frac{1}{\tau\mathpzc{H}_0^2(\tau)}\int_0^\tau \frac{s\mathpzc{H}_0(s)}{1-x_0s^2}d(s)dsd\tau
=8x_0\frac{\displaystyle{\int}_0^1\frac{s\mathpzc{H}_0(s)}{1-x_0s^2}d(s)ds}{\mathpzc{H}_0^\prime(1)}.
\end{eqnarray*}
From the expression of $\mathpzc{H}_0$ and \eqref{Diff41} we recover that
$
\mathpzc{H}_0^\prime(1)=-4x_0F(1-\sqrt{2},1+\sqrt{2};2;x_0).
$
We point out that this quantity is not vanishing. This can be proved by differentiating the relation \eqref{relationFH}, which implies that
$$
F(1-\sqrt{2},1+\sqrt{2};2;x_0)=(x_0-1)F(2-\sqrt{2},2+\sqrt{2};2;x_0),
$$
and the latter term is not vanishing from the definition of hypergeometric functions.
It follows that the condition $H(1)=0$ is equivalent to 
\begin{equation}\label{condimag}
\int_0^1\frac{s\mathpzc{H}_0(s)}{1-x_0s^2}d(s)ds=0.
\end{equation}
This characterizes the elements of the range of the linearized operator. 
Now, we are in a position  to check the transversality condition. According to the expression \eqref{Difrad}, one gets by differentiating with respect to $\Omega$ that
$$
D_{\Omega,f} G_{\textnormal{rad}}(\Omega_0,f_0)h(r)=\frac{1}{2A}h(r).
$$
Recall from \eqref{kernelZZ1} and the relation \eqref{relationFH} that the kernel is generated by 
$
r\in[0,1]\mapsto \frac{\mathpzc{H}_0(r)}{1-x_0 r^2}.
$
Hence, from \eqref{condimag} the transversality assumption in the Crandall--Rabinowitz Theorem becomes
$$
\int_0^1\frac{s\mathpzc{H}_0^2(s)}{(1-x_0s^2)^2}ds\neq 0,
$$ 
which is trivially satisfied, and concludes  the announced result.
\end{proof}

\section{Linearized operator for the density equation} \label{Sec-line}
This section is devoted to the study of the linearized operator of the density equation \eqref{densityEq}. First, we will compute it with a general $f_0$ and provide a suitable formula in  the case of  quadratic profiles.  Second, we shall prove  that the  linearized operator is a Fredholm operator of index zero because it takes the form of  a compact perturbation of an invertible operator. More details about Fredholm operators can be found in Appendix \ref{Apbif}.
 Later we will focus our attention on the algebraic structure of the kernel and the range and give explicit expressions by using hypergeometric functions. We point out that the kernel  description  is done through the resolution of  a Volterra integro-differential equation.
\subsection{General formula and Fredholm index of the linearized operator}
Let $f_0$ be  an arbitrary  smooth radial function satisfying  \eqref{mutrivial} and let  us compute the linearized operator of the functional $\widehat{G}_{f_0}$ given by  \eqref{densityEq}. First, using Proposition \ref{propImpl} one gets
\begin{eqnarray}\label{TIKL}
 D_g\mathcal{N}(\Omega, 0)h(z)=z\sum_{n\geq 1} A_n z^{n}
&\triangleq& k(z),
\end{eqnarray}
where $A_n$ is given in \eqref{An} and $h\in\mathscr{C}^{1,\alpha}_s(\D)$. Therefore, differentiating with respect to $g$ yields
\begin{eqnarray*}
D_g\widehat{G}_{f_0}(\Omega,0)h&=&\partial_g G_{f_0}(\Omega,0,0)h+\partial_{\phi}G_{f_0}(\Omega,0,0) \partial_g\mathcal{N}h(\Omega,0)
\\
&=&\partial_g G_{f_0}(\Omega,0,0)h+\partial_{\phi}G_{f_0}(\Omega,0,0) k.
\end{eqnarray*}
Using the Fr\'echet derivatives from Proposition \ref{Gwelldefined}, we have that
\begin{eqnarray*}
D_g\widehat{G}_{f_0}(\Omega,0)h(z)&=&D_g \mathcal{M}_{f_0}(\Omega,f_0(z))h(z)+\frac{1}{2\pi}\int_\D \log |z-y|h(y)dA(y)-\Omega\textnormal{Re}[\overline{z}k(z)]\\
&&+\frac{1}{2\pi}\int_\D\textnormal{Re}\left[\frac{k(z)-k(y)}{z-y}\right]f_0(y)dA(y)\\
&&+\frac{1}{\pi}\int_\D \log|z-y|f_0(y)\textnormal{Re}[k'(y)]dA(y).
\end{eqnarray*}
From the definition of $\mathcal{M}_{f_0}$ in \eqref{M}  we infer that
$$ D_g \mathcal{M}_{f_0}(\Omega,f_0(z))h(z)=\frac{h(z)}{\mu(\Omega,f_0(|z|))},\quad \forall \, z\in\D.
$$
Putting together  the preceding formulas we obtain
\begin{eqnarray}\label{FormDif}
\nonumber D_g \widehat{G}_{f_0}(\Omega,0)h(z)&=&\frac{h(z)}{\mu(\Omega,f_0(r))}+\frac{1}{2\pi}\int_\D \log |z-y|h(y)dA(y)-\Omega\textnormal{Re}[\overline{z}k(z)]\\
\nonumber&&+\frac{1}{2\pi}\int_\D\textnormal{Re}\left[\frac{k(z)-k(y)}{z-y}\right]f_0(y)dA(y)\\
&&+\frac{1}{\pi}\int_\D \log|z-y|f_0(y)\textnormal{Re}[k'(y)]dA(y),
\end{eqnarray}
where  $\mu$ is given by the compatibility condition \eqref{mutrivial}. 
Next, we shall rewrite  the linearized operator for the quadratic profile, and we will omit the subscript $f_0$  for the sake of simplicity. Taking  $\mu$ as in \eqref{muQP}, we get
\begin{eqnarray}\label{FormDif2}
\nonumber D_g \widehat{G}(\Omega,0)h(z)&=&\frac18\left(\frac1x-r^2\right)h(z)+\frac{1}{2\pi}\int_\D \log |z-y|h(y)dA(y)-\Omega\textnormal{Re}[\overline{z}k(z)]\\
\nonumber&&+\frac{1}{2\pi}\int_\D\textnormal{Re}\left[\frac{k(z)-k(y)}{z-y}\right]f_0(y)dA(y)\\
&&+\frac{1}{\pi}\int_\D \log|z-y|f_0(y)\textnormal{Re}[k'(y)]dA(y),
\end{eqnarray}
where $x$ is given by \eqref{FormXX}.

In the following result we show that the linearized operator associated to a quadratic profile is a Fredholm operator with  index zero. Similar result may be obtained in the general case imposing suitable conditions on the profiles.
 \begin{proposition}\label{compactop}
 Let $f_0$ be the profile \eqref{QuadrP}, with $A\in\R^\star$ and $B\in\R$. Assume that \, $\Omega\notin \left[\frac{B}{2},\frac{B}{2}+\frac{A}{4}\right]\cup \mathcal{S}_{\textnormal{sing}}.$ Then,  
$ D_g \widehat{G}(\Omega,0):\mathscr{C}^{1,\alpha}_s(\D)\to \mathscr{C}^{1,\alpha}_s(\D)$ is a Fredholm operator with zero index.  
\end{proposition}
\begin{proof}
Using \eqref{FormDif2} we have
\begin{eqnarray*}
\nonumber D_g \widehat{G}(\Omega,0)h(z)&=&\frac18\left(\frac1x-r^2\right)h(z)+\frac{1}{2\pi}\int_\D \log |z-y|h(y)dA(y)-\Omega\textnormal{Re}[\overline{z}k(z)]\\
&&+\frac{1}{2\pi}\int_\D\textnormal{Re}\left[\frac{k(z)-k(y)}{z-y}\right]f_0(y)dA(y)\\
&&+\frac{1}{\pi}\int_\D \log|z-y|f_0(y)\textnormal{Re}[k'(y)]dA(y)\\
&\triangleq & \left[\frac18\left(\frac1x-r^2\right)\hbox{Id}+\mathcal{K}\right]h(z),
\end{eqnarray*}
where $k$ is related to $h$ through \eqref{TIKL}. The assumption on $\Omega$ entails that the smooth function $z\in \overline{\D}\mapsto \frac1x-|z|^2$ is not vanishing on the closed unit disc. Then the operator 
$$
\frac18\left(\frac1x-|z|^2\right)\hbox{Id}:\mathscr{C}^{1,\alpha}_s(\D)\to \mathscr{C}^{1,\alpha}_s(\D),
$$
 is an isomorphism. Hence, it is a Fredholm operator with zero index. To check that $\mathcal{L}$ is also a Fredholm operator with zero index, it suffices to prove that the operator $\mathcal{K}: \mathscr{C}^{1,\alpha}_s(\D)\to \mathscr{C}^{1,\alpha}_s(\D)$ is compact. To do that, we will prove that $\mathcal{K}: \mathscr{C}^{1,\alpha}(\D)\to \mathscr{C}^{1,\gamma}(\D)$, for any  $\gamma\in (0,1)$. 
We split $\mathcal{K}$ as follows
$$
\mathcal{K}h=\sum_{j=1}^4\mathcal{K}_jh,
$$
with
$$
\mathcal{K}_1h= -\Omega\textnormal{Re}[\overline{z}k(z)],\quad \mathcal{K}_2h(z)=\frac{1}{2\pi} \int_\D \log |z-y|h(y)dA(y),$$
$$
\mathcal{K}_3h(z)=\frac{1}{\pi}\int_\D \log|z-y|f_0(y)\textnormal{Re}[k'(y)]dA(y),\quad 
\mathcal{K}_4h=\frac{1}{2\pi}\int_\D\textnormal{Re}\left[\frac{k(z)-k(y)}{z-y}\right]f_0(y)dA(y).
$$
The estimate of the first term $\mathcal{K}_1h$ follows from \eqref{TIKL} and \eqref{ContinQ1}, leading to 
\begin{equation}\label{mercred1}
\| \mathcal{K}_1h\|_{\mathscr{C}^{2,\alpha}(\D)}\le C\|h\|_{\mathscr{C}^{1,\alpha}(\D)}.
\end{equation}
Concerning the term $\mathcal{K}_2$ we note that 
$$
\|\mathcal{K}_2h\|_{L^\infty(\D)}\leq C\|h\|_{L^\infty(\D)},
$$
and differentiating it, we obtain
$$
\nabla_z\mathcal{K}_2h(z)= \frac{1}{2\pi}\int_{\D}\frac{h(y)}{\overline{z}-\overline{y}}dA(y).
$$
Lemma \ref{lemkernel1} yields 
$$
\|\nabla_z \mathcal{K}_2h\|_{\mathscr{C}^{0,\gamma}(\D)}\le C\|h\|_{L^\infty(\D)},
$$
for any $\gamma\in(0,1)$, and thus
\begin{equation}\label{mercred2}
\| \mathcal{K}_2h\|_{\mathscr{C}^{1,\gamma}(\D)}\le C\|h\|_{L^\infty(\D)}.
\end{equation}
The estimate of $\mathcal{K}_3$ is similar to that of $\mathcal{K}_2$, and using \eqref{mercred2} and \eqref{ContinQ1}  we find that
\begin{equation}\label{mercred3}
 \| \mathcal{K}_3h\|_{\mathscr{C}^{1,\gamma}(\D)}\le  C\|f_0\|_{L^\infty(\D)}\| k'\|_{L^\infty(\D)}
\le C\|f_0\|_{L^\infty(\D)}\|h\|_{\mathscr{C}^{1,\alpha}(\D)}.
\end{equation}
Setting
$$
K(z,y)=\frac{k(z)-k(y)}{z-y}, \quad \forall  z\neq y,
$$
it is obvious that
$
|K(z,y)|\leq \|k^\prime\|_{L^\infty(\D)}.
$
Therefore, we have
$$
\|\mathcal{K}_4h\|_{L^\infty}\leq  C\|f_0\|_{L^\infty}\|k^\prime\|_{L^\infty(\D)}.
$$
Moreover, by differentiation we find
$$
\nabla_z\mathcal{K}_4h(z)=\frac{1}{2\pi}\hbox{Re}\int_\D\nabla_zK(z,y)f_0(y)dA(y).
$$
Straightforward computations show that
\begin{eqnarray*}
|\nabla_zK(z,y)|&\leq& C\|k^\prime\|_{L^\infty(\D)}|z-y|^{-1},\\
|\nabla_zK(z_1,y)-\nabla_zK(z_2,y)|&\leq& C|z_1-z_2|\left[\frac{\|k^{\prime\prime}\|_{L^\infty(\D)}}{|z_1-y|}+\frac{\|k^\prime\|_{L^\infty(\D)}}{|z_1-y||z_2-y|}\right].
\end{eqnarray*}
Thus, hypotheses \eqref{kernel1} are satisfied and we can use Lemma \ref{lemkernel1}  and \eqref{ContinQ1} to find
\[
\|\nabla_z\mathcal{K}_4h\|_{\mathscr{C}^{0,\gamma}(\D)}\leq C\|f_0\|_{L^\infty}\|k\|_{\mathscr{C}^{2}(\D)}
\le C\|f_0\|_{L^\infty}\|h\|_{\mathscr{C}^{1,\alpha}(\D)},
\]
obtaining
\begin{equation}\label{mercred4}
\| \mathcal{K}_4h\|_{C^{1,\gamma}(\D)}
\le  C\|f_0\|_{L^\infty(\D)}\|h\|_{\mathscr{C}^{1,\alpha}(\D)}.
\end{equation}
Combining the estimates \eqref{mercred1},\eqref{mercred2},\eqref{mercred3} and \eqref{mercred4}, we deduce 
\begin{equation*}
\| \mathcal{K}h\|_{\mathscr{C}^{1,\gamma}(\D)}
\le  C\|h\|_{\mathscr{C}^{1,\alpha}(\D)},
\end{equation*}
which concludes the proof.\end{proof}

To end  this subsection, we give a more explicit form of the linearized operator. Coming back to the general expression in \eqref{FormDif} and using Proposition \ref{integralsprop}  we get that
\begin{eqnarray*}
D_g \widehat{G}_{f_0}(\Omega,0)h(z)&=&\sum_{n\geq 1} \cos(n\theta)\left[\frac{h_n(r)}{\mu(\Omega,f_0(r))}-\frac{r}{n}\left(A_nG_n(r)+\frac{1}{2r^{n+1}}H_n(r)\right)\right]\\
&&+\frac{h_0(r)}{\mu(\Omega,f_0(r))}-\int_r^1\frac{1}{\tau}\int_0^\tau sh_0(s)dsd\tau,
\end{eqnarray*}
for 
$$r e^{i\theta}\mapsto h(re^{i\theta})=\sum_{n\in\N}h_n(r)\cos(n\theta)\in \mathscr{C}^{1,\alpha}_s(\D),$$
where
\begin{eqnarray*}
G_n(r)&\triangleq&n\Omega r^{n+1}+r^{n-1}\int_0^1sf_0(s)ds-(n+1)r^{n-1}\int_0^rsf_0(s)ds
+\frac{n+1}{r^{n+1}}\int_0^rs^{2n+1}f_0(s)ds,\\
H_n(r)&\triangleq&r^{2n}\int_r^1\frac{1}{s^{n-1}}h_n(s)ds+\int_0^rs^{n+1}h_n(s)ds,
\end{eqnarray*}
 for any $n\geq1$. The value of $A_n$  is given by \eqref{An} and recall that it was derived  from the expression $\partial_g \mathcal{N}(\Omega,0)$ when studying the boundary equation.
Moreover, there is another useful expression for $A_n$ coming from the value of $G_n(1)$
\begin{eqnarray*}
G_n(1)=n\left[\Omega-\int_0^1sf_0(s)ds+\frac{n+1}{n}\int_0^1s^{2n+1}f_0(s)ds\right]
=n\left(\Omega-\widehat{\Omega}_n\right).
\end{eqnarray*}
Those preceding identities agrees with
\begin{equation*}
A_n=-\frac{H_n(1)}{2G_n(1)}, \quad \forall n\geq 1.
\end{equation*}
In the special case of $f_0$ being a quadratic profiles of the type \eqref{QuadrP},
straightforward computations imply that
\begin{eqnarray}\label{linoperator}
D_g \widehat{G}(\Omega,0)h(z)&=&\sum_{n\geq 1} \cos(n\theta) \left[\frac{\frac{1}{x}-r^2}{8}h_n(r)-\frac{r}{n}\left(A_nG_n(r)+\frac{1}{2r^{n+1}}H_n(r)\right)\right]\nonumber\\
&&+\frac{\frac{1}{x}-r^2}{8}h_0(r)-\int_r^1\frac{1}{\tau}\int_0^\tau sh_0(s)dsd\tau,
\end{eqnarray}
with 
\begin{eqnarray}\label{functions}
 G_n(r)&=&-\frac{An(n+1)}{4(n+2)}r^{n-1}P_n(r^2),\label{Gn}\\
H_n(r)&=&r^{2n}\int_r^1\frac{1}{s^{n-1}}h_n(s)ds+\int_0^rs^{n+1}h_n(s)ds,\label{Hn}\\
P_n(r)&=&r^2-\frac{n+2}{n+1}\frac{1}{x}r-\frac{A+2B}{A}\frac{n+2}{n(n+1)},\label{Pn}\\ \label{Gn1}
 G_n(1)&=&-\frac{An(n+1)}{4(n+2)}P_n(1)
=n\left[\frac{A}{4}\left(\frac1x-1\right)+\frac{A(n+1)}{2n(n+2)}+\frac{B}{2n}\right],\\ \label{An1}
A_n&=&-\frac{H_n(1)}{2G_n(1)}
= \frac{H_n(1)}{2n\left\{\widehat{\Omega}_n-\Omega\right\}}. 
\end{eqnarray}
Remark  that $G_n(1)\neq 0$  since we are assuming that $\Omega\notin \mathcal{S}_{\textnormal{sing}}$,  the singular set  defined \mbox{in \eqref{Interv2}. }

From now on we will work only with the quadratic profiles.  Similar study could be implemented with general profiles but the analysis may turn out to be  very difficult because the spectral  study is intimately related to the distribution of the selected  profile.

\subsection{Kernel structure and negative results}
The current objective is to conduct a precise study for the kernel structure of the linearized operator \eqref{linoperator}. We must identify 
the master equation describing the dispersion relation. As a by-product we   connect the dimension of  the kernel to the number of roots of the master equation. We shall distinguish in this study between the regular case corresponding to $x\in(-\infty,1)$ and the singular case associated to $x>1.$ For this latter case we prove that the equation \eqref{Interv2} has no solution close to the trivial one. 
\subsubsection{Regular case}
Let us start with a preliminary result devoted to the explicit resolution of a second order differential equation with polynomial coefficients taking the form
\begin{eqnarray}\label{ODE}
(1-xr^2)rF''(r)-(1-xr^2)(2n-1)F'(r)+8rxF(r)=g(r),\quad \forall r\in[0,1],
\end{eqnarray}
 This will be applied later to  the study of the kernel and the range. Before stating our result we need to introduce some functions
 \begin{equation}\label{coefficients}
F_n(r)=F(a_n, b_n; c_n; r),\quad a_n=\frac{n-\sqrt{n^2+8}}{2},\quad 
b_n=\frac{n+\sqrt{n^2+8}}{2},\quad 
c_n=n+1.
\end{equation} 
where $r\in[0,1)\mapsto F(a,b;c;r)$ denotes  the Gauss hypergeometric function defined in \eqref{GaussF}. \begin{lemma}\label{lemmaODE}
Let  $n\geq 1$ be an integer,  ${x\in(-\infty,1)}$ and $g\in \mathscr{C}([0,1];\R)$. Then, the general continuous solutions of equation \eqref{ODE}
supplemented with the initial condition $F(0)=0$, are given by a one--parameter curve 
$$
F(r)=r^{2n}F_n(xr^2)\left[\frac{F(1)}{F_n(x)}-\frac{x^{n-1}}{4}\int_{xr^2}^x\frac{1}{\tau^{n+1}F_n^2(\tau)}\int_0^\tau\frac{F_n(s)}{1-s}\left(\frac{x}{s}\right)^{\frac12}g\left(\left(\frac{s}{x}\right)^{\frac12}\right)dsd\tau\right].
$$
\end{lemma}
\begin{proof}
Consider the auxiliary function
$
F(r)=\mathpzc{F}(xr^2)
$
and set $y=xr^2$. Note that $  y\in[0,x]$ when $x>0$ and $y\in[x,0]$ if $x<0$, then $r=\left(\frac{y}{x}\right)^{\frac{1}{2}}$ in both cases. Hence, the equation governing  this new function is
\begin{equation}\label{completeeq}
(1-y)y\mathpzc{F}^{\prime\prime}(y)-(n-1)(1-y)\mathpzc{F}^\prime(y)+2\mathpzc{F}(y)=\frac{1}{4x}\left(\frac{x}{y}\right)^{\frac{1}{2}}g\left(\left(\frac{y}{x}\right)^{\frac{1}{2}}\right),
\end{equation}
with the boundary condition $\mathpzc{F}(0)=0$. The strategy  to be followed consists in solving the homogeneous equation and using later the  method of variation of constants. The homogeneous problem is given by 
\begin{equation*}
(1-y)y\mathpzc{F}_0^{\prime\prime}(y)-(n-1)(1-y)\mathpzc{F}_0^\prime(y)+2\mathpzc{F}_0(y)=0.
\end{equation*}
Comparing it with the general differential equation \eqref{ODEEQ}, we  obviously find that $\mathpzc{F}_0$ satisfies a hypergeometric equation with the  parameters
\begin{equation*}
{a}=\frac{-n-\sqrt{n^2+8}}{2},\quad  
{b}=\frac{-n+\sqrt{n^2+8}}{2},\quad c=1-n.
\end{equation*}
The general theory of hypergeometric functions gives us that this differential equation is degenerate because  $c$ is a negative integer, see discussion in Appendix \ref{SecSpecialfunctions}.  However, we still get  two independent solutions generating the class of solutions to this differential equation:  one is smooth and the second  is singular and contains a logarithmic singularity at the origin. The smooth one is  given by 
$$
y\in(-\infty,1)\mapsto y^{1-{c}}{}F\left(1+{a}-{c},1+{b}-{c},2-{c},y\right).
$$
With the special parameters \eqref{coefficients}, it becomes 
\begin{equation*}
y\in(-\infty,1)\mapsto y^{n}{}F\left(a_n,b_n;c_n;y\right)= y^nF_n(y).
\end{equation*}
It is important to note that,  by Taylor  expansion, the hypergeometric function initially  defined in the unit disc  $\D$ admits an analytic continuation in the complex plane cut along the  real axis from $1$ to $+\infty$ . This comes from  the integral representation \eqref{integ}.

Next, we  use  the method of variation of  constants with the smooth homogeneous solution  and set
\begin{equation}\label{smoothsol}
\mathpzc{F}_0: (-\infty,1)\in\D\mapsto y^nF_n(y).
\end{equation}
We wish to mention that when using  the method of variation of  constants with the smooth solution we also find the trace of the singular solution. As we will notice in the next step, this singular part will not contribute for the full inhomogeneous problem due  to the required regularity and   the boundary \mbox{condition  $F(0)=0.$}  Now, we solve the equation \eqref{completeeq} by  looking  for solutions  in the form
$
\mathpzc{F}(y)=\mathpzc{F}_0(y)K(y).
$
By setting  $\mathpzc{K}\triangleq K^\prime$, one has that
$$
\mathpzc{K}^\prime(y)+\left[2\frac{\mathpzc{F}_0^\prime(y)}{\mathpzc{F}_0(y)}-\frac{n-1}{y}\right]\mathpzc{K}(y)=\frac{1}{4x}\frac{\left(\frac{x}{y}\right)^{\frac{1}{2}}g\Big(\Big(\frac{y}{x}\Big)^{\frac{1}{2}}\Big)}{y(1-y)\mathpzc{F}_0(y)},
$$
which can be integrated in the following way
$$
\mathpzc{K}(y)=\frac{y^{n-1}}{\mathpzc{F}_0^2(y)}\left\{K_1+\frac{1}{4x}\int_0^y\frac{\mathpzc{F}_0(s)}{s^n(1-s)}\left(\frac{x}{s}\right)^{\frac{1}{2}}g\Big(\Big(\frac{s}{x}\Big)^{\frac{1}{2}}\Big)ds\right\},
$$
where $K_1$ is a constant. Thus integrating successively  we  find  ${K}$ and  $\mathpzc{F}$ and from  the expression of $\mathpzc{F}_0$ we deduce
$$
\mathpzc{F}(y)=y^nF_n(y)\left[K_2-\int_y^{\textnormal{sign }y}\frac{1}{\tau^{n+1}F_n^2(\tau)}\left\{K_1+\frac{1}{4x}\int_0^\tau\frac{F_n(s)}{1-s}\left(\frac{x}{s}\right)^{\frac{1}{2}}g\Big(\Big(\frac{s}{x}\Big)^{\frac{1}{2}}\Big)ds\right\}d\tau\right],
$$
for any $y\in (-\infty,1)$, with $K_2$ a constant and where $sign$ is the sign function. 
From straightforward computations using integration by parts we get 
$$
\mathpzc{F}(0)=
-K_1\lim_{y\rightarrow 0} y^nF_n(y)\int_y^{\textnormal{sign }y}\frac{1}{\tau^{n+1}F_n^2(\tau)}d\tau=-\frac{K_1}{n}.
$$
Combined with
the initial condition $\mathpzc{F}(0)=0$ we obtain 
 $K_1=0$.  Coming back to  the original function  
$
F(r)=\mathpzc{F}(xr^2),
$
we obtain
$$
F(r)=
x^nr^{2n}F_n(xr^2)\left[K_2-\frac{1}{4x}\int_{xr^2}^{\textnormal{sign }x}\frac{1}{\tau^{n+1}F_n^2(\tau)}\int_0^\tau\frac{F_n(s)}{1-s}\left(\frac{x}{s}\right)^{\frac{1}{2}}g\left(\left(\frac{s}{x}\right)^{\frac{1}{2}}\right)dsd\tau\right].
$$
The constant $K_2$ can be computed by evaluating the preceding expression at $r=1$. We finally get 
\begin{equation}\label{solF}
F(r)=r^{2n}F_n(xr^2)\left[\frac{F(1)}{F_n(x)}-\frac{x^{n-1}}{4}\int_{xr^2}^x\frac{1}{\tau^{n+1}F_n^2(\tau)}\int_0^\tau\frac{F_n(s)}{1-s}\sqrt{x/s}\,g\left(\sqrt{s/x}\right)dsd\tau\right],
\end{equation}
for $r\in[0,1].$ Observe  from the integral representation \eqref{integ} that the function $F_n$ does not vanish on $(-\infty,1)$ for $n\geq1$. Hence, \eqref{solF} is well-defined and $F$ is $C^\infty$ in $[0,1]$ when $x<1$.  
\end{proof} 

The next goal is to  give the kernel structure of the linearized operator $D_g \widehat{G}(\Omega,0)$.  We emphasize that according to Proposition \ref{compactop}, this is a Fredholm operator of zero index, which implies in particular  that its kernel is finite--dimensional.  Before that, we introduce the singular set for $x$ connected to the singular set of $\Omega$ through the relations \eqref{FormXX} and \eqref{Interv2}
\begin{equation}\label{singularx}
\widehat{\mathcal{S}}_{\textnormal{sing}}=\left\{\widehat{x}_n=\frac{A}{4\left(\widehat{\Omega}_n-\frac{B}{2}\right)}, \quad \widehat{\Omega}_n\in\mathcal{S}_{\textnormal{sing}}\right\}.
\end{equation}
For any $n\geq 1$, consider the following sequences of functions 
\begin{equation}\label{Takk1}
\zeta_n(x)\triangleq  F_n(x)\left[1-x+\frac{A+2B}{A(n+1)} x\right]+\int_0^1 F_n(\tau x) \tau^n\left[-1+2x\tau\right] d\tau,\quad  \forall x\in(-\infty,1],
\end{equation}
 where $F_n$ has been introduced in \eqref{coefficients}. Then we prove the following. 
\begin{proposition}\label{statkernel}
Let  $A\in \R^\star$, $B\in \R$ and  
${x\in (-\infty,1)}\backslash\{ \widehat{\mathcal{S}}_{\textnormal{sing}}\cup\{0,x_0\}\}$, with \eqref{FormXX} and  \eqref{Firstzero}. Define the set
\begin{equation}\label{AX1}
\mathcal{A}_x\triangleq\Big\{n\in\N^\star,\quad  \zeta_n(x)=0\Big\}.
\end{equation}
Then, the kernel of $D_g\widehat{G}(\Omega,0)$ is finite--dimensional  and  generated by the $\mathscr{C}^\infty$ functions
$
\left\{\mathpzc{h}_n,\,\,  n\in\mathcal{A}_x\right\},
$
with
$
\mathpzc{h}_n: z\in\overline{\D}\mapsto \textnormal{Re}\left[\mathscr{G}_n(|z|^2)z^n\right]
$
and 
\[
\mathscr{G}_n(\ttt)=\frac{1}{1-x \ttt}\left[-\frac{P_n(\ttt)}{P_n(1)}+\frac{F_n(x\ttt)}{F_n(x)}-\frac{2xF_n(x\ttt)}{P_n(1)}\int_{\ttt}^1\frac{1}{\tau^{n+1}F_n^2(x\tau)}\int_0^\tau\frac{s^{n}F_n(xs)}{1-xs}P_n(s)dsd\tau\right].
\]
As a consequence,  $\displaystyle{\textnormal{dim } \textnormal{Ker } D_g\widehat{G}(\Omega,0)=\textnormal{Card }{\mathcal{A}_x}}.$  The functions $P_n$ and $F_n$ are defined in \eqref{Pn} and \eqref{coefficients}, respectively.
\end{proposition}

\begin{remark}
Notice that the set $\mathcal{A}_x$ can  be  empty;  in that case  the kernel of $D_g\widehat{G}(\Omega,0)$ is trivial. Otherwise, the set $\mathcal{A}_x$ is  finite.\end{remark}

\begin{proof}
To analyze the kernel structure, we return  to \eqref{linoperator} and solve the equations keeping in mind the relations \eqref{An1}. Thus we should solve 
\begin{align}\label{bneq}
\begin{split}
h_n(r)+\frac{4r}{n\left(r^2-\frac1x\right)}\left[-\frac{H_n(1)}{G_n(1)}G_n(r)+\frac{H_n(r)}{r^{n+1}}\right]&=0,\quad \forall r\in[0,1],\quad \forall n\in\N^\star,\\
 \frac{\frac{1}{x}-r^2}{8}h_0(r)-\int_\tau^1\frac{1}{\tau}\int_0^\tau sh_0(s)dsd\tau&=0,\quad  \forall r\in[0,1],
\end{split}
\end{align}
where the functions involved in the last expressions are given in \eqref{Gn}-\eqref{An1}. The term $(r^2-\frac{1}{x})$ is not vanishing from the assumptions on $x$. Note that the last equation for $n=0$ has been already studied in Proposition \ref{radialfunctions}, which  implies that  if $x\neq x_0$, then  the zero function is the only  solution. Hence, let us focus on the case $n\geq 1$ and solve the associated equation. To deal with this equation we  write down a  differential equation for $H_n$ and use Lemma \ref{lemmaODE}.  Firstly, we define the linear operator 
\begin{eqnarray}\label{linearoperator}
\mathscr{L}h(r)\triangleq r^{2n}\int_r^1\frac{1}{s^{n-1}}h(s)ds+\int_0^rs^{n+1}h(s)ds,
\end{eqnarray}
for any $h\in \mathscr{C}([0,1];\R)$. Then,  by differentiation we obtain
\begin{eqnarray*}
(\mathscr{L}h)^\prime(r)=2nr^{2n-1}\int_r^1\frac{1}{s^{n-1}}h(s)ds
=\frac{2n}{r}\left[\mathscr{L}h(r)-\int_0^rs^{n+1}h(s)ds\right].
\end{eqnarray*}
It is important to precise at this stage that $\mathscr{L}h$ satisfies the boundary conditions
\begin{equation}\label{conlinop}
(\mathscr{L}h)(0)=(\mathscr{L}h)^\prime(1)=0.
\end{equation}
Indeed, the second condition is obvious and to get the first one we use  that $h$ is bounded:
$$
|\mathscr{L}h(r)|\leq \|h\|_{L^\infty}\left(\frac{|r^{2n}-r^{2+n}|}{|n-2|}+\frac{r^{n+2}}{n+2}\right),
$$
for any $n\neq 2$. In the case $n=2$ we have
$$
|\mathscr{L}h(r)|\leq \|h\|_{L^\infty} \left(r^{4}|\ln r|+\frac{r^{4}}{4}\right),
$$
and for $n=1$ it is clearly verified.
Differentiating again  we obtain
\begin{eqnarray}\label{recoverbn}
 \frac{1}{2n}\left[r(\mathscr{L}h)^\prime(r)\right]^\prime-(\mathscr{L}h)^\prime(r)=-r^{n+1}h(r).
\end{eqnarray}
Since   $H_n=\mathscr{L}h_n$, one has 
\begin{eqnarray}\label{recoverbn1}
 \frac{1}{2n}\left[rH_n^\prime(r)\right]^\prime-H_n^\prime(r)=-r^{n+1}h_n(r).
\end{eqnarray}
 Using Equation (\ref{bneq}), we deduce that $H_n$ satisfies the following differential equation
\begin{equation}\label{difODE}
	(1-xr^2)rH_n''(r)-(1-xr^2)(2n-1)H_n'(r)+8rxH_n(r)=8\frac{H_n(1)x}{G_n(1)}r^{n+2}G_n(r),
\end{equation}
complemented with  the boundary conditions \eqref{conlinop}. 
Let us show how to recover the full solutions of \eqref{bneq} from this equation.  Assume that we have constructed all  the  solutions $H_n$ of \eqref{difODE},   with the boundary conditions \eqref{conlinop}. Then, to obtain  the solutions of  \eqref{bneq}, we should check the compatibility condition
$
\mathscr{L} h_n=H_n,
$
by setting
\begin{equation}\label{Trh1}
h_n(r)\triangleq \frac{4r}{n\left(r^2-\frac1x\right)}\left[H_n(1)\frac{G_n(r)}{G_n(1)}-\frac{H_n(r)}{r^{n+1}}\right].
\end{equation}
Combining \eqref{recoverbn} and \eqref{recoverbn1}, we deduce that $\mathpzc{H}\triangleq \mathscr{L}h_n-H_n$ satisfies 
\begin{equation*}
 \frac{1}{2n}\left[r\,\mathpzc{H}^\prime(r)\right]^\prime-\mathpzc{H}^\prime(r)=0.
\end{equation*}
By solving this differential equation we obtain the existence of two real constants  $\lambda_0$ and $\lambda_1$  such that
$
\mathpzc{H}(r)=\lambda_0+\lambda_1 r^{2n},$ for any $ r\in [0,1].
$
Since both $\mathscr{L}h_n$ and $H_n$ satisfy \eqref{conlinop}, then $\mathpzc{H}$ satisfies also these conditions. Hence, we find
$
\mathpzc{H}(r)=0,$ for any $r\in[0,1],
$
and this concludes that $h_n$, given by \eqref{Trh1}, is a solution of  \eqref{bneq}.
We emphasize that  $h_n$ satisfies the compatibility condition
$$
H_n(1)=\int_0^{1} r^{n+1}h_n(r)dr.
$$ 
 Indeed, integrating  \eqref{recoverbn1} from $0$ to $1$, we obtain
$$
\frac{H_n^\prime(1)}{2n}-\left(H_n(1)-H_n(0)\right)=-\int_0^{1} r^{n+1}h_n(r)dr.
$$
Thus, if $H_n$ satisfies the boundary conditions \eqref{conlinop},  then the compatibility condition is automatically verified.
Now, let us come back  to the resolution of  \eqref{difODE}. Since  $H_n(0)=0$,  then one  can apply  Lemma \ref{lemmaODE}  with
\begin{eqnarray}\label{gg}
g(r)=8\frac{H_n(1)x}{G_n(1)} r^{n+2}G_n(r).
\end{eqnarray}
Therefore, we obtain after a  change of variables  that 
\begin{eqnarray}\label{funcprop}
 H_n(r)&=&r^{2n}F_n(xr^2)H_n(1)\left[\frac{1}{F_n(x)}\right.
 \nonumber\\
 &&\hspace{3cm} \left.-\frac{2x^n}{G_n(1)}\int_{xr^2}^x\frac{1}{\tau^{n+1}F_n^2(\tau)}\int_0^\tau\frac{F_n(s)}{1-s}\left({\frac{s}{x}}\right)^{\frac{n+1}{2}}{{G_n\left(\left(\frac{s}{x}\right)^\frac12\right)}}dsd\tau\right]
 \nonumber\\
&=&  H_n(1)r^{2n}F_n(xr^2)\Bigg[\frac{1}{F_n(x)}
\nonumber\\ 
&&\hspace{3cm}-\frac{8x}{G_n(1)}\int_r^1\frac{1}{\tau^{2n+1}F_n^2(x\tau^2)}\int_0^\tau\frac{s^{n+2}F_n(xs^2)}{1-xs^2}G_n(s)dsd\tau\Bigg].
\end{eqnarray}
It remains to check the second initial condition: $H_n^\prime(1)=0$. From straightforward computations using \eqref{Gn}  and \eqref{Gn1} we find that
\begin{eqnarray}\label{derivkernel}
H_n^\prime(1)&=&\frac{H_n(1)}{F_n(x)}\left[2nF_n(x)+2xF_n^\prime(x)\right]+\frac{8xH_n(1)}{G_n(1)F_n(x)}\int_0^1\frac{s^{n+2}F_n(xs^2)}{1-xs^2}G_n(s)ds\nonumber\\
\nonumber &=&2\frac{H_n(1)}{F_n(x)}\left[\varphi_n(x)-\frac{An(n+1)x}{(n+2)G_n(1)}\int_0^1\frac{s^{2n+1}F_n(xs^2)}{1-xs^2}P_n(s^2)ds\right]\\
&=&\frac{2n H_n(1)}{F_n(x)G_n(1)}\Psi_n(x),
\end{eqnarray}
where 
$$
\Psi_n(x)\triangleq\varphi_n\left(x\right)\left[\frac{A}{4}\left(\frac{1}{x}-1\right)+\frac{A}{2}\frac{n+1}{n^2+2n}+\frac{B}{2n}\right]-\frac{A(n+1)x}{n+2}\int_0^1{}\frac{s^{2n+1} F_n\left(xs^2\right)}{1-xs^2}P_n(s^2)ds,
$$
and
\begin{equation}\label{phieq}
\varphi_n(x)\triangleq nF_n(x)+xF_n^\prime(x).
\end{equation}
Note that $G_n(1)\neq 0$ because $\Omega\notin \mathcal{S}_{\textnormal{sing}}$. Let us link $\Psi_n$  to the function $\zeta_n$. Recall from \eqref{completeeq} and \eqref{ODEEQ} that 
\begin{equation*} 
x(1-x)F_n^{\prime\prime}+(n+1)(1-x)F_n^\prime(x)+2F_n=0.
\end{equation*}
Differentiating the function $\varphi_n(x)$  and using the differential equation for $F_n$ we realize that
\begin{eqnarray*}
\varphi_n^\prime(x)&=&(n+1)F_n'(x)+xF_n''(x)\\
&=& \frac{1}{1-x}\left[(1-x)(n+1)F_n'(x)+(1-x)xF_n''(x)\right]
=-\frac{2F_n(x)}{1-x}.
\end{eqnarray*}
The change of variables $xs^2\mapsto \tau$ in the integral term yields
\begin{eqnarray*}
\int_0^1\frac{s^{2n+1} F_n\left(xs^2\right)}{1-xs^2}P_n(s^2)ds=\frac{1}{2x^{n+1}}\int_0^x\frac{\tau^{n}F_n(\tau)}{1-\tau}P_n\left(\frac{\tau}{x}\right)d\tau.
\end{eqnarray*}
Therefore we get
\begin{equation*}
\Psi_n\left(x\right)=\varphi_n(x)\left[\frac{A}{4}\left(\frac1x-1\right)+\frac{A}{2}\frac{n+1}{n^2+2n}+\frac{B}{2n}\right]+\frac{A(n+1)}{4(n+2) x^{n}}\int_0^x \varphi^\prime_n(\tau){\tau^n P_n\left(\frac{\tau}{x}\right)}d\tau.
\end{equation*}
From Definition \eqref{Pn} we find the identity
$$
\int_0^x \varphi^\prime_n(\tau){\tau^n P_n\left(\frac{\tau}{x}\right)}d\tau=\frac{1}{x^2}\int_0^x \varphi^\prime_n(\tau){\tau^n \left[\tau^2-\tau\frac{n+2}{n+1}-\frac{A+2B}{A}\frac{n+2}{n(n+1)} x^2\right]}d\tau.
$$
Integrating by parts, we deduce 
\begin{eqnarray*}
\frac{A(n+1)}{4(n+2) x^{n}}\int_0^x \varphi^\prime_n(\tau){\tau^n P_n\left(\frac{\tau}{x}\right)}d\tau&=&
\frac{A}{4}\varphi_n(x) \left(\frac{n+1}{n+2}-\frac{1}{x}-\frac{A+2B}{A n} \right)\\
&&-\frac{A(n+1)}{4 x^{n+2}}\int_0^x\varphi_n(\tau) \tau^{n-1}\left( \tau^2-\tau-\frac{A+2B}{A(n+1)}x^2 \right) d\tau.
\end{eqnarray*}
By virtue of the following identity
$$
\frac{A}{4}\left(\frac1x-1\right)+\frac{A}{2}\frac{n+1}{n^2+2n}+\frac{B}{2n}+\frac{A}{4} \left(\frac{n+1}{n+2}-\frac{1}{x}-\frac{A+2B}{A n} \right)=0,
$$
the boundary term in the integral is canceled  with the first part of $\Psi_n$. Thus
$$
\Psi_n(x)=-\frac{A(n+1)}{4 x^{n+2} }\int_0^x\varphi_n(\tau) \tau^{n-1}\left( \tau^2-\tau-\frac{A+2B}{A(n+1)}x^2 \right) d\tau,
$$
where, after a change of variables in the integral term, we get  that
$$
\Psi_n(x)=\frac{A(n+1)}{4 x}\int_0^1\varphi_n(\tau x) \tau^{n-1}\left( -x\tau^2+\tau+\frac{A+2B}{A(n+1)}x \right) d\tau .
$$
Setting 
\begin{eqnarray*}
\zeta_n(x)=\int_0^1\varphi_n(\tau x) \tau^{n-1}\left(-x \tau^2+\tau+\frac{A+2B}{A(n+1)} x\right) d\tau,
\end{eqnarray*}
we find the relation
\begin{equation*}
\Psi_n(x)=\frac{A(n+1)}{4 x}\zeta_n(x).
\end{equation*}
Observe  first that the zeroes of $\Psi_n$ and $\zeta_n$ are the same.
Coming back to \eqref{phieq} and 
integrating by parts we get the equivalent form
\begin{equation*}
\zeta_n(x)= F_n(x)\left(1-x+\frac{A+2B}{A(n+1)} x\right)+\int_0^1 F_n(\tau x) \tau^n\left(-1+2x\tau\right) d\tau.
\end{equation*}
This gives  \eqref{Takk1}.
According to \eqref{derivkernel}, the constraint  $H_n'(1)=0$  is equivalent to  $H_n(1)=0$ or $\zeta_n(x)=0$. In the first case,  we get  from \eqref{funcprop} that $H_n\equiv 0$ and inserting this into \eqref{Trh1} we find
$
h_n(r)=0,$ for any $ r\in[0,1].
$
Thus, for  $n\notin\mathcal{A}_x$, where $\mathcal{A}_x$ is given by  \eqref{AX1}, we obtain that  there is only one  solution for the kernel equation, which is  the trivial one.
As to the second condition $\zeta_n(x)=0$, which agrees with $n\in \mathcal{A}_x$, one gets from \eqref{funcprop}  and \eqref{Trh1}  that the kernel of $D_g\widehat{G}(\Omega,0)$ restricted to the level  frequency $n$ is generated by 
$$
\mathpzc{h}_n(re^{i\theta})=h^\star_n(r)\cos(n\theta),
$$
with
\begin{equation*}
\displaystyle{{h}^\star_n(r)=\frac{1}{1-xr^2}\left[-\frac{rG_n(r)}{G_n(1)}+ \frac{r^{n}F_n(xr^2)}{F_n(x)}-{8xr^{n}F_n(xr^2)}\bigints_r^1\frac{\displaystyle{\int_0^\tau\frac{s^{n+2}F_n(xs^2)}{1-xs^2}\frac{G_n(s)}{G_n(1)}ds}}{\tau^{2n+1}F_n^2(x\tau^2)}d\tau\right].}
\end{equation*}
The fact that $h_n=\frac{4x}{n}H_n(1)\, h^\star_n$ together with  $H_n(1)=\displaystyle{\int_0^1}s^{n+1}h_n(s)ds$ imply
\begin{eqnarray}\label{inthn}
\int_0^1s^{n+1}h_n^\star(s)ds=\frac{n}{4x}.
\end{eqnarray}
Using
$$
\frac{rG_n(r)}{G_n(1)}=r^n\frac{P_n(r^2)}{P_n(1)},
$$
and a suitable change of variables allows getting the formula,
\begin{eqnarray*}
\bigints_r^1\frac{\displaystyle{\int_0^\tau\frac{s^{n+2}F_n(xs^2)}{1-xs^2}\frac{G_n(s)}{G_n(1)}ds}}{\tau^{2n+1}F_n^2(x\tau^2)}d\tau&=&\frac{1}{P_n(1)}\bigints_r^1\frac{\displaystyle{\int_0^\tau\frac{s^{2n+1}F_n(xs^2)}{1-xs^2}P_n(s^2)ds}}{\tau^{2n+1}F_n^2(x\tau^2)}d\tau\\
&=&\frac{1}{4P_n(1)}\bigints_{r^2}^1\frac{\displaystyle{\int_0^\tau\frac{s^{n}F_n(xs)}{1-xs}P_n(s)ds}}{\tau^{n+1}F_n^2(x\tau)}d\tau.
\end{eqnarray*}
We have
\begin{equation*}
{h}^\star_n(r)=\frac{r^n}{1-xr^2}\left[-\frac{P_n(r^2)}{P_n(1)}+ \frac{F_n(xr^2)}{F_n(x)}-\frac{2x F_n(xr^2)}{P_n(1)}\int_{r^2}^1\frac{\displaystyle{1}}{\tau^{n+1}F_n^2(x\tau)}\int_0^\tau\frac{s^{n}F_n(xs)}{1-xs}P_n(s)ds d\tau\right].
\end{equation*}
Setting
$$
\mathscr{G}_n(\ttt)=\frac{1}{1-x\ttt}\left[-\frac{P_n(\ttt)}{P_n(1)}+\frac{F_n(x\ttt)}{F_n(x)}-\frac{2xF_n(x\ttt)}{P_n(1)}\int_{\ttt}^1\frac{\displaystyle{1}}{\tau^{n+1}F_n^2(x\tau)}\int_0^\tau\frac{s^{n}F_n(xs)}{1-xs}P_n(s)ds d\tau\right],
$$
we deduce that
\begin{eqnarray}\label{hkernel}
\mathpzc{h}_n(z)=\mathscr{G}_n(|z|^2)|z|^n\cos(n\theta)
=\Real\left[\mathscr{G}_n(|z|^2)z^n\right], \quad \forall \,\, z\in \overline{\D}.
\end{eqnarray}
We intend to check that $\mathpzc{h}_n$ belongs to $\mathscr{C}^{\infty}(\overline{\D};\R).$ To get this it is enough to   verify that $\mathscr{G}$ belongs to $\mathscr{C}^{\infty}([0,1];\R)$. Since $x\in (-\infty,1)$, then $\tau\in[0,1]\mapsto F_n(x\tau)$ is in $\mathscr{C}^\infty([0,1];\R).$ The change of variables $s=\tau \theta$ implies 
$$
\displaystyle \bigints_{\ttt}^1\frac{\displaystyle \int_0^\tau\frac{s^{n}F_n(xs)}{1-xs}P_n(s)ds}{\tau^{n+1}F_n^2(x\tau)}d\tau =\bigints_{\ttt}^1\frac{\displaystyle \int_0^1\frac{\theta^{n}F_n(x\tau \theta) P_n(\tau \theta)}{1-x\tau \theta}d\theta }{F_n^2(x\tau)}d\tau.
$$
Since $F_n$ does not vanish on $(-\infty,1)$, then the mapping $\tau\in[0,1]\mapsto \frac{1}{F_n(x\tau)}$ belongs to  $\mathscr{C}^{\infty}([0,1];\R).$  It suffices to observe that the integral function is $\mathscr{C}^\infty$ on $[0,1]$.  Then,  we have an independent element of the kernel given by $h_n^\star(\cdot)$, for any $n\in \mathcal{A}_x$. This concludes the announced result.
\end{proof}

\begin{remark}\label{n1}
The hypergeometric function $F_n$ for $n=1$ can be computed as
$
F_1(r)=F(a_1,b_1;c_1;r)=1-r.
$
Hence, the function \eqref{Takk1} becomes
\begin{eqnarray*}
\zeta_1(x)&=& F_1(x)\left[1-x+\frac{A+2B}{2A}x\right]+\int_0^1F_1(\tau x)\tau\left[-1+2x\tau\right]d\tau\\
&=&(1-x)\left[1-x+\frac{A+2B}{2A}x\right]+\int_0^1(1-\tau x)\tau\left[-1+2x\tau\right]d\tau
=(1-x)\left(\frac{B}{A}x+\frac{1}{2}\right).
\end{eqnarray*}
The root $x=1$ is not allowed since $x\notin \widehat{\mathcal{S}}_{\textnormal{sing}}$. {Therefore, the unique root is $x=-\frac{A}{2B}$. Coming back to $\Omega$ using \eqref{FormXX}, one has that $\Omega=0$.}
\end{remark}

\subsubsection{Singular case}
The singular case $x\in(1,+\infty)$ is studied in this section. Notice that from  \eqref{FormXX}, we obtain 
\begin{equation}\label{CondOmega}
\frac{B}{2}<\Omega<\frac{B}{2}+\frac{A}{4}.
\end{equation}
It is worthy to point out that this  case is degenerate  because the leading terms of the  equations of the linearized operator \eqref{bneq}  vanish inside the unit disc. 
To understand this operator one should deal with a second differential equation of hypergeometric type with a singularity. Thus, the first  difficulty amounts to solving  those equations across the singularity and invert the operator. This can be done in a straightforward way  getting   that the operator is injective with an explicit representation of its formal inverse. However, it is not an isomorphism and undergoes  a loss of regularity in the H\"{o}lder class. Despite this  bad behavior, one would  expect at least  the persistence of the injectivity for the nonlinear problem.  This  problem appears in different contexts, for instance in the inverse backscattering problem \cite{S-U}. The idea to overcome this difficulty is to prove two key ingredients. The first one concerns the coercivity of the linearized operator with a quantified loss in the H\"{o}lder class. The second point is to  use the Taylor expansion and to establish  a soft  estimate  for the reminder combined with  an interpolation argument.  This argument leads to the following result.
\begin{theorem}\label{singthm}
Let ${0<\alpha<1}$, $A>0$ and $B\in\R$ such that $\frac{B}{A}\notin\left[-1,-\frac12\right]$. Assume that  $\Omega$ satisfies
 \eqref{CondOmega} and $\Omega\notin \mathcal{S}_{\textnormal{sing}}$, where this latter set is defined in \eqref{Interv2}. Then, there exists a small \mbox{neighborhood $V$} of the origin in $\mathscr{C}^{2,\alpha}_s(\D)$ such that the nonlinear equation \eqref{densityEq} has no solution in $V$, except the origin. Notice that in the case $B<-A$,  the condition $\Omega\notin \mathcal{S}_{\textnormal{sing}}$ follows automatically  from \eqref{CondOmega}.
\end{theorem}
The proof of this theorem will be given at the end of this section. Before we should develop some tools. Let us start with solving the kernel equations, for this reason we introduce some auxiliary functions. Set
\begin{equation}\label{TODA1}
\widehat{F}_n(x)=F(-a_n,b_n;b_n-a_n+1;x),
\end{equation}
and define the functions
\begin{equation}\label{EqTQ1}	           
   \mathscr{F}_{K_1,K_2}(y)=    \left\{\begin{array}{ll}
          	y^nF_n(y)\left[\frac{K_1}{F_n(1)}-\displaystyle{\int_y^1\frac{\int_0^\tau\frac{F_n(s)}{1-s}\mathcal{R}(\frac{s}{x})ds}{\tau^{n+1}F_n^2(\tau)}d\tau}\right],\quad  y\in[0,1],\\
	{y^{a_n}\widehat{F}_n\left(\frac1y\right)}\left[\frac{K_1}{\widehat{F}_n(1)}+\displaystyle{\int_{\frac1y}^1}\frac{K_2-\int_\tau^1\frac{s^{n-a_n-1}\widehat{F}_n(s)}{1-s}\mathcal{R}(\frac{1}{sx})ds}{\tau^{n+1-2a_n}\widehat{F}_n^2(\tau)}d\tau\right], \quad y\geq1,
       \end{array}\right.
\end{equation}
and
\begin{eqnarray*}              
  \widehat{\mathscr{F}}_{K_1,K_2}(y)=    \left\{\begin{array}{ll}
          	y(1-y)\left[K_1+\displaystyle{\int_{0}^{y}\frac{\int_0^\tau\mathcal{R}(\frac{s}{x})ds}{\tau^{2}(1-\tau)^2}d\tau}\right], \quad y\in[0,1],\\
	{\frac{1}{y}\widehat{F}_1\left(\frac{1}{y}\right)}\left[\frac13\int_0^1\mathcal{R}\left(\frac{s}{x}\right)ds+\displaystyle{\int_{\frac{1}{y}}^1}\frac{K_2-\int_\tau^1\frac{s\widehat{F}_1(s)}{1-s}\mathcal{R}(\frac{1}{sx})ds}{\tau^{4}\widehat{F}_1^2(\tau)}d\tau\right],\quad y\geq1,
       \end{array}\right.
\end{eqnarray*}
with
\begin{equation}\label{TODA2}
\mathcal{R}(y)=\frac{1}{4x}y^{-\frac12}g\big({y}^{\frac12}\big),\end{equation}
where $g$ is the source term in \eqref{ODE}.
Our first result reads as follows.
\begin{lemma}\label{lemmaODEX1}
Let  $n\geq 2$ be an integer,  ${x\in(1,+\infty)}$ and $g\in \mathscr{C}([0,1];\R)$. Then, the  continuous solutions in $[0,+\infty)$ to the equation \eqref{ODE}, 
such that $F(0)=0$, are given by the  two--parameters curve 
$$
r\in[0,1]\mapsto F(r)=\mathscr{F}_{K_1,K_2}(xr^2),\quad K_1,K_2\in\R,
$$
with the constraint 
\begin{align}\label{h1}
\mathcal{R}\left(\frac1x\right)=0.
\end{align}
Moreover, if $g\in \mathscr{C}^\mu([0,1];\R)$, for some $\mu>0$, then the  above solutions  are  $\mathscr{C}^1$ on $[0,1]$ if and only if the following conditions hold true:
\begin{equation}\label{ZEDA1}
K_1=0\quad \hbox{and}\quad K_2=\frac{\widehat{F}_n(1)}{F_n(1)}\int_0^1\frac{F_n(s)\mathcal{R}(\frac{s}{x})}{1-s}ds.
\end{equation}
For $n=1$, if $F(0)=0$ and $g\in \mathscr{C}([0,1];\R)$, with the condition \eqref{h1}, then
the continuous solutions are given by the two-parameters curve
$$
r\in[0,1]\mapsto F(r)=\widehat{\mathscr{F}}_{K_1,K_2}(xr^2),\quad K_1,K_2\in\R.
$$
{If $g\in \mathscr{C}^\mu([0,1];\R)$, for some $\mu>0$,  then this solution  is $\mathscr{C}^1$  if and only if
\begin{align}\label{h2}
 \mathcal{R}(0)=\int_0^{\frac1x}\mathcal{R}\left({\tau}\right)d\tau=0,
\end{align}
and $K_1$ and $K_2$ satisfy
\begin{align}\label{K1K31} 
K_2&=-3\left(K_1+\int_0^1\frac{\int_0^\tau\mathcal{R}(\frac{s}{x})ds}{\tau^{2}(1-\tau)^2}d\tau\right).
\end{align}
Moreover, if $F'(1)=0$, we have the additional constraint
\begin{align}\label{K1K32} 
K_2&\left(\frac{1}{x^2}\right.\left.\left[\widehat{F}_1\left(\frac1x\right)+x\widehat{F}_1'\left(\frac1x\right)\right]+\frac{x}{\widehat{F}_1\left(\frac1x\right)}\right)
=\frac{x}{\widehat{F}_1(\frac1x)}\int_{\frac1x}^1\frac{s\widehat{F}_1(s)}{1-s}\mathcal{R}\left(\frac{1}{sx}\right)ds\\
&+\frac{1}{x^2}\left[\widehat{F}_1\left(\frac1x\right)+\frac{1}{x}\widehat{F}_1'\left(\frac1x\right)\right] \nonumber \times\int_{\frac1x}^1\frac{\int_\tau^1\frac{s\widehat{F}_1(s)}{1-s}\mathcal{R}(\frac{1}{sx})ds}{\tau^{4}\widehat{F}_1^2(\tau)}d\tau.
\end{align}}
\end{lemma}

\begin{proof}
We proceed as in the proof of Lemma \ref{lemmaODE}. The resolution of the equation \eqref{completeeq},  in the interval $[0,1]$, is exactly the same and we find
$$
\mathpzc{F}(y)=\mathscr{F}_{K_1,K_2}(y), \quad \forall y\in[0,1].
$$
In the interval $[1,x]$, we first solve the homogeneous equation associated \mbox{to \eqref{completeeq}}. By virtue of Appendix \ref{SecSpecialfunctions}, one gets two independent solutions, one of them is described by
$$
\mathpzc{F}_0(y)= y^{a_n}\widehat{F}_n\left(\frac1y\right), \quad \forall y\in[1,x].
$$
Using the method of variation of constants, we obtain that the general solutions to the equation  \eqref{completeeq} in this interval  take the form
$$
\mathpzc{F}(y)=\mathscr{F}_{K_1,K_2}(y), \quad \forall y\in[1,x].
$$
Since $K_1$ coincides in both sides  of \eqref{EqTQ1} and the integrals converge at $1$ by \eqref{h1}, then we deduce that $\mathpzc{F}$ is continuous in $[0,x]$. Therefore, we get the first part of the Lemma by using simply that $F(r)=\mathpzc{F}(x r^2),$ for $r\in [0,1]$.

Let us now  select  in this class those solutions who are  $\mathscr{C}^1$. Notice that the solutions $\mathcal{F}$  are $\mathscr{C}^1$ in $[0,x]\backslash\{1\}$. So it remains to study the derivatives from the left and the right on $y=1$. Since $F_n$ and $\widehat{F}_n$ have no derivatives on the left at $1$ and verify
$$
 |{F}_n^\prime(y)|\sim C\ln(1-y)\quad \hbox{and}\quad |\widehat{F}_n^\prime(y)|\sim C\ln  (1-y),
$$
for any $y\in[0,1)$, see \eqref{log1}, then the first members of \eqref{EqTQ1} have no derivatives at $1$. This implies necessary that  $K_1=0$. Moreover, one gets
\begin{equation}\label{gauche}
\mathscr{F}_{K_1,K_2}^\prime(1^-)=\frac{\int_0^1\frac{F_n(s)}{1-s}\mathcal{R}\left(\frac{s}{x}\right)ds}{F_n(1)}\cdot
\end{equation}
Since $F_n(1)>0$ and $\mathcal{R}$ is H\"older continuous, then the convergence of this integral is equivalent to  the condition $\mathcal{R}\left(\frac1x\right)=0.$ Now, using \eqref{EqTQ1}, we deduce that the right derivative of $\mathscr{F}_{K_1,K_2}$ at $1$ is given by
\begin{equation}\label{droite}
\mathscr{F}_{K_1,K_2}^\prime(1^+)=\frac{K_2}{\widehat{F}_n(1)}\cdot
\end{equation}
Combining \eqref{gauche} with  \eqref{droite}, we deduce that $F_n$ admits a derivative at $1$ if and only if 
$$
K_2=\frac{ \widehat{F}_n(1)}{F_n(1)}\int_0^1\frac{F_n(s)}{1-s}\mathcal{R}\left(\frac{s}{x}\right)ds.
$$
Thus,   the $\mathscr{C}^1-$solution to \eqref{EqTQ1} is given by 
\begin{eqnarray*} 	           
          	&&\qquad \qquad \mathpzc{F}(y)= y^nF_n(y)\displaystyle{\bigintsss_1^y\frac{\bigintsss_0^\tau\frac{F_n(s)}{1-s}\mathcal{R}\left(\frac{s}{x}\right)ds}{\tau^{n+1}F_n^2(\tau)}d\tau}\, {\bf{1}}_{[0,1]}(y)+\\
\nonumber&&	{y^{a_n}\widehat{F}_n\left(\frac1y\right)}\displaystyle{\bigintsss_{\frac1y}^1}\frac{\frac{ \widehat{F}_n(1)}{F_n(1)}\bigintsss_0^1\frac{F_n(s)}{1-s}\mathcal{R}\left(\frac{s}{x}\right)ds-\bigintsss_\tau^1\frac{s^{n-a_n-1}\widehat{F}_n(s)}{1-s}\mathcal{R}(\frac{1}{sx})ds}{\tau^{n+1-2a_n}\widehat{F}_n^2(\tau)}d\tau\,{\bf{1}}_{[1,\infty)}(y),
\end{eqnarray*}
and, therefore, the solution to \eqref{ODE} takes the form
\begin{eqnarray}  \label{EqTQ23} 	  
 F(r)=\mathpzc{F}(xr^2).
 \end{eqnarray}
This implies in particular that there is only one $\mathscr{C}^1$ solution   to \eqref{ODE} and satisfies 
$
F(x^{-\frac12})=0.
$
The case $n=1$ is  very special  since $F_1(x)=1-x$ and hence $F_1$ vanishes at $1$. As in the previous discussion, the solution of \eqref{completeeq} in $[0,1]$ is given by
$$
\mathpzc{F}(y)=   
          	y(1-y)\left[K_1+\displaystyle{\int_0^{y}\frac{\int_0^\tau\mathcal{R}(\frac{s}{x})ds}{\tau^{2}(1-\tau)^2}d\tau}\right],
$$
whereas for $y\in[1,\infty)$ the solution reads as
$$
\mathpzc{F}(y)=\frac{1}{y}\widehat{F}_1\left(\frac{1}{y}\right)\left[K_3+\displaystyle{\int_{\frac1y}^1}\frac{K_2-\int_\tau^1\frac{s\widehat{F}_1(s)}{1-s}\mathcal{R}(\frac{1}{sx})ds}{\tau^{4}\widehat{F}_1^2(\tau)}d\tau\right],
$$
where $K_1, K_2$ and $K_3$  are constants. We can check that the continuity of $\mathpzc{F}$ at $1$ is satisfied if and only if 
\begin{eqnarray*}
K_3&=&\frac{1}{\hat{F}_1(1)}\int_0^1\mathcal{R}\left(\frac{s}{x}\right)ds\\
&=&\frac{1}{3}\int_0^1\mathcal{R}\left(\frac{s}{x}\right)ds,
\end{eqnarray*}
{by using in the last line the explicit expression of $\hat{F}_1(1)$ coming from \eqref{id1}:
$$\hat{F}_1(1)=F(1,2;4;1)=\frac{\Gamma(4)\Gamma(1)}{\Gamma(3)\Gamma(2)}=3.$$}
Note that we need \eqref{h1} in order to ensure the convergence of the integral
$$
\int_\tau^1\frac{s\widehat{F}_1(s)}{1-s}\mathcal{R}\left(\frac{1}{sx}\right)ds.
$$
Let us deal with the derivative, given by
$$
\mathpzc{F}'(y)=\left\{\begin{array}{lr}
\displaystyle (1-2y)\left[K_1+\displaystyle{\int_0^y\frac{\int_0^\tau\mathcal{R}(\frac{s}{x})ds}{\tau^{2}(1-\tau)^2}d\tau}\right]+\frac{1}{y(1-y)}\int_0^y\mathcal{R}\left(\frac{s}{x}\right)ds,&\hspace{-0.6cm} y\in[0,1],\\
{-\frac{1}{y^2}\left[\widehat{F}_1\left(\frac1y\right)+\frac{1}{y}\widehat{F}_1'\left(\frac1y\right)\right]}\Bigg[\frac{1}{3}\int_0^1\mathcal{R}\left(\frac{s}{x}\right)ds 
\\ \ +\displaystyle{\int_{\frac1y}^1}\frac{K_2-\int_\tau^1\frac{s\widehat{F}_1(s)}{1-s}\mathcal{R}(\frac{1}{sx})ds}{\tau^{4}\widehat{F}_1^2(\tau)}d\tau\Bigg]
+\frac{y}{\widehat{F}_1\left(\frac1y\right)}\left[K_2-\int_{\frac1y}^1\frac{s\widehat{F}_1(s)}{1-s}\mathcal{R}\left(\frac{1}{sx}\right)ds\right],& \hspace{-0.6cm}y\geq 1.
\end{array}\right.
$$
The convergence in $0$ and $1$ of the first part comes from $\mathcal{R}(0)=\mathcal{R}\left(\frac1x\right)=\int_0^1\mathcal{R}\left(\frac{s}{x}\right)d\tau=0$. In which case, one gets
$$
\mathpzc{F}'(1^-)=-K_1-\int_0^1\frac{\int_0^\tau\mathcal{R}(\frac{s}{x})ds}{\tau^{2}(1-\tau)^2}d\tau.
$$
For the second part of $\mathpzc{F}$, one needs $\int_0^1\mathcal{R}\left(\frac{s}{x}\right)d\tau=0$ since $\widehat{F}'$ is singular at $1$ as it was mentioned before. Then,
$$
\mathpzc{F}'(1^+)=\frac13K_2.
$$
{Clearly, we have the constraint
\begin{align}\label{relK1K3}
K_2=-3\left(K_1+\int_0^1\frac{\int_0^\tau\mathcal{R}(\frac{s}{x})ds}{\tau^{2}(1-\tau)^2}d\tau\right),
\end{align}
in order to obtain $\mathscr{C}^1$ solutions.}
If, in addition, $F'(1)=0$, which agrees with $\mathpzc{F}'(x)=0$, we obtain the following additional equation for $K_2$:
\begin{eqnarray*}
&&K_2\left(\frac{1}{x^2}\left[\widehat{F}_1\left(\frac1x\right)+x\widehat{F}_1'\left(\frac1x\right)\right]+\frac{x}{\widehat{F}_1\left(\frac1x\right)}\right)
=\frac{x}{\widehat{F}_1\left(\frac1x\right)}\int_{\frac1x}^1\frac{s\widehat{F}_1(s)}{1-s}\mathcal{R}\left(\frac{1}{sx}\right)ds+\\
&&\frac{1}{x^2}\left[\widehat{F}_1\left(\frac1x\right)+\frac{1}{x}\widehat{F}_1'\left(\frac1x\right)\right]\int_{\frac1x}^1\frac{\int_\tau^1\frac{s\widehat{F}_1(s)}{1-s}\mathcal{R}(\frac{1}{sx})ds}{\tau^{4}\widehat{F}_1^2(\tau)}d\tau.
\end{eqnarray*}
{Notice that
$$
\frac{1}{x^2}\left[\widehat{F}_1\left(\frac1x\right)+x\widehat{F}_1'\left(\frac1x\right)\right]+\frac{x}{\widehat{F}_1\left(\frac1x\right)}\neq 0,
$$
using that $\widehat{F}_1(z)=F(1,2;4;z)$, for $z\in(0,1)$, which is a positive and increasing function.} Then, we can obtain the exact value of $K_2$ and then $K_1$ via relation \eqref{relK1K3}. Using that $F(r)=\mathpzc{F}(xr^2)$. The proof is now achieved.

\end{proof}
{
\begin{proposition}\label{trivialkernelsing}
Let $A>0$ and $B\in\R$ such that $\frac{B}{A}\notin\left[-1,-\frac12\right]$. Let $\Omega$ satisfy \eqref{CondOmega} with $\Omega\notin {\mathcal{S}}_{\textnormal{sing}}$, where the last set is defined in \eqref{Interv2}. Then, the following holds true:
\begin{enumerate}
\item The kernel of  $D_g\widehat{G}(\Omega,0)$ is trivial.
 \item  If  $h$ and $d$ are smooth enough with
 $$
 D_g\widehat{G}(\Omega,0)h=d, \quad h(re^{i\theta})=\sum_{n\in\N}h_n(r)\cos(n\theta), \quad d(re^{i\theta})=\sum_{n\in\N}d_n(r)\cos(n\theta),
 $$
 then, there exists an absolute constant $C>0$ such that 
 $$
  \|h_n\|_{\mathscr{C}^0([0,1])}\leq C \|d_n\|_{\mathscr{C}^1([0,1])}, \quad \forall\, n\in\N.
 $$
 \item Coercivity with loss of derivative: for any $\alpha\in(0,1)$, there exists $C>0$ such that
 $$
\|h\|_{\mathscr{C}^0(\D)}\leq C\|D_g\widehat{G}(\Omega,0) h\|_{\mathscr{C}_s^{2,\alpha}(\D)}.
 $$
 \end{enumerate}
\end{proposition}
\begin{proof}
{${\bf(1)}$}  First, note that $x\neq x_0$, where $x_0$ is defined in \eqref{Firstzero}, because $x_0\in(0,1)$. Then, Proposition \ref{radialfunctions} implies  that the last equation in \eqref{bneq} admits only the zero function as a solution.{
We will check how the condition \eqref{h1} gives us that  there are not nontrivial smooth solutions for $n=1$. This can be done easily with the explicit expression of $g$ given in \eqref{gg} for $n=1$. Since
$$
G_1(r)=-2\left[r^4-\frac{3}{2x}r^2-\frac{3}{2}\left(1-\frac{1}{x}\right)\right],
$$ 
one obtains that
$$
g(r)=-16H_1(1)xr^3\left[r^4-\frac{3}{2x}r^2-\frac{3}{2}\left(1-\frac{1}{x}\right)\right].
$$
Then, condition \eqref{h1} is equivalent to
$$
3x^2-3x+1= 0 \quad\hbox{or}\quad H_1(1)=0.
$$
Since $x>1,$  one has  $H_1(1)=0$ and $g\equiv 0$. Then, $\mathcal{R}\equiv 0$. By Lemma \ref{lemmaODEX1}, one has that the solution of \eqref{difODE} with $H_1(0)=0$ and $H_1'(1)=0$ is the trivial one: $H_1(r)\equiv 0$. Coming back to \eqref{Trh1}, we deduce that $h_1\equiv 0.$  Therefore, the only smooth solution is the zero function, which  implies finally  that the kernel is trivial.}

 Let us now deal with $n\geq2.$  Applying Lemma \ref{lemmaODEX1} to  the equation \eqref{difODE} with \eqref{TODA2} we get that this equation admits a $\mathscr{C}^1$ solution if and only if
\begin{equation}\label{VanX1}
H_n(1)G_n(x^{-\frac12})=0.
\end{equation}
Using the expression \eqref{Gn}, we find that 
$$
G_n(x^{-\frac12})=-\frac{An(n+1)}{4(n+2)}x^{\frac{1-n}{2}}P_n\left(x^{-1}\right),
$$
and from \eqref{Pn} one has
\begin{equation}\label{Van55}
P_n\left(x^{-1}\right)=-\frac{1}{n+1}\left(\frac{1}{x^2}+\frac{A+2B}{A}\frac{n+2}{n}\right).
\end{equation}
With the assumptions $\frac{A+2B}{A}\notin[-1,0]$ and $x>1$ one gets
$$
P_n\left(x^{-1}\right)\neq0, \quad \forall n\geq 2,
$$
obtaining from \eqref{VanX1} that
$$H_n(1)=0, \quad \forall n\geq 2.
$$
Coming back to \eqref{difODE} we find that the source term is vanishing everywhere. Now, from \mbox{Lemma \ref{lemmaODEX1}} and \eqref{EqTQ1} we infer that 
$$
 H_n(r)=0,\quad \forall n\geq 2, r\in[0,1].
$$
Inserting this  into \eqref{bneq}, we obtain $h_n\equiv 0$ for any $n\in\N^\star$. Finally, this implies that the vanishing function is the only element of the kernel.

\medskip
\noindent
{${\bf(2)}$} To get this result we should derive   {\it a priori estimates} for the solutions to  the equation
$$
D_g\widehat{G}(\Omega,0)h=d.
$$
The pre-image equation  is equivalent to solve the infinite-dimensional system
\begin{align}\label{rangeeqLX}
\begin{split}
\frac{\frac{1}{x}-r^2}{8}h_n(r)-\frac{r}{n}\left[A_nG_n(r)+\frac{1}{2r^{n+1}}H_n(r)\right]&=d_n(r),\quad \forall n\geq 1,\\
\frac{\frac{1}{x}-r^2}{8}h_0(r)-\int_r^1\frac{1}{\tau}\int_0^\tau sh_0(s)dsd\tau&=d_0(r),
\end{split}
\end{align}
where we use the notations of \eqref{Gn}, \eqref{Hn} and \eqref{An1}. 
We first analyze the case of large  $n$, for which we can apply the contraction principle and  get the desired   estimates. Later, we will deal with low frequencies, which is more delicate and requires the integral representation  \eqref{EqTQ23}.  Let us first work with large values of $n$. Observe that  the first equation of \eqref{rangeeqLX} can be transformed into
\begin{equation}\label{Eqfirstsing}
h_n(r)=\frac{8}{\frac1x-r^2}\left[\frac{A_n}{n}rG_n(r)+\frac{H_n(r)}{2nr^{n}}+d_n(r)\right].
\end{equation}
The estimate of   $A_n$, defined by \eqref{Hn} and \eqref{An1}, can be done as follows
 \begin{equation*}
|A_n|\le \frac{\int_0^{1}s^{n+1}|h_n(s)|ds}{2n|\widehat{\Omega}_n-\Omega|}.
\end{equation*}  
Keeping in mind the relation \eqref{FormXX} and the assumption $x\notin \widehat{\mathcal{S}}_{\textnormal{sing}}$, we obtain
$$
\inf_{n\geq1}\frac{1}{|\widehat{\Omega}_n-\Omega|}<+ \infty,
$$
and, on the other hand, it is obvious that
\begin{equation*}
\int_0^{1}s^{n+1}|h_n(s)|ds\leq \frac{\|h_n\|_{\mathscr{C}^0([0,1])}}{n+2}.
\end{equation*}  
Combining  the preceding estimates, we find
\begin{equation}\label{EsAn}
|A_n|\le \frac{C}{n^2}\|h_n\|_{\mathscr{C}^0([0,1])},
\end{equation}
for some  constant $C$ independent of $n.$  Let us remark that according to   the equation \eqref{Eqfirstsing}, one should get the compatibility condition
$$
\frac{A_n}{n}x^{-\frac12}G_n(x^{-\frac12})+\frac{H_n(x^{-\frac12})}{2nx^{-\frac{n}{2}}}+d_n(x^{-\frac12})=0.
$$ 
Hence, applying  the Mean Value Theorem, combined with \eqref{EsAn}, we obtain
\begin{eqnarray}\label{Eshn}
\|h_n\|_{\mathscr{C}^0([0,1])}&\le& C\frac{\|h_n\|_{\mathscr{C}^0([0,1])}}{n^{3}}\left\|\left(rG_n(r)\right)^\prime\right\|_{\mathscr{C}^0([0,1])} \nonumber \\&&+\frac{1}{2n}\left\|\left(\frac{H_n(r)}{r^{n}}\right)^\prime\right\|_{\mathscr{C}^0([0,1])}+\|d_n^\prime\|_{\mathscr{C}^0([0,1])},
\end{eqnarray}
where in this inequality $C$ may depend  on $x$ but not on $n$. Now, it is straightforward to check  that
\begin{equation}\label{EsGn}
\left|\left(rG_n(r)\right)^\prime\right|\le C n^2,\quad \forall \, r\in [0,1],
\end{equation}
and second
$$
\left(\frac{H_n(r)}{r^{n}}\right)^\prime=n r^{n-1}\int_r^1\frac{h_n(s)}{s^{n-1}}ds-\frac{n}{r^{n+1}}\int_0^rs^{n+1} h_n(s)ds.
$$
From this latter identity, we infer that 
\begin{eqnarray*}
 \left|\left(\frac{H_n(r)}{r^{n}}\right)^\prime\right|&\le& \left(\frac{n}{n-2}+\frac{n}{n+2}\right)\|h_n\|_{\mathscr{C}^0([0,1])}\le C\|h_n\|_{\mathscr{C}^0([0,1])},\quad \forall \, r\in [0,1],
\end{eqnarray*}
for any $n\geq3$. Consequently, we get
\begin{equation}\label{EsHn}
\left\|\left(\frac{H_n(r)}{r^{n}}\right)^\prime\right\|_{\mathscr{C}^0([0,1])}\le C\|h_n\|_{\mathscr{C}^0([0,1])}.
\end{equation}
Plugging  \eqref{EsHn} and \eqref{EsGn} into \eqref{Eshn}, we find
\begin{equation*}
\|h_n\|_{\mathscr{C}^0([0,1])}\le \frac{C}{n}\|h_n\|_{\mathscr{C}^0([0,1])}+\|d_n^\prime\|_{\mathscr{C}^0([0,1])}.
\end{equation*}
Hence, choosing $n_0$ large enough we deduce  that \begin{equation*}
\|h_n\|_{\mathscr{C}^0([0,1])}\le \|d_n^\prime\|_{\mathscr{C}^0([0,1])},\quad \forall\, n\geq n_0.
\end{equation*}
Next, we deal with the cases $1\le n\leq n_0.$ The preceding argument fails and to invert the operator we recover explicitly the solution $h_n$ from $d_n$ according to the integral representation \eqref{EqTQ23}. For this purpose we will proceed as  in the range  study  in Subsection \ref{RangeSEc}. By virtue of  \eqref{X111}, we find that $H_n$ satisfies an equation of the type \eqref{ODE}. Thus, using  \eqref{EqTQ23}, we deduce
\begin{eqnarray}  \label{EqTQ25} 	           
        H_n(r)&=&  x^nr^{2n}F_n(xr^2)\displaystyle{\bigintsss_1^{xr^2}\frac{\bigintsss_0^\tau\frac{F_n(s)}{1-s}\mathcal{R}(\frac{s}{x})ds}{\tau^{n+1}F_n^2(\tau)}d\tau}\, {\bf{1}}_{[0,x^{-\frac12}]}(r) +	{x^{a_n}r^{2a_n}\widehat{F}_n\left(\frac{1}{xr^2}\right)} \nonumber\\
 && \times \displaystyle{\bigintsss_{\frac{1}{xr^2}}^1}\frac{\frac{ \widehat{F}_n(1)}{F_n(1)}\bigintsss_0^1\frac{F_n(s)}{1-s}\mathcal{R}(\frac{s}{x})ds-\bigintsss_\tau^1\frac{s^{n-a_n-1}\widehat{F}_n(s)}{1-s}\mathcal{R}(\frac{1}{sx})ds}{\tau^{n+1-2a_n}\widehat{F}_n^2(\tau)}d\tau\,{\bf{1}}_{(x^{-\frac12},1]}(r),
\end{eqnarray}
for $n\geq 2$, and
\begin{align*}
H_1(r)= &	xr^2(1-xr^2)\left[K_1-\displaystyle{\int_{xr^2}^0\frac{\int_0^\tau\mathcal{R}(\frac{s}{x})ds}{\tau^{2}(1-\tau)^2}d\tau}\right]\, {\bf{1}}_{[0,x^{-\frac12}]}(r)\\
&+  	{\frac{1}{xr^2}\widehat{F}_1\left(\frac{1}{xr^2}\right)}\left[\frac13\int_0^1\mathcal{R}\left(\frac{s}{x}\right)ds+\displaystyle{\int_{\frac{1}{xr^2}}^1}\frac{K_2-\int_\tau^1\frac{s\widehat{F}_1(s)}{1-s}\mathcal{R}(\frac{1}{sx})ds}{\tau^{4}\widehat{F}_1^2(\tau)}d\tau\right]   \,{\bf{1}}_{(x^{-\frac12},1]}(r),
\end{align*}
 where $K_1$ and $K_2$ are given in \eqref{K1K31}-\eqref{K1K32}. From \eqref{TODA2}, we get
\begin{equation*}
\mathcal{R}(y)=\frac{1}{4x}\left[A A_n\frac{n(n+1)}{n+2} y^{n}P_n(y)-4 x ny^{\frac{n}{2}}d_n(y^{\frac12})\right],
\end{equation*}
where $P_n$ is defined in \eqref{Pn}. 

Let us relate $A_n$ with $d_n$.  This can be obtained from the \mbox{constraint $\mathcal{R}(\frac{1}{x})=0$}, which implies 
 $$
 A_n=\frac{(n+2) x^{\frac{n}{2}+1}d_n(x^{-\frac12})}{A(n+1) P_n(x^{-1})}.
 $$
Let us remark that this relation is different from \eqref{EsAn}, which is not useful for low frequencies.  Consequently, using \eqref{Van55} we infer that 
 $$
| A_n|\le C\|d_n\|_{\mathscr{C}^0([0,1])}, \quad \forall n\in[1,n_0].
 $$
Concerning $\mathcal{R}$, we can get successively
\begin{equation}\label{EsZP0}
 |\mathcal{R}(y)|\le C\left[|A_n|y^{n}+y^{\frac{n}{2}}\left|d_n(y^{\frac12})\right|\right],\quad \forall\, y\in[0,1],
\end{equation}
which implies that
$$
||\mathcal{R}||_{\mathscr{C}^0([0,1])}\leq C\left[|A_n|+||d_n||_{\mathscr{C}^0([0,1])}\right]\leq C||d_n||_{\mathscr{C}^0([0,1])},
$$
for any $n\geq 1$. In the case $n\geq 2$, we also obtain that
\begin{equation*}
|\mathcal{R}^\prime(y)|\le C\left[|A_n|y^{n-1}+y^{\frac{n}{2}-1}\left(\left|d_n(y^{\frac12})\right|+\left|d_n^\prime(y^{\frac12})\right|\right)\right], \quad \forall\, y\in[0,1],
\end{equation*}
which amounts to
\begin{equation}\label{BoundZ1}
\|\mathcal{R}^\prime\|_{\mathscr{C}^0([0,1])}\le C \left( \|d_n\|_{\mathscr{C}^0([0,1])}+\|d_n^\prime\|_{\mathscr{C}^0([0,1])}\right),\quad \forall n\geq2.
\end{equation}
Note that this last estimate can not be used for $n=1$ since $\mathcal{R}$ is only H\"older continuous. Then, we can find in this case that
$$
||\mathcal{R}||_{\mathscr{C}^{0,\gamma}([0,1])}\leq C||d_n||_{\mathscr{C}^{0,\gamma}([0,1])},
$$
for $\mu=\min\left(\frac12,\alpha\right)$.

Let us begin with $n\geq 2$. Using  the  boundedness property of $F_n$, which we shall see later in Lemma \ref{lemestim}, combined  with an integration by parts imply 
\begin{eqnarray}\label{LastX}
\nonumber\left|\bigintsss_1^{xr^2}\frac{\bigintsss_0^\tau\frac{F_n(s)}{1-s}\mathcal{R}(\frac{s}{x})ds}{\tau^{n+1}F_n^2(\tau)}d\tau\right|&\le&C\left| \bigintsss_1^{xr^2}\frac{\bigintsss_0^\tau\frac{\left|\mathcal{R}({s}/{x})\right|}{1-s}ds}{\tau^{n+1}}d\tau\right|\\
\nonumber&\le&C\int_0^1\frac{\left|\mathcal{R}(\frac{s}{x})\right|}{1-s}ds+Cr^{-2n}\int_0^{xr^2}\frac{\left|\mathcal{R}(\frac{s}{x})\right|}{1-s}ds\\
&&+\left|\int_1^{xr^2}\frac{\tau^{-n}\left|\mathcal{R}(\frac{\tau}{x})\right|}{1-\tau}d\tau\right|.
\end{eqnarray}
We discuss first the case  $r\in[0,\frac12x^{-\frac12}]$. From the compatibility assumption \eqref{ZEDA1}, we infer that $\mathcal{R}(\frac1x)=0$, and therefore we deduce from \eqref{EsZP0} and \eqref{BoundZ1} that
\begin{align*}
\int_0^1\frac{\left|\mathcal{R}(\frac{s}{x})\right|}{1-s}ds&+r^{-2n}\int_0^{xr^2}\frac{\left|\mathcal{R}(\frac{s}{x})\right|}{1-s}ds\le\int_0^1\frac{\left|\mathcal{R}(\frac{s}{x})\right|}{1-s}ds+Cr^{-2n}\int_0^{xr^2}\left|\mathcal{R}\left(\frac{s}{x}\right)\right|ds\\
&\le C \left(|A_n|+ \|d_n\|_{\mathscr{C}^0([0,1])}+\|d_n^\prime\|_{\mathscr{C}^0([0,1])}+ r^{-n+2}\|d_n\|_{\mathscr{C}^0([0,1])}\right),
\end{align*}
 for $2\le n\le n_0$. In a similar way to  the last integral term of \eqref{LastX}, we  split it  as follows using the estimates \eqref{EsZP0} and \eqref{BoundZ1}
\begin{eqnarray*}
\bigintsss_{xr^2}^1\frac{\tau^{-n}\left|\mathcal{R}(\frac{\tau}{x})\right|}{1-\tau}d\tau&=&\bigintsss_{\frac12}^{1}\frac{\tau^{-n}\left|\mathcal{R}(\frac{\tau}{x})\right|}{1-\tau}d\tau+\bigintsss_{xr^2}^{\frac12}\frac{\tau^{-n}\left|\mathcal{R}(\frac{\tau}{x})\right|}{1-\tau}d\tau\\
&\le&C \left(|A_n|+ \|d_n\|_{\mathscr{C}^0([0,1])}+\|d_n^\prime\|_{\mathscr{C}^0([0,1])}\right)+C\bigintsss_{xr^2}^{\frac12}\tau^{-n}\left|\mathcal{R}\left(\frac{\tau}{x}\right)\right|d\tau\\
&\le&C \left(|A_n|+ \|d_n\|_{\mathscr{C}^0([0,1])}+\|d_n^\prime\|_{\mathscr{C}^0([0,1])}+Cr^{-n+2}\|d_n\|_{\mathscr{C}^0([0,1])}\right).
\end{eqnarray*}
Putting together the preceding estimates we get that 
$$
\left|\bigintsss_1^{xr^2}\frac{\bigintsss_0^\tau\frac{F_n(s)}{1-s}\mathcal{R}(\frac{s}{x})ds}{\tau^{n+1}F_n^2(\tau)}d\tau\right|\le C \left(|A_n|+ \|d_n\|_{\mathscr{C}^0([0,1])}+\|d_n^\prime\|_{\mathscr{C}^0([0,1])}+Cr^{-n+2}\|d_n\|_{\mathscr{C}^0([0,1])}\right),
$$
for $r\in[0,\frac12x^{-\frac12}]$. Plugging this into  \eqref{EqTQ25} yields
\begin{equation}\label{Hn66}
|H_n(r)|\le C \left(|A_n|+ \|d_n\|_{\mathscr{C}^0([0,1])}+\|d_n^\prime\|_{\mathscr{C}^0([0,1])}\right) r^{2n}+Cr^{n+2}\|d_n\|_{\mathscr{C}^0([0,1])},
\end{equation}
 for $r\in[0,\frac12x^{-\frac12}]$ and $2\le n\leq n_0$. Now, we wish  to estimate the derivative of $\frac{H_n(r)}{r^n}$. Coming back to \eqref{EqTQ25}, we deduce from elementary computations that 
 \begin{equation*}
H_n^\prime(r)=H_n(r)\left(\frac{2n}{r}+\frac{2xr F_n^\prime(xr^2)}{F_n(xr^2)}\right)+\frac{2}{rF_n(xr^2)}\bigintsss_0^{xr^2}\frac{F_n(s)}{1-s}\mathcal{R}\left(\frac{s}{x}\right)ds,
\end{equation*}
\mbox{for $r\in[0,\frac12x^{-\frac12}]$}. By \eqref{Hn66}, we get
 \begin{eqnarray*}
\left|H_n(r)\left (\frac{2n}{r}+\frac{2xr F_n^\prime(xr^2)}{F_n(xr^2)}\right)\right| &\le&  C \left(|A_n|+ \|d_n\|_{\mathscr{C}^0([0,1])}+\|d_n^\prime\|_{\mathscr{C}^0([0,1])}\right) r^{2n-1}\\ &&+Cr^{n+1}\|d_n\|_{\mathscr{C}^0([0,1])}.
\end{eqnarray*}
 Concerning the integral term, it suffices to  apply \eqref{EsZP0} in order to get
  \begin{align*}
\frac{2}{rF_n(xr^2)}\bigintsss_0^{xr^2}\frac{F_n(s)}{1-s}\left|\mathcal{R}\left(\frac{s}{x}\right)\right|ds\le \frac{C}{r}\bigintsss_0^{xr^2}\left|\mathcal{R}\left(\frac{s}{x}\right)\right|ds\le  C \left(|A_n| r^{2n+1}+ r^{n+1}\|d_n\|_{\mathscr{C}^0([0,1])}\right).
\end{align*}
 Hence, combining the preceding estimates leads to
 \begin{equation}
 \left|H_n^\prime(r)\right|\le  C \left(|A_n|+ \|d_n\|_{\mathscr{C}^0([0,1])}+\|d_n^\prime\|_{\mathscr{C}^0([0,1])}\right) r^{2n-1}+Cr^{n+1}\|d_n\|_{\mathscr{C}^0([0,1])}.
 \end{equation}
This estimate together with  \eqref{Hn66}  allows getting
 \begin{eqnarray*}
\left|\left(\frac{H_n(r)}{r^n}\right)^\prime\right|\le  C \left(|A_n|+ \|d_n\|_{\mathscr{C}^0([0,1])}+\|d_n^\prime\|_{\mathscr{C}^0([0,1])}\right), \quad \forall\, r\in\left[0,\frac12x^{-\frac12}\right],
\end{eqnarray*}
 for  $2\le n\leq n_0$. The case $n=1$ can be done using similar ideas since we only have the singularity at $0$ in this interval. Note that $K_1$ and $K_2$ can be estimated in terms of $\mathcal{R}$ having
$$
|K_1|, |K_2|\leq ||\mathcal{R}||_{\mathscr{C}^{0,\gamma}([0,1])}.
$$ 

Let us now move to the intermediate  case $x\in [\frac12x^{-\frac12},x^{-\frac12}]$. Then, there is no singularity in this range  except for $r=x^{-\frac12}$ due to the logarithmic behavior  of $F_n$ close to this point. This logarithmic divergence can be controlled   from the smallness of the integral term in $H_n$. Let us show the idea. When we differentiate $H_n$, we obtain one term of the type
$$
2x^{n+1}r^{2n+1} F_n'(xr^2)\bigintsss_1^{xr^2}\frac{\bigintsss_0^\tau\frac{F_n(s)}{1-s}\mathcal{R}(\frac{s}{x})ds}{\tau^{n+1}F_n^2(\tau)}d\tau,
$$
where we notice the logarithmic singularity coming from $F_n'$ at $1$. However, one has
\begin{align*}
\left|2x^{n+1}r^{2n+1} F_n'(xr^2)\bigintsss_1^{xr^2}\frac{\bigintsss_0^\tau\frac{F_n(s)}{1-s}\mathcal{R}(\frac{s}{x})ds}{\tau^{n+1}F_n^2(\tau)}d\tau\right|&\leq C\left|F_n'(xr^2)(1-xr^2)\frac{\bigintsss_0^{xr^2}\frac{F_n(s)}{1-s}\mathcal{R}(\frac{s}{x})ds}{(xr^2)^{n+1}F_n^2(xr^2)}\right|\\
&\leq C\left(||\mathcal{R}||_{\mathscr{C}^0([0,1])}+||\mathcal{R}'||_{\mathscr{C}^0([0,1])}\right)\\
&\leq C\left(||d_n||_{\mathscr{C}^0([0,1])}+||d_n'||_{\mathscr{C}^0([0,1])}\right).
\end{align*}
Therefore, after straightforward efforts on \eqref{EqTQ25} using \eqref{BoundZ1}, it implies that 
 \begin{equation}
 \left|H_n^\prime(r)\right|\le  C \left(|A_n|+ \|d_n\|_{\mathscr{C}^0([0,1])}+\|d_n^\prime\|_{\mathscr{C}^0([0,1])}\right),
 \end{equation}
for $2\le n\leq n_0$. Hence, we obtain 
 \begin{eqnarray*}
\left|\left(\frac{H_n(r)}{r^n}\right)^\prime\right|\le  C \left(|A_n|+ \|d_n\|_{\mathscr{C}^0([0,1])}+\|d_n^\prime\|_{\mathscr{C}^0([0,1])}\right), \quad \forall\, r\in \left[\frac12x^{-\frac12},x^{-\frac12}\right],
\end{eqnarray*}
for $2\le n\le n_0$. In this interval, the case $n=1$ is different. This is because we have not singularity coming from the hypergeometric function but we do have it coming from the integral. Hence, some more manipulations are needed. In this case, $H_1$ reads as
$$
H_1(r)=xr^2(1-xr^2)\left[K_1+\int_{0}^{\frac12}\frac{\int_0^\tau\mathcal{R}(\frac{s}{x})ds}{\tau^{2}(1-\tau)^2}d\tau+\int_{\frac12}^{xr^2}\frac{\int_0^\tau\mathcal{R}(\frac{s}{x})ds}{\tau^{2}(1-\tau)^2}d\tau\right].
$$
The first integral term can be treated as in the previous computations in the interval $[0,\frac12 x^{-\frac12}]$. Let us focus on the singular integral term. By a change of variables, one has
\begin{align*}
\left|\int_{\frac12}^{xr^2}\frac{\int_0^\tau\mathcal{R}(\frac{s}{x})ds}{\tau^{2}(1-\tau)^2}d\tau\right|
&= \left| \frac{\int_0^{xr^2}\mathcal{R}\left(\frac{s}{x}\right)ds}{x^2r^4(1-xr^2)}-8\int_0^{\frac12}\mathcal{R}\left(\frac{s}{x}\right)ds-\int_{\frac12}^{xr^2}\frac{\tau \mathcal{R}(\frac{\tau}{x})-2\int_0^\tau\mathcal{R}(\frac{s}{x})ds}{\tau^3(1-\tau)}d\tau\right|\\
&\leq  C\left( ||\mathcal{R}||_{\mathscr{C}^0([0,1])}+||\mathcal{R}||_{\mathscr{C}^{0,\gamma}([0,1])}\right),
\end{align*}
where we have used \eqref{h1}-\eqref{h2}. Then, we obtain
$$
|H_1(r)|\leq C(1-xr^2) ||\mathcal{R}||_{\mathscr{C}^{0,\gamma}([0,1])}.
$$
Similar arguments can be done to find that
$$
 \left|\left(\frac{H_1(r)}{r}\right)^\prime\right|\le  C ||d_n||_{\mathscr{C}^{0,\gamma}([0,1])},
$$
for any $r\in[\frac12 x^{-\frac12}, x^{-\frac12}]$.

It remains to establish similar results for the case $r\in[x^{-\frac12},1]$. With this aim we use the second integral in \eqref{EqTQ25}. Notice that the only singular point is $r=x^{-\frac12}$ due to the logarithmic singularity of the hypergeometric function $\widehat{F}_n$ defined in \eqref{TODA1}. One can check that this function is strictly increasing, positive and satisfies
$$ 1\leq \widehat{F}_n(r)\le \sup_{n\in \N}\widehat{F}_n(1)<+\infty, \quad \forall r\in[0,1].
$$
As in the previous interval, the smallness of the integral term controls this singularity. The same happens for the case $n=1$.
Hence, we get 
 \begin{eqnarray*}
 \left|\left(\frac{H_n(r)}{r^n}\right)^\prime\right|\le  C \left(|A_n|+ \|d_n\|_{\mathscr{C}^0([0,1])}+\|d_n^\prime\|_{\mathscr{C}^0([0,1])}\right), \quad \forall\, r\in [x^{-\frac12},1],
\end{eqnarray*}
for $1\le n\le n_0$. Therefore,   in  all the cases we have 
\begin{equation}\label{Qii11}
\left|\left(\frac{H_n(r)}{r^n}\right)^\prime\right|\le  C \left( \|d_n\|_{\mathscr{C}^0([0,1])}+\|d_n^\prime\|_{\mathscr{C}^0([0,1])}\right),\quad \forall\, r\in [0,1],
\end{equation}
for $1\le n\le n_0$. Applying the Mean Value Theorem to \eqref{Eqfirstsing}, and using \eqref{Qii11} and \eqref{EsGn}, allow us to obtain
\begin{equation*}
\|h_n\|_{\mathscr{C}^0([0,1])}\le  C \left( \|d_n\|_{\mathscr{C}^0([0,1])}+\|d_n^\prime\|_{\mathscr{C}^0([0,1])}\right), \quad \forall\,n\in[1,n_0].
\end{equation*}
Similar arguments can be done in order to deal with the equation for $n=0$. Note that the resolution of this equation is similar to the work done in Proposition \ref{propdegen}.

Then, combining all the estimates, it yields
\begin{equation*}
\|h_n\|_{\mathscr{C}^0([0,1])}\le  C \left( \|d_n\|_{\mathscr{C}^0([0,1])}+\|d_n^\prime\|_{\mathscr{C}^0([0,1])}\right),\quad \forall\,n\in\N.
\end{equation*}
This achieves the proof of the announced result.

\medskip
\noindent
{${\bf(3)}$} Let us recall the formula for the Fourier coefficients
$$
d_n(r)=\frac{1}{\pi}\int_0^{2\pi}d(r\cos\theta,r\sin\theta) \cos(n\theta)d\theta.
$$
We can prove that
$$
\|d_n\|_{\mathscr{C}^0([0,1])}\leq \frac{C}{n^{1+\alpha}}\|d\|_{\mathscr{C}^{1,\alpha}(\D)},
$$
for $n\geq1$. This can be done integrating by parts as
$$
d_n(r)=-\frac{r}{n\pi}\int_0^{2\pi}\nabla d(r\cos\theta,r\sin\theta) \cdot (-\sin\theta, \cos\theta) \sin(n\theta)d\theta,
$$
and writing it as
\begin{align*}
d_n(r)=&\frac{r}{n\pi}\int_0^{2\pi}\nabla d\left(r\cos\left(\theta+\frac{\pi}{n}\right),r\sin\left(\theta+\frac{\pi}{n}\right)\right) \cdot \left(-\sin\left(\theta+\frac{\pi}{n}\right), \cos\left(\theta+\frac{\pi}{n}\right)\right) \sin(n\theta)d\theta\\
=&\frac{r}{2n\pi}\int_0^{2\pi}\Big[\nabla d\left(r\cos\left(\theta+\frac{\pi}{n}\right),r\sin\left(\theta+\frac{\pi}{n}\right)\right) \cdot \left(-\sin\left(\theta+\frac{\pi}{n}\right), \cos\left(\theta+\frac{\pi}{n}\right)\right) \\
&-\nabla d\left(r\cos\theta,r\sin\theta\right) \cdot (-\sin\theta, \cos\theta)\Big]\sin(n\theta)d\theta.
\end{align*}
Consequently,
$$
|d_n(r)|\leq \frac{C}{n^{1+\alpha}}||d||_{\mathscr{C}^{1,\alpha}(\D)}.
$$
With similar arguments, one achieves that 
$$
\|d_n\|_{\mathscr{C}^0([0,1])}+\|d_n^\prime\|_{\mathscr{C}^0([0,1])}\leq \frac{C}{n^{1+\alpha}}\|d\|_{\mathscr{C}^{2,\alpha}(\D)},
$$
and then
$$
\|h_n\|_{\mathscr{C}^0([0,1])}\le \frac{C}{n^{1+\alpha}}\|d\|_{\mathscr{C}^{2,\alpha}(\D)}.
$$
Therefore, we obtain
\begin{eqnarray*}
\|h\|_{\mathscr{C}^0(\D)}\leq \sum_{n\in\N}\|h_n\|_{\mathscr{C}^0([0,1])}\le C\|d\|_{\mathscr{C}^{2,\alpha}(\D)},
\end{eqnarray*}
which completes the proof.
\end{proof}
}
The next target is  to provide  the proof of Theorem \ref{singthm}.
\begin{proof}[Proof of Theorem \ref{singthm}]
In a small neighborhood of the origin we have the  following  decomposition  through Taylor expansion at the second order 
$$
\widehat{G}(\Omega,h)=D_g\widehat{G}(\Omega,0)(h)+\frac12 D^2_{g,g}\widehat{G}(\Omega,0)(h,h)+\mathscr{R}_2(\Omega, h),
$$
where $\mathscr{R}_2(\Omega, h)$ is the remainder term, which verifies
$$
\|\mathscr{R}_2(\Omega, h)\|_{\mathscr{C}^{2,\alpha}(\D)}\leq \frac16 \|D^3_{g,g,g}\widehat{G}(\Omega,g)(h,h,h)\|_{\mathscr{C}^{2,\alpha}(\D)},
$$
where $g=th$ for some $t\in(0,1)$. We intend to show the following,
\begin{align}\label{Estim2}
\left\|\frac12 D^2_{g,g}\widehat{G}(\Omega,0)(h,h)+\mathscr{R}_2(\Omega, h)\right\|_{\mathscr{C}^{2,\alpha}(\D)}\leq C\|h\|_{\mathscr{C}^{2,\alpha}(\D)}^{\delta_1} \|h\|_{\mathscr{C}^{0}(\D)}^{\delta_2},
\end{align}
for some $\delta_1>0$ and $\delta_2\geq 1$, getting the last bound will  be crucial in our argument. First, let us deal with the second derivative of $\widehat{G}$. Straightforward computations, similar to what was done in Proposition \ref{Gwelldefined}, lead to
\begin{align*}
D^2_{g,g}\widehat{G}(\Omega,0)&(h,h)(z)=-\frac{h(z)^2}{8A}+\frac{\textnormal{Re}}{\pi}\int_\D\frac{\partial_g \phi(\Omega,0)(h)(z)-\partial_g \phi(\Omega,0)(h)(y)}{z-y}h(y)dA(y)\\
&+\frac{2}{\pi}\int_\D\log|z-y|h(y)\textnormal{Re}\left[\partial_g \phi(\Omega,0)(h)'(y)\right]dA(y)\\
&-\frac{\textnormal{Re}}{2\pi}\int_\D\frac{(\partial_g \phi(\Omega,0)(h)(z)-\partial_g \phi(\Omega,0)(h)(y))^2}{(z-y)^2}f_0(y)dA(y)\\
&+\frac{2\textnormal{Re}}{\pi}\int_\D\frac{\partial_g \phi(\Omega,0)(h)(z)-\partial_g \phi(\Omega,0)(h)(y)}{z-y}f_0(y)\textnormal{Re}\left[\partial_g \phi(\Omega,0)(h)'(y)\right]dA(y)\\
&-\Omega|\partial_g \phi(\Omega,0)(h)(z)|^2+\frac{1}{\pi}\int_\D \log|z-y|f_0(y)|\partial_g \phi(\Omega,0)(h)'(y)|^2dA(y)\\
&-\Omega \textnormal{Re}\left[\partial^2_{g,g} \phi(\Omega,0)(h,h)(z)\right]\\
&+\frac{\textnormal{Re}}{2\pi}\int_\D \frac{\partial^2_{g,g} \phi(\Omega,0)(h,h)(z)-\partial^2_{g,g} \phi(\Omega,0)(h,h)(y)}{z-y}f_0(y)dA(y)\\
&+\frac{1}{\pi}\int_\D \log|z-y|f_0(y)\textnormal{Re}\left[\partial^2_{g,g} \phi(\Omega,0)(h,h)'(y)\right]dA(y).
\end{align*}
Recall the relation between $\partial_{g} \phi(\Omega,0)(h)$ and $h$ 
$$
\partial_{g} \phi(\Omega,0)(h)(z)=z\sum_{n\geq 1} A_n z^{n},\quad A_n={\frac{\displaystyle{\int_0^1s^{n+1}h_n(s)ds}}{2n\big(\widehat{\Omega}_n-\Omega\big)}}\cdot
$$
By Proposition \ref{propImpl}, one has that
$$
\|\partial_g\phi(\Omega, 0)h\|_{\mathscr{C}^{2,\alpha}(\D)}\leq  C\|h\|_{\mathscr{C}^{1,\alpha}(\D)}.
$$
We claim that one can reach the limit case, 
\begin{equation}\label{limitcaseX}
\|\partial_g\phi(\Omega, 0)h\|_{\mathscr{C}^{1}(\D)}\leq ||h||_{\mathscr{C}^0(\D)}.
\end{equation}
The last estimate can be done using Proposition \ref{equivnorms} in order to work in $\T$ as follows
\begin{align*}
\left|\partial_g\phi(\Omega, 0)h'(e^{i\theta})\right|&=\left|\sum_{n\geq 1}\int_0^1 s^{n+1}h_n(s)ds e^{in\theta}\right|\\
&= \frac12\left| \sum_{n\geq 1}\int_0^{2\pi}\int_0^1 s^{n+1}h(se^{i\theta'})\cos(n\theta')e^{in\theta}dsd\theta'\right|\\
&= \frac14 \left|\sum_{n\geq 1}\int_0^{2\pi}\int_0^1 s^{n+1}h(se^{i\theta'})\big({e^{in\theta'}+e^{-in\theta'}}\big)e^{in\theta}dsd\theta'\right|\\
&= \frac14 \left|\sum_{n\geq 1}\int_0^{2\pi}\int_0^1 s^{n+1}h(se^{i\theta'})\left(e^{in(\theta+\theta')}+e^{in(\theta-\theta')}\right)dsd\theta'\right|.
\end{align*}
Using Fubini we deduce that
\begin{eqnarray*}
\left|\partial_g\phi(\Omega, 0)h'(e^{i\theta})\right|&=& \frac14  \left|\int_0^{2\pi}\int_0^1 s^{2}h(se^{i\theta'})\left(\frac{e^{i(\theta+\theta')}}{1-se^{i(\theta+\theta')}}+\frac{e^{i(\theta-\theta')}}{1-se^{i(\theta-\theta')}}\right)dsd\theta'\right|\\
&\leq& C||h||_{\mathscr{C}^0(\D)}.
\end{eqnarray*}
The latter estimate follows from the convergence of the double integral
$$
\int_0^{2\pi}\int_0^1\frac{ds d\theta}{|1-se^{i\theta}|}<\infty.
$$
The next step is to deal with $\partial^2_{g,g} \phi(\Omega,0)(h,h)$. By differentiating it, similarly to the proof of  in Proposition \ref{propImpl}, we obtain
\begin{align*}
\partial^2_{g,g} \phi(\Omega,0)(h,h)=&-\frac12 \partial_\phi F(\Omega,0,0)^{-1}\left[\partial^2_{g,g}F(\Omega,0,0)(h,h)+2\partial^2_{g,\phi}F(\Omega,0,0)(h, \partial_g\phi(\Omega,0)h)\right.\\
& \left. +\partial^2_{\phi,\phi}F(\Omega,0,0)(\partial_g\phi(\Omega,0)h,\partial_g\phi(\Omega,0)h)\right]\\
=& -\frac12 \sum_{n\geq 1}\frac{\rho_n}{n(\Omega-\widehat{\Omega}_n)}z^{n+1},
\end{align*}
where $\rho(w)=\sum_{n\geq 1}\rho_n\sin(n\theta)$ and
\begin{align*}
\rho(w)=&\partial^2_{g,g}F(\Omega,0,0)(h,h)(w)+2\partial^2_{g,\phi}F(\Omega,0,0)(h, \partial_g\phi(\Omega,0)h)(w)\\
&+\partial^2_{\phi,\phi}F(\Omega,0,0)(\partial_g\phi(\Omega,0)h,\partial_g\phi(\Omega,0)h)(w)\\
=& \textnormal{Im}\Bigg[ -\frac{w\partial_g\phi(\Omega,0)(h)'(w)}{\pi}\int_\D\frac{h(y)}{w-y}dA(y)
\\&+\frac{w}{\pi}\int_\D \frac{\partial_g\phi(\Omega,0)(h)(w)-\partial_g\phi(\Omega,0)(h)(y)}{(w-y)^2}h(y)dA(y)\\
&\left.-\frac{2w}{\pi}\int_\D \frac{h(y)}{w-y}\textnormal{Re}\left[\partial_g\phi(\Omega,0)(h)'(y)\right]dA(y)\right.
\left. +2\Omega\,\overline{\partial_g\phi(\Omega,0)(h)(w)}\partial_g\phi(\Omega,0)(h)'(w)w\right.\\
&\left. +\frac{w\partial_g\phi(\Omega,0)(h)'(w)}{\pi}\int_\D \frac{\partial_g\phi(\Omega,0)(h)(w)-\partial_g\phi(\Omega,0)(h)(y)}{(w-y)^2}f_0(y)dA(y)\right.\\
&\left.-\frac{2w\partial_g\phi(\Omega,0)(h)'(w)}{\pi}\int_\D \frac{f_0(y)}{w-y}\textnormal{Re}\left[\partial_g\phi(\Omega,0)(h)'(y)\right]dA(y)\right.\\
&\left. -\frac{w}{\pi} \int_\D\frac{\big[\partial_g\phi(\Omega,0)(h)(w)-\partial_g\phi(\Omega,0)(h)(y)\big]^2}{(w-y)^3}f_0(y)dA(y)\right.\\
&\left.+\frac{2w}{\pi}\int_\D \frac{\partial_g\phi(\Omega,0)(h)(w)-\partial_g\phi(\Omega,0)(h)(y)}{(w-y)^2}f_0(y) \textnormal{Re}\left[\partial_g\phi(\Omega,0)(h)'(y)\right]dA(y)\right.\\
&-\frac{w}{\pi}\int_\D \frac{f_0(y)}{w-y}|\partial_g\phi(\Omega,0)(h)'(y)|^2dA(y) \Bigg].
\end{align*}
By Proposition \ref{ContinQ1} and due to the fact that  $\partial^2_{g,g}\phi(\Omega,0)$ can be seen as a convolution operator, one has that
$$
||\partial^2_{g,g}\phi(\Omega,0)||_{\mathscr{C}^{k,\alpha}(\D)}\leq C ||\rho||_{\mathscr{C}^{k,\alpha}(\T)}.
$$
Moreover, we claim that
\begin{align}\label{Estim1}
||\rho||_{\mathscr{C}^{k,\alpha}(\T)}\leq C||h||_{\mathscr{C}^{2,\alpha}(\D)}^{\sigma_1}||h||_{\mathscr{C}^0(\D)}^{\sigma_2}, \quad k=0,1,2,
\end{align}
with  $\sigma_2\geq 1$. First, we use the interpolation inequalities for H\"older spaces
\begin{align}\label{Estim3}
||h||_{\mathscr{C}^{k,\alpha}(\D)}\leq C ||h||_{\mathscr{C}^{k_1,\alpha_1}(\D)}^{\beta}||h||_{\mathscr{C}^{k_2,\alpha_2}(\D)}^{1-\beta},
\end{align}
for $k$, $k_1$ and $k_2$ non negative integers, $0\leq \alpha_1,\alpha_2\leq 1$ and
$$
k+\alpha=\beta(k_1+\alpha_1)+(1-\beta)(k_2+\alpha_2),
$$
where $\beta\in(0,1)$. The proof of the interpolation inequality can be found in \cite{Helms}.  In order to get the announced results, we would need to use some classical results in Potential Theory dealing with the Newtonian potential and the Beurling transform, see Appendix \ref{Appotentialtheory} or for instance \cite{EncisoPoyatoSoler, Kress, MateuOrobitgVerdera, Miranda}.
Now, let us show the idea behind \eqref{Estim1}. To estimate  the first term of $\rho$, we combine  \eqref{limitcaseX} with the law products in H\"{o}lder spaces, as follows,
\begin{align*}
\Bigg\|\partial_g\phi(\Omega,0)(h)'(\cdot)&\int_\D \frac{h(y)}{(\cdot)-y}dA(y)\Bigg\|_{\mathscr{C}^{k,\alpha}(\T)}\leq C \Big\|\partial_g\phi(\Omega,0)(h)'(\cdot)\Big\|_{\mathscr{C}^{0}(\T)} \left\|\int_\D \frac{h(y)}{(\cdot)-y}dA(y)\right\|_{\mathscr{C}^{k,\alpha}(\T)}\\
&+C\left\|\partial_g\phi(\Omega,0)(h)'(\cdot)\right\|_{\mathscr{C}^{k,\alpha}(\T)}\left\|\int_\D \frac{h(y)}{(\cdot)-y}dA(y)\right\|_{\mathscr{C}^{0}(\T)}\\
\leq &\ C\left( \left\|h\right\|_{\mathscr{C}^{0}(\D)}\left\|h\right\|_{\mathscr{C}^{k,\alpha}(\D)}+\left\|h\right\|_{\mathscr{C}^{k,\alpha}(\D)}\left\|h\right\|_{\mathscr{C}^{0}(\D)}\right).
\end{align*}
Then, \eqref{Estim1} is satisfied for the first term. Let us deal with the second term of $\rho$. {For $k=0$, one has
\begin{align*}
\Big\|\int_\D &\frac{\partial_g\phi(\Omega,0)(h)(\cdot)-\partial_g\phi(\Omega,0)(h)(y)}{((\cdot)-y)^2}h(y)dA(y)\Big\|_{\mathscr{C}^{0,\alpha}(\T)}\\
& \leq  C\left\|\partial_g\phi(\Omega,0)(h)(\cdot)\, \textnormal{p.v.}\int_\D \frac{h(y)}{((\cdot)-y)^2}dA(y)\right\|_{\mathscr{C}^{0,\alpha}(\T)}\\
&\qquad\qquad \qquad\qquad+C\left\|\, \textnormal{p.v.}\int_\D \frac{\partial_g\phi(\Omega,0)(h)(y)\, h(y)}{((\cdot)-y)^2}dA(y)\right\|_{\mathscr{C}^{0,\alpha}(\T)}\\
\leq & C\left\|\partial_g\phi(\Omega,0)(h)(\cdot)\right\|_{\mathscr{C}^{0,\alpha}(\T)}\left\|\, \textnormal{p.v.}\int_\D \frac{h(y)}{((\cdot)-y)^2}dA(y)\right\|_{\mathscr{C}^{0}(\T)}+C\left\|\partial_g\phi(\Omega,0)(h)\, h\right\|_{\mathscr{C}^{0,\alpha}(\T)}\\
&\qquad\qquad \qquad\qquad+C\left\|\partial_g\phi(\Omega,0)(h)(\cdot)\right\|_{\mathscr{C}^{0}(\T)}\left\|\, \textnormal{p.v.}\int_\D \frac{h(y)}{((\cdot)-y)^2}dA(y)\right\|_{\mathscr{C}^{0,\alpha}(\T)}\\
\leq &C\left\|h\right\|_{\mathscr{C}^{0}(\D)}\left\|h\right\|_{\mathscr{C}^{0,\alpha}(\D)}.
\end{align*}}
For $k=1, 2$, we would need the use of the interpolation inequalities:
\begin{align*}
\Big\|\int_\D &\frac{\partial_g\phi(\Omega,0)(h)(\cdot)-\partial_g\phi(\Omega,0)(h)(y)}{((\cdot)-y)^2}h(y)dA(y)\Big\|_{\mathscr{C}^{k,\alpha}(\T)}\\
\leq& C\left\|\partial_g\phi(\Omega,0)(h)(\cdot)\, \textnormal{p.v.}\int_\D \frac{h(y)}{((\cdot)-y)^2}dA(y)\right\|_{\mathscr{C}^{k,\alpha}(\T)}\\
&+C\left\|\, \textnormal{p.v.}\int_\D \frac{\partial_g\phi(\Omega,0)(h)(y)\, h(y)}{((\cdot)-y)^2}dA(y)\right\|_{\mathscr{C}^{k,\alpha}(\T)}\\
\leq & C\left\|\partial_g\phi(\Omega,0)(h)(\cdot)\right\|_{\mathscr{C}^{k,\alpha}(\T)}\left\|\, \textnormal{p.v.}\int_\D \frac{h(y)}{((\cdot)-y)^2}dA(y)\right\|_{\mathscr{C}^{0}(\T)}\\
&+C\left\|\partial_g\phi(\Omega,0)(h)(\cdot)\right\|_{\mathscr{C}^{0}(\T)}\left\|\, \textnormal{p.v.}\int_\D \frac{h(y)}{((\cdot)-y)^2}dA(y)\right\|_{\mathscr{C}^{k,\alpha}(\T)}+C\left\|\partial_g\phi(\Omega,0)(h)\, h\right\|_{\mathscr{C}^{k,\alpha}(\T)}\\
\leq & \left\|h\right\|_{\mathscr{C}^{k-1,\alpha}(\D)}\left\|h\right\|_{\mathscr{C}^{0,\alpha}(\D)}+C\left\|h\right\|_{\mathscr{C}^{0}(\D)}\left\|h\right\|_{\mathscr{C}^{k,\alpha}(\D)}.
\end{align*}
It remains to use the interpolation inequalities in order to conclude. For the case $k=1$, we need to use \eqref{Estim3} in order to get
$$
\left\|h\right\|_{\mathscr{C}^{0,\alpha}(\D)}\left\|h\right\|_{\mathscr{C}^{0,\alpha}(\D)}\leq C \left\|h\right\|_{\mathscr{C}^{0}(\D)}^{2-\frac{2\alpha}{2+\alpha}}\left\|h\right\|_{\mathscr{C}^{2,\alpha}(\D)}^{\frac{2\alpha}{2+\alpha}}.
$$
We use again \eqref{Estim3} for $k=2$: 
$$
\left\|h\right\|_{\mathscr{C}^{1,\alpha}(\D)}\left\|h\right\|_{\mathscr{C}^{0,\alpha}(\D)}\leq C \left\|h\right\|_{\mathscr{C}^{0}(\D)}^{\frac{3}{2+\alpha}}\left\|h\right\|_{\mathscr{C}^{2,\alpha}(\D)}^{\frac{1+2\alpha}{2+\alpha}}.
$$
Note that in every case, the exponent of the $\mathscr{C}^0$-norm is bigger than 1. As to  the remaining  terms of $\rho$,  we develop similar estimates with the same order of difficulties leading to the announced inequality \eqref{Estim1}.

Once we have these preliminaries estimates, we can check that \eqref{Estim2} holds true. For example, let us illustrate the basic  idea to implement through the second term.
\begin{align*}
\left\|\int_\D\log|(\cdot)-y|h(y)\textnormal{Re}\left[\partial_g \phi(\Omega,0)(h)'(y)\right]dA(y)\right\|_{\mathscr{C}^{2,\alpha}(\D)}\leq&C \left\|h\partial_g \phi(\Omega,0)(h)'\right\|_{\mathscr{C}^{0,\alpha}(\D)}.
\end{align*}
Using the classical law products and \eqref{limitcaseX} we find 
\begin{align*}
\left\|\int_\D\log|(\cdot)-y|h(y)\textnormal{Re}\left[\partial_g \phi(\Omega,0)(h)'(y)\right]dA(y)\right\|_{\mathscr{C}^{2,\alpha}(\D)}
\leq&C \left\|h\right\|_{\mathscr{C}^{0}(\D)} \left\|\partial_g \phi(\Omega,0)(h)'\right\|_{\mathscr{C}^{0,\alpha}(\D)}\\
&+ C\left\|h\right\|_{\mathscr{C}^{0,\alpha}(\D)} \left\|\partial_g \phi(\Omega,0)(h)'\right\|_{\mathscr{C}^{0}(\D)}\\
\leq& C\left\|h\right\|_{\mathscr{C}^{0,\alpha}(\D)} \left\|h\right\|_{\mathscr{C}^{0}(\D)}.
\end{align*}
The other terms can be estimated in a similar way, achieving \eqref{Estim2}. The same arguments  applied to the remainder term lead to
$$
\left\|\mathscr{R}_2(\Omega, h)\right\|_{\mathscr{C}^{2,\alpha}(\D)}\leq C\|h\|_{\mathscr{C}^{2,\alpha}(\D)}^{\delta_1} \|h\|_{\mathscr{C}^{0}(\D)}^{\delta_2}.
$$
The computations are  long here but the analysis is straightforward.

Let us see how to achieve the argument. Assume that $h$ is a zero to $ \widehat{G}$ in a small neighborhood of the origin, then 
$$
D_g\widehat{G}(\Omega,0)h=-\frac12 D^2_{g,g}\widehat{G}(\Omega,0)(h,h)-\mathscr{R}_2(\Omega, h).
$$
Applying Lemma \ref{trivialkernelsing}--$(3)$ we deduce that 
\begin{align*}
\|h\|_{\mathscr{C}^0(\D)}&\le C\|D_g\widehat{G}(\Omega,0)h\|_{\mathscr{C}^{2,\alpha}(\D)}\\
&\le C\|\frac12 D^2_{g,g}\widehat{G}(\Omega,0)(h,h)+\mathscr{R}_2(\Omega, h)\|_{\mathscr{C}^{2,\alpha}(\D)}\\
&\le C\|h\|_{\mathscr{C}^{2,\alpha}}^{\delta_1} \|h\|_{\mathscr{C}^0(\D)}^{\delta_2},
\end{align*}
Consequently, if $\|h\|_{\mathscr{C}^{2,\alpha}(\D)}^{\delta_1}<C^{-1}$, then necessary $\|h\|_{\mathscr{C}^0(\D)}=0$ since $\delta_2\geq 1$. Therefore, we deduce that  there is only the trivial solution in this ball.
\end{proof}
\begin{remark}\label{Rempoly}
The   quadratic profiles are  particular cases of the polynomial profiles studied in Section \ref{polprofile}; $f_0(r)=Ar^{2m}+B$. Here, we briefly show how to develop this case. Studying the kernel for this case is equivalent to study the equations
\begin{eqnarray*}
h_n(r)+\frac{r}{n}\frac{m(2m+2)r^{2m-2}}{\left(r^{2m}-\frac{1}{x_m}\right)}\left[-\frac{H_n(1)}{G_n(1)}G_n(r)+\frac{H_n(r)}{r^{n+1}}\right]&=&0,\quad \forall r\in[0,1],\quad \forall n\in\N^\star,\\
\frac{\frac{1}{x_m}-r^2}{8}h_0(r)-\int_\tau^1\frac{1}{\tau}\int_0^\tau sh_0(s)dsd\tau&=&0,\quad  \forall r\in[0,1],
\end{eqnarray*} 
where the functions $H_n$ and $G_n$ are defined in \eqref{Gn}-\eqref{Hn} and 
$$
\frac{1}{x_m}=\frac{2m+2}{A}\left(\Omega-\frac{B}{2}\right).
$$
Thus $H_n$ verifies the following equation
\begin{eqnarray*}
r(1-x_mr^{2m})H_n''(r)&-&(2n-1)(1-x_mr^{2m})H_n'(r)\\
&&+2m(2m+2)r^{2m-1}x_mH_n(r)=2m(2m+2)x_mr^{n+2m}\frac{H_n(1)}{G_n(1)}G_n(r).
\end{eqnarray*}
Using the change of variables $y=x_mr^{2m}$ and setting $H_n(r)=\mathcal{F}(x_mr^{2m})$, one has that
\begin{eqnarray*}
y(1-y)\mathcal{F}''(y)&+&\frac{m-n}{m}(1-y)\mathcal{F}'(y)+\frac{m+1}{m}\mathcal{F}(y)\\
&&=\frac{m+1}{m}\left(\frac{y}{x_m}\right)^{\frac{n+1}{2m}}\frac{H_n(1)}{G_n(1)}G_n\left(\left(\frac{y}{x_m}\right)^{\frac{1}{2m}}\right).
\end{eqnarray*}
The homogeneous equation of the last differential equation can be solved in terms of hypergeometric functions as it was done in the quadratic profile. Then, similar arguments can be applied to this case.
\end{remark}

\subsection{Range structure}\label{RangeSEc}
Here, we provide an algebraic description of the range. This will be useful when studying the transversality assumption of the Crandall-Rabinowitz Theorem. Our result reads as follows.
\begin{proposition}\label{Proprange}
Let  $A\in \R^\star$, $B\in \R$ and $x_0$ be given by \eqref{Firstzero}. Let   ${x\in (-\infty,1)}\backslash\{ \widehat{\mathcal{S}}_{\textnormal{sing}}\cup\{0,x_0\}\}$, where  the set $\widehat{\mathcal{S}}_{\textnormal{sing}}$ is defined in \eqref{singularx}. Then
$$
\textnormal{Im}\, D_g\widehat{G}(\Omega,0)= \left\{d\in \mathscr{C}_s^{1,\alpha}(\D):\quad  \int_\D d(z)\mathpzc{K}_n(z)dz=0,\, \,  n\in\mathcal{A}_x\right\},
$$
where
$$
\mathpzc{K}_n(z)=\textnormal{Re}\left[\frac{F_n(x|z|^2)}{1-x|z|^2}z^n\right],\quad \Omega=\frac{B}{2}+\frac{A}{4x},
$$
and the set $\mathcal{A}_x$ is defined by \eqref{AX1}.
\end{proposition}
\begin{proof}
In order to describe the range of the $D_g\widehat{G}(\Omega,0):\mathscr{C}_s^{1,\alpha}(\D)\to \mathscr{C}_s^{1,\alpha}(\D)$, we should solve the equation 
$$
D_g\widehat{G}(\Omega,0)h=d,\quad h(r e^{i\theta})=
\sum_{n\geq 0}h_n(r)\cos(n\theta),\quad   d(r e^{i\theta})=
\sum_{n\geq 0}d_n(r)\cos(n\theta).
$$
From the structure  of the linearized operator seen in \eqref{linoperator}, this problem is equivalent to
\begin{align}\label{rangeeq}
\begin{split}
\frac{\frac{1}{x}-r^2}{8}h_n(r)-\frac{r}{n}\left[A_nG_n(r)+\frac{1}{2r^{n+1}}H_n(r)\right]&=d_n(r),\quad \forall n\geq 1\\
\frac{\frac{1}{x}-r^2}{8}h_0(r)-\int_r^1\frac{1}{\tau}\int_0^\tau sh_0(s)dsd\tau&=d_0(r),
\end{split}
\end{align}
where the functions involved in the last expressions are defined in \eqref{Gn}-\eqref{An1}.  By Proposition \ref{radialfunctions}, the case $n=0$ can be analyzed through the Inverse Function Theorem getting  a unique  solution.  Let us focus on the case $n\geq 1$ and proceed as in the  preceding study for  the kernel. We use  the linear operator defined in \eqref{linearoperator},

$$
\mathscr{L}h= r^{2n}\int_r^1\frac{1}{s^{n-1}}h(s)ds+\int_0^rs^{n+1}h(s)ds,
$$
for any $h\in \mathscr{C}([0,1];\R)$, which satisfies the boundary conditions  in \eqref{conlinop} and 
$$
 \frac{1}{2n}(r(\mathscr{L}h)^\prime(r))^\prime-(\mathscr{L}h)^\prime(r)=-r^{n+1}h(r).
$$
Taking $H_n\triangleq\mathscr{L}h_n$ and using  \eqref{rangeeq} we find  that $H_n$ solves
\begin{align}\label{X111}
\nonumber (1-xr^2)rH_n''(r)-(1-xr^2)(2n-1)H_n'(r)&+8rxH_n(r)\\
&=-16A_nxr^{n+2}G_n(r)-16xnr^{n+1}d_n(r),
\end{align}
complemented with the  boundary conditions $H_n(0)=H_n'(1)=0$. This  differential equation is equivalent to \eqref{rangeeq}.
Once we have a solution of the differential equation \eqref{X111} we have to verify that $\mathscr{L}h_n=H_n$, where
$$
h_n(r)\triangleq\frac{8x}{1-xr^2}\left[d_n(r)+\frac{A_nr}{n}G_n(r)+\frac{1}{2nr^n}H_n(r)\right].
$$
Denote $\mathpzc{H}\triangleq\mathscr{L}h_n-H_n$, then it satisfies 
\begin{equation*}
 \frac{1}{2n}\left[r\,\mathpzc{H}^\prime(r)\right]^\prime-\mathpzc{H}^\prime(r)=0.
\end{equation*}
From the boundary conditions one obtains that $\mathpzc{H}=0$ and thus  $\mathscr{L}h_n=H_n$. 
Now, since $H_n(0)=0$, Lemma \ref{ODE} can be applied with
$$
g(r)=-16A_nxr^{n+2}G_n(r)-16xnr^{n+1}d_n(r).
$$
Thus, the solutions are given by
\begin{eqnarray*}
H_n(r)&=&r^{2n}F_n(xr^2)\left[\frac{H_n(1)}{F_n(x)}+4A_nx^n\int_{xr^2}^x\frac{1}{\tau^{n+1}F_n^2(\tau)}\int_0^\tau \frac{F_n(s)}{1-s}\left({\frac{s}{x}}\right)^{\frac{n+1}{2}}G_n\left(\left(\frac{s}{x}\right)^{\frac{1}{2}}\right)dsd\tau\right.\\
&&+\left.4nx^n\int_{xr^2}^x\frac{1}{\tau^{n+1}F_n^2(\tau)}\int_0^\tau \frac{F_n(s)}{1-s}\left({\frac{s}{x}}\right)^{\frac{n}{2}}d_n\left(\left(\frac{s}{x}\right)^{\frac{1}{2}}\right)dsd\tau\right].
\end{eqnarray*}
A change of variables combined with \eqref{An1} yield
\begin{eqnarray*}
H_n(r)&=&H_n(1)r^{2n}F_n(xr^2)\left[\frac{1}{F_n(x)}-\frac{8x}{G_n(1)}\int_r^1\frac{1}{\tau^{2n+1}F_n^2(x\tau^2)}\int_0^\tau\frac{s^{n+2}F_n(xs^2)}{1-xs^2}G_n(s)dsd\tau\right]\\
&&+16nxr^{2n}F_n(xr^2)\int_r^1\frac{1}{\tau^{2n+1}F_n^2(x\tau^2)}\int_0^\tau\frac{s^{n+1}F_n(xs^2)}{1-xs^2}d_n(s)dsd\tau.
\end{eqnarray*}
Note that when $d_n\equiv 0$, the function $H_n$ agrees with the one obtained for  the kernel. It remains to check the boundary conditions. Clearly $H_n(0)=0$, then we focus on  proving $H_n^\prime(1)=0$. Following  the  computations leading to \eqref{derivkernel}, we obtain
$$
H_{n}^\prime(1)=\frac{2n H_n(1)}{F_n(x)G_n(1)}\Psi_n(x)-\frac{16nx}{F_n(x)}\int_0^1\frac{s^{n+1}F_n(xs^2)}{1-xs^2}d_n(s)ds.
$$
We will distinguish two cases. In the first one  $n\notin \mathcal{A}_x$. By virtue of Proposition \ref{statkernel} and Proposition \ref{compactop}, we obtain that  $D_g\widehat{G}(\Omega,0)$  is an isomorphism. Otherwise,  we have  $\Psi_n(x)=0$ and the boundary condition is equivalent to
\begin{eqnarray}\label{conditionrange}
\int_0^1\frac{r^{n+1}F_n(xr^2)}{1-xr^2}d_n(r)dr=0.
\end{eqnarray}
Define
$
z\in \overline{\D}\mapsto \mathpzc{K}_n(z)=\textnormal{Re}\left[\frac{F_n(x|z|^2)}{1-x|z|^2}z^n\right],
$
and consider  the linear form 
$
T_{\mathpzc{K}_n}: \mathscr{C}_s^{1,\alpha}(\D)\mapsto \R$ given by
$$
T_{\mathpzc{K}_n} d\triangleq\int_{\D} d(y)\mathpzc{K}_n(y)dA(y).
$$
Condition \eqref{conditionrange} leads to $T_{\mathpzc{K}_n} d=0$, which follows from 
$$
\int_{0}^{2\pi} d(re^{i\theta})\cos( n\theta)d\theta=\pi d_n(r).
$$
Since $\mathpzc{K}_n$ belongs to  $\mathscr{C}^\infty(\overline{\D};\R)$, we deduce that $T_{\mathpzc{K}_n}$ is continuous. Thus, $\textnormal{Ker }T_{\mathpzc{K}_n}$ is closed and of co-dimension one. In addition, from the preceding analysis one has that
\begin{eqnarray*}
\textnormal{Im}\, D_g\widehat{G}(\Omega,0)\subseteq \left\{d\in \mathscr{C}_s^{1,\alpha}(\D)\, :\, \, \int_\D d(z)\mathpzc{K}_n(z)dz=0,\, \,  n\in\mathcal{A}_x\right\}
\subseteq \bigcap_{\atop n\in \mathcal{A}_x}\textnormal{Ker }T_{\mathpzc{K}_n} .
\end{eqnarray*}
The elements of the family $\left\{\mathpzc{K}_n: \, n\in \mathcal{A}_x\right\}$ are independent, and thus  $\bigcap_{\atop n\in \mathcal{A}_x}\textnormal{Ker }T_{\mathpzc{K}_n}$ is  closed  and of co--dimension  $\textnormal{card}\mathcal{A}_x$. As a consequence of Proposition \ref{statkernel},  $\textnormal{Ker }D_g\widehat{G}(\Omega,0)$ is of dimension $\textnormal{card}\mathcal{A}_x$. Using   Proposition \ref{compactop},  $D_g\widehat{G}(\Omega,0)$ is a Fredholm  operator of index zero. Consequently, $\textnormal{Im }D_g\widehat{G}(\Omega,0)$ is of co--dimension $\textnormal{card}\mathcal{A}_x$, and thus
\begin{equation*}
\textnormal{Im}\, D_g\widehat{G}(\Omega,0)=\bigcap_{\atop n\in \mathcal{A}_x}\textnormal{Ker }T_{\mathpzc{K}_n}.
\end{equation*}
This achieves the proof of the announced result.
\end{proof}

\section{Spectral study}
The aim of this section is to study some qualitative properties of the roots  of the spectral function \eqref{Takk1} that will be needed when we apply bifurcation arguments.   For instance, to identify the eigenvalues and  explore the kernel structure of the linearized operator, we should  carefully  analyze   the existence and uniqueness  of  roots  $x_n$ of   \eqref{Takk1} at each frequency level $n$ and  study  their  monotonicity.  This part is highly  technical and requires  cautious  manipulations on hypergeometric functions and their asymptotics with respect to $n$. Notice that for some special  regime in $A$ and $B$, the monotonicity turns to be very  intricate  and   it is only established for higher frequencies through   refined expansions of the eigenvalues $x_n$ with respect to $n$. Another problem that one should face is connected to the separation between the eigenvalues set and the singular set associated to \eqref{Interv1}. It seems that the two sequences admit the same leading term and the separation is obtained  at  the second asymptotics  level, which requires much more efforts because the sequence $\{x_n\}$ converges to $1$, which is a singular point for the hypergeometric function involved in  \eqref{Takk1}. Recall that $n$ and $m$ are non negative integers.

\subsection{Reformulations of the dispersion  equation}

In what follows, we intend  to write down various formulations for the dispersion equation \eqref{Takk1} describing the set
\eqref{disper-dos}.
This set is given by the zeroes of \eqref{Takk1} and the elements of this set are called ``eigenvalues''.  As we shall notice, the study of some qualitative behavior of the zeroes  will be much more tractable through the use of different representations  connected to some specific algebraic  structure of the hypergeometric equations. Recall the use of the notation $F_n(x)=F(a_n,b_n;c_n;x),$ where the coefficients are given by \eqref{coefficients}. The Kummer quadratic transformations introduced in Appendix \ref{SecSpecialfunctions} leads to the following result:

\begin{lemma}\label{LMZ1}
 The following identities hold true:
\begin{align}\label{special}
\zeta_n(x)=&\left[1+x\left(\frac{A+2B}{A(n+1)}-1\right)\right]F(a_n,b_n; n+1;x)\\\nonumber&+\frac{2x-1}{n+1}F(a_n,b_n; n+2;x)-\frac{2x}{(n+1)(n+2)}F(a_n,b_n; n+3;x),
\end{align}
\begin{align}\label{Takk3}
\zeta_n(x)=&\frac{A+2B}{A(n+1)} xF(a_n,b_n;n+1;x)\\
\nonumber&+\frac{n-(n+1)x}{n+1}F(a_n,b_n;n+2;x)
+\frac{2nx}{(n+1)(n+2)}F(a_n,b_n;n+3;x),
\end{align}
for any $n\in\N$ and $x\in(-\infty,1)$, where we have used the notations  \eqref{coefficients}.
\end{lemma}
\begin{proof}
Let us begin with \eqref{special}. 
The integral term in \eqref{Takk1} can be written  as follows
$$
\int_0^1 F_n(\tau x) \tau^n\left[-1+2x\tau \right] d\tau =(2x-1)\int_0^1 F_n(\tau x) \tau^n d\tau-2x\int_0^1 F_n(\tau x) \tau^n(1-\tau) d\tau.
$$
This leads to
\begin{eqnarray}\label{Takk2}
\nonumber \zeta_n(x)&=& F_n(x)\left[1-x+\frac{A+2B}{A(n+1)} x\right]+(2x-1)\int_0^1 F_n(\tau x) \tau^n d\tau\\
&&-2x\int_0^1 F_n(\tau x) \tau^n(1-\tau) d\tau.
\end{eqnarray}
We use \eqref{FormInt1} in order to get successively
$$
\int_0^1 F(a_n,b_n;n+1;\tau x) \tau^{n} d\tau=\frac{F(a_n,b_n; n+2;x)}{n+1}
$$
and
$$
\int_0^1 F(a_n,b_n;n+1;\tau x) \tau^{n}\left[1-\tau\right] d\tau=\frac{F(a_n,b_n; n+3;x)}{(n+1)(n+2)},
$$
 for any $x\in(-\infty,1)$. Taking into account these identities, we can rewrite \eqref{Takk2}  as \eqref{special}.

\medskip
\noindent
In oder to obtain \eqref{Takk3}, we use \eqref{Kumt} with $a=a_n,\, b=b_n$ and $c=n+1$, which yields
\begin{eqnarray*}
F(a_n,b_n;n+1)(x-1)&=&\frac{(n+3)x-(n+1)}{n+1}F(a_n,b_n;n+2;x)\\
&&+\frac{(a_n-(n+2))((n+2)-b_n)x}{(n+1)(n+2)}F(a_n,b_n;n+3;x)\\
&=&\frac{(n+3)x-(n+1)}{n+1}F(a_n,b_n;n+2;x)\\
&&-\frac{(2n+2)x}{(n+1)(n+2)}F(a_n,b_n;n+3;x),
\end{eqnarray*}
where we have taken into account  the identities $a_n+b_n=n$ and $a_nb_n=-2$. By virtue of  the first assertion of this lemma we obtain
\begin{eqnarray*}
\zeta_n(x)&=&\frac{A+2B}{A(n+1)} xF(a_n,b_n;n+1;x)\\
&&+\frac{(n+1)-(n+3)x}{n+1}F(a_n,b_n;n+2;x)
+\frac{(2n+2)x}{(n+1)(n+2)}F(a_n,b_n;n+3;x)\\
&&+\frac{2x-1}{n+1}F(a_n,b_n;n+2;x)-\frac{2x}{(n+1)(n+2)}F(a_n,b_n;n+3;x)\\
&=&\frac{A+2B}{A(n+1)} xF(a_n,b_n;n+1;x)\\
&&+\frac{n-(n+1)x}{n+1}F(a_n,b_n;n+2;x)
+\frac{2nx}{(n+1)(n+2)}F(a_n,b_n;n+3;x).
\end{eqnarray*}
This achieves the proof of the second identity \eqref{Takk3}.
\\
Let us also remark that using  \eqref{f4} we can deduce another useful equivalent expression for $\zeta_n$
\[ \label{special1}
\zeta_n(x)=I_n^1(x)F(a_n,b_n;n+1;x)+I_n^2(x)F(a_n+1,b_n;n+2;x)+I_n^3(x)F(a_n,b_n;n+3;x),
\]
where
\begin{eqnarray*}
I_n^1(x)&\triangleq&\frac{n-a_n}{n+1-a_n}+x\left(\frac{A+2B}{A(n+1)}-\frac{n-1{+}a_n}{n+1-a_n}\right),\\
I_n^2(x)&\triangleq&-\frac{a_n(2x-1)}{(n+1)(n+1-a_n)},\\
I_n^3(x)&\triangleq&-\frac{2x}{(n+1)(n+2)}.
\end{eqnarray*}

 \end{proof}

\subsection{Qualitative properties of  hypergeometric functions} 
The main task of this section is to provide  suitable properties  about the analytic continuation of  the mapping  $(n,x)\mapsto F(a_n,b_n;c_n;x)$ and some partial monotonicity behavior.
First, applying  the integral representation  \eqref{integ} with  the special coefficients \eqref{coefficients}, we find
\begin{equation}\label{integ2}
F(a_n,b_n;c_n;x)=\frac{\Gamma(n+1)}{\Gamma(n-a_n)\Gamma(1+a_n)}\int_0^1 \tau^{n-a_n-1} (1-\tau)^{a_n} (1-\tau x)^{-a_n}d\tau,
\end{equation}
for $x\in(-\infty,1]$.
Notice that, due  to \eqref{id1} we can evaluate it at 1, obtaining
\begin{equation}\label{Yah11}
F(a_n,b_n;c_n;1)=\frac{\Gamma(n+1)}{\Gamma(n-a_n+1)\Gamma(1+a_n)}, \end{equation}
for any $n\geq 2$, where we have used the identity $\Gamma(x+1)=x\Gamma(x).$  We observe that the representation  \eqref{integ2} fails for the case $n=1$  because $a_1=-1$.  This does not matter since as we have already mentioned in Remark \ref{n1}, the case $n=1$ is explicit and the study of $\zeta_1$ can be done by hand. 
It is a well-known fact  that the Gamma function  can be extended analytically to $\C\backslash \{0,-1,-2,\dots\}$. Therefore, the  map $n\in \N^\star\backslash\{1\}\mapsto F(a_n,b_n;c_n;1)$ admits a $\mathscr{C}^\infty$- extension given by
$$
\mathcal{F}:\ttt\in]1,+\infty[\mapsto \frac{\Gamma(\ttt+1)}{\Gamma(\ttt-a_\ttt+1)\Gamma(1+a_\ttt)},
$$
with $a_\ttt=-\frac{4}{\ttt+\sqrt{\ttt^2+8}}.$
\\
The first result that we should discuss concerns some  useful asymptotic behaviors for  $\ttt\mapsto \mathcal{F}(\ttt)$.
\begin{lemma}\label{lemX1}
The following properties are satisfied:
\begin{enumerate}
\item Let $\ttt\geq1,$ then the function $x\in(-\infty,1]\mapsto F(a_\ttt,b_\ttt;c_\ttt;x)$ is positive and strictly decreasing.
\item For large $\ttt\gg 1$, we have
\begin{equation*}
\mathcal{F}(\ttt)
=1-2\frac{\ln \ttt}{\ttt}-\frac{2\gamma}{\ttt}+O\left(\frac{\ln^2 \ttt}{\ttt^2}\right)
\end{equation*}
and
$$
\mathcal{F}^\prime(\ttt)
=2\frac{\ln \ttt}{\ttt^2}+\frac{2(\gamma-1)}{\ttt^2}+O\left( \frac{\ln^2 \ttt}{\ttt^3}\right),
$$
where  $\gamma$  is  the Euler constant.
In particular, we have the asymptotics 
$$
F(a_n,b_n;c_n;1)=1-2\frac{\ln n}{n}-\frac{2\gamma}{n}+O\left(\frac{\ln^2 n}{n^2}\right)
$$
and
$$
\frac{d}{dn}F(a_n,b_n;c_n;1)
=2\frac{\ln n}{n^2}+\frac{2(\gamma-1)}{n^2}+O\left( \frac{\ln^2 n}{n^3}\right),
$$
for large $n$. 
\end{enumerate}
\end{lemma}
\begin{proof}
${\bf (1)}$ 
The case $\ttt=1$ follows obviously  from the explicit expression given in Remark \ref{n1}. Now let us  consider $\ttt>1$. According to \eqref{Diff41}, we can differentiate   $F$  with respect to $x$:
$$
F^\prime(a_\ttt,b_\ttt;c_\ttt;x)=\frac{a_\ttt b_\ttt}{c_\ttt}F(a_\ttt+1,b_\ttt+1;c_\ttt+1;x),\quad \forall x\in (-\infty,1).
$$
Using the integral representation  \eqref{integ} and the positivity of Gamma function, one has that  for any $c>b>0$
\begin{equation}\label{PositX1}
 F(a,b;c,x)>0,\quad \forall x\in (-\infty,1).
\end{equation}
Since $a_\ttt\in(-1,0), b_\ttt, c_\ttt>0$, then we deduce that $F(a_\ttt+1,b_\ttt+1;c_\ttt+1;x)>0$ and thus
$$
F^\prime(a_\ttt,b_\ttt;c_\ttt;x)<0,\quad \forall x\in (-\infty,1).
$$
This implies  that  $ x\mapsto F(a_n,b_n;c_n;x)$ is strictly decreasing and together with \eqref{Yah11} we obtain  
$$
F(a,b;c,x)\geq F(a_\ttt,b_\ttt;c_\ttt,1)>0,\quad \forall x\in (-\infty,1).
$$

\medskip

\noindent
${\bf (2)}$ The  following asymptotic behavior
$$
\frac{\Gamma(\ttt+\alpha)}{\Gamma(\ttt+\beta)}=\sum_{n\in\N} C_n(\alpha-\beta,\beta) \ttt^{\alpha-\beta-n},
$$
holds as $\ttt\to +\infty$ by using \cite[Identity 12]{TricomiErdelyi}, where the coefficients $C_n(\alpha-\beta,\beta)$ can be obtained recursively and are polynomials on the variables $\alpha,\beta$. In addition, the first coefficients can be calculated explicitly
$$
C_0(\alpha-\beta,\beta)=1\quad\hbox{and}\quad C_1(\alpha-\beta,\beta)=\frac12(\alpha-\beta)\left(\alpha+\beta-1\right).
$$
Taking  $\alpha=1$ and $\beta=1-a_\ttt$ we deduce
\begin{eqnarray*}
\frac{\Gamma(\ttt+1)}{\Gamma(\ttt+1-a_\ttt)}= \ttt^{a_\ttt}+\frac12 a_\ttt\left(1-a_\ttt\right)\ttt^{a_\ttt-1}+O\left(\frac{1}{\ttt^2}\right)
=  \ttt^{a_\ttt}+O\left(\frac{1}{\ttt^2}\right),
\end{eqnarray*}
where we have used that $a_\ttt\sim -\frac{2}{\ttt}.$ From the  following expansion
\begin{eqnarray*}
\ttt^{a_\ttt}=e^{a_\ttt\ln \ttt}
= 1-2\frac{\ln \ttt}{\ttt}+O\left(\frac{\ln^2 \ttt}{\ttt^2}\right),
\end{eqnarray*}
we get
\begin{eqnarray}\label{ASS1}
\frac{\Gamma(\ttt+1)}{\Gamma(\ttt+1-a_\ttt)}&=&1-2\frac{\ln \ttt}{\ttt}+O\left(\frac{\ln^2 \ttt}{\ttt^2}\right).\end{eqnarray}
Using again  Taylor expansion, we find
$
\Gamma(1+a_\ttt)=1+a_\ttt\Gamma^\prime(1)+O\left(a_\ttt^2\right).
$
Therefore, combining this  with $\Gamma^\prime(1)=-\gamma$ and $a_\ttt\sim -\frac{2}{\ttt}$ yields
$$
\Gamma(1+a_\ttt)=1+\frac{2\gamma}{\ttt}+O\left(\frac{1}{\ttt^2}\right).
$$
Consequently, it follows that  $\mathcal{F}$ admits the following asymptotic behavior at infinity
\begin{eqnarray*}
\mathcal{F}(\ttt)=\frac{1-2\frac{\ln \ttt}{\ttt}+O(\ttt^{-2}\ln^2 \ttt)}{1+\frac{2\gamma}{\ttt}+O(\ttt^{-2})
}
= 1-2\frac{\ln \ttt}{\ttt}-2\frac{\gamma}{\ttt}+O\left(\frac{\ln^2 \ttt}{\ttt^2}\right).
\end{eqnarray*}
Since $\Gamma$ is real  analytic, we have that  $\mathcal{F}$ is also real analytic and one may deduce the asymptotics at $+\infty$ of the derivative $\mathcal{F}^\prime$  through the  differentiation  term by term the asymptotics of $\mathcal{F}$. Thus, we obtain the second expansion in assertion (2).
\end{proof}
Our next purpose is  to provide some useful estimates  for   $F(a_n,b_n;c_n;x)$ and its partial  derivatives.  More precisely, we state the following result.

\begin{lemma}\label{Lem2}
With the notations \eqref{coefficients}, the following assertions hold  true.
\begin{enumerate}
\item  The sequence $n\in[1,+\infty)\mapsto F(a_n,b_n;c_n;x)$ is strictly increasing, for any $x\in (0,1]$, and strictly decreasing, for any ${x\in(-\infty,0)}$.
\item { Given $n\geq1$ we have}
$$
 \left|\partial_n F_n(x)\right|\leq \frac{-2x F_n(x)}{(n+1)^2},
$$
for any $x\in(-\infty,0]$.
\item There exists $C>0$ such that
$$
|\partial_xF(a_n,b_n;c_n;x)|\leq C+C|\ln (1-x)|,
$$
for any $x\in [0,1]$, and $ n\geq 2$.
\item There exists $C>0$ such that
$$
|\partial_n F(a_n,b_n;c_n;x)|\le C\frac{\ln n}{n^2},
$$
for any $x\in [0,1]$, and $ n\geq2$.
\item There exists $C>0$ such that
$$
|\partial_{xx}F(a_n,b_n;c_n;x)|\leq \frac{C}{1-x},
$$
for any $x\in [0,1]$ and $ n\geq2$.
\end{enumerate}
\end{lemma}
\begin{proof}
\textbf{(1)}
Recall that $F_n$ solves the equation
$$
x(1-x)F_n''(x)+(n+1)(1-x)F_n'(x)+2F_n(x)=0,
$$
with $F_n(0)=1$ and $F_n'(0)=\frac{a_nb_n}{c_n}=\frac{-2}{n+1}.$ As we have mentioned in the beginning of this  section the dependence with respect to $n$ is smooth, here we use $n$ as a continuous parameter instead \mbox{of $\ttt$.} Then, differentiating  with respect to $n$ we get
$$
x(1-x)(\partial_nF_n)''(x)+(n+1)(1-x)(\partial_nF_n)'(x)+2(\partial_nF_n)=-(1-x)F_n'(x),
$$
with $(\partial_nF_n)(0)=0$ and $(\partial_nF_n)'(0)=\frac{2}{(n+1)^2}$. We can explicitly solve the last differential equation by using the variation of the constant and keeping in mind that  $x\mapsto F_n(x)$ is a homogeneous solution. Thus, we obtain
$$
\partial_n F_n(x)=F_n(x)\left[K_2-\int_x^1\frac{1}{F_n^2(\tau)\tau^{n+1}}\left(K_1-\int_0^\tau F_n'(s)F_n(s)s^nds\right)d\tau\right],
$$
where the constant $K_1$ must be zero to remove the singularity at 0, in a similar way to the proof of Lemma  \ref{lemmaODE}. Since $(\partial_nF_n)(0)=0$, we deduce that
$$
K_2=-\int_0^1\frac{1}{F_n^2(\tau)\tau^{n+1}}\int_0^\tau F_n'(s)F_n(s)s^n ds \, d\tau ,
$$
and then
\begin{equation}\label{IdenDer0}
\partial_n F_n(x)=-F_n(x)\int_0^x\frac{1}{F_n^2(\tau)\tau^{n+1}}\int_0^\tau F_n'(s)F_n(s)s^ndsd\tau.
\end{equation}
The change of variables $s=\tau\theta$ leads to 
\begin{equation}\label{IdenDer}
\partial_n F_n(x)=-F_n(x)\int_0^x\frac{1}{F_n^2(\tau)}\int_0^1 F_n'(\tau\theta)F_n(\tau\theta)\theta^nd\theta d\tau.
\end{equation}
Hence,  it is clear that $\partial_n F_n(x)>0$, for $x\in[0,1)$, using Lemma \ref{lemX1}-$(1)$  . In the case $x\in(-\infty,0]$ we similarly get $\partial_n F_n(x)<0$.  Let us observe that  the compatibility condition  $(\partial_nF_n)'(0)=\frac{2}{(n+1)^2}$ can be directly checked from the preceding representation. Indeed, one has
\begin{eqnarray*}
(\partial_nF_n)'(0)&=&\lim_{x\to0}\frac{\partial_nF_n(x)}{x}
=-\lim_{x\rightarrow 0^+} \frac{F_n'(0)\int_0^1\theta ^nd\theta}{F_n(0)}
=\frac{2}{(n+1)^2}.
\end{eqnarray*}

\medskip
\noindent
\textbf{(2)}
First, notice that
$$
F^\prime_n(x)=-\frac{2}{n+1}F(a_n+1,b_n+1,c_n+1,x),
$$
and from \eqref{PositX1} we deduce that
$$
|F^\prime_n(x)|\leq\frac{2}{n+1}F(a_n+1,b_n+1,c_n+1,x).
$$

Now studying the variation of   $x\in(-\infty,1)\mapsto F(a_n+1,b_n+1,c_n+1,x)$ by means of the integral representation \eqref{integ}, we can show that it is  strictly increasing and positive, which implies in turn that
$$
 0<F(a_n+1,b_n+1,c_n+1,x)\leq F(a_n+1,b_n+1,c_n+1,0)=1,\quad \forall x\in (-\infty,0].
$$
This allows us to get,
\begin{equation}\label{NNG1}
 0\le -F^\prime_n(x)\leq\frac{2}{n+1},\quad \forall x\in (-\infty,0].
\end{equation}
Lemma \ref{lemX1}-$(1)$ implies in particular that
$$
F_n(x)\geq1,\quad \forall x\in (-\infty,1),
$$
and coming back  to \eqref{IdenDer} we find 
\begin{eqnarray*}
 |\partial_n F_n(x)|&\le&\frac{2 F_n(x)}{n+1}\int_x^0\frac{1}{F_n^2(\tau)}\int_0^1 F_n(\tau\theta)\theta^nd\theta d\tau
\le \frac{2 F_n(x)}{(n+1)^2}\int_{x}^0\frac{d\tau}{F_n(\tau)} 
\le \frac{-2x F_n(x)}{(n+1)^2},
\end{eqnarray*}

for $x\in(-\infty,0]$. This achieves the proof of the announced inequality.
\medskip
\noindent
\\
{\bf{(3)}} From previous computations we have
$$
\partial_xF(a_n,b_n;c_n;x)=\frac{-2}{n+1} F(1+a_n,n+1-a_n; n+2;x),
$$
which admits the integral representation
\begin{eqnarray}\label{DerV1}
\nonumber\partial_xF(a_n,b_n;c_n;x)
&=&\frac{-2\Gamma(n+2)}{(n+1)\Gamma(n+1-a_n)\Gamma(1+a_n)} \int_0^1 \tau^{n-a_n} (1-\tau)^{a_n} (1-\tau x)^{-1-a_n}d\tau\\
\nonumber&=&\frac{-2(n+2)}{n+1}\mathcal{F}(n) \int_0^1 \tau^{n-a_n} (1-\tau)^{a_n} (1-\tau x)^{-1-a_n}d\tau\\
&=& \frac{-2(n+2)}{n+1}\mathcal{F}(n) J_n (x),
\end{eqnarray}
where
$$
J_n(x)\triangleq\int_0^1 \tau^{n-a_n} (1-\tau)^{a_n} (1-\tau x)^{-1-a_n}d\tau.
$$
Using
$$
\sup_{n\geq1}\frac{2(n+2)}{n+1}\leq 3,
$$
and the first assertion of Lemma \ref{lemX1}, we have $\mathcal{F}(n)\in[0,1]$, and
\begin{equation}\label{DeXW1}
\sup_{n\geq1}\frac{2(n+2)\mathcal{F}(n)}{n+1}\leq 3.
\end{equation}
Consequently,
\begin{equation}\label{DeXWW1}
|\partial_xF(a_n,b_n;c_n;x)|\leq 3J_n.
\end{equation}
To estimate $J_n$ we simply write
\begin{eqnarray*}
J_n \leq   \int_0^1 (1-\tau)^{a_n} \left(1-\frac{\tau}{2}\right)^{-1-a_n}d\tau
\leq C\int_0^1 (1-\tau)^{a_n} d\tau
\le \frac{C}{1+a_n}
\leq C,
\end{eqnarray*}
for some constant $C$ independent of $n\geq2$, and for $x\in\left[0,\frac{1}{2}\right]$. In the case $x\in\left[\frac{1}{2},1\right)$, making the change of variable
\begin{equation}\label{chang1}
\tau=1-\frac{1-x}{x}\tau',
\end{equation}
and denoting  the new variable again by  $\tau$, we obtain
\begin{eqnarray*}
J_n&\leq&x^{-a_n-1}\int_0^{\frac{x}{1-x}}\tau^{a_n}(1+\tau)^{-1-a_n}d\tau\\
&\leq &x^{-a_n-1}\int_0^{1}\tau^{a_n}(1+\tau)^{-1-a_n}d\tau+x^{-a_n-1}\int_{1}^{\frac{x}{1-x}}\tau^{a_n}(1+\tau)^{-1-a_n}d\tau\\
&\le&C+C\int_{1}^{\frac{x}{1-x}}\left(\frac{1+\tau}{\tau}\right)^{-a_n}(1+\tau)^{-1}d\tau
\leq C+C|\ln(1-x)|,
\end{eqnarray*}
which  achieves the proof.

\medskip
\noindent
\textbf{(4)} According to \eqref{IdenDer0} and Lemma \ref{lemX1}$-(1)$ we may write 
\begin{eqnarray}\label{FNX1}
\nonumber |\partial_n F_n(x)|&\leq& F_n(x)\int_0^x\frac{1}{F_n^2(\tau)\tau^{n+1}}\int_0^\tau |F_n'(s)|F_n(s)s^ndsd\tau\\
&\le&\frac{1}{F_n^2(1)}\int_0^1{\tau^{-n-1}}\int_0^\tau |F_n'(s)s^ndsd\tau,
\end{eqnarray}
for $x\in[0,1]$ and $n\geq2$. Using \eqref{DeXWW1}, the definition of $J_n$ and the fact that $0<-a_n<1$,  we obtain 
\begin{eqnarray*}
 |F_n'(x)|\le 3\int_{0}^1\tau^{n-a_n}(1-\tau)^{a_n}(1-\tau x)^{-1-a_n} d\tau
 \leq \frac{3}{1-x}\int_{0}^1\tau^{n-a_n}(1-\tau)^{a_n} d\tau.
 \end{eqnarray*}
Now, recall the classical result on the Beta function $B$ defined as follows
\begin{eqnarray}\label{Betafunction}
\int_{0}^1\tau^{n-a_n}(1-\tau)^{a_n} d\tau = \textnormal{B}(n+1-a_n,1+a_n)
=\frac{\Gamma(n+1-a_n)\Gamma(1+a_n)}{\Gamma(n+2)},
 \end{eqnarray}
which implies in view of \eqref{Yah11} that
 \begin{equation*}
\int_{0}^1\tau^{n-a_n}(1-\tau)^{a_n} d\tau=\frac{1}{(n+1)F_n(1)}.
 \end{equation*}
 Consequently
 $$
  |F_n'(x)|\leq\frac{3}{(1-x)(n+1)F_n(1)}.
 $$
 Inserting this inequality into \eqref{FNX1} we deduce that
 \begin{eqnarray*}
 |\partial_n F_n(x)|
&\le&\frac{1}{(n+1)F_n^3(1)}\int_0^1{\tau^{-n-1}}\int_0^\tau \frac{s^n}{1-s}dsd\tau,
\end{eqnarray*}
and integrating by parts we find
  \begin{eqnarray*}
\int_0^1{\tau^{-n-1}}\int_0^\tau \frac{s^n}{1-s}dsd\tau&=&\left[\frac{1-\tau^{-n}}{n}\int_0^\tau\frac{s^n}{1-s} ds\right]_0^1+\frac{1}{n}\int_0^1\frac{s^n-1}{s-1} ds\\
&=&\frac{1}{n}\int_0^1\sum_{k=0}^{n-1}s^k ds
=\frac{1}{n}\sum_{k=1}^{n}\frac{1}{k}.
\end{eqnarray*}
Thus, it follows from  the classical inequality
 $
 \sum_{k=1}^{n}\frac{1}{k}\leq 1+\ln n
 $
that
$$
 |\partial_n F_n(x)|
\le \frac{1+\ln n}{n^2F_n^3(1)},
$$
 for $n\geq2$. Since $F_n(1)>0$ and converges to $1$, as $n$ goes to $\infty$, then one can find an absolute constant $C>0$ such that
 $$
 |\partial_n F_n(x)|
\le C\frac{\ln n}{n^2},
$$
for any $n\geq2$, which achieves the proof of the estimate.

\medskip
\noindent
{\bf{(5)}} Differentiating the integral representation \eqref{DerV1} again with respect to $x$ we obtain
\begin{equation*}
\partial_{xx}F(a_n,b_n;c_n;x)
= { \frac{(a_n+1)a_n(n-a_n)(n+2)}{n+1}}\mathcal{F}(n) \widehat{J}_n(x),
\end{equation*}
with
$$
\widehat{J}_n(x)\triangleq\int_0^1 \tau^{n-a_n+1} (1-\tau)^{a_n} (1-\tau x)^{-2-a_n}d\tau.
$$
According to \eqref{DeXW1} one finds that
\begin{equation}\label{DerV2}
\big|\partial_{xx}F(a_n,b_n;c_n;x)\big|
\leq 4\widehat{J}_n(x).
\end{equation}
The procedure for estimating $\widehat{J}_n$ matches the one given for $J_n$ in the previous assertion. Indeed,  we have the uniform bound
$
\widehat{J}_n(x)\leq C,
$
for $x\in \left[0,\frac{1}{2}\right]$. Now,  the change of variables \eqref{chang1} leads to
\begin{eqnarray*}
\widehat{J}_n(x) \leq \int_0^1 (1-\tau)^{a_n} (1-\tau x)^{-2-a_n}d\tau
\leq \frac{x^{-a_n-1}}{1-x}\int_0^{\frac{x}{1-x}}\tau^{a_n}(1+\tau)^{-2-a_n}d\tau 
\leq  \frac{C}{1-x},
\end{eqnarray*}
where  we have used the bounds $0<-a_n<\frac{2}{1+\sqrt{3}}<1$, which are verified for any $n\geq2$ and  $x\in\left[\frac{1}{2},1\right)$. Inserting this estimate into \eqref{DerV2} we obtain the announced inequality.
\end{proof}

Next we shall  prove the following.
\begin{lemma}\label{lemestim}
There exists $C>0$ such that  
\begin{eqnarray}\label{TSS00}
|F(a_n,b_n;n+1;x)-1| &\leq & C\frac{\ln n}{n}, \\
\label{estim2}
1\leq F(a_n+1,b_n;n+2;x) &\leq & C n, \\
\label{warda0}
|F(a_n,b_n;n+3;x)-1| &\leq & C\frac{\ln n}{n},
\end{eqnarray}
for any  $n\geq 2$ and any $x\in[0,1].$
\end{lemma}
\begin{proof}
The estimate  \eqref{TSS00} follows easily from  the second assertion of Lemma \ref{Lem2}, combined with the monotonicity of $F_n$ and Lemma \ref{lemX1}. Indeed,
\begin{eqnarray*}
|F(a_n,b_n;c_n;0)-F(a_n,b_n;c_n;x)|\le F(a_n,b_n;c_n;0)-F(a_n,b_n;c_n;1)
\le C\frac{\ln n}{n}.
\end{eqnarray*}

\noindent
In the case \eqref{estim2}, applying 
 similar arguments as in the first assertion of Lemma \ref{lemX1}, we conclude that the function $x\in[0,1]\to F(a_n+1,b_n;n+2;x)$ is positive and strictly increasing. Hence,
$$
1\leq F(a_n+1,b_n;n+2;x) \leq  F(a_n+1,b_n;n+2;1).
$$
Combining  \eqref{id1} and \eqref{ASS1} we obtain the estimate:
\begin{eqnarray*}
1\leq F(a_n+1,b_n;n+2;x) \leq  F(a_n+1,b_n;n+2;1)
\le \frac{\Gamma(n+2)}{\Gamma(n+1-a_n)\Gamma(2+a_n)}
\le  C n.
\end{eqnarray*}

\noindent
To check \eqref{warda0}, we use 
the first assertion of Lemma \ref{lemX1}, 
$$
0\le 1-F(a_n,b_n;n+3;x)\leq 1-F(a_n,b_n;n+3;1).
$$
Moreover, by virtue of  \eqref{id1} one has
$$
F(a_n,b_n;n+3;1)=\frac{\Gamma(n+3)}{\Gamma(n+3-a_n)}\frac{\Gamma(3)}{\Gamma(3+a_n)}.
$$
As a consequence of \eqref{ASS1} and  $a_n\sim -\frac2n$, we obtain
$$
\frac{\Gamma(n+3)}{\Gamma(n+3-a_n)}=1-2\frac{\ln n}{n}+O\left(\frac{1}{n}\right) \quad \mbox{ and } \quad
\frac{\Gamma(3)}{\Gamma(3+a_n)}=1+O\left(\frac{1}{n}\right).
$$
Therefore, the following asymptotic expansion
$$
F(a_n,b_n;n+3;1)=1-2\frac{\ln n}{n}+O\left(\frac{1}{n}\right)
$$
holds and the estimate follows easily.
\end{proof}

Another useful property deals with the behavior of the hypergeometric function with respect to the third variable $c$.
{
\begin{lemma}\label{monotC}
Let $n\geq1$, then 
the mapping $c\in (b_n,\infty)\mapsto F(a_n,b_n;c;x)$ is strictly increasing for $x\in(0,1)$ and strictly decreasing  for $x\in {(-\infty,0)}$.
\end{lemma}
}

\begin{proof}
First, we check the case $n=1$ that comes by
$$
F(a_1,b_1;c;x)=1-\frac{2}{c} x.
$$
Let $n>1$ and recall that the hypergeometric function $F(a_n,b_n;c;x)$ solves  the differential equation
$$
x(1-x)\partial^2_{xx}F(a_n,b_n;c;x)+[c-(n+1)x]\partial_xF(a_n,b_n;c;x)+2F(a_n,b_n;c;x)=0,
$$
with $F(a_n,b_n;c;0)=1$ and $\partial_xF(a_n,b_n;c;0)=-\frac{2}{c}$. Hence by differentiation  it is easy to check that $\partial_c F(a_n,b_n;c;x)$ solves
\begin{eqnarray*}
&&x(1-x)\partial^2_{xx}(\partial_c F(a_n,b_n;c;x))+[c-(n+1)x]\partial_x(\partial_c F(a_n,b_n;c;x))+2(\partial_ cF(a_n,b_n;c;x))\\
&& \hspace{2cm}=-\partial_x F(a_n,b_n;c;x),
\end{eqnarray*}
with initial conditions 
$$\partial_ cF(a_n,b_n;c;0)=0\quad \hbox{and}\quad \partial_x(\partial_ cF(a_n,b_n;c;0))=\frac{2}{c^2}.
$$ 
Note that a homogeneous solution of the last differential equation is $F(a_n,b_n;c;x)$. By the variation of constant method one can look for the full solution to the differential equation in the form  
$$\partial_ cF(a_n,b_n;c;x)=K(x)F(a_n,b_n;c;x),
$$ and from straightforward computations we find that  $T=K'$ solves the first order differential equation, 
$$
T'(x)+\left[2\frac{\partial_x F(a_n,b_n;c;x)}{F(a_n,b_n;c;x)}+\frac{c-(n+1)x}{x(1-x)}\right]\, T(x)=-\frac{\partial_x F(a_n,b_n;c;x)}{x(1-x)F(a_n,b_n;c;x)}.
$$
The general solution to this latter equation is given by 
$$
T(x)=\frac{(1-x)^{c-(n+1)}}{F(a_n,b_n;c;x)^2|x|^c}\left[K_1-\int_{0}^x \frac{|s|^{c} s^{-1}F(a_n,b_n;c;s)\partial_sF(a_n,b_n;c;s)}{(1-s)^{c-n}}ds\right], 
$$
for $x\in (-\infty,1)$ 
where $K_1\in\R$ is a real constant.
 Thus
\begin{eqnarray*}
&&\partial_c F(a_n,b_n;c;x)=F(a_n,b_n;c;x)\Bigg[K_2+\\
&&\qquad\quad\displaystyle \int_{x_0}^x\frac{(1-\tau)^{c-(n+1)}}{F(a_n,b_n;c;\tau)^2|\tau|^c}\left\{K_1-\int_{0}^\tau \frac{|s|^{c}s^{-1}F(a_n,b_n;c;s)\partial_sF(a_n,b_n;c;s)}{(1-s)^{c-n}}ds\right\}d\tau\Bigg],
\end{eqnarray*}
where $K_1, K_2\in\R$ and $x_0\in(-1,1)$. Since $\partial_c F(a_n,b_n;c;x)$ is not singular at $x=0$, we get that $K_1=0$. Then, changing the constant $K_2$ one can take $x_0=0$ getting
\begin{eqnarray*}
&&\partial_c F(a_n,b_n;c;x)=F(a_n,b_n;c;x)\Bigg[K_2-\\
&&\qquad\quad \left.\int_{0}^x\frac{(1-\tau)^{c-(n+1)}}{F(a_n,b_n;c;\tau)^2|\tau|^c}\int_{0}^\tau \frac{|s|^{c}s^{-1}F(a_n,b_n;c;s)\partial_sF(a_n,b_n;c;s)}{(1-s)^{c-n}}dsd\tau\right].
\end{eqnarray*}
The initial condition $\partial_c F(a_n,b_n;c;0)=0$ implies that $K_2=0$ and hence
\begin{eqnarray*}
\partial_c F(a_n,b_n;c;x)&=&-F(a_n,b_n;c;x)\times\\
&&\int_{0}^x{(1-\tau)^{c-(n+1)}}\int_{0}^\tau \frac{|s|^{c}s^{-1}F(a_n,b_n;c;s)\partial_sF(a_n,b_n;c;s)}{|\tau|^cF(a_n,b_n;c;\tau)^2(1-s)^{c-n}}dsd\tau.
\end{eqnarray*}
Similarly to the proof of Lemma \ref{lemX1}$-(1)$ one may obtain that  
$$F(a_n,b_n;c;x)>0\quad \hbox{and}\quad \partial_x F(a_n,b_n;c;x)<0,
$$ 
for any $c>b_n$ and $ x\in{(-\infty,1)}$ . This  entails that $\partial_c F(a_n,b_n;c;x)$ is positive for $x\in(0,1)$, and negative when $x\in{(-\infty,0)},$ which  concludes the proof.
\end{proof}

\subsection{Eigenvalues}
The existence of eigenvalues, that are the elements of the dispersion set  defined in \eqref{disper-dos}, is connected to the problem of studying the roots of the equation 
introduced in \eqref{Takk1}. 
Here, we will develop different cases illustrating strong discrepancy on the structure of the dispersion set. Assuming $A>0$ and $B<-A$, we find that the dispersion set is infinite. However, for the case $A>0$ and $B\geq -\frac{A}{2}$, the dispersion set is finite.  Notice that the transient regime corresponding to $-A\leq B\leq -\frac{A}{2}$ is not covered by the current study and turns to be more complicate due to the complex structure  of  the spectral function \eqref{Takk1}.
\\
Let us begin with studying   the cases
\begin{eqnarray}
& A>0,\quad & A+B<0,\label{condX1}\\
& A>0,\quad & A+2B\leq 0,\label{condXZ1}\\
& A>0,\quad  & A+2B\geq 0\label{Cond2}.
\end{eqnarray}

Our first main result reads as follows. 
\begin{proposition}\label{Prop-eig}
The following assertions hold true:
\begin{enumerate}
\item  Given $A, B$ satisfying \eqref{condX1}, there exist $n_0\in\N^\star$, depending only on $A$ and $B$, and a unique root $x_n\in (0,1)$ of  \eqref{Takk1}, i.e.
$
\zeta_n({x}_n)=0,$
for any $n\geq n_0$.
In addition,
\begin{equation*}
 x_n\in\left(0,1+\frac{A+B}{An}\right),
\end{equation*}
and the sequence $n\in[n_0,+\infty)\mapsto {x}_n$ is strictly increasing.
\item  Given $A, B$ satisfying the weak condition \eqref{condXZ1} and   $n\in\N^\star$, then $\zeta_n$ has no solution in ${(-\infty,0]}$.
\item Given  $A, B$ satisfying \eqref{Cond2}, then $\zeta_n$ has no solution in $ [0,1]$, for $n\geq2$.
\end{enumerate}
\end{proposition}
\begin{proof}
{$\bf{(1)}$} The expression of the spectral equation \eqref{special1} agrees with
$$
 \hspace{1cm}\zeta_n(x)=I_n^1(x)F(a_n,b_n;n+1;x)+I_n^2(x)F(a_n+1,b_n;n+2;x)+I_n^3(x)F(a_n,b_n;n+3;x),
 $$
where  $I_n^1, I_n^2$ and $I_n^3$ are defined in Lemma \ref{LMZ1}.
From this  expression we get  $\zeta_n(0)=\frac{n}{n+1}.$  To find a solution in $(0,1)$ we shall apply the Intermediate Value Theorem, and for this purpose we need to check that  $\zeta_n(1)<0$. 
Applying  \eqref{id1} we get
\begin{eqnarray*}
\zeta_n(1)&=&\left[\frac{1}{n+1-a_n}+\frac{A+2B}{A(n+1)}\right]\frac{\Gamma(n+1)}{\Gamma(n+1-a_n)\Gamma(1+a_n)}\\
&&-\frac{a_n\Gamma(n+1)}{\Gamma(n+2-a_n)\Gamma(2+a_n)}-\frac{4\Gamma(n)}{\Gamma(n+3-a_n) \Gamma(3+a_n)}.
\end{eqnarray*}
Using the following  expansion for large $n\gg1$
$$
\frac{1}{n+1-a_n}=\frac{1}{n+1}+O\left(\frac{1}{n^3}\right),
$$
and \eqref{ASS1}, we find
$$
\zeta_n(1)=2\frac{A+B}{n+1}+O\left(\frac{\ln n}{n^2}\right).
$$
Thus, under the hypothesis \eqref{condX1},  there  exists $n_0\in\N^\star$, depending on $A,B$,  such that
$$\zeta_n(1)<0,\quad\forall n\geq n_0.
$$
This proves the existence of  at least one  solution ${x}_n\in (0,1)$ to the equation
$
\zeta_n({x}_n)=0,
$
for any $n\geq n_0$.
The next objective is  to localize this root and show that  $ x_n\in\left(0,1+\frac{A+B}{An}\right)$. For this goal it suffices to verify that
$$
\zeta_n(1-\EE)<0, \, \quad \forall \EE\in \left(0,-\frac{A+B}{An}\right).
$$
Let us begin with the first term $I_n^1(x)$ in the expression \eqref{special1} which implies that
\begin{eqnarray*}
I_n^1(1-\EE)&=&\left[\frac{1-2a_n}{n+1-a_n}+\frac{A+2B}{A(n+1)}\right]-\EE\left[\frac{A+2B}{A(n+1)}-\frac{n-1+a_n}{n+1-a_n}\right].
\end{eqnarray*}
Now, it is straightforward to check the  following asymptotic expansions,
\begin{eqnarray*}
\frac{1-2a_n}{n+1-a_n}+\frac{A+2B}{A(n+1)}=2\frac{A+B}{A(n+1)}-\frac{(2n+1)a_n}{(n+1)(n+1-a_n)}
=2\frac{A+B}{A(n+1)}+O\left(\frac{1}{n^2}\right),
\end{eqnarray*}
and
\begin{eqnarray*}
-\EE\left[\frac{A+2B}{A(n+1)}-\frac{n-1+a_n}{n+1-a_n}\right]&=& -\EE\left[-1+\frac{A+2B}{A(n+1)}+\frac{2-2a_n}{n+1-a_n}\right]\\
&\leq & \EE+O\left(\frac{1}{n^2}\right)
\leq   -\frac{A+B}{An}+O\left(\frac{1}{n^2}\right),
\end{eqnarray*}
Therefore, we obtain
$$
I_n^1(1-\EE)\leq  \frac{A+B}{An}+O\left(\frac{1}{n^2}\right).
$$
Thanks to \eqref{TSS00}, we deduce
\begin{eqnarray*}
| 1-F(a_n,b_n;n+1;1-\EE)| 
\leq  C\frac{\ln n}{n},
\end{eqnarray*}
which yields in turn
\begin{equation}\label{AZQ0}
I_n^1(1-\EE)F(a_n,b_n;n+1;1-\EE)\leq\frac{A+B}{A(n+1)}+O\left(\frac{\ln n}{n^2}\right).
\end{equation}
Next, we will deal with the second term $I_n^2$ of \eqref{special1}. Directly from \eqref{estim2} we get
\begin{equation}\label{AZQ1}
 |I_n^2(x)|F(a_n+1,b_n;n+2;x)=\frac{|a_n(2x-1)|}{(n+1)(n+1-a_n)}F(a_n+1,b_n;n+2;x)
\le \frac{C}{n^2} .
\end{equation}
Similarly, the estimate \eqref{warda0} implies that
\begin{equation}\label{AZQ2}
 |I_n^3(x)|F(a_n,b_n; n+3;x)=\frac{2xF(a_n,b_n;n+3;x)}{(n+1)(n+2)}
\le \frac{C}{n^2}.
\end{equation}
Inserting \eqref{AZQ0}, \eqref{AZQ1} and \eqref{AZQ2} into the expression of $\zeta_n$ we find
$$
\zeta_n(1-\EE)\leq \frac{A+B}{A(n+1)}+O\left(\frac{\ln n}{n^2}\right),
$$
for any $\EE\in\left(0,-\frac{A+B}{An}\right)$.
From this we deduce the existence of $n_0$ depending on $A$ and $B$ such that
$$
\zeta_n(1-\EE)\leq\frac{A+B}{2A(n+1)}<0,
$$
for any $n\geq n_0$. Then $\zeta_n$ has no zero in $(1+\frac{A+B}{An},1)$ and this achieves the proof of the first result.

Next we shall prove that ${x}_n$ is the only zero of $\zeta_n$ in $(0,1)$. For this purpose it appears to be more convenient to  use the expression for $\zeta_n$ given by \eqref{Takk1}. Let us differentiate  $\zeta_n$ with respect to $x$ as follows
\begin{eqnarray*}
\partial_x\zeta_n(x)&=&F_n^\prime(x)\left[1-x+\frac{A+2B}{A(n+1)} x\right]+F_n(x)\left[-1+\frac{A+2B}{A(n+1)} \right]\\
&& + \int_0^1 F_n^\prime(\tau x) \tau^{n+1}\left[-1+2x\tau\right] d\tau+2\int_0^1 F_n(\tau x) \tau^{n+1} d\tau.
\end{eqnarray*}
From Lemma \ref{lemX1}$-(1)$, we recall that  $F_n>0$ and $F_n^\prime<0$. Hence for $ A+2B<0$ and $x\in(0,1)$ we get
\[
\partial_x\zeta_n(x)\le F_n^\prime(x)\frac{A+2B}{A(n+1)} x-F_n(x)
+ \int_0^1 F_n^\prime(\tau x) \tau^{n+1}\left[-1+2x\tau\right] d\tau
+2\int_0^1 F_n(\tau x) \tau^{n+1} d\tau.
\]
Applying the  third assertion of Lemma \ref{Lem2} we find 
$$
 |F_n^\prime(x)|\leq C\ln n,\quad \forall n\geq2,
$$
for any $x\in (0,1+\frac{A+B}{An})$ and with $C$ a constant depending only on  $A$ and $B$. It follows that
\[
\partial_x\zeta_n(x)\le C\frac{|A+2B|}{A} \frac{\ln n}{n}-F_n(x)+C\frac{\ln n}{n}+\frac{C}{n}
\le -1+C\frac{2A+2|B|}{A} \frac{\ln n}{n}+(1-F_n(x)),
\]
for $x\in \left[0,1+\frac{A+B}{An}\right)$, which implies according to \eqref{TSS00} that
\[
\partial_x\zeta_n(x)
\le-1+C\frac{A+|B|}{A} \frac{\ln n}{n}.
\]
 Hence, there exists $n_0$ such that
$$
\partial_x\zeta_n(x)\le-\frac12, \quad \forall x\in \left[0,1+\frac{A+B}{An}\right),\quad \forall n\geq n_0 .
$$
Thus, the function $x\in \left[0,1+\frac{A+B}{An}\right)\mapsto\zeta_n(x)$ is strictly decreasing  and admits only one zero that we have   denoted by ${x}_n$.

It remains to show that $n\in[n_0,+\infty)\mapsto {x}_n$ is strictly increasing, which implies in particular that
\begin{equation}\label{Diff}
\zeta_m({x}_n)\neq0, \quad \forall n\neq m\geq n_0.
\end{equation}
For this aim, it suffices to show that  the mapping  $n\in[n_0,+\infty)\mapsto \zeta_n(x)$ is strictly increasing, for any $x\in (0,1)$. Setting
\begin{equation}\label{DTQ1}
F_n(x)=1+\rho_n(x),
\end{equation}
we can write
\begin{eqnarray}\label{EqWA1}
\nonumber \zeta_n(x)&=&\frac{n}{n+1}-\frac{n}{n+2} x+\frac{A+2B}{A(n+1)}x+\rho_n(x)\left[1-x+\frac{A+2B}{A(n+1)}x\right]\\ && +\int_0^1\rho_n(\tau x) \tau^{n}\left[-1+2x\tau\right]d\tau
\triangleq\frac{n}{n+1}-\frac{n}{n+2} x+\frac{A+2B}{A(n+1)}x+R_n(x) .
\end{eqnarray}
Since $F_n$ is analytic with respect to its parameters and we can think in $n$ as a continuous parameter,  $n\mapsto \zeta_n(x)$ is also analytic. Therefore,
differentiating with respect to $n$, we deduce that
$$
\partial_n\zeta_n(x)=\frac{1}{(n+1)^2}-2\frac{x}{(n+2)^2} -\frac{A+2B}{A(n+1)^2}x+\partial_nR_n(x).
$$
Consequently,
\[
\partial_n\zeta_n(x)\geq  \frac{1}{(n+1)^2}\left[1-2x -\frac{A+2B}{A}x\right]+\partial_nR_n(x)
\geq \frac{1}{(n+1)^2}\left[1-\frac{3A+2B}{A}x\right]+\partial_nR_n(x).
\]
We use the following trivial bound
$$
1 -\frac{3A+2B}{A}x\geq\min\left(1,\kappa\right), \quad \forall x\in [0,1],
$$
where $\kappa=-2\frac{A+B}{A}$ is strictly positive due to the assumptions \eqref{condX1}.
Therefore, we can rewrite the bound for $\partial_n\zeta_n(x)$ as follows
\begin{equation}\label{ZZ1}
\partial_n\zeta_n(x)\geq \frac{\min\left(1,\kappa\right)}{(n+1)^2}+\partial_nR_n(x).
\end{equation}
To estimate $\partial_nR_n(x)$ we shall differentiate \eqref{EqWA1}  with respect to $n$, 
\begin{eqnarray*}
\partial_nR_n(x)&=&\partial_nF_n(x)(1-x)+\partial_nF_n(x)\frac{A+2B}{A(n+1)}x
-\rho_n(x)\frac{A+2B}{A(n+1)^2}x\\ &&+ \int_0^1(\partial_nF_n(\tau x)) \tau^{n}\left[-1+2x\tau\right]d\tau
+\int_0^1\rho_n(\tau x) \tau^{n}\ln \tau\left[-1+2x\tau\right]d\tau.
\end{eqnarray*}
From \eqref{TSS00} and Lemma \ref{Lem2}, we deduce
\begin{eqnarray*}
\partial_nR_n(x)&\geq&\partial_nF_n(x)\frac{A+2B}{A(n+1)}x-\rho_n(x)\frac{A+2B}{A(n+1)^2}x\\
&& + \int_0^1\partial_nF_n(\tau x) \, \tau^{n}\left[-1+2x\tau\right]d\tau+\int_0^1\rho_n(\tau x) \tau^{n}\ln \tau\left[-1+2x\tau\right]d\tau,
\end{eqnarray*}
and
\begin{eqnarray*}
\left|\partial_nF_n(x)\frac{A+2B}{A(n+1)}x-\rho_n(x)\frac{A+2B}{A(n+1)^2}x\right|&\le& C\frac{|A+2B|}{A}\frac{\ln n}{n^3}.
\end{eqnarray*}
Observe that the  first integral can be bounded as follows
\[
\left|\int_0^1\partial_nF_n(\tau x) \, \tau^{n}\left[-1+2x\tau\right]d\tau\right| \le C\frac{\ln n}{n^2}\int_0^1 \tau^{n}|-1+2x\tau|d\tau
\le C\frac{\ln n}{n^3},
\]
while for the second one we have 
\[
\left|\int_0^1\rho_n(\tau x) \tau^{n}\ln \tau\left[-1+2x\tau\right]d\tau\right| \le C\frac{\ln n}{n}\int_0^1\tau^n|\ln \tau| d\tau\\
\le C\frac{\ln n}{n^3},
\]
where we have used \eqref{DTQ1} and Lemma \ref{lemestim}. Plugging  these estimates  into \eqref{ZZ1}, we find
$$
\partial_n\zeta_n(x)\geq\frac{\min\left(1,\kappa\right)}{(n+1)^2}-C\frac{|A+2B|}{A}\frac{\ln (n+1)}{(n+1)^3}, \quad \forall x\in [0,1].
$$
Then, there exists $n_0$ depending only on $A,B$ such that
$$
\partial_n\zeta_n(x)\geq\frac{\min\left(1,\kappa\right)}{2(n+1)^2},
$$
for any $n\geq n_0$ and any $ x\in [0,1]$. This implies that $n\in [n_0,+\infty[\mapsto \zeta_n(x)$ is strictly increasing and thus \eqref{Diff} holds.

\medskip
\noindent
${\bf{(2)}}$  From   \eqref{Takk3} one has 
\begin{eqnarray*}
\zeta_n(x)&=&\frac{A+2B}{A(n+1)}xF(a_n,b_n;n+1;x)-\frac{x}{n+1}F(a_n,b_n;n+2;x)\\
&&+ \frac{n(1-x)}{n+1}F(a_n,b_n;n+2;x)+\frac{2nx}{(n+1)(n+2)}F(a_n,b_n;n+3;x).
\end{eqnarray*}
Remark that  the involved  hypergeometric functions are strictly positive, which implies that
\begin{eqnarray*}
\zeta_n(x)>
 \frac{n}{n+1}\left((1-x)F(a_n,b_n;n+2;x)+\frac{2x}{n+2}F(a_n,b_n;n+3;x)\right).
\end{eqnarray*}
To get the announced result, it is enough to check that
$$
 (1-x)F(a_n,b_n;n+2;x)+\frac{2x}{(n+2)}F(a_n,b_n;n+3;x)\geq {1},\quad \forall x\in {(-\infty,0)},
$$
which follows from  Lemma \ref{monotC}:
\begin{eqnarray*}
(1-x)F(a_n,b_n;n+2;x)&+&\frac{2x}{(n+2)}F(a_n,b_n;n+3;x)\\
&\geq&F(a_n,b_n;n+2;x){\left(1-\frac{n x}{n+2}\right)}
\geq{F(a_n,b_n;n+2;x)}
\geq {1},
\end{eqnarray*}
 for  any {$x\leq0$}. Thus, $\zeta_n(x)>0$ for any $x\in {(-\infty,0)}$ and this concludes the proof.

\medskip
\noindent
{${\bf(3)}$} Let us use the expression of $\zeta_n$ given in \eqref{Takk3} obtaining
\begin{eqnarray*}
\zeta_n(x)&=&\frac{A+2B}{A(n+1)} xF(a_n,b_n;n+1;x)+\frac{n(1-x)}{n+1}F(a_n,b_n;n+2;x)\\
&&-\frac{x}{n+1}F(a_n,b_n;n+2;x)+\frac{2nx}{(n+1)(n+2)}F(a_n,b_n;n+3;x).
\end{eqnarray*}
Then
$$
\zeta_n(x)>
\frac{x}{n+1}\left(-F(a_n,b_n;n+2;x)+\frac{2n}{n+2}F(a_n,b_n;n+3;x)\right),
$$
 for any $x\in [0,1)$. From Lemma \ref{monotC} we deduce
\[
-F(a_n,b_n;n+2;x)+\frac{2n}{(n+2)}F(a_n,b_n;n+3;x) \geq \frac{n-2}{n+2}F(a_n,b_n;n+2;x)
\geq 0,
\]
 for any $n\geq2$ and $x\in(0,1)$. This implies that $\zeta_n(x)>0$, for $x\in(0,1)$ and $n\geq2$. The case $n=1$ can be checked directly by the explicit expression stated in Remark \ref{n1}. 
\end{proof}
In the following result, we investigate more the case  \eqref{Cond2}. We mention that according to \mbox{Proposition \ref{Prop-eig}$-(3)$} there are no eigenvalues in $(0,1)$. Thus, it remains to explore the region $(-\infty,0)$ and study whether one can find eigenvalues there. Our result reads as follows.
\begin{proposition}\label{Propexistcase2}
Let $n\geq2$ and  $A, B\in\R$ satisfying \eqref{Cond2}. Then, the following assertions hold true:
\begin{enumerate}
\item If 
$
n\leq \frac{B}{A}+\frac18,
$
there exists a unique $x_n\in(-1,0)$ such that $\zeta_n(x_n)=0$.
\item If 
$
n\leq {\frac{2B}{A}},
$
there exists a unique $x_n\in(-\infty,0)$ such that $\zeta_n(x_n)=0$, with
\begin{equation}\label{Trd1}
\frac{1}{1-\frac{A+2B}{A(n+1)}}<x_n<0.
\end{equation}
In addition,  the map $x\in(-\infty,0]\mapsto \zeta_n(x)$ is strictly increasing.
\item If 
$
n\geq  \frac{B}{A}+1,
$
then $\zeta_n$ has no solution in $[-1,0]$.

\item If 
$
n\geq {\frac{2B}{A}+2},
$
then $\zeta_n$ has no solution in $(-\infty,0]$.
\end{enumerate}
\end{proposition}

\begin{proof}
\noindent
{\bf{(1)}}
Thanks to  \eqref{Takk1} we have that
$
\zeta_n(0)=\frac{n}{n+1}>0.
$ So to apply the Intermediate Value Theorem and prove that $\zeta_n$ admits a solution in $[-1,0]$ it suffices to guarantee that
$\zeta_n(-1)<0.$ Now coming back to \eqref{Takk1} and using that $x\in (-1,1)\mapsto F_n(x)$ is strictly decreasing we get
\begin{eqnarray*}
 \zeta_n(-1)&=&F_n(-1)\left(2-\frac{A+2B}{A(n+1)}\right)-\int_0^1F_n(-\tau)\tau^n\big(1+2\tau) d\tau\\
&<&F_n(-1)\left(2-\frac{A+2B}{A(n+1)}\right)-\int_0^1\tau^n\big(1+2\tau) d\tau\\
&<&F_n(-1)\left(2-\frac{A+2B}{A(n+1)}\right)-\frac{3n+4}{(n+1)(n+2)}, \quad \forall x\in [-1,0).
\end{eqnarray*}
Consequently, to get $\zeta_n(-1)<0$ we impose the condition
\begin{equation}\label{DQW1}
2-\frac{A+2B}{A(n+1)}\leq \frac{3n+4}{(n+1)(n+2)F_n(-1)}.
\end{equation}
Coming back to the integral representation, one gets 
\begin{eqnarray}\label{DQWM1}
 F_n(-1)\leq F_n(0) 2^{-a_n}
\le 2,
\end{eqnarray}
due to the fact $a_n\in (-1,0)$. In addition, it is easy to check that 
$$
 \frac{3n+4}{n+2}\geq \frac52,\quad \forall n\geq2,
$$
and the assumption \eqref{DQW1} is satisfied if 
\begin{equation*}
2-\frac{A+2B}{A(n+1)}\leq \frac{5}{4(n+1)}
\end{equation*}
holds, or equivalently, if
\begin{equation}\label{DQW2}
2\leq n\leq \frac{B}{A}+\frac18.
\end{equation}
In conclusion, under the assumption \eqref{DQW2}, the function $\zeta_n$ admits a solution $x_n\in (-1,0).$ 
Now, we localize this zero.
Since $F_n$ is strictly positive  in $[-1,1]$, then the  second term in \eqref{Takk1} is always strictly negative. Let us analyze the sign of the first term
$$
F_n(x)\left[1-x+\frac{A+2B}{A(n+1)}x\right],
$$
which has a unique root 
\begin{eqnarray}\label{xc}
x_c=\frac{1}{1-\frac{A+2B}{A(n+1)}}.
\end{eqnarray}
This root belongs to $(-\infty,0)$ if and only if $n<\frac{2B}{A}$, which follows  automatically from \eqref{DQW2}. Moreover the mapping $x\mapsto 1-x+\frac{A+2B}{A(n+1)}x$ will be strictly increasing.  Hence, \mbox{if $x_c\leq-1$}, then $x_n>x_c$. So let us assume that $x_c\in(-1,0)$, then 
$$
 F_n(x)\left[1-x+\frac{A+2B}{A(n+1)}x\right]<0,\quad \forall x\in [-1,x_c],
$$
which  implies that
$$
\zeta_n(x)<0,\quad \forall x\in [-1,x_c].
$$
Therefore, the solution $x_n$ must belong to $(x_c,0)$, and equivalently
\begin{equation*}
\frac{1}{1-\frac{A+2B}{A(n+1)}}<x_n<0.
\end{equation*}
The uniqueness of this solutions comes directly from the second assertion.

\medskip
\noindent
$\bf{(2)}$
As in the previous argument we have that
$
\zeta_n(0)=\frac{n}{n+1}>0
$ and the idea is to apply also the Intermediate Value Theorem. We intend to find the asymptotic behavior of $\zeta_n$ for $x$ going to $-\infty$. We first  find an asymptotic behavior of $F_n$. For 
this purpose we use the identity \eqref{Line-HGF}, which implies that
$$
F_n(x)=(1-x)^{-a_n}F\left(a_n,a_n+1,n+1,\frac{x}{x-1}\right),\quad \forall x\leq0.
$$
Setting
\begin{equation}\label{HYpG1}
\varphi_n(y)\triangleq F\left(a_n,a_n+1,n+1,y\right),\quad \forall y\in [0,1],
\end{equation}
we obtain
$$
\varphi_n^\prime(y)=\frac{a_n(1+a_n)}{n+1}F\left(a_n+1,a_n+2,n+2,y\right),\quad \forall y\in [0,1)
$$
By a monotonicity argument we deduce that
\[
 |\varphi_n^\prime(y)|\leq C|F\left(a_n+1,a_n+2,n+2,y\right)|
\le C|F\left(a_n+1,a_n+2,n+2,1\right)|
\le C_n,\quad \forall\, y\in[0,1],
\]
 with $C_n$ a constant depending on $n$. However, the dependence with respect to $n$ does not matter because we are interested in the asymptotics for large negative $x$  but for  a fixed $n$. Then, let us drop $n$ from the subscript of the constant $C_n$. Applying the Mean Value Theorem we get
 $$
  |\varphi_n(y)-\varphi_n(1)|\leq C(1-y),\quad \forall\, y\in[0,1].
 $$
 Combining this estimate with \eqref{HYpG1} and $a_n\in(-1,0),$ we obtain  
 
 \[
|F_n(x)-(1-x)^{-a_n}\varphi_n(1)|\leq C (1-x)^{-a_n-1}
\leq C,\quad \forall x\leq 0,
 \]
 which implies in turn that
  \begin{equation}\label{EqXV1}
 |F_n(x)-(1-x)^{-a_n}\varphi_n(1)|
\leq C,\quad \forall x\leq -1.
 \end{equation}
 Consequently, we deduce that
 \begin{eqnarray*}
 \int_0^1F_n(\tau x)\tau^n\big(-1+2x\tau) d\tau&\sim& 2x(-x)^{-a_n}\varphi_n(1)\int_0^1\tau^{-a_n+n+1}d\tau\\
&\sim& \frac{2}{n+2-a_n}x(-x)^{-a_n}\varphi_n(1), \quad \forall\, -x\gg 1.
\end{eqnarray*}
Coming back to \eqref{Takk1} and using once again \eqref{EqXV1}  we get the asymptotic behavior
 \begin{eqnarray*}
 \zeta_n(x)&\sim&\varphi_n(1)  (-x)^{-a_n}\left(1-x+\frac{A+2B}{A(n+1)} x+ \frac{2}{n+2-a_n}x\right)\\
&\sim& \varphi_n(1) \left(\frac{A+2B}{A(n+1)}-1+ \frac{2}{n+2-a_n}\right)(-x)^{-a_n} x,\quad \forall\, -x\gg 1.
\end{eqnarray*}
The condition
\begin{equation}\label{DQW3}
n\leq\frac{2B}{A},
\end{equation}
 implies that 
$$
\frac{A+2B}{A(n+1)}-1+ \frac{2}{n+2-a_n}>0.
$$
Since $\varphi_n(1)>0$ , we deduce that
$$
\lim_{x\to-\infty}\zeta_n(x)=-\infty.
$$
Therefore we deduce from  the Intermediate Value Theorem that  under the assumption \eqref{DQW3}, the function $\zeta_n$ admits a solution $x_n\in (-\infty,0).$ Moreover, by the previous proof we get \eqref{Trd1}.

It remains to prove the uniqueness of this solution. For this goal we check that the mapping $x\in(-\infty,0)\mapsto \zeta_n(x)$ is strictly increasing when $n\in\left[1,\frac{2B}{A}\right].$
Differentiating  $\zeta_n$ with respect to $x$ yields
\begin{eqnarray*}
\zeta_n'(x)&=&F_n'(x)\left[1-x+\frac{A+2B}{A(n+1)}x\right]+F_n(x)\left[\frac{A+2B}{A(n+1)}-1\right]\\
&&+\int_0^1F_n'(\tau x)\tau^{n+1}(-1+2x\tau)d\tau+2\int_0^1F_n(\tau x)\tau^{n+1}d\tau.
\end{eqnarray*}
From Lemma \ref{lemX1}-(1) we infer  \mbox{that $F_n'(x)<0$,} for $x\in(-\infty,0)$, and therefore we get 
$$
 \zeta_n^\prime(x)>0,\quad \forall \, x\in (-\infty,x_c).
$$ 
Let  $x\in(x_c,0)$, then by a monotonicity argument we get
\begin{eqnarray*}
0\leq 1-x+\frac{A+2B}{A(n+1)}x\le 1,
\end{eqnarray*}
and  thus
\begin{eqnarray*}
F_n^\prime(x)\left[1-x+\frac{A+2B}{A(n+1)}x\right]+2\int_0^1F_n(\tau x)\tau^{n+1}d\tau&\geq&F_n'(0)+\frac{2}{n+2}\geq-\frac{2}{(n+1)(n+2)},
\end{eqnarray*}
by using that $F_n^\prime(0)=-\frac{2}{n+1}$ and  that $ F_n^\prime(x)$ is decreasing and negative in $ (-\infty,1)$. From the assumption \eqref{DQW3} and the positivity of $F_n$ we get
$$
 F_n(x)\left(\frac{A+2B}{A(n+1)}-1\right)\geq0,\quad \forall x\in (-\infty,0].
$$
Therefore, putting together the preceding estimates we deduce that
\begin{eqnarray*}
\zeta_n^\prime(x)&>&-\frac{2}{(n+1)(n+2)}+\int_0^1F_n'(\tau x)\tau^{n+1}(-1+2x\tau)d\tau\\
&>&-\frac{2}{(n+1)(n+2)}-\int_0^1F_n^\prime(\tau x)\tau^{n+1}d\tau,\quad \forall x\in (-\infty,0).
\end{eqnarray*}
At this stage it suffices to make appeal to \eqref{NNG1} in order to obtain
$$
 -\int_0^1F_n^\prime(\tau x)\tau^{n+1}d\tau\geq\frac{2}{(n+1)(n+2)},\quad \forall x\in (-\infty,0],
$$
from which  it follows
\begin{equation*}
\zeta_n^\prime(x)>0,\quad \forall x\in (-\infty,0],
\end{equation*}
which implies that $\zeta_n$ is strictly increasing in $(-\infty,0]$, and thus $x_n$ is the only solution in this interval.

\medskip
\noindent
${\bf(3)}$ 
Using the definition of $\zeta_n$ in \eqref{Takk1} and the monotonicity of $F_n$, one has
\begin{equation}\label{TPP1}
\zeta_n(x)>F_n(x)\left[1-x+\frac{A+2B}{A(n+1)}x-\frac{1}{n+1}+\frac{2x}{n+1}\right].
\end{equation}
Now it is easy to check that 
$$
1-x+\frac{A+2B}{A(n+1)}x-\frac{1}{n+1}+\frac{2x}{n+1}\geq\min\left\{\frac{n}{n+1}, \frac{2n-1}{n+1}-\frac{A+2B}{A(n+1)}  \right\},
$$
for any $x\in[-1,0]$. This claim can be derived from the fact that the left-hand-side term is polynomial in $x$ with degree one.
Consequently, if we assume 
$$
n\geq 1+\frac{B}{A},
$$
we get $\frac{2n-1}{n+1}-\frac{A+2B}{A(n+1)}\geq0$, and therefore \eqref{TPP1} implies
$$
 \zeta_n(x)>0,\quad \forall x\in [-1,0].
$$ 
Then, $\zeta_n$ has no solution in $[-1,0]$.

\medskip
\noindent
{\bf{(4)}}
Using the expression of $\zeta_n$ in \eqref{Takk1} and the monotonicity of $F_n$, one has
\begin{eqnarray}\label{TPP2}
\nonumber \zeta_n(x)&>&F_n(x)\left[1-x+\frac{A+2B}{A(n+1)}x-\frac{1}{n+1}+\frac{2x}{n+1}\right]\\
&\geq &F_n(x)\left[\frac{n}{n+1}-x\left(\frac{n-1}{n+1}-\frac{A+2B}{A(n+1)}\right)\right],\quad \forall \, x\in(-\infty,0).
\end{eqnarray}
The assumption
$$
n\geq 2+\frac{2B}{A},
$$
yields $\frac{n-1}{n+1}-\frac{A+2B}{A(n+1)}\geq0$, and therefore \eqref{TPP2} implies
$$
 \zeta_n(x)>0,\quad \forall x\in (-\infty,0].
$$ 
Thus, $\zeta_n$ has no solution in $(-\infty,0]$.
\end{proof}

In the next task we discuss the localization of the  zeroes of $\zeta_n$ and, in particular, we improve the lower bound  \eqref{Trd1}. Notice that $B>0$ in order to get solutions of $\zeta_n$ in $(-\infty,0]$ in the case $A>0$ and  $n\geq2$, by using Proposition \ref{Propexistcase2}-$(4)$. Our result reads as follows.
\begin{proposition}\label{Case2Prop}
Let  ${A, B>0}$ and $n\geq2$. If $x_n\in(-\infty,0)$ is any solution of $\zeta_n$, then the following properties are satisfied:
\begin{enumerate}
\item We have 
$ \displaystyle
\mathscr{P}_n(x_n)<0,  \hbox{ with } \mathscr{P}_n(x)= \frac{n}{n+1}+x\left[-\frac{n}{n+2}+\frac{A+2B}{A(n+1)}\right].
$

\item If  $x_n\in(-1,0),$ then $\displaystyle
x_\star\triangleq-\frac{2n+1}{2(n+1)}\frac{1}{\frac{A+2B}{A(n+1)}-\frac{n+1}{n+2}}<x_n.
$

\item We always have  $\displaystyle
x_n<-\frac{A}{2B}.
$

\end{enumerate}
\end{proposition}

\begin{proof}
$\bf{(1)}$
Since  $F_n(\tau x)< F_n(x)$, for any $\tau\in[0,1)$ and $x\in(-\infty,0),$ then we deduce from the expression  \eqref{Takk1} that
\begin{equation}\label{AZZ1}
\zeta_n(x)> F_n(x)\Big[1-x+\frac{A+2B}{A(n+1)}x-\frac{1}{n+1}+\frac{2x}{n+2}\Big]=F_n(x)\mathscr{P}_n(x).
\end{equation}  
As $F_n$ is strictly positive in $(-\infty,1)$, then 
$
\mathscr{P}_n(x_n)<0,
$ 
for any root of $\zeta_n$.

\medskip
\noindent
$\bf{(2)}$
Recall from Proposition \ref{Propexistcase2}--$(3)$ that if $\zeta_n$ admits a solution in $(-1,0)$ with $n\geq2$ then necessary $2\leq n\leq 1+\frac{B}{A}$. This implies that $n\leq \frac{2B}{A}$ and hence the mapping $x\mapsto 1-x+\frac{A+2B}{A(n+1)}x$ is increasing. Combined with  the definition of  \eqref{xc}  and \eqref{Trd1} we deduce that
$$
1-x+\frac{A+2B}{A(n+1)}x\geq 0,	\quad \forall x\in (x_c,0)
$$
and $x_n\in(x_c,0)$. Using the monotonicity of $F_n$ combined with the bound \eqref{DQWM1} we find from   \eqref{Takk1} 
\begin{eqnarray*}
 \zeta_n(x)<2\left[1-x+\frac{A+2B}{A(n+1)}x\right]-\frac{1}{n+1}+\frac{2x}{n+2},\quad \forall x\in (x_c,0).
\end{eqnarray*}
Evaluating at any root $x_n$ we obtain
$$
-\frac{2n+1}{n+1}<2x_n\left[\frac{A+2B}{A(n+1)}-\frac{n+1}{n+2}\right].
$$
 Keeping  in mind that $n\leq\frac{2B}{A} $, we get 
$
\frac{A+2B}{A(n+1)}-\frac{n+1}{n+2}>0,
$ and therefore 
we find the announced lower bound for $x_n$.

\medskip
\noindent
$\bf{(3)}$
In a similar way to the upper bound for $x_n$, we turn to \eqref{AZZ1}, and evaluate this inequality at $-\frac{A}{2B}.$ Then we find
\[
 \zeta_n\left(-\frac{A}{2B}\right)>F_n\left(-\frac{A}{2B}\right)\mathscr{P}_n\left(-\frac{A}{2B}\right)
 >\mathscr{P}_n\left(-\frac{A}{2B}\right).
\]
Explicit computations yield
$$
\mathscr{P}_n\left(-\frac{A}{2B}\right)=\frac{n-1}{n+1}+\frac{A}{2B}\Big[\frac{n}{n+2}-\frac{1}{n+1}\Big]
$$
Since $\frac{A}{2B}>0$ and $n\geq2$, then we infer that $\mathscr{P}_n\left(-\frac{A}{2B}\right)>0$ and $\zeta_n\left(-\frac{A}{2B}\right)>0.$ Now we recall from Proposition \ref{Propexistcase2}-$(2)$ that $x\in(-\infty,0)\mapsto \zeta_n(x)$ is strictly increasing. Thus combined this property with the preceding one we deduce that
$$
x_n<-\frac{A}{2B},
$$
which achieves the proof.
\end{proof}

Notice that from Proposition \ref{Propexistcase2}-$(4)$ when $B\leq 0$, the function $\zeta_n$ has no solution in $(-\infty,0]$ for any $n\geq2$. Moreover, in the case that $0<B\leq\frac{A}{4}$, Proposition \ref{Propexistcase2}-$(4)$ and Proposition \ref{Case2Prop}-$(1)$ give us again that  $\zeta_n$ has no solution in $(-\infty,0]$ for any $n\geq2$. Combining these facts 
with  Proposition \ref{Prop-eig}--$(3)$,  we immediately get the following result.
\begin{coro}\label{CorCaseM2}
Let $A>0$ and $B$ satisfying
$$
-\frac{A}{2}\leq B\leq \frac{A}{4}.
$$
Then, the function  $\zeta_n$ has no solution   in $(-\infty,1]$ for any $n\geq2$. However,  the function $\zeta_1$ admits the solution $x_1=-\frac{A}{2B}$. Notice that this latter solution belongs to $(-\infty,1)$ if and only if {$B\notin\left[ -\frac{A}{2},0\right].$}
\end{coro}

{
In the next result, we study the case when $x_1\in\left(0,\frac{1+\epsilon}{2}\right]$ for some $0<\epsilon<1$, showing that there is no intersection with other eigenvalues.

\begin{proposition}\label{Propn1intersec}
Let $A>0$. There exists $\epsilon\in(0,1)$ such that if $B\leq-\frac{A}{1+\epsilon}$, then $\zeta_n(x_1)\neq 0$ for any $n\geq 2$ and $x_1=-\frac{A}{2B}$, with
$
\epsilon\approx 0,0581.
$
\end{proposition}
\begin{proof}
From \eqref{Takk1}, one has
$$
\zeta_n(x_1)=F_n(x_1)\frac{n}{n+1}(1-x_1)+\int_0^1F_n(x_1\tau)\tau^n\left[-1+2x_1\tau\right]d\tau.
$$
By the integral representation of $F_n$ given in \eqref{integ}, we obtain that
$$
F_n(x)>\frac{\Gamma(n+1)}{\Gamma(n-a_n)\Gamma(a_n+1)}\int_0^1 \ttt^{n-a_n-1}(1-\ttt)^{a_n}d\ttt(1-x)^{-a_n}=(1-x)^{-a_n},
$$
for any $x\in(0,1)$, using the Beta function \eqref{Betafunction}.
By the monotonicity of $F_n(x)$ with respect to $x$ and the above estimate, we find that
$$
(1-x_1)^{-a_n}< F_n(x_1)<1,
$$
for any $n\geq 2$ and $x_1\in(0,1)$, which agrees with the hypothesis on $A$ and $B$. Hence, we have that
$$
\zeta_n(x_1)>\frac{n}{n+1}(1-x_1)^{1-a_n}+2x_1\frac{(1-x_1)^{-a_n}}{n+2}-\frac{1}{n+1}.
$$
The above expression is increasing with respect to $n$, which implies that
$$
\zeta_n(x_1)>\frac{2}{3}(1-x_1)^{1-a_2}+\frac{x_1}{2}(1-x_1)^{-a_2}-\frac{1}{3}.
$$
Since $a_2=1-\sqrt{3}$, we get
$$
\zeta_n(x_1)>\frac{2}{3}(1-x_1)^{\sqrt{3}}+\frac{x_1}{2}(1-x_1)^{\sqrt{3}-1}-\frac{1}{3}\triangleq \mathcal{P}(x_1).
$$
The function $\mathcal{P}$ decreases in $(0,1)$ and admits a unique  root $\overline{x}$ whose approximate value is given by
$
\overline{x}=0,52907.
$
Hence, $\mathcal{P}(x_1)\geq0$ for $x_1\in(0,\overline{x}]$, and consequently we get
$$
\zeta_n(x_1)>0,
$$
for $B\leq-\frac{A}{2\overline{x}}$ , achieving the announced result.
\end{proof}
}

We finish this section by the following result concerning  the monotonicity of the eigenvalues.
\begin{proposition}\label{CorCase2}
Let  $A>0$ and $2B>A$. Then, the following assertions hold true: 

\begin{enumerate}
\item Let $x\in (-\infty,0)$, then $n\in\left[1,\frac{2B}{A}\right]\mapsto \zeta_n(x)$ is strictly increasing. In addition, we have 
$$
\Big\{x\in (-\infty,0], \,\zeta_n(x)= 0\Big\}\cap \Big\{x\in (-\infty,0], \,\zeta_m(x)=0\Big\}=\emptyset,
$$
for any $n\neq m\in\left[1,\frac{2B}{A}\right],$ and each set contains at most one element. 
\item The sequence $n\in\left[1, \frac{B}{A}+\frac18\right]\mapsto x_n$ is strictly decreasing, where the $\{x_n\}$ are constructed in Proposition \ref{Propexistcase2}.
\item {If   $m\in \left[1, \frac{2B}{A}-2\right],$ then 
$$
 \zeta_n(x_m)\neq 0, \quad \forall n\in \N^\star\backslash\{ m\}.
$$}

\end{enumerate}
\end{proposition}
\begin{proof}
{${\bf(1)}$}
We shall  prove that  the mapping $n\in\left[1,\frac{2B}{A}\right]\mapsto \zeta_n(x)$ is strictly increasing for fixed $x\in (-\infty,0]$ and $2B\geq A$. 
Differentiating \eqref{Takk1} with respect to $n$ we get
\begin{eqnarray*}
\partial_n\zeta_n(x)&=&  \partial_nF_n(x)\left[1-x+\frac{A+2B}{A(n+1)} x\right]+\int_0^1 \partial_nF_n(\tau x) \tau^n\left[-1+2x\tau\right] d\tau\\
&&-F_n(x)\frac{A+2B}{A(n+1)^2} x+\int_0^1 F_n(\tau x) \tau^n \ln \tau \left[-1+2x\tau\right] d\tau.
\end{eqnarray*}
Using Lemma \ref{Lem2}-(1) and the positivity of $F_n$, we deduce that 
$$
\int_0^1 \partial_nF_n(\tau x) \tau^n\left[-1+2x\tau\right] d\tau>0,\quad  \int_0^1 F_n(\tau x) \tau^n \ln \tau \left[-1+2x\tau\right] d\tau>0,\quad \forall  x\in (-\infty,0).
$$
Due to the assumption   $n\in\left[1,\frac{2B}{A}\right]$, we have that $x_c<0$, where  $x_c$ is defined in \eqref{xc}.
If $x\in(-\infty,x_c]$, we find that
$$
1-x+\frac{A+2B}{A(n+1)} x\leq 0,
$$
which implies
$$ \partial_n \zeta_n (x)>0,\quad \forall x\in(-1,x_c),\forall n\geq 1.
$$
We obtain 
\begin{equation}\label{HPL1}
0<1-x+\frac{A+2B}{A(n+1)} x<1,
\end{equation}
 for $x\in(x_c,0)$, which yields in view of  Lemma \ref{Lem2}-(2)
\begin{eqnarray*}
\partial_n\zeta_n(x)&>&  \partial_nF_n(x)\left[1-x+\frac{A+2B}{A(n+1)} x\right]-F_n(x)\frac{A+2B}{A(n+1)^2} x\\
&>& \partial_nF_n(x)-F_n(x)\frac{A+2B}{A(n+1)^2} x
>\frac{-2xF_n(x)}{(n+1)^2}\left[-1+\frac{A+2B}{2A}\right].
\end{eqnarray*}
Taking into account $2B\geq A$, one gets
$
 \partial_n\zeta_n(x)>0, \forall\, x\in(-\infty,0].$
It remains to discuss the  case $x_c\leq-1$. Remark that the estimate \eqref{HPL1} is satisfied for any $x\in(-\infty,0)$, and then the foregoing  inequality holds, and one gets finally
$$
  \partial_n\zeta_n(x)>0,\quad  \forall\, x\in[-1,0],\forall n\in\left[1,\frac{2B}{A}\right].
$$
Consequently, we deduce that the mapping $n\in\left[1,\frac{2B}{A}\right]\mapsto \zeta_n(x)$ is strictly increasing for any $x\in (-\infty,0)$. This implies in particular that the functions $\zeta_n$ and $\zeta_m$ have no common zero in $(-\infty,0)$ for $n\neq m\in \left[1,\frac{2B}{A}\right]$.

\medskip
\noindent
{${\bf(2)}$} This follows  by combining that $x\in (-\infty,0)\mapsto\zeta_n(x)$ and $n\in\left[1,\frac{2B}{A}\right]\mapsto \zeta_n(x)$ are strictly increasing, proved in Proposition \ref{Propexistcase2}$-(1)$ and Proposition \ref{CorCase2}$-(1)$.

\medskip
\noindent
{{${\bf(3)}$} By the last assertions, this is clear for $n\leq \frac{2B}{A}$ and it is also true for $n\geq \frac{2B}{A}+2$, since $\zeta_n$ has not roots in $(-\infty,1)$, by Proposition \ref{Prop-eig} and Proposition \ref{Propexistcase2}. Then, let us study the case $n\in(\frac{2B}{A},\frac{2B}{A}+2)$. First, using \eqref{Trd1}, we get that $x_m$, which is a solution of $\zeta_m=0$ with $m\leq \frac{2B}{A}-2$, verifies
\begin{eqnarray}\label{bound2}
x_m\geq -\frac{2B-A}{2A}.
\end{eqnarray}
Now, the strategy is to show that $\mathscr{P}_n(x_m)>0$ for $n\in(\frac{2B}{A},\frac{2B}{A}+2)$, and then  Proposition \ref{Case2Prop} will imply that $\zeta_n(x_m)\neq 0$. By definition, we have that
$$
\mathscr{P}_n(x_m)=\frac{n}{n+1}+x_m\left[-\frac{n}{n+2}+\frac{A+2B}{A(n+1)}\right].
$$
If $-\frac{n}{n+2}+\frac{A+2B}{A(n+1)}\leq 0$, then $\mathscr{P}_n(x_m)>0$. Otherwise, we use \eqref{bound2} getting
\begin{eqnarray*}
\mathscr{P}_n(x_m)&>&\frac{n}{n+1}-\frac{2B-A}{2A}\left[-\frac{n}{n+2}+\frac{A+2B}{A(n+1)}\right]\\
&=&\frac{1}{2(n+1)(n+2)}\left[n^2\left(1+\frac{2B}{A}\right)+n\left(4+\frac{2B}{A}-4\frac{B^2}{A^2}\right)+2\left(1-4\frac{B^2}{A^2}\right)\right].
\end{eqnarray*}
Straightforward computations yield that the above parabola is increasing in $n\in(\frac{2B}{A},\frac{2B}{A}+2)$. Evaluating at $n=\frac{2B}{A}$ in the parabola, we find
$$
\mathscr{P}_n(x_m)>\frac{8\frac{B}{A}+2}{2(n+1)(n+2)}>0.
$$}
\end{proof}

\subsection{Asymptotic expansion of the eigenvalues}
When solving  the boundary equation in  Proposition \ref{propImpl}, one requires that the angular velocity is located outside the singular set \eqref{Interv1}.  Consequently, in order to  apply the bifurcation  argument for the density equation we should check that  the eigenvalues  $\{x_n\}$ constructed in Proposition \ref{Prop-eig}  do not intersect the singular set. This problem sounds to be very technical and in the case $A+B<0$, where we know that the dispersion set is infinite, we reduce  the problem to  studying the asymptotic behavior of each sequence.  Let us start with a preliminary result.
\begin{lemma}\label{lemint}
Let $(x_n)_{n\in\N}$ be a sequence of real  numbers in $(-1,1)$ such that  $x_n=1-\frac{\kappa}{n}+o\left(\frac{1}{n}\right)$, for some strictly positive number $\kappa.$ Then the following asymptotics
$$
F(a_n+1,b_n;n+2;x_n)=n\left(\eta + o(1)\right), 
$$
holds, with $\displaystyle \eta=\kappa \int_0^{+\infty} \frac{\tau e^{-\kappa\tau }}{1+\tau}d\tau$.
\end{lemma}
\begin{proof}
The integral representation of hypergeometric functions \eqref{integ} allows us to write
\[
F(a_n+1,b_n;n+2;x_n)=\frac{\Gamma(n+2)}{\Gamma(n-a_n)\Gamma(2+a_n)}\ell_n
=\frac{n(n+1)\Gamma(n)}{\Gamma(n-a_n)\Gamma(2+a_n)}\ell_n,
\]
where
$$
\ell_n\triangleq\int_0^1\tau^{n-a_n-1}(1-\tau)^{1+a_n}(1-\tau x_n)^{-1-a_n}d\tau.
$$
Set
$
\varepsilon_n=\frac{1-x_n}{x_n},
$  making the change of variables
$
\tau=1-\varepsilon_n \tau^\prime,
$ and  keeping the same notation $\tau$ to the new variable, we obtain
$$
\ell_n=x_n^{-a_n-1}\varepsilon_n\int_0^{\frac{1}{\varepsilon_n}}\left(1-\varepsilon_n\tau\right)^{n-a_n-1}\tau^{1+a_n}(1+\tau)^{-a_n-1}d\tau.
$$
From the first order expansion of $x_n$,f one has the pointwise convergence
$$
\lim_{n\to+\infty} \left(1-\varepsilon_n\tau\right)^{n-a_n-1}=e^{-\kappa\tau},
$$
for any $\tau>0$. Since  the sequence $n\mapsto \left(1-\frac{\kappa}{n}\right)^n$ is increasing,  the Lebesgue Theorem leads to
$$
\lim_{n\to+\infty}\int_0^{\frac{1}{\varepsilon_n}}\left(1-\varepsilon_n\tau\right)^{n-a_n-1}\tau^{1+a_n}(1+\tau)^{-a_n-1}d\tau=\int_0^{+\infty}e^{-\kappa \tau}\frac{\tau}{1+\tau}d\tau.
$$
Therefore,  we obtain the equivalence
$
\ell_n\sim \frac{\eta}{n}.
$
Combining the previous estimates with \eqref{ASS1} we find the announced estimate.
\end{proof}
The next objective is to give the asymptotic expansion of the eigenvalues.

\begin{proposition}\label{Asym1}
Let $A$ and $B$ be such that \eqref{condX1} holds. Then, the sequence  $\left\{x_n, n\geq n_0\right\}$,   constructed in Proposition $\ref{Prop-eig}$, admits the following asymptotic behavior
$$
x_n=1-\frac{\kappa}{n}+\frac{c_\kappa}{n^2}+o\left(\frac{1}{n^2}\right),
$$
where
$$
\kappa=-2\frac{A+B}{A}\quad\hbox{and}\quad {c_\kappa=\kappa^2-2+2\int_0^{+\infty}\frac{e^{-\kappa \tau}}{(1+\tau)^2}d\tau}.
$$
\end{proposition}
\begin{proof}
Let us consider the ansatz
$$
x_n=1-\frac{\kappa}{n+1}-u_n,\quad u_n=o\left(\frac{1}{n}\right).
$$
From \eqref{special1}, we know that  $x_n$ satisfies the equation
\begin{eqnarray*}
 \zeta_n(x_n)&=&I_n^1(x_n)F(a_n,b_n;n+1;x_n)+I_n^2(x_n)F(a_n+1,b_n;n+2;x_n)\\
 &&+I_n^3(x_n)F(a_n,b_n;n+3;x_n)
 =0,
\end{eqnarray*}
with
$$
I_n^1(x_n)=\frac{n-a_n}{n+1-a_n}-\left(1-\frac{\kappa}{n+1}-u_n\right)\left(\frac{1+\kappa}{n+1}+\frac{n-1-a_n}{n+1-a_n}\right),$$
$$ I_n^2=-\frac{a_n(2x_n-1)}{(n+1)(n+1-a_n)}\quad\hbox{and}\quad I_n^3=-\frac{2x_n}{(n+1)(n+2)}.
$$
It is a simple matter to have
$$
I_n^1=-\frac{\kappa(1-\kappa)}{(n+1)^2}+u_n+o(u_n)+O\left(\frac{1}{n^3}\right),
$$
$$
I_n^2=\frac{2}{(n+1)^3}+o\left(\frac{1}{n^3}\right)
\quad \mbox{ and } \quad
I_n^3=\frac{-2}{(n+1)^2}+O\left(\frac{1}{n^3}\right),
$$
where we have used $a_n\sim -\frac{2}{n}$. By virtue of the above estimates, Lemma \ref{lemestim} and Lemma \ref{lemint}, the expansion of $u_n$ reads as
\begin{eqnarray*}
u_n= \frac{\kappa(1-\kappa)+2-2\eta}{(n+1)^2}+o\left(\frac{1}{n^2}\right)
= \frac{\kappa-\kappa^2+2-2\eta}{n^2}+o\left(\frac{1}{n^2}\right).
\end{eqnarray*}
Using
$
\frac{1}{n+1}=\frac1n-\frac{1}{n^2}+o\left(\frac{1}{n^2}\right),
$
we find
\begin{eqnarray*}
x_n= 1-\frac{\kappa}{n}+\frac{\kappa^2-2+2\eta}{n^2}+o\left(\frac{1}{n^2}\right).
\end{eqnarray*}
The final expression holds as a consequence of the following integration by parts 
\begin{eqnarray*}
\eta = 1-\kappa\int_0^{+\infty}\frac{e^{-\kappa \tau}}{1+\tau}d\tau
= \int_0^{+\infty}\frac{e^{-\kappa \tau}}{(1+\tau)^2}d\tau.
\end{eqnarray*}
\end{proof}
\subsection{Separation of  the singular and dispersion sets}
In Section \ref{boundaryeq} we have established some conditions in order to solve the boundary equation. If we want to apply Proposition \ref{propImpl}, we must verify that $\Omega$ does not lie in the singular set $\mathcal{S}_{\textnormal{sing}}$ given in \eqref{Interv2}. Moreover, from the last analysis, we have checked that the dispersion set $\mathcal{S}$, defined in \eqref{disper-dos}, contains different sets depending on the assumptions on $A$ and $B$.  We will prove the following results.

\begin{proposition}\label{propsingset}
Let $A$ and $B$ such that \eqref{condX1} holds. Denote by
$$
\Omega_n\triangleq\frac{B}{2}+\frac{A}{4 {x}_n},
$$ 
where the sequence $({x}_n)_{n\geq n_0}$ has been defined in Proposition $\ref{Prop-eig}.$ If  $n_0$ is large enough depending on $A$ and $B$, then
$$
\Omega_n\neq \widehat{\Omega}_{np}, \qquad \forall\, n\geq n_0, \ \forall p\in\N^\star,
$$
where $\widehat{\Omega}_{np}$ belongs to the set $ \mathcal{S}_{sing}$  introduced in \eqref{Interv2}. 
Moreover, there exists $\kappa_c>0$ such that for any $\kappa>\kappa_c$ we find $n_0\in\N$ such that
$$
\Omega_n\neq \widehat{\Omega}_{m}, \qquad \forall\, n,m\geq n_0.
$$
The number {$\kappa_c\in(0,2)$} is the unique  solution of the equation
$$
{\kappa_c-2\int_0^{+\infty}\frac{e^{-\kappa_c\tau}}{(1+\tau)^2} d\tau=0.}
$$
\end{proposition}
\begin{proof}
It is a simple matter to have
$$
\frac4A\left(\widehat{\Omega}_{n}-\frac{B}{2}\right)=1+\frac{\kappa}{n}+\frac{2}{n^2}+O\left(\frac{1}{n^3}\right).
$$
Setting
$\widehat{x}_n=\frac{1}{\frac4A\left(\widehat{\Omega}_{n}-\frac{B}{2}\right)},$
we obtain
\begin{equation}\label{yn}
\widehat{x}_n=1-\frac{\kappa}{n}+\frac{\kappa^2-2}{n^2}+O\left(\frac{1}{n^3}\right).
\end{equation}
Thus, condition  $\Omega_n\neq\widehat\Omega_m$  is equivalent to ${x}_n\neq \widehat{x}_m$. According to Proposition \ref{Asym1}, we have
$$
x_n=1-\frac{\kappa}{n}+\frac{c_\kappa}{n^2}+o\left(\frac{1}{n^2}\right),
$$
for $n\geq n_0$, where
$$
c_\kappa=\kappa^2-2+2\int_0^{+\infty}\frac{e^{-\kappa \tau}}{(1+\tau)^2}d\tau.
$$
This implies that ${x}_n\neq \widehat{x}_n$ for large $n$.
Moreover,
\begin{eqnarray*}
\widehat{x}_{np}=1-\frac{\kappa}{n p}+\frac{\kappa^2-2}{n^2 p^2}+O\left(\frac{1}{n^3}\right)
=1-\frac{\kappa}{n p}+O\left(\frac{1}{n^2}\right),
\end{eqnarray*}
for any $p\in\N^\star$ and  $O\left(\frac{1}{n^2}\right)$ being  uniform on $p$. Therefore,  we get that
$
\widehat{x}_{np}>x_n,$ for any $ k\geq 2,
$
with $n\geq n_0t$ and $n_0$ large enough. Consequently, we deduce that
$x_n\neq \widehat{x}_{np}$, for any $ n\geq n_0,$ and $ p\in\N^\star$.
To establish the second assertion, we will use an asymptotic expansion for $\widehat{x}_{n+1}$. Thanks to \eqref{yn} we can write
$$
\widehat{x}_{n+1}=1-\frac{\kappa}{n}+\frac{\kappa^2+\kappa-2}{n^2}+O\left(\frac{1}{n^3}\right).
$$
{Then, since $(\widehat{x}_n)_{n\geq n_0}$ is strictly increasing, to prove $\widehat{x}_n\neq {x}_n$ it suffices just to check that
$\widehat{x}_{n+1}> {x}_n$, which leads to
$$
g(\kappa)\triangleq\kappa- 2\int_0^{+\infty}\frac{e^{-\kappa \tau}}{(1+\tau)^2}d\tau>0.
$$
Since $g$ is strictly increasing on $[0,+\infty)$ and satisfies
$g(0)=-2$ and $g(2)>0$,
there exists only one solution $\kappa_c\in(0,2)$ for the equation $g(\kappa)=0$. This concludes the  proof.}
\end{proof}
The next task is to discuss the separation problem when the dispersion set is finite.

\begin{proposition}\label{Propseparation2}
Let $A>0, B\in\R$ and  $\widehat{\mathcal{S}}_{sing}$ being the set defined in \eqref{singularx}. Then the following assertions hold true:
\begin{enumerate}
\item If $B\notin[-\frac{A}{2},-\frac{A}{4}]$, then
$x_1=-\frac{A}{2B}\notin \widehat{\mathcal{S}}_{sing}.
$
\item If $B>A$, then  the sequence $m\in\left[2, \frac{2B}{A}\right]\mapsto x_m$  defined  in Proposition $\ref{Propexistcase2}$ satisfies
$$
x_m\neq\widehat{x}_{nm},\quad \forall n\geq 1.
$$
\end{enumerate}
\end{proposition}

\begin{proof}
Recall from the definition of the set $\widehat{\mathcal{S}}_{sing}$ given in  \eqref{singularx} that
$$
\frac{1}{\widehat{x}_n}=1-\frac{2(n+1)}{n(n+2)}-\frac{2B}{An},
$$
where we have used \eqref{Interv2} and \eqref{FormXX}. Notice  that when $A$ and $B$ are positive then  $n\mapsto \frac{1}{\widehat{x}_n}$ is strictly  increasing.

\medskip
\noindent
${\bf (1)}$
Let us prove that
\begin{equation}\label{Daba1}
 -\frac{2B}{A}\neq 1-\frac{2(n+1)}{n(n+2)}-\frac{2B}{An},\quad \forall n\geq1.
\end{equation}
 Note that for $n=1$ this constraint is always satisfied since  we get
$$
-\frac{1}{3}-\frac{2B}{A}\neq -\frac{2B}{A}.
$$
Thus \eqref{Daba1}  is equivalent to 
$$
\frac{n^2-2}{n^2+n-2}\neq -\frac{2B}{A}, \quad \forall  n\geq2.
$$
One can easily check that left part is strictly increasing on $[2,+\infty)$ and so
$$
 \frac12\leq \frac{n^2-2}{n^2+n-2}\leq1,\quad \forall  n\geq2.
$$
Consequently, if $-\frac{2B}{A} \notin [\frac{1}{2},1]$ then the condition \eqref{Daba1} is satisfied and this ensures the first point.

\medskip
\noindent
${\bf (2)}$ From the expression of $\mathscr{P}_n$ given in Proposition \ref{Case2Prop}, we may write  
$$
\mathscr{P}_n(x_m)=\frac{n}{n+1}\left(1-\frac{x_m}{\widehat{x}_n} \right),
$$
and  $\mathscr{P}_m(x_m)<0$, which implies
$
\frac{1}{\widehat{x}_m}<\frac{1}{x_m}.
$
By the monotonicity of $\frac{1}{\widehat{x}_n}$, it is enough to prove
\begin{equation}\label{Esw1}
\frac{1}{x_m}<\frac{1}{\widehat{x}_{2m}},
\end{equation}
in order to conclude. According to \eqref{Trd1}, we have
$
\frac{1}{x_m}<1-\frac{A+2B}{A(m+1)}.
$
To obtain \eqref{Esw1}, it is enough to establish
$$
1-\frac{A+2B}{A(m+1)}\leq \frac{1}{\widehat{x}_{2m}}=1-\frac{2m+1}{2m(m+1)}-\frac{B}{Am},
$$
which is equivalent to
$$
\frac{A+2B}{A(m+1)}\geq \frac{2m+1}{2m(m+1)}+\frac{B}{Am}.
$$
This latter one agrees with
$$
\frac{B}{A}\geq \frac{1}{2(m-1)},
$$
which holds true   since $\frac{B}{A}\geq 1$.
\end{proof}

\subsection{Transversal property}
This section is devoted to the transversality assumption concerning the fourth hypothesis of the Crandall--Rabinowitz Theorem \ref{CR}. We shall reformulate an equivalent tractable statement, where the problem reduces to check the non-vanishing of a suitable integral. However, it is slightly hard to check this property for all the eigenvalues.  We give positive results for higher frequencies using the asymptotics, which have been developed in the preceding sections for some  special  regimes on $A$ and $B$. The first result in this direction is summarized as follows.

\begin{proposition}\label{proptransv1}
Let  $A>0$ and  $B\in \R$, ${x\in (-\infty,1)}\backslash\{ \widehat{\mathcal{S}}_{\textnormal{sing}}\cup\{0, x_0\}\}$ and $n\in\mathcal{A}_x$ where all the elements involved can be found in  \eqref{FormXX} \eqref{Firstzero}-\eqref{singularx} and \eqref{AX1}. Let 
$$r e^{i\theta}\in\overline{\D}\mapsto \mathpzc{h}_n(re^{i\theta})=h_n^\star(r)\cos(n\theta) \in \textnormal{Ker }D_g\widehat{G}(\Omega,0) ,
$$
Then
$$
D_{\Omega,g}\widehat{G}(\Omega,0)\mathpzc{h}_n\notin\textnormal{Im} \, D_g\widehat{G}(\Omega,0)
$$
if and only if $h_n^\star$ satisfies 
\begin{eqnarray}\label{conditiontransv}
\int_0^1\frac{s^{n+1}F_n(xs^2)}{1-xs^2}\left[\frac{h^\star_n(s)}{2A}-A_ns^{n+2}+A_n\frac{sG_n(s)}{G_n(1)}\right]ds\neq0,
\end{eqnarray}
where
$$
A_n=-\frac{\int_0^1s^{{n+1}}h_n^\star(s)ds}{2G_n(1)},
$$
and $G_n$ is defined by \eqref{Gn}.
\end{proposition}
\begin{proof} 
Differentiating the expression of the linearized operator \eqref{linoperator}  with respect to $\Omega$ we obtain
\[
D_{\Omega,g}\widehat{G}(\Omega,0)h(re^{i\theta})=\sum_{n\geq 1}\cos(n\theta)\left[\frac{1}{2A}h_n(r)-\frac{r}{n}\Big((\partial_\Omega A_n)G_n(r)+A_n\,\partial_\Omega G_n(r)\Big)\right]\\
+\frac{1}{2A}h_0(r).
\]
Differentiating  the identity \eqref{An1} with respect to $\Omega$ yields
\[
\partial_\Omega A_n = \frac{H_n(1)}{2n(\Omega-\widehat{\Omega}_n)^2}
= -\frac{nA_n}{G_n(1)}.
\]
Similarly, we get from \eqref{Gn}, \eqref{Pn} and  the relation \eqref{FormXX} between $x$ and $\Omega$ that
\[
\partial_\Omega G_n(r)=\frac{A n}{4}r^{n+1}\partial_\Omega\left(\frac1x\right)
=n r^{n+1}.
\]
Putting together the preceding identities we find
\begin{eqnarray*}
D_{\Omega,g}\widehat{G}(\Omega,0)h(re^{i\theta})
&=&\sum_{n\geq 1}\cos(n\theta)\left[\frac{1}{2A}h_n(r)+A_n\frac{rG_n(r)}{G_n(1)}-A_nr^{n+2}\right]+\frac{1}{2A}h_0(r).
\end{eqnarray*}
Evaluating this formula at $\mathpzc{h}_n$ yields
\begin{eqnarray*}
D_{\Omega,g}\widehat{G}(\Omega,0)\mathpzc{h}_n(re^{i\theta})&=&\left[\frac{h_n^\star(r)}{2A}-A_nr^{n+2} +A_n\frac{rG_n(r)}{G_n(1)}\right]\cos(n\theta),
\end{eqnarray*}
where $A_n$ is related to $h_n^\star$ via $\eqref{An}$. Now applying Proposition \ref{Proprange} we obtain that this element  does not belong to $\textnormal{Im} D_g\widehat{G}(\Omega,0)$  if and only if the function 
$$
r\in[0,1]\mapsto d_n^\star(r)\triangleq \frac{h_n^\star(r)}{2A}-A_nr^{n+2} +A_n\frac{rG_n(r)}{G_n(1)}
$$
verifies
$$
\int_0^1\frac{s^{n+1}F_n(xs^2)}{1-xs^2}d^\star_n(s)ds\neq0,
$$
which gives the announced result.
\end{proof}
The next goal is to check the condition  \eqref{conditiontransv} for large $n$ in the regime \eqref{condX1}. We need first to rearrange the function $d_n^\star$ defined above and use the explicit expression of $\mathpzc{h}_n$ given in \eqref{hkernel}. From  \eqref{inthn} we get  
$$
A_n=-\frac{n}{8xG_n(1)}.
$$ 
Then, multiplying $d_n^\star$  by $\frac{1}{A_n}$, we obtain 
\begin{eqnarray*}
\frac{1}{A_n}\left[\frac{h^\star_n(r)}{2A}-A_nr^{n+2}+A_n\frac{rG_n(r)}{G_n(1)}\right]=\left[\frac{h_n^\star(r)}{2AA_n}-r^{n+2} +\frac{rG_n(r)}{G_n(1)}\right]
\triangleq r^n\mathcal{H}(r^2),
\end{eqnarray*}
where the function $\mathcal{H}$ takes the form 

\begin{eqnarray*}
\mathcal{H}(\ttt)\hspace{-0.3cm}&=&\hspace{-0.3cm} \frac{4xG_n(1)}{An(1-x\ttt)}\left[\frac{P_n(\ttt)}{P_n(1)}-\frac{F_n(x\ttt)}{F_n(x)}+\frac{2xF_n(x\ttt)}{P_n(1)}\int_{\ttt}^1\frac{1}{\tau^{n+1}F_n^2(x\tau)}\int_0^\tau\frac{s^{n}F_n(xs)}{1-xs}P_n(s)dsd\tau\right]\\
&&\qquad\qquad\qquad\qquad\qquad\qquad\quad-\ttt+\frac{P_n(\ttt)}{P_n(1)}.
\end{eqnarray*}
With the change of variables $s\leadsto s^2$ in the integral, condition \eqref{conditiontransv} is equivalent to
\begin{eqnarray}\label{transversalcondition2}
\int_0^1\frac{s^nF_n(xs)}{1-xs}\mathcal{H}(s)ds\neq 0.
\end{eqnarray}

\subsubsection{Regime $A>0$ and $A+B<0$}
Here, we study the transversality assumption in the regime \eqref{condX1}. Notice that the existence of infinite countable set of eigenvalues has been already established in Proposition \ref{Prop-eig}. However, due to the complex structure  of the integrand in \eqref{transversalcondition2} it appears   quite difficult to check  the non-vanishing of the integral for a given frequency. Thus, we have to overcome this difficulty  using an asymptotic behavior of the integral and checking by this way the transversality only for high frequencies. More precisely, we prove the following result.
\begin{proposition}\label{proptransv2}
Let $A$ and $B$ satisfying  \eqref{condX1} and $\{x_n\}$ the sequence  constructed in Proposition $\ref{Prop-eig}-(1)$. Then there exists $n_0\in\N$ such that 
$$
\int_0^1\frac{s^nF_n(x_ns)}{1-x_ns}\mathcal{H}(s)ds=-\frac{n}{2}\big(1+o(1)\big),\quad \forall n\geq n_0.
$$

\end{proposition}

\begin{proof}
We proceed with studying  the asymptotic behavior of the above integral for large frequencies $n$. We write  $\mathcal{H}=\mathcal{H}_1+\mathcal{H}_2+\mathcal{H}_3,$
with
\begin{eqnarray*}
\mathcal{H}_1(\ttt)&=&\frac{4x_nG_n(1)}{An(1-x_n\ttt)}\left[\frac{P_n(\ttt)}{P_n(1)}-\frac{F_n(x_n\ttt)}{F_n(x_n)}\right],\\
\mathcal{H}_2(\ttt)&=&\frac{4x_nG_n(1)}{An(1-x_n\ttt)}\frac{2x_nF_n(x_n\ttt)}{P_n(1)}\int_{\ttt}^1\frac{1}{\tau^{n+1}F_n^2(x_n\tau)}\int_0^\tau\frac{s^{n}F_n(x_ns)}{1-x_ns}P_n(s)dsd\tau,\\
\mathcal{H}_3(\ttt)&=&-\ttt+\frac{P_n(\ttt)}{P_n(1)}.
\end{eqnarray*}
Let us start with the function $\mathcal{H}_1$. Proposition \ref{Asym1} leads to
\begin{eqnarray*}
\frac{1}{x_n}&=&1+\frac{\kappa}{n}-\frac{c_\kappa-\kappa^2}{n^2}+O\left(\frac{1}{n^3}\right),
\end{eqnarray*}
which, together with the expression of $G_n(1)$ given in \eqref{Gn1}, imply
\begin{equation}\label{dim01}
G_n(1)\sim-\frac{A}{4n}\left(c_\kappa-\kappa^2+2\right).
\end{equation}
Recall that the inequality 
$
F_n(1)\leq F_n({x}_n{\ttt})\leq1
$
holds for any $\ttt\in[0,1]$, and thus  Lemma \ref{lemX1} gives  that
$$
\frac12\leq F_n(1), \quad  \forall n\geq n_0.
$$
Hence,  $\mathcal{H}_1$ can be bounded as follows
$$
 \left|\mathcal{H}_1(\ttt)\right|\le \frac{C}{n^2(1-x_n\ttt)}\left[\frac{|P_n(\ttt)|}{|P_n(1)|}+1\right],\quad \forall \ttt\in[0,1],
$$
with $C$ a constant depending in $A$ and $B$. 
From the definition of $P_n$ in \eqref{Pn} we may obtain
\begin{equation}\label{dim0}
 |P_n(\ttt)|\leq C \left[1-x_n\ttt+\frac{1}{n}\right],\quad  \forall \ttt\in[0,1]. 
\end{equation}
Moreover, plugging \eqref{dim01} into \eqref{Gn1} we deduce that
\begin{equation}\label{Tarn1}
 P_n(1)\sim -\frac{4}{An}G_n(1)
\sim \frac{c_\kappa-\kappa^2+2}{n^2} .
\end{equation}
Putting everything together one gets 
$$
  |\mathcal{H}_1(\ttt)|\leq C\left[\frac{1}{n(1-x_n\ttt)}+1\right],\quad  \forall\, \ttt\in[0,1],
$$
from which we infer that
\[
\left|\int_0^1\frac{s^nF_n(x_ns)}{1-x_ns}\mathcal{H}_1(s)ds\right| \le \frac{C}{n}\int_0^1\frac{s^n}{(1-x_ns)^2}ds+C\int_0^1\frac{s^n}{1-x_ns}ds
 \triangleq C\left(\frac{I_1}{n}+I_2\right).
\] 
To estimate the  second integral we use the change of variables $s=1-\varepsilon_n\tau$, with $\varepsilon_n=\frac{1-x_n}{x_n}$, leading to the asymptotic behavior
\begin{equation}\label{dim1}
 I_2=\frac{1}{x_n}\int_0^{\frac{1}{\varepsilon_n}}\frac{(1-\varepsilon_n\tau)^n}{1+\tau}d\tau
\sim \int_0^{\frac{n}{\kappa}}\frac{(1-\frac{\kappa \tau}{n})^n}{1+\tau}\sim \int_0^{+\infty} \frac{e^{-\kappa \tau}}{1+\tau}d\tau,
\end{equation}
where we have used the expansion of $x_n$ given by  Proposition \ref{Asym1}. 
As  to the first integral  we just have to integrate by parts  and use the previous computations,
$$
I_1=\frac{1}{x_n}\frac{1}{1-x_n}-\frac{n}{x_n}\int_0^1\frac{s^{n-1}}{1-x_ns}ds=\left(\frac{1}{\kappa}-\int_0^\infty \frac{e^{-\kappa \tau}}{1+\tau}d\tau\right)n +o(n),
$$
and consequently,
\begin{equation}\label{dim3}
\sup_{n\geq n_0}\Big|\int_0^1\frac{s^nF_n(x_ns)}{1-x_ns}\mathcal{H}_1(s)ds\Big|<+\infty.
\end{equation}
The estimate $\mathcal{H}_2$ it  is straightforward. Indeed, from \eqref{Gn1} we may write
\begin{equation}\label{dim5}
|\mathcal{H}_2(\ttt)|\leq \frac{C}{1-{x}_n\ttt}\int_\ttt^1\frac{1}{\tau^{n+1}}\int_0^\tau\frac{s^{n}}{1-{x}_ns}|P_n(s)|dsd\tau.
\end{equation}
Using once again \eqref{dim0} and Proposition \ref{Asym1}  we get
\begin{eqnarray*}
\int_0^\tau\frac{s^{n}}{1-{x}_ns}|P_n(s)|ds&\le&\frac{C}{n}\int_0^\tau\frac{s^{n}}{1-{x}_ns}ds+C\int_0^\tau{s^{n}}ds \nonumber\\
&\le&\frac{C}{n(1-x_n)}\frac{\tau^{n+1}}{n+1}+C\frac{\tau^{n+1}}{n+1}
\leq \frac{C}{n}\tau^{n+1}, \quad \forall \tau\in[0,1].
\end{eqnarray*}
Hence, we deduce that
$$
\int_\ttt^1\frac{1}{\tau^{n+1}}\int_0^\tau \frac{s^n}{1-x_ns}|P_n(s)|ds\leq C\frac{1-\ttt}{n}.
$$
Since  the function $\ttt\in[0,1]\mapsto \frac{1-\ttt}{1-x_n \ttt}$ is strictly decreasing , then we have
$$
0\le \frac{1-\ttt}{1-x_n \ttt}\leq1,\quad \forall \ttt\in[0,1].
$$
Consequently, inserting the preceding two estimates into  \eqref{dim5}, we obtain
$$
 |\mathcal{H}_2(t)|\leq \frac{C}{n},\quad \forall \ttt\in[0,1].
$$
Thus we infer
\begin{eqnarray*}
\Big|\int_0^1\frac{s^nF_n(x_ns)}{1-x_ns}\mathcal{H}_2(s)ds\Big|&\le& \frac{C}{n}\int_0^1 \frac{s^n}{1-x_ns}ds,
\end{eqnarray*}
which implies 
\begin{equation}\label{dim7}
\sup_{n\geq n_0}\Big|\int_0^1\frac{s^nF_n(x_ns)}{1-x_ns}\mathcal{H}_2(s)ds\Big|\leq \frac{C}{n},
\end{equation}
where we have used  \eqref{dim1}. It remains to estimate the integral term associated with  $\mathcal{H}_3$ which takes the form
$$
\int_0^1\frac{s^nF_n(x_ns)}{1-x_ns}\mathcal{H}_3(s)ds=-\int_0^1\frac{s^{n+1}F_n(x_ns)}{1-x_ns}ds+\frac{1}{P_n(1)}\int_0^1\frac{s^nF_n(x_ns)}{1-x_ns}P_n(s)ds.
$$
Similarly to \eqref{dim1}, one has
\begin{equation*}
\sup_{n\geq n_0}\int_0^1\frac{s^{n+1}F_n(x_ns)}{1-x_ns}ds<+\infty.
\end{equation*}
To finish we just have to deal with the second integral term. Observe from \eqref{Pn}  that $P_n$ is a monic  polynomial of degree two, and thus from Taylor formula one gets
$$
P_n(\ttt)=P_n\left(\frac{1}{x_n}\right)+\left(\ttt-\frac{1}{x_n}\right)P_n^\prime\left(\frac{1}{x_n}\right)+\left(\ttt-\frac{1}{x_n}\right)^2.
$$ 
It is easy to  check the following behaviors 
$$P_n\left(\frac{1}{x_n}\right)\sim \frac{\kappa}{n},\quad P_n'\left(\frac{1}{x_n}\right)\sim 1.$$
Hence, we obtain 
\begin{eqnarray*}
\hspace{-1cm}\int_0^1\frac{s^nF_n(x_ns)}{1-x_ns}P_n(s)ds&=&P\left(\frac{1}{x_n}\right)\int_0^1\frac{s^nF_n(x_ns)}{1-x_ns}ds-\frac{1}{x_n}P_n'\left(\frac{1}{x_n}\right)\int_0^1s^nF_n(x_ns)ds\\
&&+\frac{1}{x_n^2}\int_0^1s^nF_n(x_ns)(1-x_ns)ds.
\end{eqnarray*}
Concerning the last term we use the asymtotics of $x_n$, leading to
\[
\frac{1}{x_n^2}\int_0^1s^nF_n(x_ns)(1-x_ns)ds \le C \int_0^1s^n(1-x_ns)ds
\le C \left(\frac{1}{n+1}-\frac{x_n}{n+2}\right)
\le\frac{C}{n^2}.
\]
For the first and second terms we use the estimate  $|F_n(x_nt)-1|<C\frac{\ln n}{n}$ coming from Lemma \ref{lemestim}. Hence, we find 
$$
-\frac{1}{x_n}P_n^\prime\left(\frac{1}{x_n}\right)\int_0^1s^nF_n(x_ns)ds=-\frac{1}{n}\big(1+o(1)\big).
$$
Again from \eqref{dim1} we find
\[
\int_0^1\frac{s^nF_n(x_ns)}{1-x_ns}ds = \int_0^1\frac{s^n}{1-x_ns}ds+o(1)
 = \int_0^{+\infty} \frac{e^{-\kappa \tau}}{1+\tau}d\tau+o(1).
\]
Putting together the preceding estimates, we obtain
$$
\int_0^1\frac{s^nF_n(x_ns)}{1-x_ns}P_n(s)ds=\frac{\displaystyle{\kappa\int_0^{+\infty}\frac{e^{-\kappa\tau}}{1+\tau}d\tau-1}}{n}+o\left(\frac{1}{n}\right),
$$
and combining this estimate with \eqref{Tarn1}, we infer
$$
\int_0^1\frac{s^nF_n(x_ns)}{1-x_ns}\mathcal{H}_3(s)ds= \frac{\kappa\displaystyle{\int_0^{+\infty}}\frac{e^{-\kappa\tau}}{1+\tau}d\tau-1}{c_\kappa-\kappa^2+2}n+o(n).
$$
Using the explicit expression of  above constants  defined in Proposition \ref{Asym1}, we get that
$$
\frac{\kappa\displaystyle{\int_0^{+\infty}}\frac{e^{-\kappa\tau}}{1+\tau}d\tau-1}{c_\kappa-\kappa^2+2}=-\frac{1}{2},
$$
and therefore 
$$
\int_0^1\frac{s^nF_n(x_ns)}{1-x_ns}\mathcal{H}_3(s)ds= -\frac{n}{2}+o(n).
$$
Combining this estimate with \eqref{dim7} and \eqref{dim3}, we deduce that
$$
\int_0^1\frac{s^nF_n(x_ns)}{1-sx_n}\mathcal{H}(s)ds= -\frac{n}{2} +o(n),
$$
which achieves the proof of the announced result.
\end{proof}
\subsubsection{Regime $B>A>0$}
In this special regime there is only a finite number of eigenvalues that can be indexed by a decreasing sequence, see Proposition \ref{CorCase2}-(1).   In what follows we shall prove that the transversality assumption is always satisfied without any additional constraint on the parameters. More precisely, we prove the following result.
\begin{proposition}\label{Proptrans2}
Let $B>A>0$. Then, the transversal property \eqref{conditiontransv} holds, for every subsequence $\Big\{x_n; n\in \left[2, \mathscr{N}_{A,B}\right]\Big\}$ defined in  Proposition $\ref{Propexistcase2}$, where
\begin{equation*}
\mathscr{N}_{A,B}\triangleq\max\left(\frac{B}{A}+\frac{1}{8},\frac{2B}{A}-\frac{9}{2}\right).
\end{equation*}
\end{proposition}

\begin{proof}
Let us start with the case $n\in \left[2, \frac{B}{A}+\frac{1}{8}\right]$.  
Using the expression of $P_n$ introduced in  \eqref{Pn} one has
$$
x_nP_n(1)=x_n\left[1-\frac{A+2B}{A}\frac{n+2}{n(n+1)}\right]-\frac{n+2}{n+1}.
$$
Moreover, from the  definition of $\mathscr{P}$ seen in Proposition \ref{Case2Prop}, we get
\begin{eqnarray*}
\mathscr{P}_n(x_n)&=&\frac{n}{n+1}+x_n\left[-\frac{n}{n+2}+\frac{A+2B}{A(n+1)}\right]\\
&=&\frac{n}{n+2}\left(\frac{n+2}{n+1}+x_n\left[-1+\frac{(A+2B)(n+2)}{An(n+1)}\right]\right)
=-\frac{n}{n+2}x_n P_n(1)<0.
\end{eqnarray*}
This implies that  $P_n(1)<0$. Since $\ttt\in[0,1]\mapsto P_n(\ttt)$ is strictly increasing with $x_n<0$, we deduce that  
\begin{equation}\label{NegT}
 P_n(\ttt)<0, \quad \forall  \ttt\in[0,1].
\end{equation}
 Therefore, we get from  \eqref{Gn1} that  $G_n(1)>0$.
Let us study every term involved in \eqref{transversalcondition2} by using the decomposition $\mathcal{H}=\mathcal{H}_1+\mathcal{H}_2+\mathcal{H}_3$  of   Proposition \ref{proptransv2}. From the preceding properties, it is clear that
$
 \mathcal{H}_2(\ttt)>0,$ for $  \ttt\in(0,1).
$
In addition, we also have that $\ttt\in[0,1]\mapsto -\ttt+\frac{P_n(\ttt)}{P_n(1)}$ is strictly decreasing and thus
\[
\mathcal{H}_3(\ttt)=-\ttt+\frac{P_n(\ttt)}{P_n(1)}
\geq \mathcal{H}_3(1)=0, \quad \forall \ttt\in[0,1].
\]
Concerning $\mathcal{H}_1$ we first note that the mapping $\ttt\in[0,1]\mapsto \frac{P_n(\ttt)}{P_n(1)}-\frac{F_n(x_n\ttt)}{F_n(x_n)}$ is strictly decreasing which follows from \eqref{NegT} and the  fact that $F_n$ is decreasing and $P_n$ is increasing. Thus,
\begin{eqnarray*}
 \frac{P_n(\ttt)}{P_n(1)}-\frac{F_n(x_n\ttt)}{F_n(x_n)}&\geq&\frac{P_n(1)}{P_n(1)}-\frac{F_n(x_n)}{F_n(x_n)}=0, \quad \forall \ttt\in[0,1].
\end{eqnarray*}
Combining \eqref{NegT} with \eqref{Gn1}, we deduce 
$$
\mathcal{H}_1(\ttt)=\frac{4x_nG_n(1)}{An(1-x_n\ttt)}\left[\frac{P_n(\ttt)}{P_n(1)}-\frac{F_n(x_n\ttt)}{F_n(x_n)}\right]<0,\quad \ttt\in(0,1).
$$
We continue our analysis assuming that
\begin{eqnarray}\label{bound}
\left|\frac{4x_nG_n(1)}{An(1-x_n\ttt)}\right|\le 1,\quad \forall\, \ttt\in[0,1),
\end{eqnarray}
holds, we see how to conclude with. Since $\mathcal{H}_2$ is always positive  then $\mathcal{H}$ will be strictly positive if one can show  that 
$
 \mathcal{H}_1(\ttt)+\mathcal{H}_3(\ttt)>0,$ for any $ \ttt\in(0,1).
$
With \eqref{bound} in mind, one gets
$$
\mathcal{H}_1(\ttt)+\mathcal{H}_3(\ttt)\geq-\left[\frac{P_n(\ttt)}{P_n(1)}-\frac{F_n(x_n\ttt)}{F_n(x_n)}\right]-\ttt+\frac{P_n(\ttt)}{P_n(1)}=\frac{F_n(x_n\ttt)}{F_n(x_n)}-\ttt,\quad \forall \ttt\in[0,1].
$$
Computing the derivatives of the  function in the right-hand side term, we find
$$
\partial_\ttt \left(\frac{F_n(x_n\ttt)}{F_n(x_n)}-\ttt\right)=\frac{x_nF_n'(x_n\ttt)}{F_n(x_n)}-1,\quad \partial^2_{\ttt\ttt}\left(\frac{F_n(x_n\ttt)}{F_n(x_n)}-\ttt\right)=\frac{x_n^2F_n''(x_n\ttt)}{F_n(x_n)}<0.
$$
The latter fact implies that the  first derivative is decreasing, and  thus
\[
\partial_\ttt \left(\frac{F_n(x_n\ttt)}{F_n(x_n)}-\ttt\right)\leq \frac{x_nF_n'(0)}{F_n(x_n)}-1
\leq \frac{-2x_n}{(n+1)F_n(x_n)}-1
\le \frac{2}{n+1}-1<0,\quad \forall \ttt\in[0,1],
\]
where we have used  Lemma \ref{Lem2}$-(1)$.
Therefore, we conclude that the mapping 
$ \ttt\in[0,1]\mapsto \frac{F_n(x_n\ttt)}{F_n(x_n)}-t$ decreases and, since it  vanishes at $\ttt=1$, we get
$$
\mathcal{H}_1(\ttt)+\mathcal{H}_3(\ttt)\geq\frac{F_n(x_n\ttt)}{F_n(x_n)}-\ttt>0,\quad \forall \ttt\in[0,1).
$$
This implies that $\mathcal{H}(\ttt)>0$ for any $\ttt\in[0,1)$ and hence the transversality assumption \eqref{transversalcondition2} is satisfied.
Let us now  turn to the proof of \eqref{bound} and observe that from \eqref{Gn1}
\begin{eqnarray}\label{TM1}
\nonumber\left|\frac{4x_nG_n(1)}{An(1-x_nt)}\right|&=&\frac{1}{1-x_nt}\frac{n+1}{n+2}x_nP_n(1)\\
&\leq& \frac{n+1}{n+2}\left\{x_n\left[1-\frac{A+2B}{A}\frac{n+2}{n(n+1)}\right]-\frac{n+2}{n+1}\right\}.
\end{eqnarray}
Using $x_\star$, defined in Proposition \ref{Case2Prop}, and the fact that $\frac{A+2B}{A}\frac{n+2}{n(n+1)}-1>0$, we obtain
\begin{eqnarray*}
\left|\frac{4x_nG_n(1)}{An(1-x_nt)}\right|&\le&\frac{n+1}{n+2}\left\{x_\star\left[1-\frac{A+2B}{A}\frac{n+2}{n(n+1)}\right]-\frac{n+2}{n+1}\right\}\\
&\leq&\frac{n+1}{n+2}\left\{\frac{2n+1}{2(n+1)}\frac{\frac{A+2B}{A}\frac{n+2}{n(n+1)}-1}{\frac{A+2B}{A(n+1)}-\frac{n+1}{n+2}}-\frac{n+2}{n+1}\right\}.
\end{eqnarray*}
Consequently \eqref{bound} is satisfied provided that
$$
\frac{\frac{A+2B}{A}\frac{n+2}{n(n+1)}-1}{\frac{A+2B}{A(n+1)}-\frac{n+1}{n+2}}\leq 4\frac{n+2}{2n+1}.
$$
Since $\frac{A+2B}{A(n+1)}-\frac{n+1}{n+2}>0$, then the preceding inequality is true if and only if
$$
\frac{A+2B}{A(n+1)}\geq\frac{2n^2+3n}{2n^2+3n-2}.
$$
It is easy to check that the sequence $n\geq2\mapsto \frac{2n^2+3n}{2n^2+3n-2}$ is decreasing, and then the foregoing inequality is satisfied for any $n\geq2$  if
$$
\frac{A+2B}{A(n+1)}\geq\frac76.
$$
Fom the assumption $n+1\leq\frac{B}{A}+\frac98$, the above inequality holds if
$$
\frac{B}{A}\geq \frac{5}{16},
$$
which follows from the condition $B>A.$

Let us now move on the case  $n\in \left[2, \frac{2B}{A}-\frac{9}{2}\right]$. Note that we cannot use $x_\star$ coming from Proposition \ref{Case2Prop} since we do not know if $x_n\in(-1,0)$. In fact Proposition \ref{Propexistcase2} gives us that $x_n<-1$, for $n\geq \frac{B}{A}+1$. Hence, we should slightly modify the arguments used to \eqref{bound}. From \eqref{TM1}, we deduce that  \eqref{bound} is satisfied if
$$
x_n\geq \frac{2\frac{n+2}{n+1}}{1-\frac{A+2B}{A}\frac{n+2}{n(n+1)}}.
$$
By \eqref{Trd1}, it is enough to prove that 
$$
\frac{1}{1-\frac{A+2B}{A(n+1)}}\geq \frac{2\frac{n+2}{n+1}}{1-\frac{A+2B}{A}\frac{n+2}{n(n+1)}}.
$$
Straightforward computations give that the last inequality is equivalent to
$$
\frac{A+2B}{A(n+1)}\geq \frac{(n+3)n}{(n+2)(n-1)}.
$$
Therefore, if  $n$ satisfies $2\leq n\leq \frac{2B}{A}-\frac{9}{2}$, then we have
$$
\frac{n+1+\frac{9}{2}}{n+1}\leq \frac{A+2B}{A(n+1)}.
$$
Hence, one can check that
$$
\frac{n+1+\frac{9}{2}}{n+1}\geq \frac{(n+3)n}{(n+2)(n-1)},
$$
for any  $n\geq2$, and thus  \eqref{bound} is verified.
\end{proof}
\subsubsection{One-fold case}
The main objective is to make a complete study in the case $n=1$. It is very particular because $F_1$ is explicit according to Remark \ref{n1} and, therefore, we can get a compact formula for the integral of \eqref{transversalcondition2}.   Our main result reads as follows.
\begin{proposition}\label{one-foldX1}
Let $n=1$ and  $x=-\frac{A}{2B}$, then we have the formula
$$
\int_0^1\frac{sF_1(xs)}{1-xs}\mathcal{H}(s)ds=\frac{x-1}{2x}.
$$
In particular the transversal assumption  \eqref{conditiontransv} is satisfied  if and only if
$
x\notin\{0,1\}.
$
\end{proposition}

\begin{proof}
Note that from \eqref{Pn}-\eqref{Gn1} one has
$$
P_1(\ttt)=\ttt^2-\frac{3}{2x}\ttt-\frac{3}{2}\left[1-\frac{1}{x}\right], \quad P_n(1)=-\frac{1}{2}, \quad G_1(1)=\frac{A}{12}.
$$
Moreover,  we get $F_1(\ttt)=1-\ttt$ using   Remark \ref{n1}, and thus
\begin{eqnarray}\label{TIO1}
\nonumber\mathcal{H}(\ttt)&=&{\frac13}\frac{x}{(1-x\ttt)}\left[\frac{P_1(\ttt)}{P_1(1)}-\frac{F_1(x\ttt)}{F_1(x)}+\frac{2xF_1(x\ttt)}{P_1(1)}\hspace{-0.1cm}\int_{\ttt}^1\frac{1}{\tau^{2}F_1^2(x\tau)}\int_0^\tau\frac{sF_1(xs)}{1-xs}P_1(s)dsd\tau\right]\\
\nonumber&&-\ttt+\frac{P_1(\ttt)}{P_1(1)}\\
&\triangleq& {\frac13}\frac{x}{(1-x\ttt)}\widehat{\mathcal{H}}(\ttt)-\ttt+\frac{P_1(\ttt)}{P_1(1)}.
\end{eqnarray}
From straightforward computations we deduce 
\begin{eqnarray*}
\int_{\ttt}^1\frac{1}{\tau^{2}F_1^2(x\tau)}\int_0^\tau\frac{sF_1(xs)}{1-xs}P_1(s)dsd\tau&=&\int_{\ttt}^1\frac{\frac{\tau^2}{4}-\frac{1}{2x} \tau-\frac34(1-\frac{1}{x})}{(1-\tau x)^2}d\tau
\end{eqnarray*}
Denoting by $\varphi(\tau)=\frac{\tau^2}{4}-\frac{1}{2x} \tau-\frac34(1-\frac{1}{x})$ and integrating by parts, we get
\begin{eqnarray*}
\int_{\ttt}^1\frac{\varphi(\tau)}{(1-\tau x)^2}d\tau=\frac1x\left[\frac{\varphi(1)}{1-x}-\frac{\varphi(\ttt)}{1-\ttt x}\right]+\frac{1-\ttt}{2x^2}
=\frac1x\left[\frac{1-2x}{4x(1-x)}+\frac{1-\ttt}{2x}-\frac{\varphi(\ttt)}{1-\ttt x}\right].
\end{eqnarray*}
Therefore, after standard computations,  we get the simplified  formula
$$
\widehat{\mathcal{H}}(\ttt)=\frac{3(1-\ttt x)(\ttt-1)}{x}.
$$
Inserting this into \eqref{TIO1}, we find
\[
\mathcal{H}(\ttt)=-1+\frac{P_1(\ttt)}{P_1(1)}
=-2\ttt^2+\frac{3}{x} \ttt+2-\frac3x.
\]
Plugging it into the integral of \eqref{transversalcondition2}, it yields
$$
\int_0^1s\mathcal{H}(s)ds=\frac{x-1}{2x}.
$$
\end{proof}

\section{Existence of non--radial time--dependent rotating solutions}\label{SecResult}

At this stage we are able to give the full statement of our main result, by using the analysis of the previous sections. In order to apply the Crandall--Rabinowitz Theorem, let us introduce the m--fold symmetries in our spaces. 
\begin{eqnarray*}
\mathscr{C}^{k,\alpha}_{s,m}({\D})&\triangleq&\left\{g\in \mathscr{C}^{k,\alpha}_{s}({\D})\,:\quad g(e^{i\frac{2\pi}{m}}z)=g(z),\quad \forall z\in\D\right\},\\
\mathscr{C}^{k,\alpha}_{a,m}(\T)&\triangleq&\left\{\rho\in \mathscr{C}^{k,\alpha}_{a}(\T)\,: \quad\rho(e^{i\frac{2\pi}{m}}w)=\rho(w),\quad \forall w\in\T\right\},\\
\mathscr{H}\mathscr{C}^{k,\alpha}_m({\D})&\triangleq&\left\{\phi\in \mathscr{H}\mathscr{C}^{k,\alpha}({\D})\,: \quad\phi(e^{i\frac{2\pi}{m}}z)=e^{i\frac{2\pi}{m}}\phi(z),\quad \forall z\in\D\right\}.
\end{eqnarray*}
Note that the functions $g\in \mathscr{C}^{k,\alpha}_{s,m}({\D}),\, \rho\in \mathscr{C}^{k,\alpha}_{a,m}(\T)$ and $\phi\in \mathscr{H}\mathscr{C}^{k,\alpha}_m({\D})$ admit the following representation:
$$
g(re^{i\theta})=\sum_{n\geq 0} g_{nm}(r)\cos(nm\theta),\quad \rho(e^{i\theta})=\sum_{n\geq 0}\rho_n\sin(nm\theta)\quad\textnormal{and}\quad \phi(z)=z\sum_{n\geq 1} z^{nm},
$$
where $z\in\D,\, r\in[0,1]$ and $\theta\in[0,2\pi].$ With these spaces, the functional $F$ defined in \eqref{F}, concerning the boundary equation, is also well--defined:

\begin{proposition}\label{PropReqX2}
Let $\E\in (0,1)$, then 
 $$
 F: \R\times B_{\mathscr{C}_{s,m}^{1,\alpha}}(0,\E)\times B_{\mathscr{H}\mathscr{C}^{2,\alpha}_m}(0,\E)\mapsto \mathscr{C}^{1,\alpha}_{a,m}({\T})$$ is well--defined and of class $\mathscr{C}^1$,  where the balls were defined in \eqref{BAL}.
\end{proposition}
\begin{proof}
Thanks to Proposition \ref{PropReqX}, it remains to prove that $F(\Omega,g,\phi)$ satisfies $F(\Omega,g,\phi)(e^{i\frac{2\pi}{m}}w)=F(w)$, with $w\in\T$:
\begin{eqnarray*}
F(\Omega,g,\phi)(e^{i\frac{2\pi}{m}}w)\hspace{-0.2cm}&=&\hspace{-0.2cm}\textnormal{Im}\left[ \left(\Omega\,\overline{\Phi(e^{i\frac{2\pi}{m}}w)}-\frac{1}{2\pi}\int_{\D}\frac{f(y)|\Phi^\prime(y)|^2}{{\Phi(e^{i\frac{2\pi}{m}}w)}-{\Phi(y)}} dA(y)\right){\Phi^\prime(e^{i\frac{2\pi}{m}}w)}{e^{i\frac{2\pi}{m}}w}\right]\\
\hspace{-0.2cm}&=&\hspace{-0.2cm}\textnormal{Im}\left[ \left(\Omega\,e^{-i\frac{2\pi}{m}}\overline{\Phi(w)}-\frac{1}{2\pi}\int_{\D}\frac{f(e^{i\frac{2\pi}{m}}y)|\Phi^\prime(e^{i\frac{2\pi}{m}}y)|^2}{{\Phi(e^{i\frac{2\pi}{m}}w)}-{\Phi(e^{i\frac{2\pi}{m}}y)}} dA(y)\right){\Phi^\prime(w)}{e^{i\frac{2\pi}{m}}w}\right]\\
\hspace{-0.2cm}&=&\hspace{-0.2cm}\textnormal{Im}\left[ \left(\Omega\,e^{-i\frac{2\pi}{m}}\overline{\Phi(w)}-\frac{e^{-i\frac{2\pi}{m}}}{2\pi}\int_{\D}\frac{f(y)|\Phi^\prime(y)|^2 }{{\Phi(w)}-{\Phi(y)}}dA(y)\right){\Phi^\prime(w)}{e^{i\frac{2\pi}{m}}w}\right]\\
\hspace{-0.2cm}&=&\hspace{-0.2cm}F(\Omega,g,\phi)(w),
\end{eqnarray*}
where we have used that $\phi(e^{i\frac{2\pi}{m}}z)=e^{i\frac{2\pi}{m}}\phi(z)$, $\phi'(e^{i\frac{2\pi}{m}}z)=\phi'(z)$ and  $g(e^{i\frac{2\pi}{m}}z)=g(z).$ 
\end{proof}

We must define the singular set \eqref{Interv2} once we have introduced the symmetry in the spaces. Fixing $f_0$ as a quadratic profile \eqref{QuadrP}, the singular set \eqref{Interv2} becomes
\begin{eqnarray*}
\mathcal{S}^m_{\textnormal{sing}}\triangleq\left\{\widehat{\Omega}_{mn}\triangleq\frac{A}{4}+\frac{B}{2}-\frac{A(nm+1)}{2nm(nm+2)}-\frac{B}{2nm}\, ,\quad  n\in\N^\star\cup\{+\infty\}  \right\}.
\end{eqnarray*}

For the density equation defined in \eqref{densityEq} and the new spaces we obtain the following result.

\begin{proposition}\label{wddensityeq}
Let $I$ be an open interval with  $\overline{I}\subset\R\backslash \mathcal{S}^m_{\textnormal{sing}}$. Then, there exists $\varepsilon>0$ such that
$$\widehat{G}:I\times B_{\mathscr{C}_{s,m}^{1,\alpha}(\D)}(0,\E)\to \mathscr{C}_{s,m}^{1,\alpha}(\D)
$$ 
is well--defined and of class $\mathscr{C}^1$, where  $\widehat{G}$ is defined in \eqref{densityEq} and   $ B_{\mathscr{C}_{s}^{1,\alpha}(\D)}(0,\E)$  in \eqref{BAL}.
\end{proposition}

\begin{proof}
Similarly to the previous result, from Proposition \ref{Gwelldefined} we just have to check that $\widehat{G}(\Omega,g)(e^{i\frac{2\pi}{m}}z)=\widehat{G}(\Omega,g)(z)$, for $z\in\D$. From Proposition \ref{propImpl}, there exists $\varepsilon>0$ such that the conformal map $\phi$ is given by $(\Omega,g)$, and lies in $B_{\mathscr{H}\mathscr{C}^{2,\alpha}_m}(0,\varepsilon)$. Then, using  $\widehat{G}(\Omega,g)=G(\Omega,g,\phi(\Omega,g))$, we have
\begin{eqnarray*}
G(\Omega,g, \phi)(e^{i\frac{2\pi}{m}}z)\hspace{-0.3cm}&=&\hspace{-0.2cm}\frac{4\Omega-B}{8A}f(e^{i\frac{2\pi}{m}}z)-\frac{f(e^{i\frac{2\pi}{m}}z)^2}{16A}\\
&&+\frac{1}{2\pi}\int_{\D}\log|\Phi(e^{i\frac{2\pi}{m}}z)-\Phi(y)| f(y)|\Phi^\prime(y)|^2 dA(y)-\frac{\Omega|\Phi(e^{i\frac{2\pi}{m}}z)|^2}{2}-{\lambda}\\
&=&\hspace{-0.3cm}\frac{4\Omega-B}{8A}f(z)-\frac{f(z)^2}{16A}\\
&&\hspace{-0,3cm}+\frac{1}{2\pi}\int_{\D}\log|\Phi(e^{i\frac{2\pi}{m}}z)-\Phi(e^{i\frac{2\pi}{m}}y)| f(e^{i\frac{2\pi}{m}}y)|\Phi^\prime(e^{i\frac{2\pi}{m}}y)|^2 dA(y)\\
&&-\frac{\Omega|e^{i\frac{2\pi}{m}}\Phi(z)|^2}{2}-{\lambda}\\
&=& G(\Omega,g,\phi)(z),
\end{eqnarray*}
where we have used the properties of functions $g$ and $\phi$.
\end{proof}

We have now all the tools we need to prove the first three points of Theorem \ref{TP1}, which can be detailed as follows:

\begin{theorem}\label{TP2}
Let $A>0$, $B\in\R$ and $f_0$ a quadratic profile \eqref{QuadrP}. Then the following results hold true.
\begin{enumerate}
\item If  $A+B<0$, then there is $m_0\in \N$ (depending only on $A$ and $B$)  such that for any $m\geq m_0$ there exists
\begin{itemize}
\item $\Omega_m=\frac{A+2B}{4}+\frac{A\kappa}{4m}+\frac{A}{4}\frac{\kappa^2-c_\kappa}{m^2}+o\left(\frac{1}{m^2}\right), $ 
\item$V$ a neighborhood of $(\Omega_m,0,0)$ in $\R\times \mathscr{C}^{1,\alpha}_{s,m}(\D)\times \mathscr{H}\mathscr{C}_m^{2,\alpha}(\D)$,
\item a continuous curve $\xi\in(-a,a)\mapsto(\Omega_\xi,f_\xi, \phi_\xi)\in V$, 
\end{itemize}
such that \eqref{sol},
\begin{eqnarray*}
\omega_0=(f\circ\Phi^{-1}){\bf{1}}_{\Phi(\D)},\quad f=f_0+f_\xi,\quad \Phi=\textnormal{Id}+\phi_\xi,
\end{eqnarray*}
defines a curve of non radial solutions of Euler equations that rotates at constant angular velocity $\Omega_\xi$. 
The constants $\kappa$ and $c_\kappa$ are defined in \eqref{Asym1}.

\item If $B>A>0$, then for any integer $m\in \left[1, \frac{2B}{A}-\frac92\right]$ or $m\in \left[1, \frac{B}{A}+\frac18\right]$  there exists 
\begin{itemize}
\item $0\leq\Omega_m<\frac{B}{2}$,  
\item$V$ a neighborhood of $(\Omega_m,0,0)$ in $\R\times \mathscr{C}^{1,\alpha}_{s,m}(\D)\times \mathscr{H}\mathscr{C}_m^{2,\alpha}(\D)$,
\item a continuous curve $\xi\in(-a,a)\mapsto(\Omega_\xi,f_\xi, \phi_\xi)\in V$, 
\end{itemize}
such that \eqref{sol} defines a curve of non radial solutions of Euler equations. {However, there is no bifurcation with any symmetry  \mbox{$m\geq \frac{2B}{A}+2$}.}

\item If $B>0$ or $B\leq -\frac{A}{1+\epsilon}$, for some $0,0581<\epsilon<1$, then there exists
\begin{itemize}
\item$V$ a neighborhood of $(0,0,0)$ in $\R\times \mathscr{C}^{1,\alpha}_{s}(\D)\times \mathscr{H}\mathscr{C}^{2,\alpha}(\D)$,
\item a continuous curve $\xi\in(-a,a)\mapsto(\Omega_\xi,f_\xi, \phi_\xi)\in V$, 
\end{itemize}
such that \eqref{sol} defines a curve of one-fold non radial solutions of Euler equations. 
\item If $- \frac{A}{2}\le B\le 0$, then there is no bifurcation with any  symmetry $m\geq1$. However, in the case that $0<B< \frac{A}{4}$, there is no bifurcation with any  symmetry $m\geq2$.
\end{enumerate}
\end{theorem}

\begin{proof} {\bf (1)} Let us prove the first assertion in the case  $A+B<0$.
We will implement the Crandall--Rabinowitz Theorem \eqref{CR} to $\widehat{G}$, defined in \eqref{densityEq}. First, we must concrete the domain of $\Omega$. From Proposition \ref{Prop-eig}, there exist $n_{0,1}$ and a unique solution $x_m\in(0,1)$ of $\zeta_m(x)=0$ for any $m\geq n_{0,1}$. Then, we fix $\Omega_m=\frac{A}{4x_m}+\frac{B}{2}>\frac{A}{4}+\frac{B}{2}$. Note that one can take $n_{0,1}$ large enough to avoid $\Omega_0$ defined in \eqref{Firstzero}. Moreover, Proposition \ref{propsingset} gives us that $\Omega_m\neq \hat{\Omega}_{mn}$, for $n\in \N$, with $\widehat{\Omega}_{mn}\in \mathcal{S}^m_{\textnormal{sing}}$. Therefore, let $I$ be an interval with $\Omega_m\in I$ and
$$
\overline{I}\cap \mathcal{S}^m_{\textnormal{sing}}=\varnothing, \quad \Omega_0\notin \overline{I}.
$$ 
By virtue of Proposition \ref{propImpl}, we know that there exists $\varepsilon>0$ and  a $\mathscr{C}^1$ function  
$\mathcal{N}: I\times B_{\mathscr{C}_{ s,m}^{1,\alpha}}(0,\varepsilon)\longrightarrow B_{\mathscr{H}\mathscr{C}_m^{2,\alpha}}(0,\varepsilon),
$
such that  
$$
 F(\Omega,g,\phi)=0\Longleftrightarrow \phi=\mathcal{N}(\Omega,g)
$$
holds, for any $ (\Omega,g,\phi)\in I\times B_{\mathscr{C}_{s,m}^{1,\alpha}}(0,\varepsilon)\times B_{\mathscr{H}\mathscr{C}_m^{2,\alpha}}(0,\varepsilon)$. Hence, the conformal map is defined through the density  for that $\varepsilon$. We define the density equation,
$$
\widehat{G}: I\times B_{\mathscr{C}_{s,m}^{1,\alpha}}(0,\varepsilon) \rightarrow \mathscr{C}_{s,m}^{1,\alpha}(\D),
$$
with the expression given in \eqref{densityEq}. Thanks to Proposition \ref{wddensityeq}, the function $\widehat{G}$ is well-defined in these spaces and is $\mathscr{C}^1$ with respect to $(\Omega,g)$. It remains to check the spectral properties of the Crandall--Rabinowitz Theorem. Using Proposition \ref{statkernel}, we know that the dimension of the kernel of the linearized operator $D_g\widehat{G}(\Omega,0)$ is given by the number of elements of the set $\mathcal{A}_x$ defined in \eqref{AX1}.
Note that we have introduced the symmetry $m$ in our spaces, and consequently in all the functions involved in the linearized operator. Take $n=1$. We know that $x_m$ is the unique root $x_m$ of $\zeta_m(x)$, and the sequence $n\in[n_0,+\infty[\mapsto x_n$ is strictly increasing. Since we  fixed $n=1$, we get that
$
\mathcal{A}_{x_m}=\{m\}.
$
As $\Omega_m\notin\mathcal{S}^m_{\textnormal{sing}}$, we have that the kernel is one dimensional. Since $D_g\widehat{G}(\Omega,0)$ is a Fredholm operator of zero index, we have that the codimension of the range is also one. About the transversality condition, note that using Proposition \ref{proptransv1} and Proposition \ref{proptransv2} we have that there exists $n_{0,2}$ such that if $m>n_{0,2}$ the transversality condition is satisfied. Taking $n_0=\hbox{max}\{n_{0,1},n_{0,2}\}$, the Crandall-Rabinowitz Theorem can be applied to $\widehat{G}$ obtaining $V$ a neighborhood of $(\Omega_m,f_0,\hbox{Id})$ in $\R\times \mathscr{C}^{1,\alpha}_{s,m}(\D)\times \mathscr{H}\mathscr{C}_m^{2,\alpha}(\D)$,
and a continuous curve $\xi\in(-a,a)\mapsto(\Omega_\xi,f_\xi, \phi_\xi)\in V$, 
solutions of $\widehat{G}(\Omega,g)=0$. The conformal map $\phi_\xi$ is the one associated to $f_\xi$ via Proposition \ref{propImpl}. 
Moreover, thanks to $\Omega_\xi \neq \Omega_0$, Proposition \ref{radialfunctions} gives us that $f_\xi$ can not be radial because of $f_\xi\neq f_0$.
Furthermore, by Proposition \ref{Asym1} we know the asymtotics of $x_m$ obtaining
$$
\frac{1}{x_m}=1+\frac{\kappa}{m}+\frac{\kappa^2-c_\kappa}{m^2}+o\left(\frac{1}{m^2}\right),
$$
where $\kappa$ and $c_\kappa$ are defined in \eqref{Asym1}.  Using the relation between $x_m$ and $\Omega_m$ in \eqref{FormXX}, we get  
\begin{equation}
\frac{A+2B}{4} < \Omega_m=\frac{A+2B}{4}+\frac{A\kappa}{4m}+\frac{A}{4}\frac{\kappa^2-c_\kappa}{m^2}+o\left(\frac{1}{m^2}\right).
\end{equation}
 Therefore, we obtain that 
 \begin{eqnarray}\label{sol}
\omega_0=(f\circ\Phi^{-1}){\bf{1}}_{\Phi(\D)},\quad f=f_0+f_\xi,\quad \Phi=\hbox{\rm{Id} }+\phi_\xi,
\end{eqnarray}
  defines a non radial solution of Euler equation, which rotates at constant angular velocity $\Omega_\xi$.

\medskip
\noindent {\bf (2)}
Now, we are concerned with the existence of m--fold non radial solutions of the type \eqref{sol} in the case  $B>A>0$,  for any integer $m\in \left[1, \frac{2B}{A}-\frac92\right]$ or $m\in \left[1, \frac{B}{A}+\frac18\right]$. In this part of the theorem we also prove that  there is no bifurcation with the symmetry $m$, for any $m\geq \frac{2B}{A}+2.$
As in Assertion {\bf (1)}, we check that the Crandall--Rabinowitz Theorem can be applied.  From Proposition \ref{Propexistcase2} and Proposition \ref{Prop-eig}, there is a unique solution $x_m\in(-\infty,1)$ of $\zeta_m(x)=0$. In fact, $x_m<0$. Then, we fix $\Omega_m=\frac{A}{4x_m}+\frac{B}{2}$. Note that by \eqref{Trd1} and Proposition \ref{Case2Prop} we get the bounds for $\Omega_m$. Moreover, Proposition \ref{Propseparation2} gives that  ${\Omega}_{m}\notin \mathcal{S}^m_{\textnormal{sing}}$. Then, let $I$ be an interval such that $\Omega_m\in I$ and
$
\overline{I}\cap \mathcal{S}^m_{\textnormal{sing}}=\varnothing.
$
 
 Using again Proposition \ref{propImpl} and Proposition \ref{wddensityeq}, we get that $
\widehat{G}: I\times B_{\mathscr{C}_{s,m}^{1,\alpha}}(0,\varepsilon) \rightarrow \mathscr{C}_{s,m}^{1,\alpha}(\D),
$
is well-defined and $\mathscr{C}^1$ with respect to $(\Omega,g)$. 

About the spectral properties, we have stated in the previous proof that the dimension of the kernel of the linearized operator is given by the roots of $\zeta_{nm}$.  Taking $n=1$, we know that $x_m$ is the unique root of $\zeta_m(x)$. By Corollary \ref{CorCase2} we get that $\zeta_{nm}(x_m)\neq 0$, for any $n\geq 2$. Hence
$
\mathcal{A}_{x_m}=\{m\}.
$
Due to $\Omega_m\notin\mathcal{S}^m_{\textnormal{sing}}$, we have that the kernel is one dimensional. 
 $D_g\widehat{G}(\Omega,0)$ is a Fredholm operator of zero index, then we have that the codimension of the range is also one. About the transversal condition, note that using Proposition \ref{Proptrans2}, we have that the transversal condition is satisfied. Similarly to the previous proof, the Crandall-Rabinowitz Theorem can be applied to $\widehat{G}$ obtaining a curve $\xi\in(-a,a)\mapsto (\Omega_\xi, f_\xi)$ solutions of $\widehat{G}(\Omega,g)=0$. 
 Moreover, thanks to $\Omega_\xi \neq \Omega_0\in(0,1)$, Proposition \ref{radialfunctions} gives that $f_\xi$ can not be radial since $f_\xi\neq f_0$. 

First,  note that  $\Omega\notin\left[\frac{B}{2}, \frac{B}{2}+\frac{A}{4}\right]$ is equivalent to $x<1$. By Proposition \ref{Prop-eig} and Proposition \ref{Propexistcase2} we get that $\zeta_n$ has not solutions in $(-\infty,1]$, for $m\geq \frac{2B}{A}+2$, and then there is no bifurcation with that symmetry. { In the opposite case, $x>1$, there is no bifurcation according to \mbox{Theorem \ref{singthm}.}}

\medskip
\noindent {\bf (3)} 
Here, we are concerning with the case $A,B>0$ or $B\leq -\frac{A}{1+\epsilon}$ for some $\epsilon\in(0,1)$, with $-\frac{A}{2B}\neq x_0$, where $x_0$ is defined through \eqref{Firstzero}. We work as in {\bf(1)}-{\bf (2)} checking the hypothesis of Crandall--Rabinowitz Theorem. Fixing $\Omega_1=0$ agrees with $x_1=-\frac{A}{2B}$, where we use \eqref{FormXX}.  Proposition \ref{Propseparation2} allows us to have that $x_1\notin \widehat{\mathcal{S}}_{\textnormal{sing}}$. Then, we can take an interval $I$ such that $0\in I$, and 
$$
\overline{I}\cap \mathcal{S}_{\textnormal{sing}}=\varnothing, \quad \Omega_0\notin \overline{I}.
$$
Again, Proposition \ref{propImpl} and Proposition \ref{Gwelldefined} imply that
$$
\widehat{G}: I\times B_{\mathscr{C}_{s}^{1,\alpha}}(0,\varepsilon) \rightarrow \mathscr{C}_{s}^{1,\alpha}(\D),
$$
is well-defined and is $\mathscr{C}^1$ in $(\Omega,g)$.

We must check the spectral properties. Due to the assumptions on $A$ and $B$, we get that $x_1\leq 1$. By Proposition \ref{Prop-eig}, Proposition \ref{Case2Prop} and Proposition \ref{Propn1intersec} we have that $x_1\neq x_n$ if there exists $x_n\in(-\infty,1)$ solution of $\zeta_n$. Note that such $\epsilon$ comes from the Proposition \ref{Propn1intersec}. Hence, by Corollary \ref{CorCase2}, we obtain that the kernel of $D_g \widehat{G}(0,0)$ is one dimensional, and is generated by \eqref{hkernel}, for $n=1$. Moreover, Proposition \ref{compactop} implies that $D_g \widehat{G}(0,0)$ is a Fredholm operator of zero index, and then the codimension of the range is one. The transversal property is verified by virtue of Proposition \ref{one-foldX1}. Then, Crandall--Rabinowitz Theorem can be applied obtaining the announced result. Note that the bifurcated solutions are not radial due to Proposition \ref{radialfunctions}.

\medskip
\noindent {\bf (4)} 
The first assertion concerning the non bifurcation result comes from Proposition \ref{Prop-eig} and Proposition \ref{Propexistcase2} due to the fact that $\zeta_m$ has not solutions in $(-\infty,1]$, for $m\geq 2$. Moreover, by  Corollary \ref{CorCaseM2} and Theorem \ref{singthm} we get that there is no bifurcation for $m=1$ since the only possibility agrees with $\Omega=0$, which satisfies \eqref{CondOmega}.  { Finally, the bifurcation with $x>1$ is forbidden due again to Theorem \ref{singthm}.} 

The last assertion follows from Corollary \ref{CorCaseM2} and Theorem \ref{singthm}.
  \end{proof}

\section{Dynamical system and orbital analysis}\label{Secorbits}
In this section we wish to investigate the particle trajectories inside the support of the rotating vortices that we have constructed in Theorem \ref{TP1}. We will show that in the frame of these V-states the trajectories are organized through concentric periodic orbits  around the origin.  This allows to provide an equivalent  reformulation of the density equation \eqref{densityeq} via the study of the associated dynamical system. It is worth pointing out that some of the material developed in this section about periodic trajectories and the regularity of the period is partially known in the literature and for the convenience of the reader we will provide the complete proofs.

Assuming that \eqref{rotatingsol} is a solution of the Euler equations,  the level sets of $\Psi(x)-\Omega \frac{|x|^2}{2}$, where $\Psi$ is the stream function associated to \eqref{rotatingsol}, are given by the collection of the  particle trajectories,
\begin{eqnarray*}
\partial_t{\varphi}(t,x)&=&(v(\varphi(t,x))-\Omega \varphi(t,x)^\perp)
=\nabla^{\perp} \left(\Psi-\Omega\frac{|\cdot|^2}{2}\right)({\varphi}(t,x)),\\
{\varphi}(0,x)&=&x\in\Phi(\overline{\D}).
\end{eqnarray*}
In the same way we have translated the problem to the unit disc $\D$ using the conformal map $\Phi$ via the vector field $W(\Omega,f,\Phi)$ in \eqref{W},  we analyze the analogue in the level set context. 
We define the flow associated to $W$ as
\begin{eqnarray} \label{flow}
{\partial_t \psi(t,z)}={W(\Omega,f,\Phi)(\psi(t,z))},\quad 
{\psi}(0,z)=z\in\overline{\D}.
\end{eqnarray}
Since $v(x)-\Omega x^\perp$ is divergence free, via Lemma \ref{div}, we obtain that $W(\Omega,f,\Phi)$ is incompressible, and then the last system is  also Hamiltonian. In the following result, we highlight the relation between $\varphi$ and $\psi$.

\begin{lemma}
The following identity 
$$
{\varphi}(\eta_z(t),\Phi(z))=\Phi({\psi}(t,z)), \quad \forall z\in\overline{\D},
$$
holds, where
$$
\eta_z'(t)=|\Phi'(\Phi^{-1}(\varphi(\eta_z(t),\Phi(z)))|^2,\quad \eta_z(0)=0.
$$
\end{lemma}

\begin{proof}
Let us check that $Y(t,z)=\Phi^{-1}(\varphi(t,\Phi(z)))$ verifies a similar equation as $\psi(t,z)$ sets,
\begin{eqnarray*}
\partial_t Y(t,z)&=&(\Phi^{-1})'(\Phi(Y(t,z))\partial_t \varphi(t,\Phi(z))
=\frac{\overline{\Phi'(Y(t,z))}}{|\Phi'(Y(t,z))|^2}(v(\Phi(Y(t,z)))-\Omega \Phi(Y(t,z))^\perp)\\
&=&\frac{W(\Omega,f,\Phi)(Y(t,z))}{|\Phi'(Y(t,z))|^2}.
\end{eqnarray*}
Now, we rescale the time through the function $\eta_z$, and  $Y(\eta_z(t),z)$ satisfies,
\begin{eqnarray*}
\partial_t Y(\eta_z(t),z)=\eta_z'(t) (\partial_t Y)(\eta_z(t),z))
=\eta_z'(t)\frac{W(\Omega,f,\Phi)(Y(\eta(t),z))}{|\Phi'(Y(\eta(t),z))|^2}
=W(\Omega,f,\Phi)(Y(\eta(t),z)).
\end{eqnarray*}
Since $Y(\eta_z(0),z)=Y(0,z)=\Phi^{-1}(\varphi(0,\Phi(z))=z$, we obtain the announced result.
\end{proof}

The next  task  is to connect the solutions constructed in Theorem \ref{TP1} with the orbits of the associated dynamical system through the following result:

\begin{theorem}\label{mainth2}
Let $m\geq 1$ and $\xi\in(-a,a)\mapsto(\Omega_\xi,f_\xi, \phi_\xi)$ be one of the  solutions constructed in Theorem $\ref{TP1}.$ The flow $\psi$ associated to $W(\Omega_\xi,f_\xi,\Phi_\xi)$, defined in \eqref{flow}, 
verifies the following properties:
\begin{enumerate}
\item  $\psi\in \mathscr{C}^1(\R, \mathscr{C}^{1,\alpha}(\D))$.
\item The trajectory $t\mapsto \psi(t,z)$ is $T_z$ periodic, located inside the unit disc and invariant by the dihedral group $D_m$. Moreover, if $m\geq4$ then  the period map $z\in \overline{\D}\mapsto T_z$ belongs to $\mathscr{C}^{1,\alpha}(\D)$.

\item The family $(\psi(t))_{t\in\R}$ generates a group of diffeomorphisms of the closed unit disc.
\end{enumerate}
\end{theorem}
The  proof will be given in  Subsection \ref{proofmainth2}.

\subsection{Periodic orbits}
Here we explore sufficient conditions for Hamiltonian vector fields defined on the unit disc whose orbits are all periodic.  More precisely, we shall establish the following result. 

\begin{proposition}\label{thm1}
Let $W:\overline{\mathbb{D}}\to \C$ be a vector field in $\mathscr{C}^1(\overline{\mathbb{D}})$ satisfying the following conditions:
\begin{enumerate}
\item[i)] It has divergence free.
\item[ii)] It is tangential to the boundary $\mathbb{T}$, i.e.
$
\textnormal{Re}\left(W(z)\overline{z}\right)=0, \quad \forall\, z\in\T.
$
\item[iii)] It vanishes only at the origin.
\end{enumerate}
Then, we have
\begin{enumerate}
\item All the trajectories are periodic orbits located inside the unit disc, enclosing a simply connected domain containing the origin.
\item The family $(\psi(t))_{t\in\R}$  generates a group of diffeomorphisms of  the  closed  unit  disc.
\item  If $W$  is antisymmetric with respect to the real axis, that is,
\begin{eqnarray}\label{symmetrW}
\overline{W(z)}=-W(\overline{z}), \quad \forall z\in \overline{\mathbb{D}}.
\end{eqnarray}
then the orbits are symmetric with respect to the real axis.

 \item If $W$ is invariant by a rotation centered at zero  with angle $\theta_0$, i.e.
$
 W(e^{i\theta_0} z)=e^{i\theta_0} W(z), \, \forall z\in \overline{\mathbb{D}},
$
then all the orbits are invariant by this rotation. 

\end{enumerate}
\end{proposition}
\begin{proof}
{\bf{(1)}}
Let $\psi$ be the solution associated to the flux $W$
\begin{equation}\label{Eq02}
\left\lbrace
\begin{array}{l}
\partial_t\psi(t,z)=W(\psi(t,z)),\\
\psi(0,z)=z\in \overline{\mathbb{D}}.
\end{array}
\right.
\end{equation}
From the Cauchy--Lipschitz Theorem we know that  the trajectory $t\mapsto \psi(t,z)$ is defined in a maximal time interval $(-T_\star,T^\star),$ with $T_\star, T^\star>0$, for  each  $z\in\overline{\mathbb{D}}$. Note that when  $z$ belongs to the boundary, then the second condition listed above implies necessarily that its trajectory  does not leave  the boundary. Since the vector field does not vanish anywhere  on the boundary according to the third condition,  the trajectory will cover all the unit disc. As the equation is autonomous, this ensures that  the unit disc is  a periodic orbit. \\
By condition $\textnormal{i})$ we get that \eqref{Eq02} is a Hamiltonian system. Let $H$ be the Hamiltonian function such that $W=\nabla^\perp H$. Since $H$ is $C^1$ in $\overline{\D}$ and constant on the boundary $\T$ according to the assumption $\textnormal{ii})$, then from $\textnormal{iii})$ the origin  corresponds to an extremum point.\\
Now, taking $|z|<1$, the solution is globally well--posed in time, that is, $T_\star=T^\star=+\infty$. This follows easily from the fact that different orbits never  intersect and consequently we should get 
$$
 |\psi(t,z)|<1, \quad \forall t\in  (-T_\star,T^\star),
$$
meaning that the solution is bounded and does not touch the boundary so it is globally defined according to a classical blow--up criterion.

We will check that all the orbits are periodic inside the unit disc. This follows from some straightforward considerations on the level sets of the Hamiltonian $H$. Indeed, the $\omega-$limit of a point $z\neq0$ cannot contain the origin because it is the only critical point and  the level sets of $H$ around this point are periodic orbits. Thus we deduce from Poincaré--Bendixon Theorem that  the $\omega-$limit of $z$  will be a periodic orbit.  As the level sets cannot be limit cycles then we find that the trajectory of $z$ coincides with the periodic orbit.  \\
{\bf{(2)}} This follows from classical results on autonomous differential equation. In fact, we know that  the flow  
$\psi: \R\times \overline{\D}\longrightarrow \overline{\D}
$
 is well--defined and $\mathscr{C}^1$. For any $t\in\R$, it realizes   a bijection with
$
\psi^{-1}(t,\cdot)=\psi(-t,\cdot),
$ 
and $(\psi(t))_{t\in\R}$ generates a group of diffeomorphisms on $\overline{\D}$.
\\
{\bf{(3)}}
The symmetry of the orbits with respect to the real axis is a consequence of the following elementary fact.  Given $z\in \D$ and $t\mapsto \psi(t,z)$ its trajectory, then it follows that $t\mapsto \overline{\psi(-t,z)}$ is also a  solution of the same Cauchy problem and by uniqueness we find the identity
\begin{equation*}
\psi(t,z)=\overline{\psi(-t,z)},\quad \forall t\in\R.
\end{equation*}
\\
{\bf{(4)}} Assume that $W$ is invariant by the rotation $R_{\theta_0}$  centered at zero and with angle $\theta_0$. Let $z\in {\mathbb{D}}$, then we shall first check the identity
\begin{equation}\label{Eq5}
 e^{i\theta_0}\psi(t,z)=\psi(t, e^{i\theta_0}z), \quad \forall t\in\R.
\end{equation}
To do that, it suffices to verify that both functions satisfy the same differential equation with the same initial data, and thus the identity follows from the uniqueness of the Cauchy problem. Note that  \eqref{Eq5} means that the rotation of a trajectory is also a trajectory. Denote by $D_{z_0}$ and $e^{i\theta_0}D_{z_0}$ the domains delimited by the curves $t\mapsto \psi(t,z_0)$ and  $t\mapsto e^{i\theta_0}\psi(t,z_0)$, respectively. Then, it is a classical result  that those domains are necessary simply connected and they contain the origin according to $1)$. Since different  trajectories never intersect, then we have only two possibilities: $D_{z_0}\subset e^{i\theta_0}D_{z_0}$ or the converse.  Since the rotation is a Lebesgue preserving  measure,  then $D_{z_0}= e^{i\theta_0}D_{z_0}$, which implies that the periodic orbit $t\mapsto \psi(t,z_0)$ is invariant by the rotation $R_{\theta_0}.$
 \end{proof}
\subsection{Reformulation with the trajectory map}
In this section we discuss a new representation of solutions to the equations of the type
\begin{equation}\label{Eq1}
W(z)\cdot\nabla f(z)=0, \quad \forall z\in \overline{\mathbb{D}},
\end{equation}
with  $W$ a vector field as in Proposition \ref{thm1} and 
 $f:\overline{\mathbb{D}}\to \R$ a $\mathscr{C}^1$ function.
\begin{proposition}\label{propequiv}
Let $W:\overline{\mathbb{D}}\to \C$ be a vector field satisfying the assumptions $\textnormal{i)}, \textnormal{ii)}$ and $\textnormal{iii)}$ of Proposition \ref{thm1}. Then,  \eqref{Eq1} is equivalent to the formulation
\begin{equation}\label{Eq3}
 f(z)-\frac{1}{T_z}\int_0^{T_z}f(\psi(\tau,z))d\tau=0, \quad \forall \, z\in \overline{\mathbb{D}},
\end{equation}
with $T_z$ being the period of the trajectory $t\mapsto\psi(t,z)$.
\end{proposition}
\begin{proof}
We  first check that \eqref{Eq1} is equivalent to
\begin{equation}\label{Eq2}
  f(\psi(t,z))=f(z), \quad \forall t\in\R,\, \forall |z|\leq1.
\end{equation}
Although for simplicity we can assume that the equivalence is done pointwise, where we need $f \in \mathscr{C}^1$, the equivalence is perfectly valid in a weak sense without nothing more than assuming H\"older regularity  on $f$. Indeed, if $f$ is a  $\mathscr{C}^1$ function satisfying  \eqref{Eq2}, then by differentiating in time we get
$$
 (W\cdot\nabla f)(\psi(t,z))=0, \quad \forall t\in\R,\, \forall |z|\leq1.
$$
According to Proposition \ref{thm1}, we have that \eqref{Eq1} is satisfied everywhere in the closed unit disc, for any $t$, $\psi(t,\overline{\mathbb{D}})=\overline{\mathbb{D}}$.
Conversely, if $f$ is a $\mathscr{C}^1$ solution to  \eqref{Eq1}, then differentiating with respect to $t$ the function $t\mapsto f(\psi(t,z))$ we get
$$
\frac{d}{dt}f(\psi(t,z))=(W\cdot\nabla f)(\psi(t,z))=0.
$$
Therefore,  we have \eqref{Eq2}.
Now, we will verify that \eqref{Eq3} is in fact equivalent to \eqref{Eq2}. The implication $\eqref{Eq2}\Longrightarrow \eqref{Eq3}$ is elementary. So it remains to check the converse. From \eqref{Eq3} one has
\begin{equation}\label{Eq4}
 f(\psi(t,z))-\frac{1}{T_{\psi(t,z)}}\int_0^{T_{\psi(t,z)}}f(\psi(\tau,\psi(t,z))d\tau=0.
\end{equation}
Since the vector field is autonomous, then all the points located at the same orbit generate periodic trajectories with the same  period, and of course with the same orbit. Therefore, we have
$
T_{\psi(t,z)}=T_z.
$
Using 
$
\psi(\tau,\psi(t,z))=\psi(t+\tau,z),
$
a change of variables, and the $T_z-$periodicity of $\tau\mapsto f(\psi(\tau,z))$, then we deduce
\begin{eqnarray*}
\frac{1}{T_{\psi(t,z)}}\int_0^{T_{\psi(t,z)}}f(\psi(\tau,\psi(t,z))d\tau=\frac{1}{T_z}\int_{t}^{t+T_z}f(\psi(\tau,z))d\tau
=\frac{1}{T_z}\int_{0}^{T_z}f(\psi(\tau,z))d\tau
=f(z).
\end{eqnarray*}
Combining this with \eqref{Eq4},  we get  \eqref{Eq2}. This completes the proof.
\end{proof}

\subsection{Persistence of the symmetry}
We shall consider a vector field $W$ satisfying  the assumptions of  Proposition \ref{thm1} and \eqref{symmetrW} and let $\psi$ be its  associated  flow. We define the operator $f\mapsto Sf$ by 
$$
 Sf(z)=f(z)-\frac{1}{T_z}\int_0^{T_z}f(\psi(\tau,z))d\tau, \quad \forall z\in \overline{\mathbb{D}}.
$$
We shall prove that $f$ and $Sf$ share the same planar group of invariance in the following sense.

\begin{proposition}
Let $f: \overline{\mathbb{D}}\mapsto \R$ be a smooth function. The following assertions hold true:
\begin{enumerate}
\item If $f$ is invariant by reflection with respect to the real axis, then  $Sf$ is invariant too. This means that
$$
 f(\overline{z})=f(z), \quad \forall z\in \overline{\mathbb{D}} \Longrightarrow  Sf(\overline z)=Sf(z), \quad \forall z\in \overline{\mathbb{D}}.
$$
\item If $W$ and $f$ are invariant  by  the rotation $R_{\theta_0}$  centered at zero  with  angle  $\theta_0\in\R$, then $Sf$ commutes with the same rotation. This means that
$$
 f(e^{i\theta_0}{z})=  f(z), \quad \forall z\in \overline{\mathbb{D}}\Longrightarrow  (Sf)(e^{i\theta_0}{z})=Sf(z), \quad \forall z\in \overline{\mathbb{D}}.
$$
\end{enumerate}

\end{proposition}
\begin{proof}
${\bf(1)}$ Let $z\in \overline{\mathbb{D}}$, it is a simple matter to check that
$
\psi(t,\overline{z})=\overline{\psi(-t,z)},
$
which implies
$
T_{\overline{z}}=T_z,
$
and then 
\begin{eqnarray*}
Sf(\overline{z})=f(z)-\frac{1}{T_z}\int_0^{T_z}f\left(\overline{\psi(-\tau,z)}\right)d\tau
=f(z)-\frac{1}{T_z}\int_0^{T_z}f\left({\psi(-\tau,z)}\right)d\tau
=Sf(z).
\end{eqnarray*}

\medskip
\noindent
${\bf(2)}$ According to Proposition \eqref{thm1} and the fact that the vector-field $W$ is invariant by the rotation $R_{\theta_0}$,  then  we have that the orbits are symmetric with respect to this rotation and 
$$
T_{e^{i\theta_0}{z}}=T_z\quad \hbox{and}\quad \psi(t,e^{i\theta_0}{z})=e^{i\theta_0}\psi(t,z),
$$
where we have used \eqref{Eq5}, which implies that
\begin{eqnarray*}
Sf(e^{i\theta_0}{z})=f(z)-\frac{1}{T_z}\int_0^{T_z}f\left(e^{i\theta_0}\psi(\tau,z)\right)d\tau
=f(z)-\frac{1}{T_z}\int_0^{T_z}f\left({\psi(\tau,z)}\right)d\tau
=Sf(z).
\end{eqnarray*}
This concludes the proof.
\end{proof}

 \subsection{Analysis of the regularity}

Next, we are interested in studying the  regularity of the the flow map  \eqref{Eq02} and the period map. The following result  is classical, see for instance  \cite{Hartmann}.
\begin{proposition}\label{propo2}
Let $\alpha\in(0,1)$, $W:\overline{\mathbb{D}}\mapsto \R^2$  be a vector-field  in $\mathscr{C}^{1,\alpha}({\mathbb{D}})$ satisfying the condition $\textnormal{ii})$ of Proposition $\ref{thm1}$ and $\psi:\R\times \overline{\mathbb{D}}\mapsto \overline{\mathbb{D}}$ its flow map. Then $\psi\in \mathscr{C}^1(\R, \mathscr{C}^{1,\alpha}({\mathbb{D}}))$ and there exists $C>0$ such that
$$
 \|\psi^{\pm 1}(t)\|_{\mathscr{C}^{1,\alpha}({\mathbb{D}})}\leq e^{C\|W\|_{\mathscr{C}^{1,\alpha}({\mathbb{D}}}|t|}\left(1+\|W\|_{\mathscr{C}^{1,\alpha}({\mathbb{D}})}|t|\right),\quad \forall t\in\R,
$$
holds.
\end{proposition}

Now we intend  to study the regularity of the function $z=(x,y)\in \overline{\mathbb{D}}\mapsto T_z$. This is a classical subject in dynamical systems and several results are obtained in this direction for smooth Hamiltonians. Notice that in the most studies in the literature  the regularity is measured with respect to the energy and not with respect to the positions as we propose to do here.  

Let $z\in \overline{\mathbb{D}}$ be a  given non equilibrium point, the orbit $t\mapsto\psi(t,z)$ is periodic and $T_z$ is  the first strictly  positive time such that 
\begin{equation}
\label{Period1}
\psi(T_z,z)-z=0.
\end{equation}
This is an implicit function equation, from which we expect to deduce some regularity properties. Our result reads as follows.

\begin{proposition}\label{prop5}
Let $W$ be a vector field in $\mathscr{C}^1( \overline{\mathbb{D}})$,  satisfying the assumptions  of \mbox{Proposition $\ref{thm1}$} and \eqref{symmetrW} and such that
$$
 W(z)=i zU(z), \quad \forall \, z\in \overline{\mathbb{D}},
$$
with 
\begin{equation}\label{Degen1}
 \textnormal{Re}\{U(z)\}\neq0, \quad \forall\, z\in \overline{\mathbb{D}}.
\end{equation} Then the following assertions hold true:
\begin{enumerate}
\item The map $z\in \overline{\mathbb{D}}\mapsto T_z$ is continuous and verifies the upper bound
\begin{equation}\label{Ass1}
 0<T_z\le \frac{2\pi}{\inf_{z\in \overline{\mathbb{D}}}{|\textnormal{Re}U(z)|}},\quad  \forall z\in \overline{\mathbb{D}}.
\end{equation}

\item If  in addition $U\in \mathscr{C}^{1,\alpha}({\mathbb{D}})$, then $z \mapsto T_z$  belongs to $\mathscr{C}^{1,\alpha}({\mathbb{D}})$. 
\end{enumerate}
\end{proposition}
\begin{remark}
Since the origin is an equilibrium point for the dynamical system, then its trajectory is periodic with any period. However, as we will see in the proof, there is a minimal strictly positive period denoted by $T_z$, for any curves passing through a non vanishing \mbox{point $z$.} The mapping $z\in \overline{\D}\backslash\{0\}\mapsto T_z$ is not only well--defined, but it  can be extended continuously to zero. Thus we shall make the following convention
$$
T_0\equiv\lim_{z\to 0} T_z.
$$
\end{remark}
\begin{remark}
The upper bound in \eqref{Ass1} is ``almost optimal'' for radial profiles, where 
$
U(z)=U_0(|z|) \in\R,
$
and explicit computations yield 
$$
T_z=\frac{2\pi}{\left|U_0(|z|)\right|}.
$$
\end{remark}

\begin{proof}
${\bf{(1)}}$ We shall describe the trajectory parametrization using polar coordinates. Firstly, we may write for $z=r e^{i\theta}$
$$
W(z)=\left[W^r(r,\theta)+iW^\theta(r,\theta)\right] e^{i\theta},
$$
with
$$
W^\theta(r,\theta)= r \textnormal{Re}(U(r e^{i\theta}))\quad \hbox{and}\quad W^r(r,\theta)=- r \textnormal{Im}(U(r e^{i\theta})).
$$
Given $0<|z|\leq1 $, we look for  a polar parametrization of the trajectory passing through $z$,
$$
\psi(t,z)=r(t) e^{i\theta(t)},\, r(0)=|z|, \,\theta(0)=\textnormal{Arg}(z).
$$
Inserting into \eqref{Eq02} we obtain the system
\begin{eqnarray}\label{Sys1}
\dot{r}(t)&=&-r(t)  \textnormal{Im}\left\{U\left(r(t) e^{i\theta(t)}\right)\right\}
\nonumber \triangleq P(r(t),\theta(t))\\
\nonumber\dot{\theta}(t)&=& \textnormal{Re}\left\{U\left(r(t) e^{i\theta(t)}\right)\right\}
\nonumber \triangleq Q(r(t),\theta(t)).
\end{eqnarray}
From   the assumption  \eqref{symmetrW} we find
$$
 \overline{U(z)}=U(\overline{z}), \quad \forall\, z\in  \overline{\mathbb{D}},
$$
which implies in turn that
$$
 P(r,-\theta)=-P(r,\theta)\quad {\hbox{and}}\quad Q(r,-\theta)=Q(r,\theta), \quad \forall r\in[0,1],\forall \theta\in\R.
$$
Thus, we have the Fourier expansions
$$
P(r,\theta)=\sum_{n\in\N^\star}P_n(r)\sin(n\theta)\quad\hbox{and}\quad Q(r,\theta)=\sum_{n\in\N}Q_n(r)\cos(n\theta).
$$
Denoting by $T_n$ and $U_n$ the classical Chebyshev polynomials that satisfy the identities
\begin{eqnarray*}
\cos(n\theta))&=&T_n(\cos\theta)\\
\sin(n\theta)&=&\sin(\theta) U_{n-1}(\cos\theta),
\end{eqnarray*}
we obtain
\begin{eqnarray*}
 P(r,\theta)&=&\sin\theta \sum_{n\in\N^\star}P_n(r)U_{n-1}(\cos\theta)
 \equiv \sin\theta \,\,F_1(r,\cos \theta).\\
Q(r,\theta)&=&\sum_{n\in\N}Q_n(r)T_n(\cos\theta)
\equiv F_2(r,\cos \theta).
\end{eqnarray*}
Coming back to  \eqref{Sys1}, we get
\begin{align}\label{Eqz44}
\begin{split}
\dot{r}(t)&=\sin\theta(t) \,\,F_1(r(t),\cos \theta(t))\\
\dot{\theta}(t)&= F_2(r(t),\cos \theta(t)).
\end{split}
\end{align}
We look for  solutions in the form
$$
r(t)=f_{z}(\cos(\theta(t)),\quad \hbox{with}\quad  f_{z}:[-1,1]\to\R,
$$
and then $f $ satisfies the differential equation
\begin{equation}\label{Eqz1}
f_{z}^\prime(s)=-\left(\frac{F_1}{F_2}\right)\left(f_{z}(s),s\right), f(\cos\theta)=|z|.
\end{equation}
Note that  the  preceding fraction is well--defined since the assumption \eqref{Degen1} is equivalent to
$$
 F_2(r,\cos \theta)\neq0, \quad \forall r\in[0,1], \theta\in\R.
$$
Theorem \ref{thm1}-$ii)$ agrees with
$$
 F_1(0,s)=F_1(1,s)=0, \quad \forall |s| \leq 1,
$$
which implies that the system \eqref{Eqz1} admits a unique solution $f_{z}:[-1,1]\to\R_+$ such that
$$
0\leq f_{z}(s)\leq1, \quad \forall \, s\in [-1,1].
$$
Hence, integrating the second equation of \eqref{Eqz44} we find after a change of variable
$$
\int_{\theta_0}^{\theta(t)}\frac{1}{F_2\left(f_z(\cos s), \cos s\right)}ds=t,
$$
and, therefore, the following formula for the period is obtained
\begin{eqnarray}
\label{form1}
 T_z=\Big|\int_{\theta_0}^{\theta_0+2\pi}\frac{1}{F_2\left(f_z(\cos s), \cos s\right)}ds\Big|
=\int_{0}^{2\pi}\frac{1}{\big|F_2\left(f_z(\cos s), \cos s\right)\big|}ds.
\end{eqnarray}
This gives the bound of the period stated in \eqref{Ass1}. The continuity $z\mapsto T_z$ follows from the same property of $z\mapsto f_z$, which can be derived from the continuous dependence with respect to  the initial conditions. 

\medskip
\noindent
${\bf{(2)}}$
Now, we will study the regularity of the period in $\mathscr{C}^{1,\alpha}$. Note that  $\eqref{form1}$ involves the function $f_z$ which is not smooth enough because  the initial condition $z\mapsto f_z(\cos\theta)$ is only Lipschitz. So it seems  quite complicate to follow the regularity in $\mathscr{C}^{1,\alpha}$ from that formula. The alternative way  is to study the regularity of the period using the implicit equation \eqref{Period1}. Thus, differentiating this equation
with respect to $x$ and $y$  we obtain
\begin{eqnarray*}
(\partial_xT_z)\partial_t\psi(T_z,x,y)+\partial_x\psi(T_z,x,y)-1&=&0,\\
(\partial_yT_z)\partial_t\psi(T_z,x,y)+\partial_y\psi(T_z,x,y)-i&=&0.
\end{eqnarray*}
From the flow equation and the periodicity condition we get
$$
\partial_t\psi(T_z,z)=W(z),
$$
which implies
\begin{eqnarray}\label{Def1}
\left\{ \begin{array}{lll}
(\partial_xT_z)W(z)+\partial_x\psi(T_z,x,y)-1&=&0,\\
(\partial_yT_z)W(z)+\partial_y\psi(T_z,x,y)-i&=&0.
\end{array} \right.
\end{eqnarray}
Due to the assumption on $W$, the flow equation can be written as
$$
\partial_t\psi(t,z)=i\psi(t,z) U(\psi(t,z)), \quad\psi(0,z)=z,
$$
which can be integrated, obtaining
$$
\psi(t,z)=z \, e^{i\int_0^tU(\psi(\tau,z))d\tau}.
$$
By differentiating this identity with respect to $x$, it yields
\begin{eqnarray*}
\partial_x\psi(t,z)= e^{i\int_0^tU(\psi(\tau,z))d\tau}\left[1+iz\int_0^t\partial_x\left\{U(\psi(\tau,z))\right\}d\tau\right].
\end{eqnarray*}
Since $\psi(T_z,z)=z$, we have
$$
e^{i\int_0^{T_z}U(\psi(\tau,z))d\tau}=1,
$$
and thus
\begin{eqnarray*}
\partial_x\psi(T_z,z)= 1+iz\int_0^{T_z}\partial_x\left\{U(\psi(\tau,z))\right\}d\tau.
\end{eqnarray*}
Similarly, we find
\begin{eqnarray*}
\partial_y\psi(T_z,z)= i+iz\int_0^{T_z}\partial_y\left\{U(\psi(\tau,z))\right\}d\tau.
\end{eqnarray*}
Combining these identities with \eqref{Def1}, we obtain
\begin{eqnarray*}
\left\{ \begin{array}{lll}
(\partial_xT_z)W(z)&=&-iz\int_0^{T_z}\partial_x\left\{U(\psi(\tau,z))\right\}d\tau,\\
(\partial_yT_z)W(z)&=&-iz\int_0^{T_z}\partial_y\left\{U(\psi(\tau,z))\right\}d\tau,
\end{array} \right.
\end{eqnarray*}
which, using the structure of $W$, reads as
\begin{eqnarray*}
\left\{ \begin{array}{lll}
(\partial_xT_z)U(z)&=&-i\int_0^{T_z}\partial_x\left\{U(\psi(\tau,z))\right\}d\tau,\\
(\partial_yT_z)U(z)&=&-i\int_0^{T_z}\partial_y\left\{U(\psi(\tau,z))\right\}d\tau.
\end{array} \right.
\end{eqnarray*}
Now, notice that from  Theorem \ref{thm1}-$iv)$, the vector field  $W$  vanishes only at zero and since $U(0)\neq 0$, we find 
$$
 U(z)\neq0, \quad \forall z\in \overline{\mathbb{D}}.
$$
This implies that $z\in \overline{\mathbb{D}}\mapsto \frac{1}{U(z)}$ is well--defined and belongs to $\mathscr{C}^{1,\alpha}({\mathbb{D}})$. Therefore, we can write
\begin{equation}\label{smooth1}
\nabla_z T_z=-\frac{i}{U(z)}\int_0^{T_z}\nabla_z\left\{U(\psi(\tau,z))\right\}d\tau,
\end{equation}
where we have used the notation $\nabla_z=(\partial_x,\partial_y)$. According to Proposition \ref{propo2} and classical composition laws, we obtain
$$
\tau\mapsto  \nabla_z\left\{U(\psi(\tau,\cdot))\right\}\in \mathscr{C}(\R; \mathscr{C}^{0,\alpha}({\mathbb{D}})).
$$
Since $z\mapsto T_z$ is continuous, then  we find by composition that
$$
\varphi:z\in \overline{\mathbb{D}}\mapsto \int_0^{T_z}\nabla_z\left\{U(\psi(\tau,z))\right\}d\tau\in \mathscr{C}(\overline{\mathbb{D}}).
$$
Combining this information with \eqref{smooth1}, we deduce that
$z\mapsto T_z\in \mathscr{C}^1(\overline{\mathbb{D}}).
$
Hence, we find in turn that
$
\varphi \in   \mathscr{C}^{0,\alpha}({\mathbb{D}})
$
by composition. Using \eqref{smooth1} again, it follows that
$
z\mapsto \nabla_z T_z\in  \mathscr{C}^{0,\alpha}({\mathbb{D}}).
$
Thus, $z\mapsto T_z\in \mathscr{C}^{1,\alpha}({\mathbb{D}}).$ This achieves the proof.
\end{proof}

\subsection{Application to the nonlinear problem}\label{proofmainth2}
We intend in this section to prove  \mbox{Theorem \ref{mainth2}.} Let us point out that  from Proposition \ref{propImpl}, the nonlinear vector field  $W(\Omega,f,\Phi)$ is chosen in order to  be   tangent to the boundary everywhere. We will see
that not only this assumption  but all the  assumptions of  Proposition \ref{thm1} are satisfied  if $f$ is chosen close  to a suitable radial profile.

\begin{lemma}\label{lemm11}
Let $g\in \mathscr{C}_{s,m}^{1,\alpha}({\D})$ and $\phi\in \mathscr{H}\mathscr{C}_{m}^{2,\alpha}({\D})$, then $W(\Omega,f,\Phi)\in \mathscr{C}^{1,\alpha}({\D})$ and satisfies the symmetry properties \eqref{symmetrW} and
$$
W(\Omega,f,\Phi)(e^{i\frac{2\pi}{m}}z)=e^{i\frac{2\pi}{m}}W(\Omega,f,\Phi)(z).
$$ 
Moreover, if $m\geq4$ then 
\begin{equation}\label{T1}
W(\Omega,f,\Phi)(z)= iz U(z),
\end{equation}
with $U\in \mathscr{C}^{1,\alpha}({\D})$
\end{lemma}
\begin{proof}

Using Proposition \ref{PropReqX} and Proposition \ref{PropReqX2}, then the regularity and the symmetry properties of $W(\Omega,f,\Phi)$ are verified.
Let us now check  \eqref{T1}. Firstly, since $\Phi(0)=0$ and $\phi\in \mathscr{H}\mathscr{C}^{2,\alpha}({\D})$, we have
$$
\Phi(z)=z \Phi_1(z), \quad \Phi_1\in \mathscr{C}^{1,\alpha}({\D}).
$$
In addition, $\overline{\Phi^\prime}\in \mathscr{C}^{1,\alpha}({\D})$, and thus we find
$$
i\Omega\,\Phi(z)\,\overline{\Phi^\prime}(z)=iz U_1(z)\quad \textnormal{with}\quad U_1\in \mathscr{C}^{1,\alpha}({\D}).
$$
Now, to get \eqref{T1} it is enough to check that
$$
I(f,\Phi)(z)=z U_2(z),\quad \textnormal{with}\quad U_2\in \mathscr{C}^{1,\alpha}({\D}),
$$
where
$$
I(f,\Phi)(z)\triangleq\frac{1}{2\pi}\int_{\D}\frac{f(y)}{\overline{\Phi(z)}-\overline{\Phi(y)}}|\Phi^\prime(y)|^2 dA(y).
$$
We  look for the first order Taylor expansion around the origin of $I(f,\Phi)(z)$. Using the m-fold symmetry of $f$ and $\Phi$, it is clear that
\begin{eqnarray*}
I(f,\Phi)(0)=- \int_{\D}\frac{f(y)}{\overline{\Phi(y)}}|\Phi^\prime(y)|^2 dA(y)
=- \int_{\D}\frac{f(e^{i\frac{2\pi}{m}}y)}{\overline{\Phi(e^{i\frac{2\pi}{m}}y)}}|\Phi^\prime(e^{i\frac{2\pi}{m}}y)|^2 dA(y)
=e^{i\frac{2\pi}{m}}I(f,\Phi)(0),
\end{eqnarray*}
which leads to  
\begin{equation}\label{Tes1}
I(f,\Phi)(0)=0,
\end{equation}
for $m\geq2$. This implies that one can always write $I(f,\Phi)(z)=z U_2(z)$, but $U_2$ is only bounded:
\begin{eqnarray*}
\|U_2\|_{L^\infty(\overline{\D})} \leq \|\nabla_zI_2\|_{L^\infty({\D})}
\leq \|I\|_{\mathscr{C}^{1,\alpha}(\D)}.
\end{eqnarray*}
We shall see how the extra symmetry helps to get more regularity for $U_2$. According to 
Taylor expansion one gets
$$
I(f,\Phi)(z)=a z+b \overline{z}+c z^2+d \overline{z}^2+ e |z|^2+\textnormal{l.o.t.}\quad\textnormal{with}\quad a,b,c,d,e\in\C.
$$
Using the reflection invariance with respect to the real axis of $f$ and $\Phi$, we obtain
$$
\overline{I(f,\Phi)(z)}=I(\overline{z}),
$$
which implies that $a,b,c,d,e\in\R$. 
Now, the rotation invariance leads to
$$
I(f,\Phi)(e^{i\frac{2\pi}{m}}z)=e^{i\frac{2\pi}{m}}I(z).
$$
Then, we obtain
$$
e=0,\quad b(e^{i\frac{4\pi}{m}}-1)=0, \quad c(e^{i\frac{2\pi}{m}}-1)=0,\quad\textnormal{and}\quad d(e^{i\frac{6\pi}{m}}-1)=0.
$$
This implies that $b=c=d=0$, whenever $m\geq4$. Thus, we have
$$
I(f,\Phi)(z)=a z+h(z)
$$
with $h\in \mathscr{C}^{2,\alpha}({\D})$, and 
$$
h(0)=0,\, \nabla_zh(0)=0,\quad\hbox{and}\quad \nabla_z^2h(0)=0.
$$
From this we claim that
$$
h(z)=zk(z),\quad\hbox{with}\quad k\in \mathscr{C}^{1,\alpha}({\D}),
$$
which concludes the proof.
\end{proof}
Now we are in  a position to prove Theorem $\ref{mainth2}$.
\begin{proof}[Proof of Theorem $\ref{mainth2}$]
The existence of the conformal map $\Phi$ comes directly from Proposition \ref{propImpl}, which gives the boundary equation \eqref{boundaryeq}. Moreover, Lemma \ref{lemm11} gives the decomposition \eqref{T1} and provides the necessary properties  to apply Proposition \ref{thm1}. Furthermore, we can use Proposition \ref{propequiv} in order to obtain the equivalence between a rotating solution \eqref{rotatingsol} of the Euler equations   and the solution of  \eqref{Eq3}. Proposition \ref{propo2} gives the regularity of the flow map, and it remains to check  \eqref{Degen1}, using Proposition \ref{prop5}, in order to get the regularity of the period function. 
In the case of radial profile, we know that
$$
W(\Omega,f_0,\textnormal{Id})(z)=izU_0(z),\quad  U_0(z)=-\Omega+\frac{1}{r^2}\int_0^r s f_0(s)ds,
$$
which implies
\begin{eqnarray*}
W(\Omega,f,\Phi)(z)-W(\Omega,f_0,\hbox{Id})(z)&=&-\Omega\phi(z)\left(1+\overline{g^\prime(z)}\right)-\Omega z\overline{g^\prime(z)}+\overline{g^\prime(z)}\,I(f,\Phi)(z)\\
&& + I(f,\Phi)(z)-I(f_0,\hbox{Id})(z).
\end{eqnarray*}
It is easy to see that
$$
\Big|\frac{\phi(z)}{z}\left(1+\overline{g^\prime(z)}\right)\Big|\leq\|\phi^\prime\|_{L^\infty}\left(1+\|g^\prime\|_{L^\infty(\D)}\right).
$$
From \eqref{Tes1} and Lemma \ref{lemkernel2}, we have
\begin{eqnarray*}
\Big|\overline{g^\prime(z)}\frac{I(f,\Phi)(z)}{z}\Big|\leq \|g^\prime\|_{L^\infty(\D)}|\nabla_z I(f,\Phi)\|_{L^\infty(\D)}
\leq C(\Phi)\|f\|_{\mathscr{C}^{1,\alpha}(\D)} \|g^\prime\|_{L^\infty(\D)}.
\end{eqnarray*}
Using again \eqref{Tes1}, we deduce
\begin{eqnarray*}
\Big|\frac{I(f,\Phi)(z)-I(f_0,\Phi_0)(z)}{z}\Big|&\leq& \Big\|\nabla_z \left[I(f,\Phi)-I(f_0,\Phi_0)\right]\Big\|_{L^\infty(\D)}.
\end{eqnarray*}
Straightforward computations imply 
\begin{eqnarray*}
I(f,\Phi)(z)-I(f_0,\Phi_0)(z)&=&\frac{1}{2\pi}\int_{\D}\frac{ g(y)}{\overline{z}-\overline{y}}dA(y)+\frac{1}{2\pi}\int_{\D}\frac{f(y)\left[\phi(\overline{y})-\phi(\overline{z})\right]}{\left(\overline{y}-\overline{z}\right)\left(\Phi(\overline{y})-\Phi(\overline{z})\right)}dA(y)\\
& & +\frac{1}{2\pi}\int_{\D}\frac{2\textnormal{Re}\{\phi^\prime({y})\}+|\phi^\prime(y)|^2}{\Phi(\overline{z})-\Phi(\overline{y})} f(y)dA(y).
\end{eqnarray*}
From this and using Lemma \ref{lemkernel2}, we claim that
\begin{eqnarray*}
\Big\|\nabla_z \left[I(f,\Phi)-I(f_0,\Phi_0)\right]\Big\|_{L^\infty(\D)}&\le& C(\Phi)\|g\|_{\mathscr{C}^{1,\alpha}(\D)}.
\end{eqnarray*}
Combining the preceding estimates we find
\begin{eqnarray*}
W(\Omega,f,\Phi)(z)=z\left[iU_0(z)+\frac{W(\Omega,f,\Phi)(z)-W(\Omega,f_0,\textnormal{Id})(z)}{z}\right]
\equiv z\left[iU_0(z)+\widehat{W}(\Omega,f,\Phi)(z)\right],
\end{eqnarray*}
with
$$
\|\widehat{W}(\Omega,f,\Phi)\|_{L^\infty(\D)}\leq C(f,\Phi)\left(\|h\|_{\mathscr{C}^{1,,\alpha}(\D)}+\|\phi\|_{\mathscr{C}^{2,\alpha}(\D)}\right).
$$
Now, we take $h,\phi$ small enough such that
$$
C(f,\Phi)\left(\|h\|_{\mathscr{C}^{1,,\alpha}(\D)}+\|\phi\|_{\mathscr{C}^{2,\alpha}(\D)}\right)\leq \varepsilon,
$$
and $\varepsilon$ verifies
$$
0<2\varepsilon\leq   \inf_{0< r\leq1}\Big|\Omega-\frac{1}{r^2}\int_0^r s f_0(s)ds\Big|
=  \inf_{z\in\D}|U_0(z)|,
$$
in order to have
$$
W(\Omega,f,\Phi)(z)=iz U(z),\quad\hbox{with}\quad 2\,\big|\textnormal{Re}\{U(z)\}\big|\geq{\inf_{0< r\leq1}\Big|\Omega-\frac{1}{r^2}\int_0^r s f_0(s)ds\Big|}.
$$
To end the proof let us check that this infimum is strictly positive  for the quadratic profile
$f_0(r)=Ar^2+B$, where $A>0$. Take  $\xi\in(-a,a)\mapsto(\Omega_\xi,f_\xi, \phi_\xi)$ the bifurcating curve from Theorem \ref{TP1}. Then  for $a$ small enough
$$
\Omega_\xi=\frac{B}{2}+\frac{A}{4x_\xi}, x_\xi<1\quad\hbox{and}\quad \Omega_\xi \notin \mathcal{S}_{\textnormal{sing}}
$$  
Thus
\begin{eqnarray*}
\Omega_\xi-\frac{1}{r^2}\int_0^rsf_0(s)ds=\frac{A}{4}\left(\frac{1}{x_\xi}-r^2\right).
\end{eqnarray*}
Consequently
\begin{eqnarray*}
\inf_{r\in[0,1]}\Big|\Omega_\xi-\frac{1}{r^2}\int_0^rsf_0(s)ds\Big|=\inf_{r\in[0,1]}\frac{A}{4}\left|\frac{1}{x_\xi}-r^2\right|
> 0.
\end{eqnarray*}
This achieves the proof.\end{proof}

\appendix

The remaining part of this paper   is devoted to  some basic materials that we have used in the preceding proofs.
\section{Bifurcation theory and Fredholm operators}\label{Apbif}
 We  shall recall   Crandall-Rabinowitz Theorem which is a fundamental tool in bifurcation theory. Notice that  the main aim of this theory is to explore the topological transitions of the phase portrait through the variation of some parameters. A particular case is to understand this transition in  the equilibria set for   the stationary problem    $F(\lambda,x)=0,$ where $F:\R\times X\rightarrow Y$ is a a smooth  function and the spaces $X$ and $Y$ are Banach spaces. Assuming that one has a particular solution,  $F(\lambda,0)=0$ for any $\lambda\in\R$, we would like to explore the bifurcation diagram close to this trivial solution,  and see whether we can find multiple branches of solutions bifurcating  from  a given point  $(\lambda_0,0)$. When this occurs we say that the pair $(\lambda_0,0)$ is a bifurcation point. When the linearized operator around this point generates a Fredholm operator, then one  may  use Lyapunov-Schmidt reduction in order to reduce the infinite-dimensional problem to a finite-dimensional one. For the latter problem we just formulate suitable assumptions so that the    Implicit Function Theorem can be applied. For more discussion in this subject, we refer to see \cite{Kato, Kielhofer}.

In what follows we shall recall some basic results on Fredholm operators.
\begin{definition}
Let $X$ and $Y$ be  two Banach spaces. A continuous linear mapping $T:X\rightarrow Y,$  is a  Fredholm operator if it fulfills the following properties,
\begin{enumerate}
\item $\textnormal{dim Ker}\,  T<\infty$,
\item $\textnormal{Im}\, T$ is closed in $Y$,
\item $\textnormal{codim Im}\,  T<\infty$.
\end{enumerate}
The integer $\textnormal{dim Ker}\, T-\textnormal{codim Im}\, T$ is called the Fredholm index of $T$.
\end{definition}
Next, we shall discuss  the index persistence through compact perturbations, see  \cite{Kato, Kielhofer}.
\begin{proposition}
The index of a Fredholm operator remains unchanged under compact perturbations.
\end{proposition}
Now, we recall the classical Crandall-Rabinowitz Theorem whose proof can be found  in \cite{CrandallRabinowitz}.

\begin{theorem}[Crandall-Rabinowitz Theorem]\label{CR}
    Let $X, Y$ be two Banach spaces, $V$ be a neighborhood of $0$ in $X$ and $F:\mathbb{R}\times V\rightarrow Y$ be a function with the properties,
    \begin{enumerate}
        \item $F(\lambda,0)=0$ for all $\lambda\in\mathbb{R}$.
        \item The partial derivatives  $\partial_\lambda F_{\lambda}$, $\partial_fF$ and  $\partial_{\lambda}\partial_fF$ exist and are continuous.
        \item The operator $F_f(0,0)$ is Fedholm of zero index and $\textnormal{Ker}(F_f(0,0))=\langle f_0\rangle$ is one-dimensional. 
                \item  Transversality assumption: $\partial_{\lambda}\partial_fF(0,0)f_0 \notin \textnormal{Im}(\partial_fF(0,0))$.
    \end{enumerate}
    If $Z$ is any complement of  $\textnormal{Ker}(F_f(0,0))$ in $X$, then there is a neighborhood  $U$ of $(0,0)$ in $\mathbb{R}\times X$, an interval  $(-a,a)$, and two continuous functions $\Phi:(-a,a)\rightarrow\mathbb{R}$, $\beta:(-a,a)\rightarrow Z$ such that $\Phi(0)=\beta(0)=0$ and
    $$F^{-1}(0)\cap U=\{(\Phi(\xi), \xi f_0+\xi\beta(\xi)) : |\xi|<a\}\cup\{(\lambda,0): (\lambda,0)\in U\}.$$
\end{theorem}

\section{Potential estimates}\label{Appotentialtheory}
This appendix is devoted to  some classical estimates on potential theory that we have used before in Section \ref{Secboundaryequation}. We shall deal in particular with truncated operators whose kernels are singular along  the diagonal. They have the form
\begin{equation}\label{OPP1}
\mathscr{L} f(z)\triangleq\int_{\D}K(z,y) f(y)dA(y), z\in\D
\end{equation}
The action  of such operators over various function spaces and its connection to the singularity of the kernel  is  widely studied in the literature, see \cite{EncisoPoyatoSoler,Helms, Kress, Miranda2, Miranda, Wittmann}. In the case of Caldéron-Zygmund operator we refer to the recent papers \cite{Cruz-M-O,Cruz-Tol, MateuOrobitgVerdera} and the references therein. In what follows we shall establish some useful estimates whose proofs are classical  and  for the convenience  of the reader we decide to provide most  the  details.
The first result reads as follows.
\begin{lemma}\label{lemkernel1}
Let $\alpha \in(0,1)$ and $K:\D\times \D\to \C$ is smooth off the diagonal and satisfies 
\begin{equation}\label{kernel1}
|K(z_1,y)|\le \frac{C_0}{|z_1-y|}\quad\hbox{and}\quad |K(z_1,y)-K(z_2,y)|\le C_0\frac{|z_1-z_2|}{|z_1-y||z_2-y|},
\end{equation}
for any $ z_1,z_2\neq y\in\D$, with   $C_0$  a real positive constant. The operator defined in \eqref{OPP1} 
 $$\mathscr{L}: L^\infty(\D)\to \mathscr{C}^{0,\alpha}(\D)$$ 
 is continuous, with the estimate
$$
\|\mathscr{L} f\|_{\mathscr{C}^{0,\alpha}(\D)}\leq C C_0 \|f\|_{L^\infty(\D)},\quad  \forall  f\in L^\infty(\D),
$$
where $C$ is a constant depending on $\alpha$.
\end{lemma}
\begin{proof}
It is easy to see that
$$
\|\mathscr{L} f\|_{L^\infty(\D)}\leq CC_0\|f\|_{L^\infty(\D)}.
$$
Using  \eqref{kernel1} combined with an interpolation argument we may write
\begin{eqnarray*}
|K(z_1,y)-K(z_2,y)|&\leq&|K(z_1,y)-K(z_2,y)|^\alpha\big(|K(z_1,y)|^{1-\alpha}+|K(z_2,y)|^{1-\alpha}\big)\\
&\le& CC_0{|z_1-z_2|^\alpha}\left[\frac{1}{|z_1-y|^{1+\alpha}}+\frac{1}{|z_2-y|^{1+\alpha}}\right],
\end{eqnarray*}
Thus from the inequality
 \begin{eqnarray*}
\sup_{z\in\D}\int_{\D}\frac{|f(y)|}{|z-y|^{1+\alpha}}dA(y)\leq CC_0\|f\|_{L^\infty(\D)}
\end{eqnarray*}
we  deduce  the desired result.
\end{proof}

Before giving the second result, we need to recall  Cauchy--Pompeiu's formula that we shall use later. Let  $D$  be a simply connected domain and $\varphi:\overline{D}\to\C$  be a $\mathscr{C}^1$ complex function, then
\begin{eqnarray}\label{Cauchy-Pompeiu}
-\frac{1}{\pi}\int_D \frac{\partial_{\overline{y}}\varphi(y)}{w-y}dA(y)=\frac{1}{2\pi i}\int_{\partial D}\frac{\varphi(w)-\varphi(\xi)}{w-\xi}d\xi.
\end{eqnarray}

The following result deals with a specific type of integrals that we have already  encountered in Proposition \ref{PropReqX}. More precisely we shall be concerned with the integral
\begin{eqnarray}\label{operatorF}
\mathscr{F}[\Phi](f)(z)\triangleq\int_{\D}\frac{ f(y)}{\Phi(z)-\Phi(y)} |\Phi^\prime(y)|^2 dA(y),\,z\in\D.
\end{eqnarray}
\begin{lemma}\label{lemkernel2}
Let $\alpha \in(0,1)$ and $\Phi:\D\to \Phi(\D)\subset  \C$ be a conformal bi-Lipschitz function  of class $\mathscr{C}^{2,\alpha}(\D)$. Then
$$\mathscr{F}[\Phi]: \mathscr{C}^{1,\alpha}(\D)\to \mathscr{C}^{1,\alpha}(\D),$$
is continuous.
 Moreover, the functional $\mathscr{F}:\Phi\in \mathscr{U} \mapsto \mathscr{F}[\Phi]$ is continuous, with
 $$
\mathscr{U}\triangleq\Big\{\Phi\in \mathscr{C}^{2,\alpha}(\D) : \Phi \textnormal{ is bi-Lipschitz and conformal}\Big\}.
$$
\end{lemma}
\begin{proof}
We start with splitting $\mathscr{F}[\Phi] f$ as follows
\begin{eqnarray*}
\mathscr{F}[\Phi] f(z)&=&\int_{\D}\frac{ f(y)-f(z)}{\Phi(z)-\Phi(y)} |\Phi^\prime(y)|^2 dA(y)+ f(z)\int_{\D}\frac{ |\Phi^\prime(y)|^2}{\Phi(z)-\Phi(y)}  dA(y)\\
&\triangleq& \mathscr{F}_1[\Phi] f(z)+f(z) \mathscr{F}_2[\Phi]f(z).
\end{eqnarray*}
Let us estimate the first term $\mathscr{F}_1[\Phi] f$. The $L^\infty(\D)$ bound is straightforward  and comes from
$$
\sup_{z,y\in\D} \frac{| f(y)-f(z)|}{|\Phi(z)-\Phi(y)|} \leq \frac{\|f\|_{\textnormal{Lip}}}{ \|\Phi^{-1}\|_{\textnormal{Lip}}}\cdot
$$
Setting
$$
K[\Phi](z,y)=\nabla_z \left(\frac{ f(y)-f(z)}{\Phi(z)-\Phi(y)}\right),
$$
then one can easily check that $K$ satisfies the assumptions 
\begin{eqnarray*}
 |K[\Phi]( z,y)|&\leq& C\|f\|_{ \mathscr{C}^{1,\alpha}(\D)}\frac{1}{|z-y|},\\
|K[\Phi](z_1,y)-K[\Phi](z_2,y)|&\leq& C\|f\|_{ \mathscr{C}^{1,\alpha}(\D)}\max_{i,j\in\{1,2\}}\left[\frac{|z_i-z_j|^\alpha}{|z_i-y|}+\frac{|z_i-z_j|}{|z_i-y||z_j-y|}\right],
\end{eqnarray*}
where the constant $C$ depends on $\Phi.$
Thus, Lemma \eqref{lemkernel1} yields
\begin{equation}\label{TMP1}
\|\nabla \mathscr{F}_1[\Phi] f\|_{ \mathscr{C}^{0,\alpha}(\D)}\leq  C\|f\|_{ \mathscr{C}^{1,\alpha}(\D)}
\end{equation}
and hence we find 
$$
\|\mathscr{F}_1[\Phi] f\|_{ \mathscr{C}^{1,\alpha}(\D)}\leq  C\|f\|_{ \mathscr{C}^{1,\alpha}(\D)}
$$
Let us now check the continuity of the operator $\mathscr{F}_1:\Phi\in \mathscr{U} \mapsto \mathscr{F}_1[\Phi].$
Taking $\Phi_1,\Phi_2\in \mathscr{U}$ we may write,
\begin{eqnarray*}
|\mathscr{F}_1[\Phi_1] f(z)-\mathscr{F}_1[\Phi_2] f(z)|&\leq&\int_\D \frac{|f(y)-f(z)|}{|\Phi_1(y)-\Phi_1(z)|}\Big||\Phi_1'(y)|^2-|\Phi_2'(y)|^2\Big|dA(y)\\
&+&\int_\D |f(y)-f(z)||\Phi_2'(y)|\Bigg|\frac{1}{\Phi_1(z)+\Phi_1(y)}-\frac{1}{\Phi_2(z)-\Phi_2(y)}\Bigg|dA(y)\\
&\leq& C\|f\|_{\textnormal{Lip}}\left(\|\Phi_1'-\Phi_2'\|_{L^\infty(\D)}+\|\Phi_1-\Phi_2\|_{L^\infty(\D)}\right).
\end{eqnarray*}
Similarly we get
\begin{eqnarray*}
\nabla (\mathscr{F}_1[\Phi_1] f(z)-\mathscr{F}_1[\Phi_2] f(z))&=&\int_D K[\Phi_1]( z,y)(|\Phi_1'(y)|^2-|\Phi_2'(y)|^2)dA(y)\\
&+&\int_\D \left(K[\Phi_1]( z,y)-K[\Phi_2]( z,y)\right)|\Phi_2'(y)|^2dA(y).
\end{eqnarray*}
Now, performing the same arguments as for \eqref{TMP1} allows to get,
$$
\Bigg\|\int_D K[\Phi_1]( \cdot,y)(|\Phi_1'(y)|^2-|\Phi_2'(y)|^2)dA(y)\Bigg\|_{\mathscr{C}^{0,\alpha}(\D)}\leq C\|\Phi_1'-\Phi_2'\|_{\mathscr{C}^{0,\alpha}(\D)}.
$$
For the second integral term, we proceed first with splitting $K[\Phi]$ in the following way
\begin{eqnarray*}
K[\Phi](z,y)&=&-\frac{\nabla_z f(z)}{\Phi(z)-\Phi(y)}+(f(y)-f(z))\nabla_z\left(\frac{1}{\Phi(z)-\Phi(y)}\right)\\
&\triangleq&-\nabla_z f(z) K_1[\Phi](z,y)+K_2[\Phi](z,y).
\end{eqnarray*}
Let us check that $K_1[\Phi_1]-K_1[\Phi_2]$ obeys to the assumptions  of Lemma \ref{lemkernel1}. For the first one, it is clear from elementary computations that
$$
|K_1[\Phi_1](z,y)-K_1[\Phi_2](z,y)|\leq C\|\Phi_1-\Phi_2\|_{\textnormal{Lip}(\D)}|z-y|^{-1}.
$$
Adding and subtracting adequately we obtain
\begin{eqnarray*}
\Big|(K_1[\Phi_1]-K_1[\Phi_2])\hspace{-1cm}&&(z_1,y)-(K_1[\Phi_1]-K_1[\Phi_2])(z_2,y)\Big|\\
&\leq& \left|\frac{((\Phi_1-\Phi_2)(z_1)-(\Phi_1-\Phi_2)(y))(\Phi_1(z_2)-\Phi_1(y))(\Phi_2(z_2)-\Phi_2(z_1))}{(\Phi_1(z_1)-\Phi_1(y))(\Phi_1(z_2)-\Phi_1(y))(\Phi_2(z_1)-\Phi_2(y))(\Phi_2(z_2)-\Phi_2(y))}\right|\\
&+&\left|\frac{(\Phi_2(z_1)-\Phi_2(y))(\Phi_1(z_2)-\Phi_1(y))((\Phi_1-\Phi_2)(z_1)-(\Phi_1-\Phi_2)(z_2))}{(\Phi_1(z_1)-\Phi_1(y))(\Phi_1(z_2)-\Phi_1(y))(\Phi_2(z_1)-\Phi_2(y))(\Phi_2(z_2)-\Phi_2(y))}\right|\\
&+&\left|\frac{((\Phi_1-\Phi_2)(z_2)-(\Phi_1-\Phi_2)(y))(\Phi_2(z_1)-\Phi_2(y))(\Phi_1(z_1)-\Phi_2(z_2))}{(\Phi_1(z_1)-\Phi_1(y))(\Phi_1(z_2)-\Phi_1(y))(\Phi_2(z_1)-\Phi_2(y))(\Phi_2(z_2)-\Phi_2(y))}\right|\\
&\leq& C \|\Phi_1-\Phi_2\|_{\textnormal{Lip}}\frac{|z_1-z_2|}{|z_1-y||z_2-y|}\cdot
\end{eqnarray*}
Therefore, by applying Lemma \ref{lemkernel1}  we deduce that
$$
\Bigg\|\int_\D \left(K_1[\Phi_1]( \cdot,y)-K_1[\Phi_2]( \cdot,y)\right)|\Phi_2'(y)|^2dA(y)\Bigg\|_{\mathscr{C}^{0,\alpha}(\D)}\leq C\|\Phi_1-\Phi_2\|_{\mathscr{C}^{2,\alpha}(\D)}.
$$
Let us deal with  $K_2$. Note that $\frac{1}{\Phi(z)-\Phi(y)}$ is holomorphic, and then we can work with its complex derivative. We write it as
\begin{eqnarray}\label{K2}
K_2[\Phi](z,y)&=&\frac{f(z)-f(y)}{(\Phi(z)-\Phi(y))^2}\Phi^\prime(z).
\end{eqnarray}
We wish to apply once again Lemma \ref{lemkernel1} and for this purpose we should check the suitable estimate for the kernel.
From straightforward computations we find that 
\begin{eqnarray*}
|K_2[\Phi_1](z,y)-K_2[\Phi_2](z,y)|&\leq& \left|\frac{f(y)-f(z)}{(\Phi_1(z)-\Phi_1(y))^2}(\Phi_1'(z)-\Phi_2'(z))\right|\\
&&\hspace{-1cm}+\left|(f(y)-f(z))\Phi_2'(z)\left(\frac{1}{(\Phi_1(z)-\Phi_1(y))^2}-\frac{1}{(\Phi_2(z)-\Phi_2(y))^2}\right)\right|\\
\hspace{0.3cm} &\leq& C\|f\|_{\textnormal{Lip}}\|\Phi_1-\Phi_2\|_{\textnormal{Lip}}|z-y|^{-1}
\end{eqnarray*}
while for the second hypothesis we write
\begin{eqnarray*}
&&\hspace{-1cm} |(K_2[\Phi_1] -K_2[\Phi_2])(z_1,y)-(K_2[\Phi_1]-K_2[\Phi_2])(z_2,y)|\\
 &\leq&\left|\frac{f(y)-f(z_1)}{(\Phi_1(z_1)-\Phi_1(y))^2}(\Phi_1'(z_1)-\Phi_2'(z_1))-\frac{f(y)-f(z_2)}{(\Phi_1(z_2)-\Phi_1(y))^2}(\Phi_1'(z_2)-\Phi_2'(z_2))\right|\\
&&+\left|(f(y)-f(z_1))\Phi_2'(z_1)\frac{((\Phi_1-\Phi_2)(z_1)-(\Phi_1-\Phi_2)(y))((\Phi_1+\Phi_2)(z_1)-(\Phi_1+\Phi_2)(y))}{(\Phi_1(z_1)-\Phi_1(y))^2(\Phi_2(z_1)-\Phi_2(y))^2}\right.\\
&&-\left.(f(y)-f(z_2))\Phi_2'(z_2)\frac{((\Phi_1-\Phi_2)(z_2)-(\Phi_1-\Phi_2)(y))((\Phi_1+\Phi_2)(z_2)-(\Phi_1+\Phi_2)(y))}{(\Phi_1(z_2)-\Phi_1(y))^2(\Phi_2(z_2)-\Phi_2(y))^2}\right|\\
&\leq&C\|f\|_{\textnormal{Lip}}(\|\Phi_1-\Phi_2\|_{\textnormal{Lip}}+\|\Phi_1'-\Phi_2'\|_{\textnormal{Lip}})\frac{|z_1-z_2|}{|z_1-y||z_2-y|}\cdot
\end{eqnarray*}
It follows from  Lemma \ref{lemkernel1} that
 $$
\Bigg\|\int_\D \left(K_2[\Phi_1]( \cdot,y)-K_2[\Phi_2]( \cdot,y)\right)|\Phi_2'(y)|^2dA(y)\Bigg\|_{\mathscr{C}^{0,\alpha}(\D)}\leq C\|\Phi_1-\Phi_2\|_{\mathscr{C}^{2,\alpha}(\D)},
$$
which concludes the proof of  the continuity of $\mathscr{F}_1$ with respect to $\Phi$.
\\
Let us now move  to the second term $\mathscr{F}_2$. By using a change of variables one may write
$$
\mathscr{F}_2[\Phi] f(z)=\int_{\Phi(\D)}\frac{ 1}{\Phi(z)-y} dA(y).
$$
First, note that
$$
  \|\mathscr{F}_2[\Phi] f\|_{L^\infty(\D)}\leq C,
$$
and
$$
 \|\mathscr{F}_2[\Phi_1] f-\mathscr{F}_2[\Phi_2] f\|_{L^\infty(\D)}\leq C\|\Phi_1-\Phi_2\|_{\mathscr{C}^{2,\alpha}(\D)}.
$$
By Cauchy--Pompeiu's formula \eqref{Cauchy-Pompeiu} we get
\begin{eqnarray*}
 \mathscr{F}_2[\Phi] f(z)&=&\pi\overline{\Phi(z)}-\frac{1}{2i}\int_{\Phi(\T)}\frac{ \overline{\xi} }{\xi-\Phi(z)}  d\xi
\\
& =& \pi\overline{\Phi(z)}-\frac{1}{2i}\int_{\T}\frac{ \overline{\Phi(\xi)} }{\Phi(\xi)-\Phi(z)}  \Phi^\prime(\xi)d\xi
=\pi\overline{\Phi(z)}-\frac{1}{2i}\mathscr{C}[\Phi](z).
\end{eqnarray*}
We observe that  the mapping $\Phi\in \mathscr{C}^{2,\alpha}(\D)\mapsto \overline\Phi\in \mathscr{C}^{2,\alpha}(\D)$ is well-defined and  continuous. So it remains to check that $\mathscr{C}[\Phi]\in \mathscr{C}^{1,\alpha}(\D)$ and prove  its continuity with respect to $\Phi$. Note first that $\mathscr{C}[\Phi]$ is holomorphic inside the unit disc and  its complex derivative is given by
$$
\left[\mathscr{C}[\Phi]\right]^\prime(z)=\Phi^\prime(z)\int_{\T}\frac{ \overline{\Phi(\xi)} }{\left(\Phi(\xi)-\Phi(z)\right)^2} \Phi^\prime(\xi)d\xi,\quad \forall\,\, z\in\D.
$$
Using a change of variables, we deduce that
\begin{eqnarray*}
\left[\mathscr{C}[\Phi]\right]^\prime(z)=-\Phi^\prime(z)\int_{\T}\frac{\overline{\Phi^\prime(\xi)}\,\,\overline{\xi}^2}{\Phi(\xi)-\Phi(z)} d\xi ,
\end{eqnarray*}
where we have used the formula
$$
\frac{d}{d\xi}\overline{\Phi(\xi)}=-\overline\xi^2\Phi^\prime(\overline{\xi}).
$$
For this last integral we can use the upcoming Lemma \ref{lemkernel4} to obtain that $[\mathscr{C}[\Phi]]'\in \mathscr{C}^{0,\alpha}(\D)$. Although the last is clear, we show here an alternative procedure useful to check  the continuity with respect to $\Phi$.
According to \cite[Lemma 6.4.8]{Rudin}, to show that $ \left[\mathscr{C}[\Phi]\right]^\prime\in  \mathscr{C}^{0,\alpha}(\D)$ it suffices to prove that
\begin{equation}\label{seconddif}
 \Big|\left[\mathscr{C}[\Phi]\right]^{\prime\prime}(z)\Big|\leq C(1-|z|)^{\alpha-1}, \quad \forall z\in \D.
\end{equation}
Then, by differentiating we get
\begin{eqnarray*}
\left[\mathscr{C}[\Phi]\right]^{\prime\prime}(z)&=&-\Phi^{\prime\prime}(z)\bigintsss_{\T}\frac{ \overline{\Phi^\prime(\xi)}\,\, \overline{\xi}^2 }{\Phi(\xi)-\Phi(z)} d\xi-\left(\Phi^{\prime}(z)\right)^2\bigintsss_{\T}\frac{ \overline{\Phi^\prime(\xi)}\,\,\overline{\xi}^2 }{\left(\Phi(\xi)-\Phi(z)\right)^2}  d\xi\\
&\triangleq &-\Phi^{\prime\prime}(z)\mathscr{C}_1[\Phi](z)-\left(\Phi^{\prime}(z)\right)^2\mathscr{C}_2[\Phi](z).
\end{eqnarray*}
For $\mathscr{C}_1[\Phi]$ we simply write
\begin{eqnarray*}
\mathscr{C}_1[\Phi](z)&=&\bigintss_{\T}\frac{ \overline{\xi}^2\frac{\overline{\Phi^\prime(\xi)} }{\Phi^\prime(\xi)}-\overline{z}^2\frac{\overline{\Phi^\prime(z)}}{\Phi^\prime(z)}}{\Phi(\xi)-\Phi(z)} \Phi^\prime(\xi) d\xi+2i\pi\frac{\overline{\Phi^\prime(z)}}{\Phi^\prime(z)}\overline{z}^2.
\end{eqnarray*}
Since  $\xi\in\overline{\D}\mapsto \overline{\xi}^2\frac{\overline{\Phi^\prime}(\xi) }{\Phi^\prime(\xi)}\in  \mathscr{C}^{0,\alpha}(\overline{\D})
$, then we have
\begin{eqnarray*}
|\mathscr{C}_1[\Phi](z)|\leq C\left(\int_{\T}\frac{ |z-\xi|^\alpha}{|z-\xi|} |d\xi|+1\right)
\leq C(1-|z|)^{\alpha-1}.
\end{eqnarray*}
It remains to estimate $\mathscr{C}_2[\Phi]$. Integration by parts implies
\begin{eqnarray*}
\mathscr{C}_2[\Phi](z)= \int_{\T}\frac{\psi(\xi)}{\Phi(\xi)-\Phi(z)} \Phi^\prime(\xi) d\xi,
\end{eqnarray*}
with
$$
\psi(\xi)=\frac{ \left(\frac{\overline{\Phi^\prime(\xi)}\overline{\xi}^2 }{\Phi^\prime(\xi)}\right)^\prime}{\Phi^\prime(\xi) }\cdot
$$
Since $\Phi\in  \mathscr{C}^{2,\alpha}(\D)$ and is bi-Lipschitz, then  $ \psi\in  \mathscr{C}^{0,\alpha}(\D)$. Writing
\begin{equation*}
\mathscr{C}_2[\Phi](z)=\int_{\T}\frac{\psi(\xi)-\psi(z)}{\Phi(\xi)-\Phi(z)} \Phi^\prime(\xi) d\xi+2i\pi \psi(z),
\end{equation*}
implies that 
\begin{eqnarray*}
|\mathscr{C}_2[\Phi]|\leq C\int_{\T}\frac{|\psi(\xi)-\psi(z)|}{|\xi-z|} | d\xi|+2\pi \|\psi\|_{L^\infty}
\le C\int_{\T}\frac{|\xi-z)|^\alpha}{|\xi-z|} | d\xi|+2\pi \|\psi\|_{L^\infty}.
\end{eqnarray*}
Thus, we have
$$
|\mathscr{C}_2[\Phi](z)|\le C(1-|z|)^{\alpha-1}.
$$
Putting together the previous estimates, we deduce \eqref{seconddif} which implies $\mathscr{F}_2[\Phi]\in\mathscr{C}^{1,\alpha}(\D)$.
\\
It remains to check the continuity of $\mathscr{F}_2$ with respect to $\Phi$.
 Splitting $\mathscr{F}_2[\Phi_1]-\mathscr{F}_2[\Phi_2]$ as 
\begin{eqnarray*}
\mathscr{F}_2[\Phi_1]f(z)-\mathscr{F}_2[\Phi_2]f(z)&=&\int_\D \frac{|\Phi_1'(y)|^2-|\Phi_2'(y)|^2}{\Phi_1(z)-\Phi_1(y)}dA(y)\\
&&+\int_\D\frac{(\Phi_2-\Phi_1)(z)-(\Phi_2-\Phi_1)(y)}{(\Phi_1(z)-\Phi_1(y))(\Phi_2(z)-\Phi_2(y))}|\Phi_2'(y)|^2dA(y),
\end{eqnarray*}
combined with Lemma \ref{lemkernel1} yield
$$
\|\mathscr{F}_2[\Phi_1]-\mathscr{F}_2[\Phi_2]\|_{\mathscr{C}^{0,\alpha}(\D)}\leq C \|\Phi_1-\Phi_2\|_{\mathscr{C}^{2,\alpha}(\D)}.
$$
Now, we need to prove a similar inequality for its derivative,
$$
\|\mathscr{F}_2'[\Phi_1]-\mathscr{F}_2'[\Phi_2]\|_{\mathscr{C}^{0,\alpha}(\D)}\leq C \|\Phi_1-\Phi_2\|_{\mathscr{C}^{2,\alpha}(\D)}.
$$
Since $\Phi\in\mathscr{C}^{2,\alpha}(\D)\mapsto  \Phi^\prime\in \mathscr{C}^{1,\alpha}(\D)$ is clearly continuous, we just have to prove that
$$
\left|[\mathscr{C}[\Phi_1]]''(z)-[\mathscr{C}[\Phi_2]]''(z)\right|\leq C \|\Phi_1-\Phi_2\|_{\mathscr{C}^{2,\alpha}(\D)}(1-|z|)^{\alpha-1}, \quad \forall z\in\D.
$$
It is enough to check the above estimate for $\mathscr{C}_1$ and $\mathscr{C}_2$. Let us show how  dealing with the first one, and the same arguments can be applied for $\mathscr{C}_2.$ We have
\begin{eqnarray*}
\left|\mathscr{C}_1[\Phi_1](z)\right.&-&\left.\mathscr{C}_1[\Phi_2](z)\right|\leq\left|\bigintss_\T\frac{\overline{\xi}^2\left(\frac{\overline{\Phi_1'(\xi)}}{\Phi_1'(\xi)}-\frac{\overline{\Phi_2'(\xi)}}{\Phi_2'(\xi)}\right)-\overline{z}^2\left(\frac{\overline{\Phi_1'(z)}}{\Phi_1'(z)}-\frac{\overline{\Phi_2'(z)}}{\Phi_2'(z)}\right)}{\Phi_1(\xi)-\Phi_1(z)}\Phi_1'(\xi)d\xi\right|\\
&&+\left|\bigintss_\T \frac{\overline{z}^2\frac{\overline{\Phi_2'(z)}}{\Phi_2'(z)}-\overline{\xi}^2\frac{\overline{\Phi_2'(\xi)}}{\Phi_2'(\xi)}}{\Phi_1(\xi)-\Phi_1(z)}(\Phi_1'(\xi)-\Phi_2'(\xi))d\xi\right|\\
&&+\left|\bigintss_\T \frac{(\Phi_1-\Phi_2)(z)-(\Phi_1-\Phi_2)(\xi)}{(\Phi_1(\xi)-\Phi_1(z))(\Phi_2(\xi)-\Phi_2(z))}\left(\overline{z}^2\frac{\overline{\Phi_2'(z)}}{\Phi_2'(z)}-\overline{\xi}^2\frac{\overline{\Phi_2'(\xi)}}{\Phi_2'(\xi)}\right)\Phi_2'(\xi)d\xi\right|\\
&&+2\pi \, \overline{z}^2\left|\frac{\overline{\Phi_1'(z)-\Phi_2'(z)}}{\Phi_1'(z)}+\overline{\Phi_2'(z)}\frac{\Phi_2'(z)-\Phi_1'(z)}{\Phi_1'(z)\Phi_2'(z)}\right|\\
&\leq& C \|\Phi_1-\Phi_2\|_{\mathscr{C}^{2,\alpha}(\D)}|(1-|z|)^{1-\alpha},
\end{eqnarray*}
where  we have used that
$$
\left|\overline{\xi}^2\left(\frac{\overline{\Phi_1'(\xi)}}{\Phi_1'(\xi)}-\frac{\overline{\Phi_2'(\xi)}}{\Phi_2'(\xi)}\right)-\overline{z}^2\left(\frac{\overline{\Phi_1'(z)}}{\Phi_1'(z)}-\frac{\overline{\Phi_2'(z)}}{\Phi_2'(z)}\right)\right|\leq C\|\Phi_1-\Phi_2\|_{\mathscr{C}^{2,\alpha}(\D}|z-\xi|.
$$
Then, we obtain the  inequality for $\mathscr{C}_1[\Phi]$, which concludes the continuity with respect to $\Phi$. \end{proof}
{
The following result deals with a Calderon-Zygmund type estimate, which will be necessary in the later development. The techniques used are related to the well-known T(1)-Theorem of Wittmann, see \cite{Wittmann}. Let us define
\begin{equation}\label{OPP2}
\mathscr{K} f(z)\triangleq\int_{\T}K(z,\xi) f(\xi)d\xi, \quad \forall z\in\D.
\end{equation}

\begin{lemma}\label{lemkernel4}
Let $\alpha \in(0,1)$ and $K:\D\times \D\to \C$ smooth outside the diagonal that satisfies 
\begin{align}
&|K(z_1,y)|\leq C_0|z_1-y|^{-1},\label{Kproperties1}\\
&|K(z_1,y)-K(z_2,y)|\leq C_0\frac{|z_1-z_2|}{|z_1-y|^2},\quad \textnormal{if }\, 2|z_1-z_2|\leq |z_1-y|,\label{Kproperties2}\\
&\mathscr{K} (\textnormal{Id})\in\mathscr{C}^{0,\alpha}(\D),\label{Kproperties3}\\
&\left|\int_{\partial (\D\cap B_{z_1}(\rho))} K(z_1,\xi)d\xi\right|<C_0,  \label{Kproperties4}
\end{align}
for any $z_1,z_2\neq y\in\D$ and $\rho>0$, where $C_0$ is a positive constant that does not depend on $z_1$, $z_2$, $y$ and $\rho$. Then,
 $$\mathscr{K}: \mathscr{C}^{0,\alpha}(\D)\rightarrow \mathscr{C}^{0,\alpha}(\D)$$ 
is continuous, with the estimate
$$
\|\mathscr{K} f\|_{\mathscr{C}^{0,\alpha}(\D)}\leq C C_0\|f\|_{\mathscr{C}^{0,\alpha}(\D)},
$$
where $C$ is a constant depending only on $\alpha$.
\end{lemma}

\begin{proof}
From \eqref{Kproperties1} and \eqref{Kproperties3}, we get easily that
\begin{align*}
|\mathscr{K}f(z)|&\leq \left|\int_\T K(z,\xi)(f(\xi)-f(z))d\xi\right|+\left|f(z)\int_\T K(z,\xi)d\xi\right|\\
&\leq C_0 \|f\|_{\mathscr{C}^{0,\alpha}(\D)}\int_\T\frac{|d\xi|}{|z-\xi|^{1-\alpha}}+C\|f\|_{\mathscr{C}^{0,\alpha}(\D)}\\
&\le  CC_0\|f\|_{\mathscr{C}^{0,\alpha}(\D)}.
\end{align*}
Taking $z_1,z_2\in\D$, we define $d=|z_1-z_2|$. We write
\begin{align*}
\int_\T K(z_1,\xi)f(\xi)d\xi&-\int_\T K(z_2,\xi)f(\xi)d\xi\\
=&\int_\T K(z_1,\xi)(f(\xi)-f(z_1))d\xi-\int_\T K(z_2,\xi)(f(\xi)-f(z_1))d\xi\\&+f(z_1)\int_\T K(z_1,\xi)d\xi-f(z_1)\int_\T K(z_2,\xi)d\xi\\
=&\int_{\T\cap B_{z_1}(3d)} K(z_1,\xi)(f(\xi)-f(z_1))d\xi-\int_{\T\cap B_{z_1}(3d)} K(z_2,\xi)(f(\xi)-f(z_1))d\xi\\
&+\int_{\T\cap B_{z_1}^c(3d)}(K(z_1,\xi)-K(z_2,\xi))(f(\xi)-f(z_1))d\xi\\
&+f(z_1)\int_\T K(z_1,\xi)d\xi-f(z_1)\int_\T K(z_2,\xi)d\xi\\
\triangleq& I_1-I_2+I_3+f(z_1)(I_4-I_5).
\end{align*}
Using \eqref{Kproperties3}, we achieve
$$
|I_4-I_5|\leq C|z_1-z_2|^\alpha.
$$
Let us work with $I_1$ using the Layer Cake Lemma, see  \cite{LiebLoss}. We use that $|\T\cap B_x(\rho)|\leq C\rho$, for any $\rho>0$ and $x\in\R^2$, which means that it is $1$-Ahlfors regular curve. In fact,  taking any $z\in\D$ and $\rho>0$ ones has that
\begin{eqnarray*}
\int_{\T\cap B_{z}(\rho)}\frac{|d\xi|}{|z-\xi|^{1-\alpha}}&=&\int_0^\infty \left|\left\{\xi\in \T\cap B_{z}(\rho) : \frac{1}{|z-\xi|^{1-\alpha}}\geq \lambda\right\}\right|d\lambda,\\
&=&\int_0^\infty \left|\left\{\xi\in \T\cap B_{z}(\rho) : |z-\xi|\leq \lambda^{\frac{-1}{1-\alpha}}\right\}\right|d\lambda\\
&=&\int_0^{\rho^{\alpha-1}}\left|\left\{\xi\in \T : |z-\xi|\leq \rho\right\}\right|d\lambda
+\int_{\rho^{\alpha-1}}^{+\infty}\left|\left\{\xi\in \T : |z-\xi|\leq \lambda^{\frac{-1}{1-\alpha}}\right\}\right|d\lambda\\
&\leq& C\left(\rho \rho^{\alpha-1}+\int_{\rho^{\alpha-1}}^{+\infty}\lambda^{\frac{-1}{1-\alpha}}d\lambda\right)
\leq C \rho^\alpha,
\end{eqnarray*}
where $|\cdot |$ inside the integral denotes the arch length measure. Applying the last estimate to $I_1$, we find
$$
|I_1|\leq C_0\|f\|_{\mathscr{C}^{0,\alpha}(\D)}\int_{\T\cap B_{z_1}(3d)}\frac{|d\xi|}{|z_1-\xi|^{1-\alpha}}\leq C C_0\|f\|_{\mathscr{C}^{0,\alpha}(\D)}|z_1-z_2|^\alpha.
$$
 For the term $I_3$, we get
 $$
 |I_3|\leq C_0\|f\|_{\mathscr{C}^{0,\alpha}(\D)}|z_1-z_2|\int_{\T\cap B_{z_1}(3d)^c}\frac{|d\xi|}{|z_1-\xi|^{2-\alpha}},
 $$
by \eqref{Kproperties2}. Now, we use again the Layer Cake Lemma, obtaining
\begin{align*}
\int_{\T\cap B_{z}(\rho)^c}\frac{|d\xi|}{|z-\xi|^{2-\alpha}}&=\int_0^\infty \left|\left\{\xi\in \T : \frac{1}{|z-\xi|^{2-\alpha}}\geq \lambda, |z-\xi|\geq \rho\right\}\right|d\lambda,\\
&=\int_0^\infty \left|\left\{\xi\in \T : |z-\xi|\leq \lambda^{\frac{-1}{2-\alpha}}, |z-\xi|\geq\rho\right\}\right|d\lambda\\
&=\int_0^{\rho^{\alpha-2}}\left|\left\{\xi\in \T : \rho\leq|z-\xi|\leq \lambda^{\frac{-1}{2-\alpha}}\right\}\right|d\lambda\\
&\leq C\int_0^{\rho^{\alpha-2}}\lambda^{\frac{-1}{2-\alpha}}d\lambda\leq C \rho^{\alpha-1}.
\end{align*}
Therefore,
$$
|I_3|\leq C_0\|f\|_{\mathscr{C}^{0,\alpha}(\D)}|z_1-z_2|^\alpha.
$$
It remains to estimate $I_2$. First, let us write it as
\begin{align*}
I_2&=\int_{\T\cap B_{z_1}(3d)} K(z_2,\xi)(f(\xi)-f(z_2))d\xi+(f(z_2)-f(z_1))\int_{\T\cap B_{z_1}(3d)} K(z_2,\xi)d\xi\\
&\triangleq H_1+(f(z_2)-f(z_1))H_2.
\end{align*}
$H_1$ can be estimated as $I_1$ noting that $B_{z_1}(3d)\subset B_{z_2}(4d)$. To finish, we just need to check that $H_2$ is bounded. Decompose it as
$$
H_2=\int_{\T\cap B_{z_2}(2d)} K(z_2,\xi)d\xi+\int_{\T\cap B_{z_1}(3d)\cap B_{z_2}(2d)^c} K(z_2,\xi)d\xi\triangleq J_1+J_2,
$$
since $B_{z_2}(2d)\subset B_{z_1}(3d)$. Note that 
$$
|J_2|\leq C_0 \int_{\T\cap B_{z_1}(3d)\cap B_{z_2}(2d)^c}\frac{|d\xi|}{|z_2-\xi|}\leq CC_0 d^{-1}|\T\cap B_{z_1}(3d)|\leq C C_0.
$$
For the last term, we write
$$
J_1=\int_{\partial(\D\cap B_{z_2}(2d)) } K(z_2,\xi)d\xi-\int_{\D\cap \partial B_{z_2}(2d)} K(z_2,\xi)d\xi
$$
By using condition \eqref{Kproperties4}, we get that the first integral is bounded. For the second one, we obtain
$$
\left|\int_{\D\cap \partial B_{z_2}(2d)} K(z_2,\xi)d\xi\right|\leq C_0 \int_{\D\cap \partial B_{z_2}(2d)} \frac{|d\xi|}{|z_2-\xi|}=\frac{1}{2d}|\D\cap \partial B_{z_2}(2d)|\leq C.
$$
Combining all the estimates, we achieved the announced result.
\end{proof}

{
In the following result, we deal with the Cauchy integral defined as
\begin{eqnarray}\label{operatorI}
\mathscr{I}[\Phi](f)(z)\triangleq\int_\T \frac{f(\xi)\Phi'(\xi)}{\Phi(z)-\Phi(\xi)}d\xi.
\end{eqnarray}
Note that this classical operator is fully studied in \cite{Kress} in the case that $\Phi=Id$, then there we will adapt that proof.

\begin{lemma}\label{lemkernel3}
Let $\alpha\in(0,1)$ and $\Phi:\D\rightarrow\Phi(\D)\subset \C$ be a conformal bi-Lipschitz function of class $\mathscr{C}^{2,\alpha}(\D)$. Therefore, we have that
$$
\mathscr{I}[\Phi]:\mathscr{C}^{0,\alpha}(\D)\rightarrow \mathscr{C}^{0,\alpha}(\D),
$$
is continuous.
Moreover, $\mathscr{I}:\Phi\in \mathscr{U} \mapsto \mathscr{I}[\Phi]$ is continuous, where $\mathscr{U}$ is defined in \mbox{Lemma $\ref{lemkernel2}.$}
\end{lemma}
\begin{proof}
Note that 
$$
\mathscr{I}[\Phi](f)(z)=\int_{\Phi(\T)}\frac{(f\circ\Phi^{-1})(\xi)}{\Phi(z)-\xi}d\xi=\mathcal{C}[f\circ\Phi^{-1}](\Phi(z)),
$$
where $\mathcal{C}$ is the Cauchy Integral. Then, it is classical, see \cite{Kress}, that
\begin{align*}
||\mathscr{I}[\Phi](f)||_{\mathscr{C}^{0,\alpha}(\D)}=||\mathcal{C}[f\circ\Phi^{-1}]\circ\Phi||_{\mathscr{C}^{0,\alpha}(\D)}&\leq C ||f\circ\Phi^{-1}||_{\mathscr{C}^{0,\alpha}(\D)}||\Phi||_{\textnormal{Lip}(\D)}\\
&\leq C ||f||_{\mathscr{C}^{0,\alpha}(\D)}||\Phi^{-1}||_{\textnormal{Lip}(\D)}||\Phi||_{\textnormal{Lip}(\D)}.
\end{align*}

To deal with the continuity with respect to the conformal map, we write
\begin{align*}
\mathscr{I}[\Phi_1]f(z)-\mathscr{I}[\Phi_2]f(z)&=&\int_\T f(\xi)\left(\frac{\Phi_1'(\xi)}{\Phi_1(z)-\Phi_1(\xi)}-\frac{\Phi_2'(\xi)}{\Phi_2(z)-\Phi_2(\xi)}\right)d\xi\triangleq \int_\T f(\xi)K(z,\xi)d\xi.
\end{align*}
We will check that $K$ verifies \eqref{Kproperties1}-\eqref{Kproperties4} in order to use Lemma \ref{lemkernel4}. Straightforward computations yield
\begin{align*}
&|K(z_1,y)|\leq C ||\Phi_1-\Phi_2||_{\mathscr{C}^{1,\alpha}(\D)}|z_1-y|^{-1},\\
&|K(z_1,y)-K(z_2,y)|\leq C ||\Phi_1-\Phi_2||_{\mathscr{C}^{1,\alpha}(\D)}\frac{|z_1-z_2|}{|z_1-y|^2},\quad \textnormal{if }\, 2|z_1-z_2|\leq |z_1-y|,
\end{align*}
using that $|z_2-y|\geq |z_1-y|-|z_1-z_2|\geq \frac12 |z_1-y|$ in the second property, which concerns  \eqref{Kproperties1}-\eqref{Kproperties2}. Moreover,
$$
\int_\T K(z,\xi)d\xi=\int_{\Phi_1(\T)}\frac{d\xi}{\Phi_1(z)-\xi}-\int_{\Phi_2(\T)}\frac{d\xi}{\Phi_2(z)-\xi}=0,
$$
which implies \eqref{Kproperties3}. In fact, 
$$
\int_{\partial(\D\cap B_z(\rho))} K(z,\xi)d\xi=\int_{\Phi_1(\partial(\D\cap B_z(\rho)))}\frac{d\xi}{\Phi_1(z)-\xi}-\int_{\Phi_2(\partial(\D\cap B_z(\rho)))}\frac{d\xi}{\Phi_2(z)-\xi}=C_0,
$$
by applying the Residue Theorem, and where  $C_0$ that does not depend on $\rho$ neither $z$, which agrees with \eqref{Kproperties4}. Then, we achieve the proof using Lemma \ref{lemkernel4}.
\end{proof}}}

We give the explicit expressions of some integrals which appear in the analysis of  the linearized operator.

\begin{proposition}\label{integralsprop} Let $\alpha\in(0,1)$. Given $h\in\mathscr{C}^{1,\alpha}_s(\D), \, k\in\mathscr{H}\mathscr{C}^{2,\alpha}(\D)$ and a  radial function $f_0\in\mathscr{C}(\D)$, the following identities 
\begin{align*}
  &\int_\D \frac{k(z)-k(y)}{(z-y)^2}f_0(y)dA(y)=2\pi\sum_{n\geq 1}A_n z^{n-1}\left[\int_0^{|z|}rf_0(r)dr-n\int_{|z|}^1rf_0(r)dr\right],\\
 &\int_\D \frac{f_0(y)}{z-y}\Real[k'(y)]dA(y)=\pi\sum_{n\geq 1}A_n(n+1) \left[-z^{n-1}\int_{|z|}^1rf_0(r)dr+\frac{\overline{z}^{n+1}}{|z|^{2(n+1)}}\int_0^{|z|}r^{2n+1}f_0(r)dr \right],\\
 &\int_\D \frac{f_0(y)}{z-y}dA(y)=2\pi\frac{\overline{z}}{|z|^2}\int_0^{|z|}rf_0(r)dr,\\
 &\int_\D\frac{h(y)}{z-y}dA(y)=\pi\sum_{n\geq 1} \left[-z^{n-1}\int_{|z|}^1\frac{1}{r^{n-1}}h_n(r)dr+\frac{\overline{z}^{n+1}}{|z|^{2n+2}}\int_0^{|z|}r^{n+1}h_n(r)dr\right],\\
 &\int_\D \log |z-y|h(y)dA(y)=-\pi\sum_{n\geq 1}\cos(n\theta)\frac{1}{n}\left[|z|^n\int_{|z|}^1\frac{1}{r^{n-1}}h_n(r)dr+\frac{1}{|z|^n}\int_0^{|z|}|z|^{n+1}h_n(r)dr\right],\\
 &\int_\D\frac{k(z)-k(y)}{z-y}f_0(y)dA(y)=2 \pi\sum_{n\geq 1}A_nz^n\int_0^1sf_0(s)ds,\\
 &\int_\D\log|z-y|f_0(y)\textnormal{Re}[k'(y)]dA(y)=-\pi\sum_{n\geq 1}A_n\frac{n+1}{n}\cos(n \theta)\\
 &\hspace{6cm}\times\left[|z|^n\int_{|z|}^1rf_0(r)dr+\frac{1}{|z|^n}\int_0^{|z|}r^{2n+1}f_0(r)dr\right],\\
 &\int_\D \log |z-y|f_0(y)dA(y)=2\pi\left[\int_0^{|z|}\frac{1}{\tau}\int_0^\tau r f_0(r)dr-\int_0^1\frac{1}{\tau}\int_0^\tau r f_0(r)dr\right]
\end{align*}
holds for $z\in\overline{\D}$.
\end{proposition}
\begin{proof}
Note that $h$ and $k$ can be given by
$$h(z)=\sum_{n\geq 1}h_n(r)\cos(n\theta), \quad k(z)=z\sum_{n\geq 1}A_nz^{n},$$
where $z=re^{i\theta}\in\D$.

\medskip
\noindent
{\bf{(1)}}  Using the expression for the function $k$,
the integral to be computed takes the form
$$\int_\D\frac{z^{n+1}-y^{n+1}}{z-y}\frac{f_0(y)}{z-y}dA(y).$$
An expansion of the function inside the integral provides
$$\int_\D\frac{z^{n+1}-y^{n+1}}{z-y}\frac{f_0(y)}{z-y}dA(y)=\sum_{k=0}^{n}z^{n-k}{\int_\D y^k\frac{f_0(y)}{z-y}dA(y)}.$$
The use of polar coordinates yields
\begin{eqnarray}\label{int}
\int_\D y^k\frac{f_0(y)}{z-y}dA(y)=i\int_0^1r^kf_0(r)\int_{\T}\frac{\xi^{k-1}}{\xi-\frac{z}{r}}d\xi.
\end{eqnarray}
We split our study in the cases $k=0$ and $k\geq 1$ by making use of the Residue Theorem. For $k=0$, we obtain
$$\int_{\T}\frac{1}{\xi}\frac{1}{\xi-\frac{z}{r}}d\xi=\left\{\begin{array}{ll}
0, & |z|\leq r,\\
-2\pi i \frac{r}{z}, & |z|\geq r,
\end{array}\right.$$
whereas, we find 
$$\int_{\T}\frac{\xi^{k-1}}{\xi-\frac{z}{r}}d\xi=\left\{\begin{array}{ll}
2\pi i\frac{z^{k-1}}{r^{k-1}}, & |z|\leq r,\\
0, & |z|\geq r,
\end{array}\right.$$
for any $k\geq 1$. This allows us to have the following expression 
\begin{eqnarray*}
\int_\D\frac{z^{n+1}-y^{n+1}}{z-y}\frac{f_0(y)}{z-y}dA(y)&=&2\pi z^{n-1}\int_0^{|z| }rf_0(r)dr-2\pi\sum_{k=1}^{n}z^{n-k} z^{k-1}\int_{|z|}^1rf_0(r)dr\\
&=&2\pi z^{n-1}\left[\int_0^{|z|}rf_0(r)dr-n\int_{|z|}^1rf_0(r)dr\right].
\end{eqnarray*}

\medskip
\noindent
{\bf{(2)}} Note that
$k'(z)=\sum_{n\geq 1} A_n (n+1) z^{n}.$
Then, the integral to be analyzed is
$$\int_\D \frac{f_0(y)}{z-y}\Real[k'(y)]dA(y)=\sum_{n\geq 1} \frac{A_n(n+1)}{2}\int_\D \frac{f_0(y)}{z-y}\left(y^{n}+\overline{y}^{n}\right)dA(y).$$
We study the two terms in the integral by using polar coordinates and the Residue Theorem. For the first one we have
\begin{eqnarray*}
\int_\D\frac{f_0(y)}{z-y}y^{n}dA(y)=i\int_0^1r^{n}f_0(r)\int_{\T}\frac{\xi^{n-1}}{\xi-\frac{z}{r}}d\xi dr
=-2\pi z^{n-1}\int_{|z|}^1rf_0(r)dr.
\end{eqnarray*}
In the same way, the second one can be written as
\begin{eqnarray*}
\int_\D\frac{f_0(y)}{z-y}\overline{y}^{n}dA(y)&=&i\int_0^1f_0(r)r^{n}\int_{\T} \frac{1}{\xi^{n+1}}\frac{1}{\xi-\frac{z}{r}}d\xi dr\\
&=&2\pi\frac{1}{z^{n+1}}\int_0^{|z|}f_0(r)r^{2n+1}dr
=\frac{2\pi\overline{z}^{n+1}}{|z|^{2n+2}}\int_0^{|z|}f_0(r)r^{2n+1}dr,
\end{eqnarray*}
that concludes the proof.

\medskip
\noindent{\bf{(3)}} This integral reads as
\begin{eqnarray*}
\int_\D\frac{f_0(y)}{z-y}dA(y)=i\int_0^1f_0(r)\int_{\T} \frac{1}{\xi}\frac{1}{\xi-\frac{z}{r}}d\xi dr
= \frac{2\pi}{z}\int_0^{|z|}rf_0(r)dr
=\frac{2\pi \overline{z}}{|z|^2}\int_0^{|z|}rf_0(r)dr.
\end{eqnarray*}

\medskip
\noindent
{\bf{(4)}} We use the expression of  $h(z)$ to deduce that
$$\int_\D \frac{h(y)}{z-y}dA(y)=\frac{1}{2}\sum_{n\geq 1} \int_\D \frac{h_n(r)(e^{in\theta}+e^{-in\theta})}{z-y}dA(y).$$
The two terms involved in the integral can be computed as follows
\begin{eqnarray*}
\int_\D\frac{h_n(r)e^{in\theta}}{z-y}dA(y)&=&i\int_0^1h_n(r)\int_{\T} \frac{\xi^{n-1}}{\xi-\frac{z}{r}}d\xi dr
=-2\pi z^{n-1}\int_{|z|}^1\frac{1}{r^{n-1}}h_n(r)dr,\\
\int_\D\frac{h_n(r)e^{-in\theta}}{z-y}dA(y)&=&i\int_0^1h_n(r)\int_{\T} \frac{1}{\xi^{n+1}}\frac{1}{\xi-\frac{z}{r}}d\xi dr\\
&=&\frac{2\pi}{z^{n+1}}\int_0^{|z|}r^{n+1}h_n(r)dr
=\frac{2\pi \overline{z}^{n+1}}{|z|^{2(n+1)}}\int_0^{|z|}r^{n+1}h_n(r)dr.
\end{eqnarray*}

\medskip
\noindent
{\bf{(5)}} Let us differentiate with respect to $r$ having that
\begin{eqnarray*}
\partial_r \int_\D \log |re^{i\theta}-y|h(y)dA(y)\hspace{-0.2cm}&=& \hspace{-0.2cm}\textnormal{Re} \left[\frac{z}{r}\int_\D \frac{h(y)}{z-y}dA(y)\right]\\
&=&\textnormal{Re}\left[\frac{z}{r}\sum_{n\geq 1}\pi\left[-z^{n-1}\int_r^1 \frac{h_n(s)}{s^{n-1}}ds+\frac{1}{z^{n+1}}\int_0^rs^{n+1}h_n(s)ds\right]\right]\\
&=& \pi \sum_{n\geq 1} \cos(n\theta)\left[-r^{n-1}\int_r^1 \frac{h_n(s)}{s^{n-1}}ds+\frac{1}{r^{n+1}}\int_0^rs^{n+1}h_n(s)ds\right].
\end{eqnarray*}
This last integral was computed before by the Residue Theorem. Now, we realize that
\begin{eqnarray*}
-\partial_r\frac{1}{n}\left[r^{n}\int_r^1 \frac{1}{s^{n-1}}h_n(s)ds\right.\hspace{-0.2cm}&+&\hspace{-0.2cm}\left.\frac{1}{r^{n}}\int_0^rs^{n+1}h_n(s)ds\right]\\
&&=-r^{n-1}\int_r^1 \frac{1}{s^{n-1}}h_n(s)ds+\frac{1}{r^{n+1}}\int_0^rs^{n+1}h_n(s)ds.
\end{eqnarray*}
Then, we obtain 
$$\int_\D \log |z-y|h(y)dA(y)=-\pi\sum_{n\geq 1}\frac{1}{n}\left[r^n\int_r^1\frac{h_n(s)}{s^{n-1}}ds+\frac{1}{r^n}\int_0^rs^{n+1}h_n(s)ds\right]\cos(n\theta)+H(\theta),$$
where $H$ is a function that only depends on $\theta$. Taking $r=0$ we have that
$$
H(\theta)=\int_\D \log (|y|) h(y)dA(y)=0.
$$
The last is equal to zero due to the form of the function $h$:
$
h(re^{i\theta})=\sum_{n\geq 1} h_n(r)\cos(n\theta).
$

\medskip
\noindent
{\bf{(6)}}  This integral can be done by spliting it as follows
\begin{eqnarray*}
\int_\D\frac{k(z)-k(y)}{z-y}f_0(y)dA(y)&=&k(z)\int_\D\frac{f_0(y)}{z-y}dA(y)-\int_\D\frac{k(y)f_0(y)}{z-y}dA(y)\\
&=&\sum_{n\geq 1} A_n\left[z^{n+1}\int_\D\frac{f_0(y)}{z-y}dA(y)-\int_\D\frac{y^{n+1}f_0(y)}{z-y}dA(y)\right].
\end{eqnarray*}
Note that these integrals have be done before. Hence, we conclude using Integral (3) for the first one and \eqref{int} for the second one.

\medskip
\noindent
{\bf{(7)}} Similarly to Integral (5), we differentiate with respect to $r$
\begin{eqnarray*}
\partial_r \int_\D \log |re^{i\theta}\hspace{-0.3cm}&-&\hspace{-0.3cm} y|  f_0(y)\textnormal{Re}\left[k'(y)\right]dA(y)=\textnormal{Re} \left[\frac{z}{r}\int_\D \frac{f_0(y)\textnormal{Re}\left[k'(y)\right]}{z-y}dA(y)\right]\\
&=& \pi\sum_{n\geq 1}A_n(n+1)\textnormal{Re}\left[\frac{z}{r} \left[-z^{n-1}\int_{r}^1sf_0(s)ds+\frac{\overline{z}^{n+1}}{r^{2(n+1)}}\int_0^{r}s^{2n+1}f_0(s)ds \right]\right],\\
&=& \pi\sum_{n\geq 1}A_n(n+1) \cos(n\theta)\left[-r^{n-1}\int_r^1sf_0(s)ds+\frac{1}{r^{n+1}}\int_0^rs^{2n+1}f_0(s)ds\right],
\end{eqnarray*}
where we use Integral (2). With the same argument than in Integral (5) we realize that
\begin{eqnarray*}
-\partial_r\frac{1}{n}\left[r^{n}\int_r^1 sf_0(s)ds\right.\hspace{-0.2cm}&+&\hspace{-0.2cm}\left.\frac{1}{r^{n}}\int_0^rs^{2n+1}f_0(s)ds\right]\\
&&=-r^{n-1}\int_r^1sf_0(s)ds+\frac{1}{r^{n+1}}\int_0^rs^{2n+1}f_0(s)ds,
\end{eqnarray*}
and hence
\begin{eqnarray*}
\int_\D \log |re^{i\theta}-y|\hspace{-0.3cm}&\hspace{-0.5cm}& \hspace{-0.5cm}f_0(y)\textnormal{Re}\left[k'(y)\right]dA(y)\\
&&=-\pi\sum_{n\geq 1}A_n\frac{n+1}{n}\cos(n\theta)\left[r^n\int_{r}^1sf_0(s)ds+\frac{1}{r^n}\int_0^{r}s^{2n+1}f_0(s)ds\right]+H(\theta),
\end{eqnarray*}
where $H$ is a function that only depends on $\theta$.  Evaluating in $r=0$ as in Integral (5), we get that $H\equiv 0$,  obtaining the announced identity.

\medskip
\noindent
{\bf{(8)}} 
As in Integral (5) and (7) we differentiate with respect to $r$ having
\begin{eqnarray*}
\partial_r \int_\D \log |re^{i\theta}-y|f_0(y)dA(y)&=&\textnormal{Re}\left[\frac{z}{r}\int_\D \frac{f_0(y)}{z-y}dA(y)\right]
= 2\pi \frac{1}{r}\int_0^r sf_0(s)ds,
\end{eqnarray*}
where the last integral is done in Integral (3).
Hence,
$$
\int_\D \log |re^{i\theta}-y|f_0(y)dA(y)=2\pi \int_0^r \frac{1}{\tau}\int_0^\tau sf_0(s)dsd\tau +H(\theta),$$
where $H$ is a function that only depends on $\theta$. Evaluating in $r=0$ we get that
$$H(\theta)=\int_0^1\int_0^{2\pi} s\log s f_0(s)dsd\theta=-2\pi \int_0^1 \frac{1}{\tau}\int_0^\tau sf_0(s)dsd\tau,
$$
concluding the proof.
\end{proof}

\section{Gauss Hypergeometric function}\label{SecSpecialfunctions}

We give a short  introduction to the Gauss hypergeometric functions and discuss some  of their basic properties. The formulae listed below were crucial in the computations of the linearized operator associated to the V-states equation and the analysis of its spectral study.  Recall that for any real numbers $a,b\in \mathbb{R},\, c\in \mathbb{R}\backslash(-\mathbb{N})$ the hypergeometric function $z\mapsto F(a,b;c;z)$ is defined on the open unit disc $\mathbb{D}$ by the power series
\begin{equation}\label{GaussF}
F(a,b;c;z)=\sum_{n=0}^{\infty}\frac{(a)_n(b)_n}{(c)_n}\frac{z^n}{n!}, \quad \forall z\in \mathbb{D}.
\end{equation}
The  Pochhammer symbol $(x)_n$ is defined by
$$
(x)_n = \begin{cases}   1,   & n = 0, \\
 x(x+1) \cdots (x+n-1), & n \geq1,
\end{cases}
$$
and verifies
\begin{equation*}
(x)_n=x\,(1+x)_{n-1},\quad (x)_{n+1}=(x+n)\,(x)_n.
\end{equation*}
The series converges absolutely for all values of $|z|<1.$ For $|z|=1$ we have that it converges absolutely if $\textnormal{Re} (a+b-c)<0$ and it diverges if $1\leq \textnormal{Re}(a+b-c)$. See \cite{Erdelyi} for more details.

We recall the integral representation of the hypergeometric function, see for instance  \cite[p. 47]{Rainville}. Assume that  $ \textnormal{Re}(c) > \textnormal{Re}(b) > 0,$ then we have
\begin{eqnarray}\label{integ}
\hspace{1cm}F(a,b;c;z)=\frac{\Gamma(c)}{\Gamma(b)\Gamma(c-b)}\int_0^1 x^{b-1} (1-x)^{c-b-1}(1-zx)^{-a}~ dx,\quad \forall{z\in \C\backslash[1,+\infty)}.
\end{eqnarray}
Notice that this representation shows that the hypergeometric function initially defined in the unit disc  admits an analytic continuation to the complex plane cut along  $[1,+\infty)$. Another useful identity is the following:  

\begin{equation}\label{Line-HGF}
{F(a,b;c;z)=(1-z)^{-a}F\left(a,c-b;c;\frac{z}{z-1}\right),\quad \forall\, |\textnormal{arg}(1-z)|<\pi},
\end{equation}
for $\Rea c>\Rea b>0$. \\
The function $\Gamma: \C\backslash\{-\N\} \to \C$ refers to the gamma function, which is the analytic continuation to the negative half plane of the usual gamma function defined on the positive \mbox{half-plane $\{\Rea\, z > 0\}$}. It is defined by the integral representation
$$
\Gamma(z)=\int_0^{+\infty}\tau^{z-1}e^{-\tau} d\tau,
$$
and satisfies the relation
$
\Gamma(z+1)=z\,\Gamma(z), \ \forall z\in \C \backslash(-\N).
$
From this we deduce the identities
\begin{equation*}
(x)_n=\frac{\Gamma(x+n)}{\Gamma(x)},\quad (x)_n=(-1)^n\frac{\Gamma(1-x)}{\Gamma(1-x-n)},
\end{equation*}
provided that all the quantities in the right terms are well-defined. \\
We can differentiate the hypergeometric function  obtaining
\begin{equation}\label{Diff41}
\frac{d^kF(a,b;c;z)}{dz^k}=\frac{(a)_k(b)_k}{(c)_k}F(a+k,b+k;c+k;z),
\end{equation}
for $k\in\N$. Depending on the parameters, the hypergeometric function behaves differently at $1$. When {$\Rea c>\Rea b>0$ and $\Rea (c-a-b)>0 $}, it can be shown that it is absolutely convergent on the closed unit disc  and one finds the expression
\begin{equation}\label{id1}
F (a,b;c;1)= \frac{\Gamma(c)\Gamma(c-a-b)}{\Gamma(c-a)\Gamma(c-b)} ,
\end{equation}
whose  proof can be found in \cite[Pag. 49]{Rainville}. {However, in the case $a+b=c$, the hypergeometric function exhibits a logarithmic singularity as follows
 \begin{align}\label{log1}
\lim_{z\rightarrow 1-}\frac{F(a,b;c;z)}{-\ln(1-z)}=\frac{\Gamma(a+b)}{\Gamma(a)\Gamma(b)},
\end{align}
see for instance \cite{Andrews} for more details.} Next we recall some Kummer's quadratic transformations of the hypergeometric series, see \cite{Rainville},
\begin{eqnarray}
 c\,F (a,b;c;z)-(c-a)\,F (a,b;c+1;z)-a\,F (a+1,b;c+1;z)&=&0, \label{f4}\\
 (b-c)F(a,b-1;c;z)+(c-a-b)F(a,b;c;z)-a(z-1)F(a+1,b;c;z)&=&0,\label{f5}\\
\frac{(2c-a-b+1)z-c}{c}F(a,b;c+1;z)+\frac{(a-c-1)(c-b+1)z}{c(c+1)} F(a,b;c+2;z)&& \nonumber \\
\hspace{-7,7cm} && \hspace{-7,7cm} = F(a,b;c;z)(z-1).\label{Kumt}
\end{eqnarray}
Other formulas  which have been used in the preceding sections  are
\begin{eqnarray}\label{FormInt1}
\left\{
\begin{array}{lll}
\displaystyle{\int_0^{1}} F(a,b;c;\tau z)\tau^{c-1} d\tau=\frac{1}{c}F(a,b;c+1;z),\\
\displaystyle{ \int_0^{1} }F(a,b;c;\tau z)\tau^{c-1}(1-\tau) d\tau=\frac{1}{c(c+1)}F(a,b;c+2;z).
\end{array} \right.
\end{eqnarray}
The last point  that we wish to recall concerns the differential equation governing   the hypergeometric equation, which   is given by,
\begin{eqnarray}\label{ODEEQ}
z(1-z)F(z)''+(c-(a+b+1)z)F(z)'-abF(z)=0,
\end{eqnarray}
with  $a, b,\in\R$ and  $c\in\R\backslash(-\N)$ given. One of the two independent solutions of the last differential equation around $z=0$ is the hypergeometric function: $F(a,b;c;z)$. It remains to identify the second independent solution. If none of $c,c-a-b,a-b$ is an integer, then the other independent solution around the singularity $z=0$ is
\begin{eqnarray}\label{secondsol}
z^{1-c}F(a-c+1,b-c+1;2-c;z).
\end{eqnarray}
We will be interested in the critical case when $c$ is a negative integer. In this case, the hypergeometric differential equation has as a smooth solution given by \eqref{secondsol}. However, the second independent solution is singular and contains a logarithmic singularity, see \cite[p. 55]{Rainville} for more details. {Real solutions around $+\infty$ may be also obtained as it is done in \cite{Andrews}. In fact, the two independent solutions are given by
$$
z^{-a}F\left(a,a+1-c;a+1-b;\frac{1}{z}\right) \quad \textnormal{and}\quad z^{-b}F\left(b,b+1-c;b+1-a;\frac{1}{z}\right).
$$}


\begin{thebibliography}{99}\frenchspacing

\bibitem{Andrews} G. R. Andrews, R. Askey, R. Roy, {\it Special Functions.} Cambridge University Press, 1999.

\bibitem{Arn} V. I. Arnold, {\it Hamiltonian nature of the Euler equations in the dynamics of a rigid body and of an ideal fluid.} In: Givental A., Khesin B., Varchenko A., Vassiliev V., Viro O. (eds) Vladimir I. Arnold - Collected Works. Vladimir I. Arnold - Collected Works, vol 2. Springer, Berlin, Heidelberg, 1969.

\bibitem{BCh} H. Bahouri, J.-Y. Chemin, {\it \'Equations de transport relatives \`a des champs de vecteurs non-lipschitziens et m\'ecanique des fluides. } Arch. Ration. Mech. Anal. {\bf 127 }(2) (1994), 159--181.

\bibitem{Erdelyi}{H. Bateman, } {\it Higher Transcendental Functions Vol. I--III.} McGraw-Hill Book Company, New York, 1953.

\bibitem{Bedro} J. Bedrossian, M. Coti Zelati, V. Vicol, {\it Vortex axisymmetrization, inviscid damping, and vorticity depletion in the linearized 2D Euler equations},  arXiv:1711.03668, 2017.

\bibitem{B-C} A. L. Bertozzi, P. Constantin, {\it Global regularity for vortex patches.} Comm. Math. Phys.
{\bf 152}(1) (1993), 19--28.

\bibitem{Burbea} { J. Burbea,} {\it Motions of vortex patches.} Lett. Math. Phys. {\bf 6} (1982), 1--16.

\bibitem{Burton} G. R. Burton, {\it Steady symmetric vortex pairs and rearrangements.} Proc. Roy. Soc. Edinburgh Sect. A. {\bf 108} (1988) 269--290.

\bibitem{CS} J. A. Carrillo, J. Soler, {\it On the evolution of an angle in a vortex patch. } J. Nonlinear Sci. {\bf 10} (2000), 23--47.

 \bibitem{Cas0-Cor0-Gom} {A. Castro, D. C\'ordoba, J. G\'omez-Serrano, }{\it Existence and regularity of rotating global solutions for the generalized surface quasi-geostrophic equations.} Duke Math. J. {\bf 165}(5) (2016), 935--984.
 
\bibitem{Cas-Cor-Gom} {A. Castro, D. C\'ordoba, J. G\'omez-Serrano, }{\it  Uniformly rotating analytic global patch solutions for active scalars}. J. Ann. PDE {\bf 2}(1) (2016),  Art. 1, 34.

\bibitem{CastroCordobaGomezSerrano} {A. Castro, D. C\'ordoba, J. G\'omez-Serrano, } {{\it Uniformly rotating smooth solutions for the incompressible 2D Euler equations,}} arXiv:1612.08964, 2016.

\bibitem{C-C-GS-2} A. Castro, D. C\'ordoba,  J. G\'omez-Serrano, {\it Global smooth solutions for the inviscid SQG equation,} arXiv:1603.03325, 2016.

\bibitem{Chemin} {J.-Y. Chemin,} {{\it Persistance de structures g\'eometriques dans les  fluides incompressibles bidimensionnels. }} Ann. Sci. Ec. Norm. Sup. {\bf 26} (1993), 1--26.

\bibitem{CSv} A. Choffrut, V. \v{S}ver\'ak, {\it Local structure of the set of steady-state solutions to the 2D incompressible Euler equations. } Geom. Funct. Anal.  {\bf 22}(1) (2012),  136--201.

\bibitem{CSz} A. Choffrut, L. Szk\'elyhidi Jr., {\it Weak solutions to the stationary incompressible Euler equations.} SIAM J. Math. Anal. {\bf 46}(6) (2014), 4060--4074. 

\bibitem{CrandallRabinowitz} {M. G. Crandall, P. H. Rabinowitz, } {\it Bifurcation from simple eigenvalues.} J. Funt. Anal. {\bf 8} (1971), 321--340.


\bibitem{Cruz-M-O}  V. Cruz, J. Mateu, J.  Orobitg, {\it Beltrami equation with coefficient  in
Sobolev and Besov spaces.} Canad. J. Math. {\bf 65}(1) (2013), 1217--1235.

\bibitem{Cruz-Tol} V. Cruz, X. Tolsa, {\it Smoothness of the Beurling transform in Lipschitz
domains.} J. Funct. Anal. {\bf 262}(10) (2012), 4423--4457.

\bibitem{DeemZabusky} {G. S. Deem, N. J. Zabusky, } {\it Vortex waves: Stationary ``V-states'', Interactions, Recurrence, and Breaking.} Phys. Rev. Lett. {\bf 40} (1978), 859--862.

 \bibitem{D-H-R} {D. G. Dritschel, T. Hmidi, C. Renault, } {\it
Imperfect bifurcation for the quasi-geostrophic shallow-water equations}, arXiv:1801.02092, 2018.

\bibitem{DelaHozHmidiMateuVerdera} {F. De la Hoz, T. Hmidi, J. Mateu, J. Verdera, } {\it Doubly connected V-states for the planar Euler equations. } SIAM J. Math. Anal. {\bf 48} (2016), 1892--1928.
 
 
 \bibitem{DHHM} {F. De la Hoz, Z. Hassainia, T. Hmidi, J. Mateu,} {\it An analytical and numerical study of steady patches in the disc.}  Anal. PDE {\bf 9}(7) (2016), 1609--1670.
 
\bibitem{LSz} C. De Lellis, L. Szk\'elyhidi Jr, {\it High dimensionality and h-principle in PDE. } Bull. Amer. Math. Soc. {\bf 54} (2017), 247--282.

\bibitem{Elg} T. M. Elgindi, I. J. Jeong, {\it Symmetries and Critical Phenomena in Fluids,} 
arXiv:1610.09701, 2016.

\bibitem{EncisoPoyatoSoler}{A. Enciso, D. Poyato, J. Soler, }{\it Stability Results, Almost Global Generalized Beltrami Fields and Applications to Vortex Structures in the Euler Equations. } Commun. Math. Phys. {\bf 360} (2018), 197--269.

\bibitem{Gra-Smets}  P. Gravejat, D. Smets,
  {\it Smooth traveling-wave solutions to the inviscid surface quasi-geostrophic equation.} International Mathematics Research Notices, page rnx177, 2017.
  
 \bibitem{Hartmann} P. Hartman, {\it Ordinary Differential Equations.} John Wiley \& Sons, Inc., New York, London, Sydney, 1964. 
  
  \bibitem{Hassa-Hmi}  Z. Hassainia, T. Hmidi, {\it  On the V-states for the generalized quasi-geostrophic equations.} Comm. Math. Phys. {\bf 337}(1) (2015), 321--377.

\bibitem{HMW} Z. Hassainia, N. Masmoudi, M. H. Wheeler, {\it Global bifurcation of rotating vortex patches},  	arXiv:1712.03085, 2017.

\bibitem{Helms}{L. L. Helms, }{\it Potential theory, } Springer-Verlag London, 2009.

\bibitem{HmidiHozMateuVerdera} {T. Hmidi, F. De la Hoz, J. Mateu, J. Verdera, }{\it  Doubly connected V-states for the planar Euler equations. } SIAM J. Math. Anal. {\bf 48}(3), 1892--1928.

\bibitem{HmidiMateu}{T. Hmidi, J. Mateu, } {\it Bifurcation of rotating patches from Kirchhoff vortices.} Discret. Contin. Dyn. Syst. {\bf 36} (2016), 5401--5422.
 
 \bibitem{H-M} T. Hmidi, J. Mateu,  {\it Existence of corotating and counter-rotating vortex pairs for active
scalar equations.} Comm. Math. Phys. {\bf 350}(2) (2017), 699--747.

\bibitem{HmidiMateuVerdera} {T. Hmidi, J. Mateu, J. Verdera, } {\it Boundary regularity of rotating vortex patches. } Arch. Ration. Mech. Anal {\bf 209} (2013), 171--208. 

 \bibitem{Renault-Hmidi} {T. Hmidi, C. Renault, }{\it  Existence of small loops in a bifurcation diagram near degenerate eigenvalues. }Nonlinearity {\bf 30}(10) (2017), 3821--3852. 
 
\bibitem{Kato}{ T. Kato, }{\it Perturbation Theory for Linear Operators. }  Springer-Verlag, Berlin-Heidelberg-New York, 1995.

\bibitem{Kielhofer}{ H. Kielh\"ofer, }{\it Bifurcation Theory: An Introduction with Applications to PDEs. }  Springer-Verlag, Berlin-Heidelberg-New York, 2004.

\bibitem{K-S} A. Kiselev, V. \v{S}ver\'ak, {\it Small scale creation for solutions of the incompressible two--dimensional
Euler equation. } Ann. of Math. {\bf 180}(3) (2014), 1205--1220.

\bibitem{Kress}{R. Kress, } {\it Linear Integral Equations. } Springer--New York--Heidelberg Dordrecht--London, 2014.

\bibitem{LiebLoss}{E. Lieb, M. Loss, } {\it Analysis. } American Mathematical Society, 1997.

\bibitem{LSh} X. Luo, R. Shvydkoy, {\it 2D homogeneous solutions to the Euler equation. } Comm. Part.
Differ. Equat. {\bf 40}(9) (2015),1666--1687.

\bibitem{MajdaBertozzi}{A. Majda, A. Bertozzi, }{\it Vorticity and Incompressible Flow.} Cambridge University Press, 2002.

\bibitem{MateuOrobitgVerdera}{J. Mateu, J. Orobitg, J. Verdera, }{\it Extra cancellation of even Calder\'on--Zygmund operators and quasiconformal mappings. } J. Math. Pures Appl. {\bf 91}(4) (2009), 402-431.

\bibitem{Meza} M. V. Melander, N. J. Zabusky, A. S. Styczek, {\it A moment model for vortex interactions of the two-dimensional Euler equations. Part 1. Computational validation of a Hamiltonian elliptical representation.} J. Fluid Mech. {\bf 167} (1986), 95--115.

\bibitem{Miranda2}{C. Miranda, }{\it Partial Differential Equations of Elliptic Type. } Springer, Berlin, 1970.

\bibitem{Miranda}{C. Miranda, }{\it Sulle proprieta di regolarita di certe trasformazioni integrali. }Mem. Acc. Lincei {\bf 7} (1965), 303--336.

\bibitem{Mo} P. J. Morrison, {\it Hamiltonian description of the ideal fluid. } Rev. Mod. Phys. {\bf 70} (1998), 467--521. 

\bibitem{Na} N. Nadirashvili, {\it On stationary solutions of two--dimensional Euler equation.} Arch. Ration. Mech.
Anal. {\bf 209}(3) (2013), 729--745.

\bibitem{Ol} P. J. Olver, {\it A Nonlinear Hamiltonian Structure for the Euler Equations. } J. Math.  
Anal. and Applications {\bf 89} (1982), 233--250.

\bibitem{O} E.A. Overman II,  {\it Steady-state solutions of the Euler equations in two dimensions. II. Local analysis of limiting V-states.} SIAM J. Appl. Math. {\bf 46}(5) (1986), 765--800.

\bibitem{Peloso}{M. M. Peloso, }{{\it Classical spaces of holomorphic functions. }} Technical report, Universit\'a di Milano, 2011.

\bibitem{Pommerenke} {Ch. Pommerenke, } {\it Boundary Behaviour of Conformal Maps.} Springer, Berlin, 1992.


\bibitem{Rainville}{E. D. Rainville, }{\it Special Functions. } The Macmillan Co., 1973.

\bibitem{Rudin}{W. Rudin, } {\it Real and Complex Analysis, 3rd ed. } McGraw-Hill, 1987. 

\bibitem{RubelShieldsTaylor}{L. A. Rubel, A. L. Shields, B. A. Taylor, }{\it Mergelyan Sets and the Modulus of Continuity of Analytic Functions. } J. Approx. Theory {\bf 15} (1975), 23--40.

\bibitem{Serf} P. Serfati, {\it Une preuve directe d'existence globale des vortex patches 2D.}
C. R. Acad. Sci. Paris S\'er. I Math. {\bf 318}(6) (1994), 515--518.

\bibitem{S-U} P. Stefanov, G. Uhlmann, {\it Linearizing non-linear inverse problems and an application to inverse backscattering.} J. Funct. Anal. {\bf 256} (2009) 2842--2866.

\bibitem{TricomiErdelyi}{ F. G. Tricomi, A. Erd\'elyi, } {\it The asymptotic expansion of a ratio of gamma functions. } Pacific J. Math.{\bf  1} (1951), 133--142.
 
 \bibitem{Yau} H.T.Yau, {\it Derivation of the Euler Equation from Hamiltonian Systems with Negligible Random Noise.} In: Schmüdgen K. (eds) Mathematical Physics X. Springer, Berlin, Heidelberg,  1992.
 
\bibitem{Yudovich}{Y. Yudovich, }{\it Nonstationary flow of an ideal incompressible liquid. } Zh. Vych. Mat. {\bf 3} (1963), 1032--1066.

\bibitem{Wittmann}{R. Wittmann, }{\it Application of a Theorem of M.G. Krein to singular integrals. } Trans. Amer.
Math. Soc. {\bf 299}(2) (1987), 581--599.	


\end{thebibliography}
\end{document}